\documentclass{amsart}

\usepackage[T1]{fontenc}
\usepackage[latin1]{inputenc}
\usepackage{url}
\usepackage{pstricks}
\usepackage{dsfont}
\usepackage{graphicx}
\usepackage{enumitem}
\usepackage{upgreek}
\usepackage{bbm}
\usepackage{supertabular}
\usepackage{marvosym}
\usepackage[colorlinks=true, linktoc=page]{hyperref}

\numberwithin{equation}{section}

\newcommand{\ie}{{\em i.e.}~}

\newcommand{\dd}{\mathrm{d}}
\newcommand{\R}{\mathbb{R}}

\newcommand{\ind}[1]{\mathds{1}_{\{#1\}}}
\newcommand{\sgn}{\mathrm{sgn}}
\newcommand{\uto}{\uparrow}
\newcommand{\dto}{\downarrow}
\newcommand{\Diff}{\mathrm{D}}

\newcommand{\ui}{\underline{i}} 
\newcommand{\oi}{\overline{i}} 
\newcommand{\uk}{\underline{k}} 
\newcommand{\ok}{\overline{k}} 
\newcommand{\uv}{\underline{v}} 
\newcommand{\ov}{\overline{v}} 
\newcommand{\ulambda}{\underline{\lambda}} 
\newcommand{\olambda}{\overline{\lambda}} 
\newcommand{\uu}{\underline{u}} 
\newcommand{\ou}{\overline{u}} 

\newcommand{\bu}{\mathbf{u}}
\newcommand{\bv}{\mathbf{v}}
\newcommand{\blambda}{\boldsymbol{\lambda}}
\newcommand{\bvarphi}{\boldsymbol{\varphi}}
\newcommand{\bm}{\mathbf{m}}
\newcommand{\bS}{\mathbf{S}}

\newcommand{\bX}{\mathbf{X}}

\newcommand{\bbX}{\mathbb{X}}
\newcommand{\bbv}{\mathbbm{v}}

\newcommand{\bU}{\mathbf{U}}

\newcommand{\rx}{\mathrm{x}}
\newcommand{\ry}{\mathrm{y}}
\newcommand{\rblambda}{\bar{\lambda}}
\newcommand{\rbmu}{\bar{\mu}}

\newcommand{\Ls}{\mathrm{L}} 
\newcommand{\Cs}{\mathrm{C}} 
\newcommand{\Ps}{\mathrm{P}} 
\newcommand{\Ws}{\mathrm{W}} 
\newcommand{\Ms}{\mathbf{M}} 
\newcommand{\dP}{\mathcal{P}} 

\newcommand{\Unif}{\mathrm{U}} 
\newcommand{\TV}{\mathrm{TV}} 

\newcommand{\ConstUSH}{L_{\mathrm{USH}}} 
\newcommand{\ConstBound}[1]{L_{\mathrm{C},#1}} 
\newcommand{\ConstLip}{L_{\mathrm{LC}}} 
\newcommand{\ConstStab}{\mathcal{L}} 
\newcommand{\Ratio}{\Theta} 
\newcommand{\ConstRar}{\rho} 

\newcommand{\x}{\mathbf{x}}
\newcommand{\y}{\mathbf{y}}
\newcommand{\z}{\mathbf{z}}

\newcommand{\Part}{{P_n^d}} 
\newcommand{\Dn}{D_n} 
\newcommand{\Dnd}{\Dn^d} 
\newcommand{\Dndeux}{D_{2n}} 
\newcommand{\Dndeuxd}{\Dndeux^d} 
\newcommand{\Drnd}{\mathcal{D}} 
\newcommand{\barB}{\bar{B}} 

\newcommand{\Rb}{\mathrm{R}} 
\newcommand{\Nb}{\mathrm{N}} 
\newcommand{\Mb}{\mathrm{M}} 
\newcommand{\Ibinter}{\mathrm{I}^{\mathrm{coll}}} 
\newcommand{\Ibself}{\mathrm{I}^{\mathrm{self}}} 
\newcommand{\Classe}{\mathrm{C}} 
\newcommand{\classe}{\mathfrak{c}} 
\newcommand{\Tbmax}{\mathrm{T}^{\mathrm{max}}} 
\newcommand{\Tmax}{T^{\mathrm{max}}} 

\newcommand{\Good}{\mathcal{G}} 

\newcommand{\clu}{\mathrm{clu}} 
\newcommand{\type}{\mathrm{type}} 

\newcommand{\tPhi}{\tilde{\Phi}} 

\newcommand{\ttinter}{\tilde{\tau}^{\mathrm{coll}}} 
\newcommand{\tinter}{\tau^{\mathrm{coll}}} 
\newcommand{\Xiinter}{\Xi^{\mathrm{coll}}} 
\newcommand{\xiinter}{\xi^{\mathrm{coll}}} 
\newcommand{\Ttau}{\mathcal{T}} 

\newcommand{\dxi}{\delta_{\xi}} 
\newcommand{\dtau}{\delta_{\tau}} 

\newcommand{\barlambda}{\bar{\lambda}}
\newcommand{\bblambda}{\bar{\boldsymbol{\lambda}}} 
\newcommand{\barmu}{\bar{\mu}}
\newcommand{\bbmu}{\bar{\boldsymbol{\mu}}} 
\newcommand{\tlambda}{\tilde{\lambda}} 
\newcommand{\tblambda}{\tilde{\boldsymbol{\lambda}}} 
\newcommand{\bmu}{\bar{\mu}} 
\newcommand{\totmass}{\mathcal{E}} 
\newcommand{\Mnu}{\mathcal{M}} 
\newcommand{\trans}{\tau} 

\newcommand{\lto}[1]{\stackrel{#1}{\to}} 

\renewcommand{\bar}[1]{\overline{#1}}

\newcommand{\sk}{\vskip 3mm}

\newtheorem{defi}{Definition}[subsection]
\newtheorem{lem}[defi]{Lemma}
\newtheorem{prop}[defi]{Proposition}
\newtheorem{theo}[defi]{Theorem}
\newtheorem{cor}[defi]{Corollary}

\newtheoremstyle{myremark}{}{}{}{0pt}{\bfseries}{.}{ }{}
\theoremstyle{myremark}
\newtheorem{rk}[defi]{Remark}

\title[Wasserstein stable semigroups solving diagonal hyperbolic systems with large data]{A multitype sticky particle construction of Wasserstein stable semigroups solving one-dimensional diagonal hyperbolic systems with large monotonic data}

\author{Benjamin Jourdain}
\address{{\bf Benjamin Jourdain}\newline
{\rm \indent Université Paris-Est, CERMICS (ENPC), INRIA, F-77455 Marne-la-Vallée.}}
\email{\href{mailto:jourdain@cermics.enpc.fr}{jourdain@cermics.enpc.fr}}

\author{Julien Reygner}
\address{{\bf Julien Reygner}\newline
{\rm \indent Sorbonne Universités, UPMC Univ Paris 06, UMR 7599, LPMA, F-75005 Paris.\newline
\indent Université Paris-Est, CERMICS (ENPC), F-77455 Marne-la-Vallée.}}
\curraddr{{\rm Laboratoire de Physique, \'Ecole Normale Sup\'erieure de Lyon, 46 all\'ee d'Italie, F-69364 Lyon.}}
\email{\href{mailto:julien.reygner@polytechnique.org}{julien.reygner@polytechnique.org}}

\thanks{This research benefited from the support of the French National Research Agency (ANR) under the program ANR-12-BLAN Stab.}

\makeatletter

\def\part{\@startsection{part}{1}%
\z@{.7\linespacing\@plus\linespacing}{.5\linespacing}%
{\large\normalfont\bfseries}}

\def\section{\@startsection{section}{1}%
\z@{.7\linespacing\@plus\linespacing}{.5\linespacing}%
{\normalfont\bfseries\centering}}

\def\@settitle{\begin{center}%
  \baselineskip14\p@\relax
    \bfseries
    \LARGE\@title
  \end{center}%
}

\def\@setauthors{%
  \begingroup
  \trivlist
  \centering\footnotesize \@topsep30\p@\relax
  \advance\@topsep by -\baselineskip
  \item\relax
  \andify\authors
  \def\\{\protect\linebreak}%
 {\Large\authors}%
  \endtrivlist
  \endgroup
}

\def\maketitle{\par
  \@topnum\z@ 
  \@setcopyright
  \thispagestyle{firstpage}
  \ifx\@empty\shortauthors \let\shortauthors\shorttitle
  \else \andify\shortauthors
  \fi
  \@maketitle@hook
  \begingroup
  \@maketitle
  \toks@\@xp{\shortauthors}\@temptokena\@xp{\shorttitle}%
  \toks4{\def\\{ \ignorespaces}}
  \edef\@tempa{%
    \@nx\markboth{\the\toks4
      \@nx{\the\toks@}}{\the\@temptokena}}%
  \@tempa
  \endgroup
  \c@footnote\z@
  \def\do##1{\let##1\relax}%
  \do\maketitle \do\@maketitle \do\title \do\@xtitle \do\@title
  \do\author \do\@xauthor \do\address \do\@xaddress
  \do\email \do\@xemail \do\curraddr \do\@xcurraddr
  \do\commby \do\@commby
  \do\dedicatory \do\@dedicatory \do\thanks \do\thankses
  \do\keywords \do\@keywords \do\subjclass \do\@subjclass
}

\makeatother

\begin{document}

\begin{abstract}
  This article is dedicated to the study of diagonal hyperbolic systems in one space dimension, with cumulative distribution functions, or more generally nonconstant monotonic bounded functions, as initial data. Under a uniform strict hyperbolicity assumption on the characteristic fields, we construct a multitype version of the sticky particle dynamics and obtain existence of global weak solutions by compactness. 
  
  We then derive a $\Ls^p$ stability estimate on the particle system uniform in the number of particles. This allows to construct nonlinear semigroups solving the system in the sense of Bianchini and Bressan [Ann. of Math. (2), 2005]. We also obtain that these semigroup solutions satisfy a stability estimate in Wasserstein distances of all orders, which encompasses the classical $\Ls^1$ estimate and generalises to diagonal systems the results by Bolley, Brenier and Loeper [J. Hyperbolic Differ. Equ., 2005] in the scalar case. 
  
  Our results are obtained without any smallness assumption on the variation of the data, and only require the characteristic fields to be Lipschitz continuous and the system to be uniformly strictly hyperbolic.
\end{abstract}

\maketitle

\setcounter{tocdepth}{2}
\tableofcontents



\section{Introduction}\label{s:intro}


\subsection{Hyperbolic systems}\label{ss:hypsys} A one-dimensional system of conservation laws is a differential equation of the form
\begin{equation}\label{eq:syst:cl}
  \partial_t \bu + \partial_x (f(\bu)) = 0, \qquad t \geq 0, \quad x \in \R,
\end{equation}
where $\bu = (u^1, \ldots, u^d) : [0,+\infty) \times \R \to \R^d$ is the vector of {\em conserved quantities}, and $f : \R^d \to \R^d$ is the {\em flux} function. When both $f$ and $\bu$ are smooth, it rewrites in the nonconservative form
\begin{equation}\label{eq:syst:nc}
  \partial_t \bu + A(\bu) \partial_x \bu = 0,
\end{equation}
where $A(\bu) = \Diff f (\bu)$ is the Jacobian matrix of the flux function. If, for all $\bu$, the matrix $A(\bu)$ is diagonalisable and has real eigenvalues $\lambda^1(\bu) \geq \lambda^2(\bu) \geq \cdots \geq \lambda^d(\bu)$, the system is called {\em hyperbolic} and the functions $\lambda^1, \ldots, \lambda^d$ are its {\em characteristic fields}. Hyperbolic systems naturally arise in continuum physics~\cite{daf10} and are the object of an intense mathematical research~\cite{serre, serre2, bressan, raviart}.

A system of the form~\eqref{eq:syst:cl} or~\eqref{eq:syst:nc} is strictly hyperbolic if $\lambda^1(\bu) > \lambda^2(\bu) > \cdots > \lambda^d(\bu)$ for all $\bu$. Global weak existence results for the strictly hyperbolic one-dimensional Cauchy problem
\begin{equation}\label{eq:syst:clCau}
  \left\{\begin{aligned}
    & \partial_t \bu + \partial_x (f(\bu)) = 0,\\
    & \bu(0,x) = \bu_0(x),
  \end{aligned}\right.
\end{equation}
go back to Glimm~\cite{glimm65}, under the assumption previously introduced by Lax~\cite{lax57} that the characteristic fields $\lambda^1, \ldots, \lambda^d$ be either genuinely nonlinear, or linearly degenerate. Under the same assumption, an alternative method to construct global weak solutions to the Cauchy problem~\eqref{eq:syst:clCau} is the Front Tracking approximation, which was introduced by Dafermos~\cite{daf72} in the scalar case $d=1$ and then extended to systems of conservation laws by DiPerna~\cite{diperna76}, see also~\cite{bressan92, risebro93, baijen98}. A version of this method that does not refer to any genuine nonlinearity nor linear degenerescence assumption on the characteristic fields was later introduced by Ancona and Marson~\cite{ancona}. Both the Glimm scheme and the Front Tracking method provide existence for initial data $\bu_0 = (u^1_0, \ldots, u^d_0)$ belonging to the class of functions with bounded variation (BV), and having a small total variation. On the other hand, the vanishing viscosity approach~\cite{bianbres00, bianbres} provides $\Ls^1$ stable semigroups defined on a set of BV functions containing functions with sufficiently small total variation, that yield weak solutions to the system~\eqref{eq:syst:nc}. The convergence of the vanishing viscosity approach, as well as the uniqueness of $\Ls^1$ stable semigroup solutions to~\eqref{eq:syst:nc}, were proved by Bianchini and Bressan~\cite{bianbres}. The Bianchini-Bressan solution was also proven to be the limit of Glimm and Front Tracking approximations~\cite{bianbres, ancona}.

Outside of the BV setting, the theory of systems of conservation laws with $\Ls^{\infty}$ initial data was developed by DiPerna~\cite{diperna83}. By compensated compactness, under weak structural conditions, it was first proved that systems of $d=2$ equations in conservative form admit global entropy solutions for $\Ls^{\infty}$ initial data. Uniqueness for such systems starting from initial data with large variation was obtained by Bressan and Colombo~\cite{brecol95:uniq} under a stability assumption on the flux function. For $d \geq 3$ equations, unless the system is in the Temple class~\cite{baibre,bianch00} or has coinciding shocks and rarefaction curves~\cite{bianch01}, no existence, uniqueness nor stability theory is available without a smallness assumption on the variation of the initial data. 


\subsection{Diagonal systems}\label{ss:diagsys} For a strictly hyperbolic system of the form~\eqref{eq:syst:nc}, let $l^1(\bu), \ldots, l^d(\bu)$ and $r^1(\bu), \ldots, r^d(\bu)$ refer to the respective left- and right-eigenvectors of the matrix $A(\bu)$. Following~\cite{serrerich}, the system~\eqref{eq:syst:nc} is {\em diagonalisable} if and only if the Frobenius condition 
\begin{equation*}
  \forall \gamma, \gamma', \gamma'' \in \{1, \ldots, d\} \quad\text{with $\gamma', \gamma'' \not= \gamma$}, \qquad l^{\gamma}\cdot\{r^{\gamma'}, r^{\gamma''}\} = 0,
\end{equation*}
is satisfied, where $\{r,r'\}=\Diff r r' - \Diff r' r$ refers to the Poisson bracket. Up to a change of variable, the system then reduces to the diagonal form
\begin{equation}\label{eq:syst:diag}
  \forall \gamma \in \{1, \ldots, d\}, \qquad \partial_t u^{\gamma} + \lambda^{\gamma}(\bu) \partial_x u^{\gamma} = 0.
\end{equation}
According to~\cite[Theorem~12.1.1]{serre2}, the diagonal system \eqref{eq:syst:diag}, when strictly hyperbolic, admits a conservative form 
\begin{equation*}
  \partial_t (g(\bu))+\partial_x (h(\bu))=0
\end{equation*}
if and only if, for all $\gamma,\gamma',\gamma''\in\{1,\ldots,d\}$ distinct, 
\begin{equation*}
  \forall \bu \in [0,1]^d, \qquad \partial_{u^{\gamma''}}\left(\frac{\partial_{u^\gamma}\lambda^{\gamma'}(\bu)}{\lambda^\gamma(\bu)-\lambda^{\gamma'}(\bu)}\right) =\partial_{u^{\gamma}}\left(\frac{\partial_{u^{\gamma''}}\lambda^{\gamma'}(\bu)}{\lambda^{\gamma''}(\bu)-\lambda^{\gamma'}(\bu)}\right).
\end{equation*}
The system is then called a {\em rich} system. Any diagonal strictly hyperbolic system of $d=2$ equations is clearly rich. On the other hand, any strictly hyperbolic system in conservative form $\partial_t \bv+\partial_x(f(\bv))=0$ composed of $d=2$ equations may be diagonalised by choosing $u^1(\bv)$ and $u^2(\bv)$ two Riemann invariants respectively associated with the first and second fields of eigenvectors of the Jacobian matrix $\Diff f(\bv)$.

This article is dedicated to the study of the Cauchy problem for the diagonal system~\eqref{eq:syst:diag} where, for all $\gamma \in \{1, \ldots, d\}$, $u^{\gamma}_0$ is a nonconstant, monotonic and bounded function on $\R$. Such initial data can be interpreted as cumulative distribution functions of bounded measures of constant sign, and up to rescaling, there is no loss of generality in assuming that these measures are probability measures. Diagonal systems with monotonic data have attracted a particular attention on account of their appearance in the dynamics of dislocation densities or in isentropic gas dynamics. We refer to the works by El Hajj and Monneau~\cite{ehm09, ehm12}, whose existence, uniqueness, regularity and stability results are discussed in~\S\ref{sss:ehm09} and~\S\ref{sss:ehm12} below.


\subsection{Main results and outline of the article} In this article, we consider the diagonal Cauchy problem
\begin{equation}\label{eq:syst}
  \forall \gamma \in \{1, \ldots, d\}, \qquad \left\{\begin{aligned}
    & \partial_t u^{\gamma} + \lambda^{\gamma}(\bu) \partial_x u^{\gamma} = 0,\\
    & u^{\gamma}(0,x) = u^{\gamma}_0(x),
  \end{aligned}\right.
\end{equation}
where $\bu = (u^1, \ldots, u^d) : [0,+\infty) \times \R \to [0,1]^d$, the characteristic functions $\lambda^1, \ldots, \lambda^d$ are defined on $[0,1]^d$ and we assume that there exist probability measures $m^1, \ldots, m^d$ on the real line such that
\begin{equation*}
  \forall \gamma \in \{1, \ldots, d\}, \qquad u^{\gamma}_0 = H*m^{\gamma},
\end{equation*}
where $H*\cdot$ refers to the convolution with the Heaviside function $H$. In other words, for all $\gamma \in \{1, \ldots, d\}$, $u^{\gamma}_0$ is the cumulative distribution function of $m^{\gamma}$.

In the scalar case $d=1$, the conservative form of~\eqref{eq:syst} is the {\em scalar conservation law}
\begin{equation}\label{eq:scalarcl}
  \left\{\begin{aligned}
    & \partial_t u + \partial_x \left(\Lambda(u)\right) = 0,\\
    & u(0,x) = u_0(x),
  \end{aligned}\right.
\end{equation}
with $\Lambda'=\lambda$ and $u_0 = H*m$, where $m$ is a probability measure on $\R$. Brenier and Grenier~\cite{bregre} proved that the entropy solution of~\eqref{eq:scalarcl} describes the large-scale behaviour of the {\em Sticky Particle Dynamics}, under which finitely many particles evolve on the real line by sticking together at collisions with preservation of the total mass and momentum. We also refer to~\cite{jourdain:sticky} for a proof of the large-scale limit in a more general framework. Independently of this representation, stability estimates in Wasserstein distance for the entropy solution of~\eqref{eq:scalarcl} were derived by Bolley, Brenier and Loeper~\cite{bolbreloe}.

\sk
In the present article, we introduce a multitype version of the Sticky Particle Dynamics, where particles have a {\em type} $\gamma \in \{1, \ldots, d\}$ and only stick with particles of the same type. Using this {\em Multitype Sticky Particle Dynamics}, we obtain the following three main results, under the generical assumption that the system~\eqref{eq:syst} be {\em uniformly} strictly hyperbolic.

Theorem~\ref{theo:existence} asserts the existence of a global weak solution for the Cauchy problem~\eqref{eq:syst}. More precisely, we show that the large-scale behaviour of the Multitype Sticky Particle Dynamics is described by functions $\bu : [0,+\infty) \times \R \to [0,1]^d$ solving the Cauchy problem~\eqref{eq:syst} in an appropriate sense, to which we refer as a {\em probabilistic solution}. We use a tightness argument for the particle system, which does not allow to identify its possibly multiple large-scale limits. 

Theorem~\ref{theo:stabMSPD} is a stability result on the Multitype Sticky Particle Dynamics. We carry out a detailed pathwise analysis of the evolution of the dynamics with two different initial configurations and thereby obtain $\Ls^p$ stability estimates, for all $p \in [1,+\infty]$. The important point here is that our stability constants are uniform with respect to the number of particles, which allows us to pass to the large-scale limit in these estimates. 

Theorem~\ref{theo:sg} combines the two previous results and finally asserts that our solutions are nonlinear semigroups, stable in Wasserstein distances of all orders (order $1$ corresponds to the usual $\Ls^1$ stability), which generalises the results of~\cite{bolbreloe} to the diagonal system~\eqref{eq:syst}. Besides, these solutions satisfy the uniqueness conditions of Bianchini and Bressan~\cite{bianbres} corresponding to our definition of probabilistic solutions. This allows us to identify all the large-scale limits of the Multitype Sticky Particle Dynamics and to finally obtain a complete convergence result for the particle system.

\sk
Our approximation procedure can be compared with the Glimm scheme or the Front Tracking method, as opposed to the vanishing viscosity approach, in the sense that it consists in constructing a piecewise constant solution to the hyperbolic system with initial data given by a discretisation of $u^1_0, \ldots, u^d_0$. Besides, similarly to~\cite{brecol95:sg, bcp00}, our stability estimates are obtained by taking the limit of uniform discrete stability estimates. 

Working with cumulative distribution functions allows us to employ classical tools from probability theory, and to some extent, from optimal transport. As an example, we shall use weak convergence and tightness of probability measures in place of the usual Helly Theorem in order to construct weak solutions. Likewise, stability estimates in Wasserstein distance shall naturally arise from discrete $\Ls^p$ estimates on our particle system when described by the increasing order of the positions.

A striking remark is that the diagonal structure of the system~\eqref{eq:syst} combined with the monotonicity of the initial data permits to obtain global existence, uniqueness and stability results without any smallness assumption on the variation of the initial data. This is done at the price of assuming that the strict hyperbolicity of the system holds uniformly on $[0,1]^d$. Let us also mention that our results involve no such condition as genuine nonlinearity or linear degenerescence of the characteristic fields.

\sk
The main definitions and results of the article are summarised and discussed in Section~\ref{s:main}. Then the article is divided into two parts. Part~\ref{part:1} is dedicated to the introduction of the Multitype Sticky Particle Dynamics and to the proof of Theorem~\ref{theo:existence}, our global weak existence result. We also describe a few properties of those solutions to the system~\eqref{eq:syst} that are obtained by Theorem~\ref{theo:existence}. Part~\ref{part:2} is concerned with stability results and contains the proof of the discrete stability estimates of Theorem~\ref{theo:stabMSPD}, as well as the construction of semigroup solutions given by Theorem~\ref{theo:sg}. Some technical proofs are postponed to an Appendix section, where a list of notations is also provided.


\subsection{Notations and conventions}\label{ss:intro:not} We shall use the following notations and conventions throughout the article. A complete notation index is provided in Appendix~\ref{app:notations}.


\subsubsection{Bold symbols} Generically, bold symbols, such as $\bu$ in~\eqref{eq:syst}, refer to objects of size $d$. Their coordinates, such as $u^1, \ldots, u^d$, are written with thin characters, and labelled with a Greek letter superscript. This letter is usually $\gamma \in \{1, \ldots, d\}$ or $\alpha, \beta$ when two distinct coordinates are at stake, in which case we take the convention that $\alpha < \beta$. 


\subsubsection{Algebraic notations} For all $x,y \in \R$, we let $x \wedge y := \min\{x,y\}$ and $x \vee y := \max\{x,y\}$. The integer part of $x \in [0,+\infty)$ is denoted by $\lfloor x \rfloor$. Given two sets $A$ and $B$, the union set $A \cup B$ shall be denoted by $A \sqcup B$ whenever $A \cap B = \emptyset$.


\subsubsection{Set of probability measures}\label{sss:not:prob} Given a metric space $E$, the set of Borel probability measures on $E$ is denoted by $\Ps(E)$. It is endowed with the topology of weak convergence, which is defined with respect to the set of continuous and bounded functions from $E$ to $\R$. 

Given two metric spaces $E$, $F$, a measurable function $g : E \to F$, and $\mu \in \Ps(E)$, the {\em image} (or {\em pushforward measure}) of $\mu$ by the function $g$, denoted by $\mu \circ g^{-1} \in \Ps(F)$, is defined by $(\mu \circ g^{-1})(B) = \mu(g^{-1}(B))$ for all Borel sets $B \subset F$.


\subsubsection{Function spaces} Given an interval $I \subset \R$, we denote by $\Cs(I,\R)$ (resp. $\Cs(I,\R^d)$) the set of continuous functions on $I$ with values in $\R$ (resp. $\R^d$). We similarly denote by $\Cs^{1,0}_{\mathrm{c}}([0,+\infty) \times \R, \R)$ (resp. $\Cs^{1,0}_{\mathrm{c}}([0,+\infty) \times \R, \R^d)$) the set of functions of $(t,x) \in [0,+\infty)\times\R$ with values in $\R$ (resp. $\R^d$) having compact support and a continuous time derivative (resp. of which each coordinate has a continuous time derivative). We finally denote by $\Cs^{1,1}_{\mathrm{c}}([0,+\infty) \times \R, \R) \subset \Cs^{1,0}_{\mathrm{c}}([0,+\infty) \times \R, \R)$ the subset of functions with a continuous space derivative.

The set of locally integrable functions on $\R$ with respect to the Lebesgue measure is denoted $\Ls^1_{\mathrm{loc}}(\R)$. Given a probability measure $m \in \Ps(\R)$, we denote by $\Ls^1(m)$ the set of integrable functions with respect to $m$.


\subsubsection{Probability measures on the space of sample-paths}\label{sss:not:probC} Given an interval $I \subset \R$, we endow the sets $\Cs(I,\R)$ and $\Cs(I,\R^d)$ with the topology of the uniform convergence if $I$ is compact, and of the locally uniform convergence otherwise. Both these topologies can be metrised.

The set of Borel probability measures on $\Cs([0,+\infty),\R^d)$ is denoted
\begin{equation*}
  \Ms := \Ps(\Cs([0,+\infty),\R^d)).
\end{equation*}
For all $\upmu \in \Ms$, we denote by $\upmu^{\gamma}_t$ the marginal distribution of the $\gamma$-th coordinate at time $t \geq 0$ under $\upmu$; that is to say, $\upmu_t^{\gamma} := \upmu \circ (\pi_t^{\gamma})^{-1}$, where 
\begin{equation*}
  \pi_t^{\gamma} : \left\{\begin{array}{ccc}
    \Cs([0,+\infty),\R^d) & \to & \R\\
    (X^1(s), \ldots, X^d(s))_{s \geq 0} & \mapsto & X^{\gamma}(t)
  \end{array}\right.
\end{equation*}
is the usual projection operator. Since $\pi_t^{\gamma}$ is continuous, the Mapping Theorem~\cite[Theorem~2.7, p.~21]{billingsley} implies that the mapping $\upmu \mapsto \upmu_t^{\gamma}$ is continuous for the topology of the weak convergence on $\Ms$ and $\Ps(\R)$.


\section{Main definitions and results}\label{s:main}

This section contains the main definitions and results of the article. The various assumptions we shall make on the characteristic fields $\lambda^1, \ldots, \lambda^d$ are gathered in Subsection~\ref{ss:ass}. A short presentation of the Multitype Sticky Particle Dynamics is given in Subsection~\ref{ss:mspd:intro}. Cumulative distribution functions play a crucial role in our work, therefore basic definitions and properties are recalled in Subsection~\ref{ss:CDF}. 

The notion of {\em probabilistic solution} to the Cauchy problem~\eqref{eq:syst} is defined in Subsection~\ref{ss:probsol}, where the weak existence result of Theorem~\ref{theo:existence} is stated. The discrete uniform stability estimates of Theorem~\ref{theo:stabMSPD} are stated in Subsection~\ref{ss:discrstab}, while our main Theorem~\ref{theo:sg} is detailed in Subsection~\ref{ss:sg}. 


\subsection{Assumptions on the characteristic fields}\label{ss:ass} Our results are stated under various assumptions on the function 
\begin{equation*}
  \blambda = (\lambda^1, \ldots, \lambda^d) : [0,1]^d \to \R^d,
\end{equation*}
that we now list.

\sk
We first introduce continuity conditions.
\begin{enumerate}[label=(C), ref=C]
  \item\label{ass:C} Continuity: for all $\gamma \in \{1, \ldots, d\}$, the function $\lambda^{\gamma}$ is continuous on $[0,1]^d$.
\end{enumerate}
Under Assumption~\eqref{ass:C}, the functions $\lambda^1, \ldots, \lambda^d$ are bounded and we define the family of finite constants $\ConstBound{p}$, $p \in [1,+\infty]$, by
\begin{equation}\label{eq:ConstBound}
  \forall p \in [1, +\infty), \quad \ConstBound{p} := \left(\sum_{\gamma=1}^d \sup_{\bu \in [0,1]^d} |\lambda^{\gamma}(\bu)|^p\right)^{1/p}, \qquad \ConstBound{\infty} := \sup_{1 \leq \gamma \leq d} \sup_{\bu \in [0,1]^d} |\lambda^{\gamma}(\bu)|.
\end{equation}

\begin{enumerate}[label=(LC), ref=LC]
  \item\label{ass:LC} Lipschitz Continuity: there exists $\ConstLip \in [0,+\infty)$ such that
  \begin{equation*}
    \forall \gamma \in \{1, \ldots, d\}, \quad \forall \bu,\bv \in [0,1]^d, \qquad |\lambda^{\gamma}(\bu) - \lambda^{\gamma}(\bv)| \leq \ConstLip\sum_{\gamma'=1}^d |u^{\gamma'}-v^{\gamma'}|.
  \end{equation*}
\end{enumerate}
Of course, Assumption~\eqref{ass:LC} is stronger than Assumption~\eqref{ass:C}.

\sk
The following Uniform Strict Hyperbolicity condition is crucial in this article, since it enables us to define the Multitype Sticky Particle Dynamics.
\begin{enumerate}[label=(USH), ref=USH]
  \item\label{ass:USH} Uniform Strict Hyperbolicity: there exists $\ConstUSH \in (0,+\infty)$ such that
  \begin{equation*}
    \forall \gamma \in \{1, \ldots, d-1\}, \qquad \inf_{\bu \in [0,1]^d} \lambda^{\gamma}(\bu) - \sup_{\bu \in [0,1]^d} \lambda^{\gamma+1}(\bu) \geq \ConstUSH.
  \end{equation*}
\end{enumerate}
Note that, under Assumptions~\eqref{ass:C} and~\eqref{ass:USH}, the triangle inequality implies that $\ConstUSH \leq \ConstBound{1}\wedge 2\ConstBound{\infty}$.


\subsection{The Multitype Sticky Particle Dynamics}\label{ss:mspd:intro} The precise construction of the Multitype Sticky Particle Dynamics (MSPD) is detailed in Section~\ref{s:mspd}. In this subsection, we only give a formal description of the MSPD and introduce the notations that will be necessary to state the $\Ls^p$ stability estimates of Theorem~\ref{theo:stabMSPD}.

The MSPD describes the evolution of $d \times n$ particles on the real line. For all $\gamma \in \{1, \ldots, d\}$ and $k \in \{1, \ldots, n\}$, the $k$-th particle of type $\gamma$ is labelled by the symbol $\gamma:k$, and we shall denote by 
\begin{equation*}
  \Part := \{ \gamma:k, \gamma \in \{1, \dots, d\}, k \in \{1, \ldots, n\}\}
\end{equation*}
the set of all such symbols.

Let us define the polyhedron $\Dn \subset \R^n$ by 
\begin{equation*}
  \Dn := \{(x_1, \ldots, x_n) \in \R^n : x_1 \leq \cdots \leq x_n\}.
\end{equation*}
The configuration space for the Multitype Sticky Particle Dynamics (MSPD) is the Cartesian product $\Dnd$, a typical element of which is denoted
\begin{equation*}
  \x = (x^{\gamma}_k)_{\gamma:k \in \Part},
\end{equation*}
so that in the configuration $\x$, the position of the particle $\gamma:k$ is $x^{\gamma}_k$.

In a configuration $\x \in \Dnd$, the {\em rank} of the particle $\gamma:k$ among the system of particles of type $\gamma' \in \{1, \ldots, d\}$ is the number of particles of type $\gamma'$ located on the left of $\gamma:k$ (\ie which position is lower than $x^{\gamma}_k$). Informally, the MSPD started at the configuration $\x$ is defined as follows:
\begin{itemize}
  \item the mass of each particle is $1/n$, and the initial velocity of a particle is determined by its rank among each system of particles of a given type,
  \item particles travel at constant velocity until they collide with other particles,
  \item when two particles of the same type collide, they stick together into a cluster, and the velocity of the cluster is determined by the conservation of mass and momentum,
  \item when two clusters of different types collide, the velocities of every particle is updated with respect to its rank in each system after the collision.
\end{itemize}
The initial velocity of the particle $\gamma:k$ as a function of its rank among each system is given under Assumption~\eqref{ass:C} by an appropriate discretisation of the function $\lambda^{\gamma}$ appearing in~\eqref{eq:syst}, see~\eqref{eq:vitesses} in Section~\ref{s:mspd}. Under the further Assumption~\eqref{ass:USH}, we show that the dynamics described above is well defined at all times and for all initial configurations. Denoting by $\Phi(\x;t) = (\Phi^{\gamma}_k(\x;t))_{\gamma:k \in \Part}$ the positions of the particles at time $t \geq 0$ in the MSPD started at the configuration $\x$, we thus define a flow $(\Phi(\cdot;t))_{t \geq 0}$ in $\Dnd$. A typical trajectory of the MSPD is plotted on Figure~\ref{fig:MSPD}.

\begin{figure}[ht]
  \begin{center}
    \vskip 2cm
    \includegraphics[clip=true,trim=0cm 0cm 1cm 0cm,angle=90,origin=c,width=.5\textwidth]{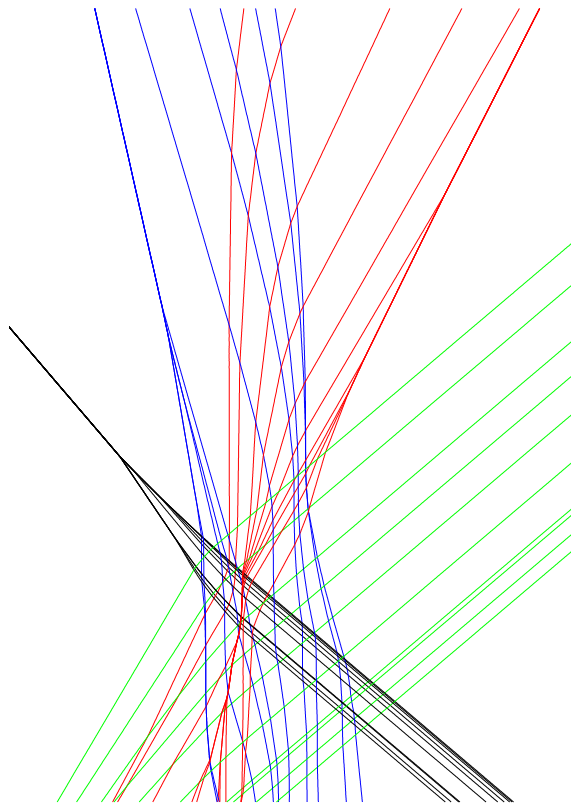}
  \end{center}
  \caption{A typical trajectory of the Multitype Sticky Particle Dynamics with $d=4$ types and $n=10$ particles per type. The horizontal coordinate refers to the physical positions of the particles, while the vertical coordinate describes the time. Each color is associated with a type of particle. Particles of the same type stick together at collisions, and the velocities may be modified at collisions with clusters of different types.}
  \label{fig:MSPD}
\end{figure}

\begin{rk}
  In the scalar case $d=1$, the MSPD reduces to the Sticky Particle Dynamics introduced by Brenier and Grenier~\cite{bregre} in the context of the study of general scalar conservation laws. The construction of such an adhesion dynamics in the physics literature is due to Zel'dovich~\cite{zeldo} and is related to the modeling of large-scale structure in the universe, as well as elementary models in turbulence~\cite{vergassola}. In particular, it played an important role in the mathematical understanding of the behaviour of pressureless gases~\cite{bouchut, grenier, eryksin, bregan}; in this direction, we highlight the recent work by Natile and Savaré~\cite{natsav} which relies on similar Wasserstein estimates as ours.
\end{rk}

\begin{rk}
  In the scalar case $d=1$, the viscous version
  \begin{equation*}
    \left\{\begin{aligned}
      & \partial_t u + \partial_x(\Lambda(u)) = \epsilon \partial_x^2 u,\\
      & u(0,x) = u_0(x),
    \end{aligned}\right.
  \end{equation*}
  of the scalar conservation law~\eqref{eq:scalarcl} is known to describe the large-scale limit of systems of rank-based interacting diffusions~\cite{bostal, bostalMC, jourdain:signed}. In~\cite{jourey:snoise}, it was proved that, when $\epsilon$ vanishes, such systems of diffusions converge to the Sticky Particle Dynamics, the large-scale limit of which is described by the entropy solution to the corresponding inviscid conservation law~\cite{bregre, jourdain:sticky}. Theoretical and numerical approximation procedures of the conservation law~\eqref{eq:scalarcl} based on this probabilistic representation and combining the small-noise and large-scale limits where constructed in~\cite{jourdain:characteristics, jourou}, where fractional diffusions are also considered.

  As far as the case $d \geq 2$ is concerned, a multitype system of rank-based interacting diffusions was introduced in~\cite[Chapitre~7]{reygner:phd} in order to approximate the solution to the parabolic system
  \begin{equation*}
    \forall \gamma \in \{1, \ldots, d\}, \qquad \left\{\begin{aligned}
      & \partial_t u^{\gamma} + \lambda^{\gamma}(\bu)\partial_x u^{\gamma} = \epsilon \partial^2_x u^{\gamma},\\
      & u^{\gamma}(0,x) = u^{\gamma}_0(x).
    \end{aligned}\right.
  \end{equation*}
  Using the arguments introduced in~\cite{jourey:snoise}, the MSPD can be shown to describe the small-noise limit of this system. \end{rk}


\subsection{Cumulative distribution functions}\label{ss:CDF} In this subsection, we give a few definitions and introduce some notations related to cumulative distribution functions (CDFs).

\begin{defi}[Cumulative distribution function]
  A cumulative distribution function on the real line is a nondecreasing and right continuous function $F : \R \to [0,1]$ such that
  \begin{equation*}
    \lim_{x \to -\infty} F(x) = 0, \qquad \lim_{x \to +\infty} F(x) = 1.
  \end{equation*}
\end{defi}

It is an elementary result of measure theory~\cite[Theorem~(4.3), p.~5]{revuz} that a function $F$ is a CDF on the real line if and only if there exists a probability measure $m \in \Ps(\R)$ such that, for all $x \in \R$, $F(x) = m((-\infty,x])$. In this case, $F$ is said to be {\em the CDF of $m$}, and we denote $F = H*m$, where $H$ refers to the Heaviside function $H(x) := \ind{x \geq 0}$.

CDFs are generically discontinuous and therefore can have {\em jumps}, defined as follows.
\begin{defi}[Jumps]
  Let $F$ be a CDF on the real line. For all $x \in \R$, the {\em jump} of $F$ at $x$ is defined by
  \begin{equation*}
    \Delta F(x) := F(x) - F(x^-),
  \end{equation*}
  where
  \begin{equation*}
    F(x^-) := \lim_{y \uto x} F(y).
  \end{equation*}
\end{defi}

Certainly, for all $x \in \R$, $\Delta F(x) = m(\{x\})$, and whenever the latter quantity is positive, then $x$ is called an {\em atom} of $m$. Note that the set of atoms of $m$ is at most countable, therefore $\dd x$-almost everywhere, $\Delta F(x)=0$.

\sk
If $F$ is the CDF of $m$, then, for all $f \in \Ls^1(m)$, the {\em expectation} of $f$ under $m$ is indifferently denoted
\begin{equation*}
  \int_{x \in \R} f(x)m(\dd x) = \int_{x \in \R} f(x)\dd F(x).
\end{equation*}
The expectation of $f$ under $m$ can also be expressed in terms of the {\em pseudo-inverse} of $F$, defined as follows.
\begin{defi}[Pseudo-inverse]
  Let $F$ be a CDF on the real line. The {\em pseudo-inverse} of $F$ is the function $F^{-1} : (0,1) \to \R$ defined by
  \begin{equation}\label{eq:pseudoinv}
    F^{-1}(v) := \inf\{x \in \R : F(x) \geq v\}.
  \end{equation}
\end{defi}

The following properties of the pseudo-inverse are straightforward.
\begin{lem}[Properties of the pseudo-inverse]\label{lem:pseudoinv}
  Let $F$ be a CDF on the real line. 
  \begin{enumerate}[label=(\roman*), ref=\roman*]
    \item\label{it:pseudoinv:0} The function $F^{-1} : (0,1) \to \R$ is nondecreasing, left continuous with right limits. It is countinuous outside of the countable set $\{v\in (0,1):\exists x<y\in\R,\;F(x)=F(y)=v\}$
    \item\label{it:pseudoinv:1} For all $v \in (0,1)$, $F(F^{-1}(v)^-) \leq v \leq F(F^{-1}(v))$.
    \item\label{it:pseudoinv:2} For all $x \in \R$, for all $v \in (0,1)$, $F^{-1}(v) \leq x$ if and only if $v \leq F(x)$.
  \end{enumerate}
\end{lem}

The expectation of $f$ under $m$ satisfies the following {\em change of variable formula}~\cite[Proposition~(4.9), p.~8]{revuz}.
\begin{lem}[Change of variable formula]\label{lem:CDFm1}
  Let $F$ be the CDF of the probability measure $m$ on $\R$. Then, for all $f \in \Ls^1(m)$,
  \begin{equation*}
    \int_{x \in \R} f(x) \dd F(x) = \int_{v=0}^1 f(F^{-1}(v))\dd v.
  \end{equation*}
\end{lem}
Let us point out the fact that, with the notations introduced in Subsection~\ref{ss:intro:not} above, a reformulation of Lemma~\ref{lem:CDFm1} is $m = \Unif \circ (F^{-1})^{-1}$, where $\Unif$ refers to the Lebesgue measure on $[0,1]$.

\begin{lem}[Weak convergence and CDFs]\label{lem:cvCDF}
  Let $(m_n)_{n \geq 1}$ be a sequence of probability measures on $\R$ and $m \in \Ps(\R)$. Let $F_n := H*m_n$ and $F := H*m$. Then $m_n$ converges weakly to $m$ if and only if, for all $x \in \R$ such that $\Delta F(x) = 0$, $F_n(x)$ converges to $F(x)$. In this case, $F_n^{-1}(v)$ converges to $F^{-1}(v)$ at all continuity points $v$ of $F^{-1}$, therefore $\dd v$-almost everywhere in $(0,1)$.
\end{lem}
The equivalence between weak convergence and convergence of the CDF outside of the atoms of the limit is a classical result, see for instance~\cite[Theorem~2.2, p.~86]{durrett}. The almost everywhere convergence of pseudo-inverses is often used as a proof of the Skorokhod Representation Theorem on the real line, see~\cite[Theorem~2.1, p.~85]{durrett}.

\sk
We finally introduce a few notations for functions $u : [0,+\infty) \times \R \to [0,1]$ such that, for all $t \geq 0$, $u(t,\cdot)$ is a CDF on the real line. For such a function, for all $t \geq 0$,
\begin{itemize}
  \item the jump of $u(t,\cdot)$ at $x \in \R$ is denoted by $\Delta_x u(t,x)$ and worth $\Delta_x u(t,x):=u(t,x)-u(t,x^-)$, where $u(t,x^-):=\lim_{y \uto x} u(t,y)$,
  \item if $m \in \Ps(\R)$ is such that $u(t,\cdot) = H*m$, then for all $f \in \Ls^1(m)$, the expectation of $f$ under $m$ is denoted
  \begin{equation*}
    \int_{x \in \R} f(x) m(\dd x) = \int_{x \in \R} f(x) \dd_x u(t,x),
  \end{equation*}
  and we have
  \begin{equation*}
    \int_{x \in \R} f(x) \dd_x u(t,x) = \int_{v=0}^1 f\left(u(t,\cdot)^{-1}(v)\right) \dd v,
  \end{equation*}
  where $u(t,\cdot)^{-1}(v)$ refers to the pseudo-inverse of the CDF $u(t,\cdot)$.
\end{itemize}


\subsection{Probabilistic solutions to the system~\texorpdfstring{\eqref{eq:syst}}{}}\label{ss:probsol} In this subsection, we introduce the notion of a {\em probabilistic solution} to the Cauchy problem~\eqref{eq:syst}. Probabilistic solutions have to be thought of as {\em weak solutions} $\bu=(u^1, \ldots, u^d)$ of~\eqref{eq:syst} having the property that $u^{\gamma}(t,\cdot)$ remains a CDF on the real line at all times. Since such functions can be discontinuous, we need to take a convention to define the product $\lambda^{\gamma}(\bu)\partial_x u^{\gamma}$. This task is carried out in~\S\ref{sss:probsol}. The existence of probabilistic solutions, based on an approximation procedure by the vector of empirical CDFs of the MSPD, is stated in~\S\ref{sss:pfexist}. A description of arbitrary probabilistic solutions in terms of trajectories in $\R^d$ is discussed in~\S\ref{sss:introtraj}, and the continuity of solutions obtained at~\S\ref{sss:pfexist} under diagonal monotonicity conditions on the characteristic fields is investigated in~\S\ref{sss:rar:intro}. Finally, the links between our results and those of~\cite{ehm09} are discussed in~\S\ref{sss:ehm09}.


\subsubsection{Definition of probabilistic solutions}\label{sss:probsol} The main difficulty in defining a notion of solution to the system~\eqref{eq:syst} is to make sense of the product $\lambda^{\gamma}(\bu)\partial_x u^{\gamma}$. Indeed, since we expect $u^{\gamma}(t, \cdot)$ to be a CDF on the real line for all $t \geq 0$, the function $\lambda^{\gamma}(\bu)$ is generically discontinuous at the atoms of the measure $\partial_x u^{\gamma}$, and therefore this product cannot be defined in the distributional sense. Although there has been several works~\cite{dallefmur, boujam98} dedicated to the problem of giving a suitable definition to the product between a discontinuous function and a Radon measure in the context of transport equations, we shall use the particular connection between $\lambda^{\gamma}(\bu)$  and $\partial_x u^{\gamma}$ in order to provide a definition such that, in the scalar case, the product $\lambda(u)\partial_x u$ coincide with the conservative form $\partial_x(\Lambda(u))$, see Remark~\ref{rk:scalarcl} below.

Let $\bu = (u^1, \ldots, u^d) : [0,+\infty) \times \R \to [0,1]^d$ be a measurable function such that, for all $\gamma \in \{1, \ldots, d\}$, for all $t \geq 0$, the function $u^{\gamma}(t,\cdot)$ is a CDF on the real line. For all $\gamma \in \{1, \ldots, d\}$, let us define the function $\lambda^{\gamma}\{\bu\} : [0,+\infty) \times \R \to \R$ by
\begin{equation}\label{eq:dlambda}
  \lambda^{\gamma}\{\bu\}(t,x) := \int_{\theta=0}^1 \lambda^{\gamma}\left(u^1(t,x), \ldots, (1-\theta)u^{\gamma}(t,x^-)+\theta u^{\gamma}(t,x), \ldots, u^d(t,x)\right) \dd \theta,
\end{equation}
which will play the role of a substitute for $\lambda^{\gamma}(\bu(t,x))$ in~\eqref{eq:syst}. Note that the function $\lambda^{\gamma}\{\bu\}$ can be rewritten under the more explicit form
\begin{equation*}
  \lambda^{\gamma}\{\bu\}(t,x) = \lambda^{\gamma}(\bu(t,x))
\end{equation*}
if $\Delta_x u^{\gamma}(t,x) = 0$, and
\begin{equation*}
    \lambda^{\gamma}\{\bu\}(t,x) = \frac{1}{\Delta_x u^{\gamma}(t,x)} \int_{w=u^{\gamma}(t,x^-)}^{u^{\gamma}(t,x)} \lambda^{\gamma}\left(u^1(t,x), \ldots, u^{\gamma-1}(t,x), w, u^{\gamma+1}(t,x), \ldots, u^d(t,x)\right) \dd w
\end{equation*}
otherwise. 

\sk
We are now ready to introduce our notion of probabilistic solution.

\begin{defi}[Probabilistic solution to~\eqref{eq:syst}]\label{defi:sol}
  Under Assumption~\eqref{ass:C}, a {\em probabilistic solution} to the hyperbolic system~\eqref{eq:syst} is a measurable function
  \begin{equation*}
    \bu = (u^1, \ldots, u^d) : [0,+\infty) \times \R \to [0,1]^d,
  \end{equation*}
  such that:
  \begin{enumerate}[label=(\roman*), ref=\roman*]
    \item\label{it:sol:1} for all $t \geq 0$, for all $\gamma \in \{1, \ldots, d\}$, $u^{\gamma}(t, \cdot)$ is a CDF on the real line,
    \item\label{it:sol:3} for all test functions $\bvarphi = (\varphi^1, \ldots, \varphi^d) \in \Cs^{1,0}_{\mathrm{c}}([0,+\infty)\times\R, \R^d)$,
    \begin{equation*}
      \begin{aligned}
        & \sum_{\gamma=1}^d \left(\int_{t=0}^{+\infty} \int_{x \in \R} \partial_t\varphi^{\gamma}(t,x) u^{\gamma}(t,x) \dd x\dd t + \int_{x \in \R} \varphi^{\gamma}(0,x) u_0^{\gamma}(x) \dd x \right)\\
        & \qquad = \sum_{\gamma=1}^d \int_{t=0}^{+\infty} \int_{x \in \R} \varphi^{\gamma}(t,x) \lambda^{\gamma}\{\bu\}(t,x) \dd_x u^{\gamma}(t,x) \dd t,
      \end{aligned}
    \end{equation*}
    where $\lambda^{\gamma}\{\bu\}$ is defined by~\eqref{eq:dlambda} above.
  \end{enumerate}
\end{defi}

\begin{rk}
  In the point~\eqref{it:sol:3} of Definition~\ref{defi:sol}, the integral term
  \begin{equation*}
    \int_{x \in \R} \varphi^{\gamma}(t,x) \lambda^{\gamma}\{\bu\}(t,x) \dd_x u^{\gamma}(t,x)
  \end{equation*}
  has to be understood as the expectation of the bounded measurable function $\varphi^{\gamma}(t,\cdot) \lambda^{\gamma}\{\bu\}(t,\cdot)$ under the probability measure with CDF $u^{\gamma}(t,\cdot)$. In addition, the point~\eqref{it:sol:3} only makes sense if the function
  \begin{equation*}
    t \mapsto \int_{x \in \R} \varphi^{\gamma}(t,x) \lambda^{\gamma}\{\bu\}(t,x) \dd_x u^{\gamma}(t,x)
  \end{equation*}
  is measurable on $[0,+\infty)$. This property is obtained by first applying the change of variable formula of Lemma~\ref{lem:CDFm1} to rewrite
  \begin{equation*}
    \int_{x \in \R} \varphi^{\gamma}(t,x) \lambda^{\gamma}\{\bu\}(t,x) \dd_x u^{\gamma}(t,x) = \int_{v=0}^1 \varphi^{\gamma}\left(t,u^{\gamma}(t,\cdot)^{-1}(v)\right) \lambda^{\gamma}\{\bu\}\left(t,u^{\gamma}(t,\cdot)^{-1}(v)\right)\dd v.
  \end{equation*}
  Now it is easily checked that the function 
  \begin{equation*}
    (t,v) \mapsto \varphi^{\gamma}\left(t,u^{\gamma}(t,\cdot)^{-1}(v)\right) \lambda^{\gamma}\{\bu\}\left(t,u^{\gamma}(t,\cdot)^{-1}(v)\right)
  \end{equation*}
  is measurable and bounded on the product space $[0,+\infty) \times (0,1)$, so that the conclusion follows from the Fubini Theorem.
\end{rk}

\begin{rk}\label{rk:scalarcl}
  In the scalar case $d=1$, and with the definition of $\lambda\{u\}$ above, we have
  \begin{equation}\label{eq:ippscal}
    \partial_x (\Lambda(u(t,x))) = \lambda\{u\}(t,x) \dd_x u(t,x),
  \end{equation}
  in the distributional sense, where we recall that $\Lambda$ is the antiderivative of $\lambda$ (this is a consequence of the chain rule formula for functions of finite variation~\cite[Proposition~(4.6), p.~6]{revuz}). As a consequence, a probabilistic solution in the sense of Definition~\ref{defi:sol} is nothing but a weak solution to the scalar conservation law~\eqref{eq:scalarcl}, which remains a CDF at all times.
\end{rk}


\subsubsection{Existence of probabilistic solutions}\label{sss:pfexist} We first define the empirical distribution and the vector of empirical CDFs of the MSPD. 

\begin{defi}[Empirical distribution and vector of empirical CDFs of the MSPD]\label{defi:muMSPD}
  Under Assumptions~\eqref{ass:C} and~\eqref{ass:USH}, for all $\x \in \Dnd$, the {\em empirical distribution of the MSPD started at $\x$} is the probability measure
  \begin{equation*}
    \upmu[\x] := \frac{1}{n} \sum_{k=1}^n \delta_{(\Phi_k^1(\x;t), \ldots, \Phi_k^d(\x;t))_{t \geq 0}} \in \Ms.
  \end{equation*}
  
  The vector $\bu[\x] = (u^1[\x], \ldots, u^d[\x])$ of {\em empirical CDFs} of the MSPD started at $\x$ is defined by, for all $\gamma \in \{1, \ldots, d\}$,
\begin{equation}\label{eq:bun}
  \forall (t,x) \in [0,+\infty) \times \R, \qquad u^{\gamma}[\x](t,x) := H*\upmu^{\gamma}_t[\x](x) = \frac{1}{n} \sum_{k=1}^n \ind{\Phi_k^{\gamma}(\x;t) \leq x},
\end{equation}
and we also let
\begin{equation}\label{eq:bun0}
  \forall x \in \R, \qquad u^{\gamma}_0[\x](x) := \frac{1}{n} \sum_{k=1}^n \ind{x^{\gamma}_k \leq x}.
\end{equation}
\end{defi}

With these definitions, we check in Section~\ref{s:existence} that that, for all $\x \in \Dnd$, the MSPD started as $\x$ satisfies the characteristic equation
\begin{equation}\label{eq:diffmspd:intro}
  \forall \gamma:k \in \Dnd, \qquad \dot{\Phi}_k^{\gamma}(\x;t) = \lambda^{\gamma}\{\bu[\x]\}(t, \Phi_k^{\gamma}(\x;t)), \qquad \text{$\dd t$-almost everywhere.}
\end{equation}
We then prove that this implies that $\bu[\x]$ is an {\em exact} probabilistic solution to the system~\eqref{eq:syst}, but with {\em discrete} initial data $(u^1_0[\x], \ldots, u^d_0[\x])$, see Proposition~\ref{prop:MSPDsol}. Taking a sequence $(\x(n))_{n \geq 1}$ of initial conditions such that $u^{\gamma}_0[\x(n)]$ approximates the initial data $u^{\gamma}_0$ of~\eqref{eq:syst}, we combine a tightness argument for the sequence of empirical distributions of the MSPD in the space of sample-paths with a closedness property of the set of probabilistic solutions to obtain the following existence theorem.

\begin{theo}[Convergence of the MSPD]\label{theo:existence}
  Let Assumptions~\eqref{ass:C} and~\eqref{ass:USH} hold, and let us fix $\bm = (m^1, \ldots, m^d) \in \Ps(\R)^d$. Let $(\x(n))_{n \geq 1}$ be a sequence of configurations such that, for all $n \geq 1$, $\x(n) \in \Dnd$, and assume that, for all $\gamma \in \{1, \ldots, d\}$, the sequence of empirical measures
  \begin{equation*}
    \frac{1}{n} \sum_{k=1}^n \delta_{x^{\gamma}_k(n)} \in \Ps(\R)
  \end{equation*}
  converges weakly to $m^{\gamma}$.
  
  Then from any subsequence of $(\upmu[\x(n)])_{n \geq 1}$, one can extract a further subsequence $(\upmu[\x(n_{\ell})])_{\ell \geq 1}$ weakly converging to some $\bar{\upmu} \in \Ms$, and such that the function $\bu = (u^1, \ldots, u^d) : [0,+\infty) \times \R \to [0,1]^d$ defined by
  \begin{equation*}
    \forall \gamma \in \{1, \ldots, d\}, \quad \forall (t,x) \in [0,+\infty) \times \R, \qquad u^{\gamma}(t,x) := H*\bar{\upmu}^{\gamma}_t(x),
  \end{equation*}
  is a probabilistic solution to the system~\eqref{eq:syst} with initial data $(u_0^1, \ldots, u_0^d)$ defined by $u^{\gamma}_0 := H*m^{\gamma}$, for all $\gamma \in \{1, \ldots, d\}$.  
\end{theo}

The tightness argument is explicited in Proposition~\ref{prop:tightness}, while the closedness property is detailed in Proposition~\ref{prop:closedness}. Of course, the probabilistic solutions that we obtain here may depend on the choice of the subsequence $(\upmu[\x(n_{\ell})])_{\ell \geq 1}$, and in the absence of a uniqueness property, nothing enables us to identify the corresponding limits. This uniqueness property is recovered by supplementing the definition of a probabilistic solution with further conditions, that are adapted from Bianchini and Bressan~\cite{bianbres}, see Subsection~\ref{ss:sg} below.

\sk
Combining the continuity of the mapping $\upmu \mapsto \upmu_t^{\gamma}$ with Lemma~\ref{lem:cvCDF}, we rewrite the result of Theorem~\ref{theo:existence} in terms of convergence of the vector of empirical CDFs of the MSPD as follows.

\begin{cor}[Convergence of the vector of empirical CDFs]\label{cor:cvCDFs}
  Under the assumptions of Theorem~\ref{theo:existence} and along the sequence $(n_{\ell})_{\ell \geq 1}$ provided by the latter, we have 
  \begin{equation*}
    \lim_{\ell \to +\infty} u^{\gamma}[\x(n_{\ell})](t,x) = u^{\gamma}(t,x),
  \end{equation*}
  for all $t \geq 0$, for all $\gamma \in \{1, \ldots, d\}$, and for all $x \in \R$ such that $\Delta_x u^{\gamma}(t,x) = 0$.
\end{cor}

In particular, for all $t \geq 0$, for all $\gamma \in \{1, \ldots, d\}$, the convergence in Corollary~\ref{cor:cvCDFs} holds $\dd x$-almost everywhere. Besides, by Dini's Theorem, if $u^{\gamma}(t, \cdot)$ is continuous on $\R$, then this convergence holds uniformly on $\R$.


\subsubsection{Trajectories associated with probabilistic solutions}\label{sss:introtraj} The equation~\eqref{eq:diffmspd:intro} for the MSPD shows that the quantiles of the probabilistic solution $\bu[\x]$ play the role of characteristics for the system~\eqref{eq:syst} --- at least between collisions. In Section~\ref{s:traj}, we address the question of whether this fact can be generalised to any probabilistic solution $\bu$, and therefore try to describe the evolution of the {\em trajectories} $(\bX_v(t))_{t \geq 0}$ in $\R^d$ associated with $\bu$, defined for all $t \geq 0$ by $\bX_v(t) = (X_v^1(t), \ldots, X_v^d(t))$, with
\begin{equation*}
  X_v^{\gamma}(t) := u^{\gamma}(t,\cdot)^{-1}(v).
\end{equation*}

We first prove in Proposition~\ref{prop:traj} that, for all probabilistic solutions $\bu$ to~\eqref{eq:syst}, $\dd v$-almost everywhere, the process $(X^{\gamma}_v(t))_{t \geq 0}$ is Lipschitz continuous and that its velocity is bounded by the minimal and maximal values of the characteristic field $\lambda^{\gamma}$. This enables us to provide a {\em probabilistic representation} of $\bu$ as the flow of marginal distributions of some stochastic process $(\bbX(t))_{t \geq 0}$ taking its values in $\R^d$. In the scalar case and for system of pressureless gases, a similar representation was constructed by Dermoune~\cite{dermoune99, dermoune01}.

We then discuss conditions under which the trajectories $(\bX_v(t))_{t \geq 0}$ satisfy the characteristic equation~\eqref{eq:diffmspd:intro}. We prove in particular, in Proposition~\ref{prop:vitesses}, that an equivalent condition to this characteristic equation is that the function $\bu$ be a {\em renormalised solution} to~\eqref{eq:syst} in the sense of DiPerna and Lions~\cite{diperna}. However, the question of whether the solutions obtained by Theorem~\ref{theo:existence} are renormalised solutions in general is left open.


\subsubsection{Continuity of rarefaction coordinates}\label{sss:rar:intro} Section~\ref{s:rar} addresses the continuity of the probabilistic solutions to~\eqref{eq:syst} obtained by Theorem~\ref{theo:existence} when a characteristic field $\lambda^{\gamma}$ satisfies some diagonal monotonicity conditions. More precisely, under Assumption~\eqref{ass:LC}, we shall say that $\gamma \in \{1, \ldots, d\}$ is a {\em rarefaction coordinate} if $\partial_{u^{\gamma}} \lambda^{\gamma} \geq 0$, and a {\em strong rarefaction coordinate} if there exists $c>0$ such that $\partial_{u^{\gamma}} \lambda^{\gamma} \geq c$. Then we prove in Corollary~\ref{cor:modulus} and Proposition~\ref{prop:strongrar} the following continuity results: if $\bu$ is a probabilistic solution obtained by Theorem~\ref{theo:existence},
\begin{itemize}
  \item for all rarefaction coordinate $\gamma \in \{1, \ldots, d\}$, if $u^{\gamma}_0$ is continuous on $\R$ then $u^{\gamma}$ is continuous on $[0,+\infty) \times \R$,
  \item for all strong rarefaction coordinate $\gamma \in \{1, \ldots, d\}$, $u^{\gamma}$ is continuous on $(0,+\infty) \times \R$, and if $u^{\gamma}_0$ is continuous on $\R$ then $u^{\gamma}$ is continuous on $[0,+\infty) \times \R$.
\end{itemize}
Let us insist on the fact that, in the two statements, the condition implying the continuity of $u^{\gamma}$ does not depend on the monotonicity of the characteristic field $\lambda^{\gamma'}$, for $\gamma' \not= \gamma$.


\subsubsection{Comparison with~\cite{ehm09}}\label{sss:ehm09} The construction of the Multitype Sticky Particle Dynamics is made under Assumptions~\eqref{ass:LC} and~\eqref{ass:USH}, and the global existence result of weak solutions stated in Theorem~\ref{theo:existence} only requires these two conditions to hold. 

El Hajj and Monneau~\cite{ehm09} obtained global existence of continuous probabilistic solutions to~\eqref{eq:syst} when the probability measures $m^1,\ldots,m^d$ admit densities with respect to the Lebesgue measure in $\Ls\log \Ls(\R)$ (that is to say, for all $\gamma\in\{1,\ldots,d\}$, $t \mapsto \partial_x u^\gamma(t,\cdot)$ remains locally bounded in $\Ls\log \Ls(\R)$), without any strict hyperbolicity condition on the characteristic fields which, in turn, are supposed to be $\Cs^\infty$, globally Lipschitz continuous and such that the matrix $(\partial_{u^{\gamma'}}\lambda^\gamma(\bu)+\partial_{u^{\gamma}}\lambda^{\gamma'}(\bu))_{\gamma,\gamma'}$ is positive semidefinite for all $\bu\in[0,1]^d$. 

Notice that this last condition implies that $\partial_{u^{\gamma}}\lambda^\gamma(\bu)\geq 0$ for all $\gamma\in\{1,\ldots,d\}$ and $\bu\in[0,1]^d$ so that all coordinates are rarefaction coordinates as defined in~\S\ref{sss:rar:intro}. By Corollary~\ref{cor:modulus} in Section~\ref{s:rar}, continuity of each rarefaction coordinate $u^\gamma(t,x)$ of our probabilistic solution holds under mere continuity of the corresponding initial condition $x\mapsto u_0^\gamma(x)$ and by Proposition~\ref{prop:strongrar}, continuity of $(t,x)\mapsto u^\gamma(t,x)$ on $(0,+\infty)\times \R$ holds as soon as the characteristic field $\lambda^\gamma$ is increasing in its $\gamma$-th coordinate.


\subsection{Discrete stability estimates}\label{ss:discrstab} For all $p \in [1,+\infty]$, let us define the following (normalised) $\Ls^p$ distances on $\Dnd$.

\begin{defi}[$\Ls^p$ distances on $\Dnd$]\label{defi:Lspdist}
  For all $\x, \y \in \Dnd$, we define
  \begin{equation}\label{eq:L1distDnd}
    \begin{aligned}
      \forall p \in [1,+\infty), \quad &||\x-\y||_p := \left(\frac{1}{n}\sum_{\gamma:k \in \Part} |x_k^{\gamma}-y_k^{\gamma}|^p\right)^{1/p},\\
      & ||\x-\y||_{\infty} := \sup_{\gamma:k \in \Part} |x_k^{\gamma}-y_k^{\gamma}|.
    \end{aligned}
  \end{equation}
\end{defi}

Section~\ref{s:stab} is dedicated to the proof of the following uniform $\Ls^p$ stability estimates on the MSPD.

\begin{theo}[Uniform $\Ls^p$ stability estimates for the MSPD]\label{theo:stabMSPD}
  Under Assumptions~\eqref{ass:LC} and \eqref{ass:USH}, for all $p \in [1, +\infty]$, there exists $\ConstStab_p \in [1,+\infty)$ such that, for all $\x, \y \in \Dnd$, for all $s,t \geq 0$,
  \begin{equation*}
    ||\Phi(\x;s) - \Phi(\y;t)||_p \leq \ConstStab_p ||\x-\y||_p + |t-s| \ConstBound{p},
  \end{equation*}
  where we recall that $\ConstBound{p}$ is defined in~\eqref{eq:ConstBound}, while $\ConstStab_p$ is an explicit function of $d$, $\ConstLip$ and $\ConstUSH$ but does not depend on $n$, see~\eqref{eq:ConstStab} below.
\end{theo}
In the scalar case $d=1$, then $\ConstStab_p=1$ for all $p \in [1,+\infty]$. For $d \geq 2$, the value of $\ConstStab_p$ is given by the following formulas:\begin{equation}\label{eq:ConstStab}
  \begin{aligned}
    & \ConstStab_1 := \left(1+ 4\Ratio(d-1)\exp\left(\Ratio(d-1)\right)\right) \exp\left(2\Ratio^2d(d-1)\exp\left(\Ratio(d-1)\right)\right),\\
    & \ConstStab_{\infty} := (1+\Ratio d \ConstStab_1) \exp(\Ratio(d-1)),\\
    & \ConstStab_p := \ConstStab_1^{1/p} \ConstStab_{\infty}^{1-1/p}, \quad \forall p \in (1, +\infty),
  \end{aligned}
\end{equation}
where  $\Ratio := 3\ConstLip/\ConstUSH$.

Theorem~\ref{theo:stabMSPD} is the cornerstone of this article. Up to technical corrections, its proof is essentially divided into two main parts. First, we assume that the initial configurations $\x$ and $\y$ are {\em close} to each other, in the sense that the trajectories of the MSPD started at both $\x$ and $\y$ share the same topological features. This permits to reduce the derivation of the stability estimates above to a purely algebraic problem, which is solved by a careful but elementary analysis and thereby provides a {\em local} stability estimate. Second, we use the geometrical properties of the trajectories of the MSPD to construct a continuous path between arbitrary initial configurations $\x$ and $\y$, along which the local stability estimate can be integrated so as to obtain a {\em global} stability estimate. We note that the idea of such a decomposition into a first {\em local} step and a second {\em interpolation} step echoes the proofs of $\Ls^1$ stability estimates for hyperbolic systems by Bressan and Colombo~\cite{brecol95:sg} and Bressan, Crasta and Piccoli~\cite{bcp00}.


\subsection{Stability and semigroup properties}\label{ss:sg} Since the discrete stability estimates obtained in Theorem~\ref{theo:stabMSPD} are uniform in the number of particles, they are expected to be consistent with the large-scale limit and therefore yield stability estimates on the solutions to the system~\eqref{eq:syst} constructed in Theorem~\ref{theo:existence}. As we shall explain below, the natural distance to extend these stability estimates is the Wasserstein distance, that we define in~\S\ref{sss:wass}. 

As a consequence of these estimates, we show that our solutions are semigroups. This property enables us to use the Bianchini-Bressan uniqueness conditions~\cite{bianbres} to roughly identify all the semigroup solutions to~\eqref{eq:syst}. These results are summed up in Theorem~\ref{theo:sg} in~\S\ref{sss:sg}.


\subsubsection{The Wasserstein distance}\label{sss:wass} Our stability estimates are stated in Wasserstein distance, an introduction to which can be found in Rachev and Rüschendorf~\cite{rachrusch} or Villani~\cite{villani}.

\begin{defi}[Wasserstein distance]\label{defi:wass}
  Let $m, m' \in \Ps(\R)$. For all $p \in [1, +\infty)$, we define the {\em Wasserstein distance} of order $p$ between $m$ and $m'$ by
  \begin{equation*}
    \Ws_p(m,m') := \inf_{\mathfrak{m} <^m_{m'}} \left(\int_{(x,x') \in \R^2} |x-x'|^p \mathfrak{m}(\dd x\dd x')\right)^{1/p},
  \end{equation*}
  where the infimum runs over all the probability measures $\mathfrak{m} \in \Ps(\R^2)$ such that, for all Borel sets $A, A' \subset \R$,
  \begin{equation*}
    \mathfrak{m}(A \times \R) = m(A), \qquad \mathfrak{m}(\R \times A') = m'(A').
  \end{equation*}
  
  The Wasserstein distance of order $\infty$ is defined by
  \begin{equation*}
    \Ws_{\infty}(m,m') := \lim_{p \to +\infty} \Ws_p(m,m').
  \end{equation*}
\end{defi}
Note that we allow the Wasserstein distances to take the value $+\infty$, therefore they should rather be referred to as {\em pseudo-distances}~\cite{villani}. For the sake of simplicity, we shall keep the denomination {\em distance}. Besides, the existence of the limit in the definition of $\Ws_{\infty}(m,m')$ follows from Hölder's inequality, which ensures that $p \mapsto \Ws_p(m,m') \in [0,+\infty]$ is nondecreasing.

It is a peculiar feature of the one-dimensional case that the measure 
\begin{equation*}
  \mathfrak{m} = \Unif \circ \left((H*m)^{-1}, (H*m')^{-1}\right)^{-1},
\end{equation*}
where $\Unif$ refers to the Lebesgue measure on $[0,1]$, realises the infimum in Definition~\ref{defi:wass} for any choice of $p$ (see for instance~\cite[Theorem~3.1.2, p.~109]{rachrusch}). We deduce the following characterisation of the Wasserstein distance.

\begin{lem}[Optimal coupling]\label{lem:coupopt}
  Let $m, m' \in \Ps(\R)$ and denote $F := H*m$, $G:=H*m'$. Then, for all $p \in [1,+\infty)$,
  \begin{equation*}
    \Ws_p(m,m') = \left(\int_{v=0}^1 |F^{-1}(v)-G^{-1}(v)|^p\dd v\right)^{1/p},
  \end{equation*}
  while
  \begin{equation*}
    \Ws_{\infty}(m,m') = \sup_{v \in (0,1)} |F^{-1}(v)-G^{-1}(v)|.
  \end{equation*}
\end{lem}

Note that, in particular, 
\begin{equation}\label{eq:L1W1}
  \Ws_1(m,m') = ||F-G||_{\Ls^1(\R)}.
\end{equation}

\begin{rk}\label{rk:wassemp}
  In the case of empirical distributions, Lemma~\ref{lem:coupopt} provides a very convenient expression of the Wasserstein distances. More precisely, let $\rx = (x_1, \ldots, x_n)$ and $\rx' = (x'_1, \ldots, x'_n) \in \Dn$, and let us define
  \begin{equation*}
    m := \frac{1}{n} \sum_{k=1}^n \delta_{x_k}, \qquad m' := \frac{1}{n} \sum_{k=1}^n \delta_{x'_k}.
  \end{equation*}
  Then, for all $p \in [1, +\infty)$,
  \begin{equation*}
    \Ws_p(m,m') = \left(\frac{1}{n} \sum_{k=1}^n |x_k-x'_k|^p\right)^{1/p},
  \end{equation*}
  and
  \begin{equation*}
    \Ws_{\infty}(m,m') = \sup_{1 \leq k \leq n} |x_k-x'_k|.
  \end{equation*}
\end{rk}

The Cartesian product $\Ps(\R)^d$ is endowed with the family of distances $\Ws^{(d)}_p$, $p \in [1,+\infty]$, defined by, for all $\bm=(m^1, \ldots, m^d), \bm'=(m'^1, \ldots, m'^d) \in \Ps(\R)^d$,
\begin{equation}\label{eq:Wpd}
  \begin{aligned}
    \forall p \in [1,+\infty), \quad & \Ws^{(d)}_p(\bm, \bm') := \left(\sum_{\gamma=1}^d \Ws_p(m^{\gamma}, m'^{\gamma})^p\right)^{1/p},\\
    & \Ws^{(d)}_{\infty}(\bm, \bm') := \sup_{1 \leq \gamma \leq d} \Ws_{\infty}(m^{\gamma}, m'^{\gamma}).
  \end{aligned}
\end{equation}

Given $\x, \y \in \Dnd$, and letting
\begin{equation*}
  \bm := \left(\frac{1}{n} \sum_{k=1}^n \delta_{x_k^1}, \ldots, \frac{1}{n} \sum_{k=1}^n \delta_{x_k^d}\right), \qquad \bm' := \left(\frac{1}{n} \sum_{k=1}^n \delta_{y_k^1}, \ldots, \frac{1}{n} \sum_{k=1}^n \delta_{y_k^d}\right),
\end{equation*}
it is a straightforward consequence of Remark~\ref{rk:wassemp} that, for all $p \in [1,+\infty]$,
\begin{equation}\label{eq:LsWs}
  ||\x-\y||_p = \Ws^{(d)}_p(\bm, \bm').
\end{equation}


\subsubsection{Construction of a stable semigroup}\label{sss:sg} The existence result of Theorem~\ref{theo:existence} does not depend on the precise way in which the sequence $(\x(n))_{n \geq 1}$ approximates the initial data $(u^1_0, \ldots, u^d_0)$ of~\eqref{eq:syst}. In order to construct semigroup solutions, it is now necessary to specify how to discretise these data. To this aim, we introduce the following {\em discretisation operator} on $\Ps(\R)^d$.

\begin{defi}[Discretisation operator]\label{defi:chi}
  For all $n \geq 1$, we define the {\em discretisation operator} $\chi_n : \Ps(\R)^d \to \Dnd$ by, for all $\bm = (m^1, \ldots, m^d) \in \Ps(\R)^d$, $\chi_n \bm = \x$, where, for all $\gamma:k \in \Part$,
  \begin{equation*}
    x_k^{\gamma} := (n+1)\int_{w=(2k-1)/(2(n+1))}^{(2k+1)/(2(n+1))} (H*m^{\gamma})^{-1}(w)\dd w.
  \end{equation*}
\end{defi}

We can now state the main result of this work, which is based on the remark that, by~\eqref{eq:LsWs}, the discrete stability estimates of Theorem~\ref{theo:stabMSPD} naturally yield Wasserstein stability estimates for the solutions obtained as limits of the MSPD.

\begin{theo}[Convergence of the MSPD to a stable semigroup solution]\label{theo:sg}
  Let Assumptions~\eqref{ass:LC} and \eqref{ass:USH} hold. 
  
  There exists a family of operators $(\bar{\bS}_t)_{t \geq 0}$ on $\Ps(\R)^d$ having the following properties:
  \begin{enumerate}[label=(\roman*), ref=\roman*]
    \item\label{it:sg:sg} for all $s,t \geq 0$, for all $\bm \in \Ps(\R)^d$, $\bar{\bS}_{s+t}\bm = \bar{\bS}_s\bar{\bS}_t\bm$,
    \item\label{it:sg:stable} for all $s,t \geq 0$, for all $\bm, \bm' \in \Ps(\R)^d$, for all $p \in [1,+\infty]$,
    \begin{equation*}
      \Ws^{(d)}_p(\bar{\bS}_s\bm, \bar{\bS}_t\bm') \leq \ConstStab_p \Ws^{(d)}_p(\bm, \bm') + |t-s|\ConstBound{p},
    \end{equation*}
    where $\ConstBound{p}$ is defined in~\eqref{eq:ConstBound} and $\ConstStab_p$ is defined in~\eqref{eq:ConstStab};
  \end{enumerate}
  and such that, for all $\bm \in \Ps(\R)^d$, the function $\bu : [0,+\infty) \times \R \to [0,1]^d$ defined by
  \begin{equation*}
    \forall \gamma \in \{1, \ldots, d\}, \qquad u^{\gamma}(t,x) := H*(\bar{S}_t^{\gamma}\bm)(x),
  \end{equation*}
  satisfies:
  \begin{enumerate}[label=(\roman*), ref=\roman*, start=3]  
    \item\label{it:sg:cvmspd} the sequence of empirical distributions $\upmu[\chi_n\bm]$ converges weakly to the measure $\bar{\upmu}[\bm] \in \Ms$ defined as the image of the Lebesgue measure $\Unif$ on $[0,1]$ by the mapping
    \begin{equation*}
      v \mapsto \left(u^1(t,\cdot)^{-1}(v), \ldots, u^d(t,\cdot)^{-1}(v)\right)_{t \geq 0},
    \end{equation*}
    \item\label{it:sg:sol} the function $\bu$ is a probabilistic solution to the system~\eqref{eq:syst} with initial data $(u^1_0, \ldots, u^d_0)$ defined by $u^{\gamma}_0 = H*m^{\gamma}$.
  \end{enumerate}
\end{theo}

The proof of Theorem~\ref{theo:sg} is detailed in Section~\ref{s:sg}. It works in two steps: we first use the stability estimates of Theorem~\ref{theo:stabMSPD} to prove that the solutions given by Theorem~\ref{theo:existence} with the sequence of initial configurations given by the discretisation operator are semigroups and satisfy the expected Wasserstein stability estimates. We then show that these semigroups are viscosity solutions in the sense of Bianchini and Bressan~\cite{bianbres}, which allows us to identify all the semigroup solutions and thus all the limits of the MSPD. We however prevent ourselves from calling our semigroup solution a viscosity solution, as we do not actually prove that it is indeed the vanishing viscosity limit of the solution to the system~\eqref{eq:syst} with viscosity.

Note that, in Theorem~\ref{theo:sg}, both sides of the inequality in \eqref{it:sg:stable} may be infinite. Let us also highlight the fact that, on account of~\eqref{eq:L1W1}, for $p=1$, the point~\eqref{it:sg:stable} rewrites as a classical $\Ls^1$ stability estimate
\begin{equation*}
  \sum_{\gamma=1}^d ||u^{\gamma}(s, \cdot) - v^{\gamma}(t, \cdot)||_{\Ls^1(\R)} \leq \ConstStab_1 \sum_{\gamma=1}^d ||u^{\gamma}(0, \cdot) - v^{\gamma}(0, \cdot)||_{\Ls^1(\R)} + |t-s| \ConstBound{1},
\end{equation*}
on the probabilistic solutions $\bu=(u^1, \ldots, u^d)$ and $\bv=(v^1, \ldots, v^d)$ to the hyperbolic system~\eqref{eq:syst} defined by
\begin{equation*}
  u^{\gamma}(t,x) := H*(\bar{S}_t^{\gamma}\bm)(x), \qquad v^{\gamma}(t,x) := H*(\bar{S}_t^{\gamma}\bm')(x).
\end{equation*}
We finally remark that the results of Sections~\ref{s:traj} and~\ref{s:rar}, namely the representation of the solutions in terms of trajectories, and the continuity properties of rarefaction coordinates, obviously apply to the probabilistic solutions to~\eqref{eq:syst} given by the semigroup $(\bar{\bS}_t)_{t \geq 0}$.


\subsubsection{Comparison with~\cite{ehm12}}\label{sss:ehm12} Besides Assumption~\eqref{ass:USH}, Theorems~\ref{theo:stabMSPD} and~\ref{theo:sg} are obtained under the sole Assumption~\eqref{ass:LC}. The assumptions made by El Hajj and Monneau in~\cite[Theorem~1.1]{ehm12} to obtain uniqueness and $\Ls^1$ stability of continuous vanishing viscosity solutions to~\eqref{eq:syst} under uniform strict hyperbolicity are more stringent: they assume moreover that the probability measures $m^1,\ldots,m^d$ admit densities in $\Ls\log \Ls(\R)$ and that $\partial_{u^{\gamma}}\lambda^\gamma(\bu)\geq 0$ for all $\gamma\in\{1,\ldots,d\}$ and $\bu\in[0,1]^d$. 

Under the assumption that the probability measures $m^1,\ldots,m^d$ admit bounded densities, they replace strict hyperbolicity by one of the following alternative conditions reinforcing the monotonicity of the characteristic fields $\lambda^\gamma$ in their $\gamma$-th coordinate: 
\begin{itemize}
   \item $\partial_{u^{\gamma'}}\lambda^\gamma(\bu)\geq 0$ for all $\bu\in[0,1]^d$ and $\gamma, \gamma' \in \{1, \ldots, d\}$ with $\gamma'\geq \gamma$,
   \item $\partial_{u^{\gamma'}}\lambda^\gamma(\bu)\leq 0$ for all $\bu\in[0,1]^d$ and $\gamma, \gamma' \in \{1, \ldots, d\}$ with $\gamma' \not= \gamma$, as well as positive semidefiniteness of the matrix $(\inf_{\bu\in[0,1]^d}\partial_{u^{\gamma'}}\lambda^\gamma(\bu)+\inf_{\bu\in[0,1]^d}\partial_{u^{\gamma}}\lambda^{\gamma'}(\bu))_{\gamma,\gamma'}$,
   \item $\partial_{u^{\gamma}}\lambda^\gamma(\bu)\geq \sum_{\gamma'\neq\gamma}(\partial_{u^{\gamma'}}\lambda^\gamma(\bu))^-$ for all $\gamma\in\{1,\ldots,d\}$ and $\bu\in[0,1]^d$, where $v^- = 0 \vee (-v)$ denotes the nonpositive part of $v$.
\end{itemize}

\part{Construction and properties of probabilistic solutions}\label{part:1}

\section{The Multitype Sticky Particle Dynamics}\label{s:mspd}

In this section, we give a formal construction of the Multitype Sticky Particle Dynamics (MSPD). We first recall some useful facts on the Sticky Particle Dynamics in Subsection~\ref{ss:spd}. The proper definition of the MSPD is given in Subsection~\ref{ss:mspd}, where a few elementary properties of this dynamics are also stated. 

\subsection{The Sticky Particle Dynamics}\label{ss:spd} In this subsection, we give a detailed introduction of the Sticky Particle Dynamics and state a few properties of this dynamics.


\subsubsection{Definition of the Sticky Particle Dynamics}\label{sss:spd} Let us fix $\rblambda = (\barlambda_1, \ldots, \barlambda_n) \in \R^n$. For all $\rx = (x_1, \ldots, x_n) \in \Dn$, the {\em Sticky Particle Dynamics started at $\rx$ with initial velocity vector $\rblambda$} is described as follows. 

First, the $k$-th particle has initial position $x_k$ and initial velocity $\barlambda_k$, while its {\em initial cluster} is determined by Definition~\ref{defi:icluSPD}.

\begin{defi}[Initial clusters]\label{defi:icluSPD}
  The {\em initial cluster} of the $k$-th particle in the Sticky Particle Dynamics started at $\rx$ with initial velocity $\rblambda$ is the largest set of consecutive indices $\{\uk, \ldots, \ok\} \subset \{1, \ldots, n\}$ such that:
  \begin{itemize}
    \item $\uk \leq k \leq \ok$,
    \item $x_{\uk} = \cdots = x_{\ok}$,
    \item either $\uk=\ok$, or for all $j \in \{\uk, \ldots, \ok-1\}$,
    \begin{equation}\label{eq:stab}
      \frac{1}{j-\uk+1} \sum_{k'=\uk}^j \barlambda_{k'} \geq \frac{1}{\ok-j} \sum_{k'=j+1}^{\ok} \barlambda_{k'}.
    \end{equation}
  \end{itemize}
\end{defi}

Clusters of particles travel at constant velocity between collisions, and stick together at collisions. The velocity of a cluster between two collisions is given by the average of the initial velocities of the particles composing the cluster. Denoting by 
\begin{equation*}
  \phi[\rblambda](\rx;t) = (\phi_1[\rblambda](\rx;t), \ldots, \phi_n[\rblambda](\rx;t)) \in \Dn
\end{equation*}
the positions of the particles at time $t \geq 0$, we obtain a continuous process $(\phi[\rblambda](\rx;t))_{t \geq 0}$ taking its values in $\Dn$, that we call the {\em Sticky Particle Dynamics started at $\rx$ with initial velocity vector $\rblambda$}. Clearly, this process has the flow property that, for all $s,t \geq 0$,
\begin{equation*}
  \phi[\rblambda](\rx;t+s) = \phi[\rblambda](\phi[\rblambda](\rx;t);s).
\end{equation*}

\begin{rk}
  It follows from a tedious but straightforward barycentric computation that if $\{\uk, \ldots, \ok\}$ and $\{\uk', \ldots, \ok'\}$ are two sets of consecutive indices in $\{1, \ldots, n\}$ satisfying the three conditions of Definition~\ref{defi:icluSPD}, then $\{\uk, \ldots, \ok\} \cup \{\uk', \ldots, \ok'\}$ also satisfies these conditions. Therefore there is no ambiguity in the definition of the initial cluster of the $k$-th particle.
\end{rk}

\begin{defi}[Clusters and their velocity]\label{defi:cluSPD} 
  We denote by $\clu_k[\rblambda](\rx;0)$ the initial cluster of the $k$-th particle, and for $t > 0$, we denote by $\clu_k[\rblambda](\rx;t)$ the largest set of indices $\{\uk, \ldots, \ok\}$ of the particles sharing the same position as the $k$-th particle at time $t$, that is, such that
  \begin{equation*}
    \phi_{\uk}[\rblambda](\rx;t) = \cdots = \phi_k[\rblambda](\rx;t) = \cdots = \phi_{\ok}[\rblambda](\rx;t).
  \end{equation*}
  
  For all $t \geq 0$, the set $\clu_k[\rblambda](\rx;t)$ is called the {\em cluster} at time $t$ of the $k$-th particle in the Sticky Particle Dynamics started at $\rx$ with initial velocity $\rblambda$.
  
  Finally, the {\em velocity} of the cluster of the $k$-th particle at time $t \geq 0$ is defined by
  \begin{equation*}
    v_k[\rblambda](\rx;t) = \frac{1}{|\clu_k[\rblambda](\rx;t)|} \sum_{k' \in \clu_k[\rblambda](\rx;t)} \barlambda_{k'},
  \end{equation*}
  where $|c|$ refers to the cardinality of the set $c$, so that 
  \begin{equation}\label{eq:vspd}
    \forall t \geq 0, \qquad \phi_k[\rblambda](\rx;t) = x_k + \int_{s=0}^t v_k[\rblambda](\rx;s)\dd s.
  \end{equation}
\end{defi}

\begin{rk}\label{rk:cluSPD}
  Definition~\ref{defi:cluSPD} can be completed by the following remarks.
  \begin{enumerate}[label=(\roman*), ref=\roman*]
    \item\label{it:rkcluSPD:1} As is shown in~\cite[Lemma~2.2]{bregre}, in the case $t>0$, the set $\clu_k[\rblambda](\rx;t)$ necessarily satisfies the condition~\eqref{eq:stab}. The latter is called the {\em stability condition}.
    \item\label{it:rkcluSPD:2} As a consequence of the definition of the velocity of a cluster, we have, for all $t \geq 0$,
    \begin{equation}\label{eq:encadrelambda}
      \forall k \in \{1, \ldots, n\}, \qquad \min_{1 \leq j \leq n} \barlambda_j \leq v_k[\rblambda](\rx;t) \leq \max_{1 \leq j \leq n} \barlambda_j.
    \end{equation}
    \item\label{it:rkcluSPD:3} For all $\rx \in \Dn$ and $s,t \geq 0$ such that $s \leq t$, for all $k \in \{1, \ldots, n\}$, 
    \begin{equation*}
      \clu_k[\rblambda](\rx;s) \subset \clu_k[\rblambda](\rx;t).
    \end{equation*}
  \end{enumerate}
\end{rk}

Let us give a representation of the process $(v_1[\rblambda](\rx;t), \ldots, v_n[\rblambda](\rx;t))_{t \geq 0}$, the proof of which can be found in~\cite[Lemma~3.4]{jourey:snoise}.

\begin{lem}[Representation of the velocities]\label{lem:cluSPD}
  For all $\rblambda \in \R^n$, for all $\rx \in \Dn$, there exist right continuous processes $(\gamma_1[\rblambda](\rx;t))_{t \geq 0}, \ldots, (\gamma_{n+1}[\rblambda](\rx;t))_{t \geq 0}$ with values in $\R$ such that, for all $t \geq 0$,
  \begin{itemize}
    \item $\gamma_1[\rblambda](\rx;t) = \gamma_{n+1}[\rblambda](\rx;t) = 0$,
    \item for all $k \in \{2, \ldots, n\}$, $\gamma_k[\rblambda](\rx;t) \geq 0$ and $\gamma_k[\rblambda](\rx;t)(\phi_k[\rblambda](\rx;t) - \phi_{k-1}[\rblambda](\rx;t)) = 0$,
  \end{itemize}
  and, for all $k \in \{1, \ldots, n\}$,
  \begin{equation*}
    v_k[\rblambda](\rx;t) = \barlambda_k + \gamma_k[\rblambda](\rx;t) - \gamma_{k+1}[\rblambda](\rx;t).
  \end{equation*}
\end{lem}

\begin{rk}
  The processes $(\gamma_1[\rblambda](\rx;t))_{t \geq 0}, \ldots, (\gamma_{n+1}[\rblambda](\rx;t))_{t \geq 0}$ introduced in Lemma~\ref{lem:cluSPD} can be interpreted as Lagrange multipliers associated with the constraint that $\phi[\rblambda](\rx;t)$ remain in the polyhedron $\Dn$. More precisely, it is shown in~\cite[Lemma~3.4]{jourey:snoise} that the process $(\phi[\rblambda](\rx;t))_{t \geq 0}$ is the unique solution, in the sense of Tanaka~\cite{tanaka}, to the {\em normally reflected equation}
  \begin{equation*}
    \forall t \geq 0, \qquad \rx(t) = \rx + \rblambda t + \kappa(t)
  \end{equation*}
  at the boundary of $\Dn$, where $\kappa(t)$ is a reflection term, the total variation of which only grows when $\rx(t)$ is at the boundary of $\Dn$.
\end{rk}

We complete this paragraph with the following lemma, which will be useful in the sequel of the article.

\begin{lem}[Extension of the stability condition]\label{lem:extstab}
  Let $\uk, \ok \in \{1, \ldots, n\}$ such that $\uk < \ok$, and such that~\eqref{eq:stab} holds for all $j \in \{\uk, \ldots, \ok-1\}$. Then, for all $\uk', \ok'$ such that $\uk \leq \uk' < \ok' \leq \ok$, we have
  \begin{equation*}
    \frac{1}{\uk'-\uk+1} \sum_{k=\uk}^{\uk'} \barlambda_k \geq \frac{1}{\ok-\ok'+1} \sum_{k=\ok'}^{\ok} \barlambda_k.
  \end{equation*}
\end{lem}
In other words, if one splits a cluster into several smaller clusters, then the leftmost and rightmost clusters tend to get closer to each other.
\begin{proof}
  If $\uk'=\ok'-1$, then the result is a straightforward application of the stability condition~\eqref{eq:stab} with $j=\uk'$. If $\uk'<\ok'-1$, then we define
  \begin{equation*}
    v_{\mathrm{left}} := \frac{1}{\uk'-\uk+1} \sum_{k=\uk}^{\uk'} \barlambda_k, \qquad v_{\mathrm{mid}} := \frac{1}{\ok'-\uk'-1} \sum_{k=\uk'+1}^{\ok'-1} \barlambda_k, \qquad v_{\mathrm{right}} := \frac{1}{\ok-\ok'+1} \sum_{k=\ok'}^{\ok} \barlambda_k.
  \end{equation*}
  Applying the stability condition~\eqref{eq:stab} with $j=\ok'-1$, we obtain
  \begin{equation*}
    (1-\rho_1) v_{\mathrm{left}} + \rho_1 v_{\mathrm{mid}} \geq v_{\mathrm{right}}, \qquad \rho_1 := \frac{\ok'-\uk'-1}{\ok'-\uk} \in (0,1);
  \end{equation*}
  and applying the stability condition~\eqref{eq:stab} with $j=\uk'$, we obtain
  \begin{equation*}
    v_{\mathrm{left}} \geq \rho_2 v_{\mathrm{mid}} + (1-\rho_2)v_{\mathrm{right}}, \qquad \rho_2 := \frac{\ok'-\uk'-1}{\ok-\uk'} \in (0,1).
  \end{equation*}
  We conclude that $v_{\mathrm{left}} \geq v_{\mathrm{right}}$.
\end{proof}


\subsubsection{Local Sticky Particle Dynamics}\label{sss:locspd} Let us fix $T > 0$, $\rx \in \Dn$, and take a set $K \subset \{1, \ldots, n\}$ having the property that
\begin{equation}\label{eq:locSPD}
  \forall k \in K, \qquad \clu_k[\rblambda](\rx;T) \subset K.
\end{equation}
In other words, $K$ is the union of a certain number of clusters at time $T$. By~\eqref{it:rkcluSPD:3} in Remark~\ref{rk:cluSPD}, for all $t \in [0,T]$, all the particles of $K$ belong to clusters contained in $K$. Writing $K = \{k_1, \ldots, k_{|K|}\}$, it is clear that the process 
\begin{equation*}
  (\phi_{k_1}[\rblambda](\rx;t), \ldots, \phi_{k_{|K|}}[\rblambda](\rx;t))_{t \geq 0}
\end{equation*}
follows the Sticky Particle Dynamics in $D_{|K|}$, with initial position vector $(x_{k_1}, \ldots, x_{k_{|K|}})$ and initial velocity vector $(\barlambda_{k_1}, \ldots, \barlambda_{k_{|K|}})$. This is a consequence of the fact that, in the Sticky Particle Dynamics, the interactions between particles are local: when some particles collide and stick together, this does not affect the motion of the other particles.

\begin{defi}[Local Sticky Particle Dynamics]\label{defi:locspd}
  As soon as $T>0$, $\rx \in \Dn$ and $K \subset \{1, \ldots, n\}$ satisfy the condition~\eqref{eq:locSPD}, the process $(\phi_{k_1}[\rblambda](\rx;t), \ldots, \phi_{k_{|K|}}[\rblambda](\rx;t))$ is said to follow the {\em Local Sticky Particle Dynamics} on $[0,T]$, in the set 
  \begin{equation*}
    D_K := \{(x_{k_1}, \ldots, x_{k_{|K|}}) \in \R^K : x_{k_1} \leq \cdots \leq x_{k_{|K|}}\},
  \end{equation*}
  with initial velocity vector $\rblambda_K := (\barlambda_{k_1}, \ldots, \barlambda_{k_{|K|}}) \in \R^K$. 
  
  For $0 \leq t_1 \leq t_2$, we shall also say that $(\phi_{k_1}[\rblambda](\rx;t), \ldots, \phi_{k_{|K|}}[\rblambda](\rx;t))$ follows the Local Sticky Particle Dynamics on $[t_1,t_2]$ if 
  \begin{equation*}
    \left(\phi_{k_1}[\rblambda](\phi[\barlambda](\rx;t_1);t-t_1), \ldots, \phi_{k_{|K|}}[\rblambda](\phi[\barlambda](\rx;t_1);t-t_1)\right)
  \end{equation*}
  follows the Local Sticky Particle Dynamics on $[0,t_2-t_1]$.
\end{defi}

For all $p \in [1, +\infty]$, we now give an estimation on the growth of the $\Ls^p$ distance between two realisations of the (Local) Sticky Particle Dynamics, with possibly distinct initial velocity vectors. 

\begin{prop}[$\Ls^p$ stability for the Local Sticky Particle Dynamics]\label{prop:contractspd}
  Let $\rx, \ry \in \Dn$ and $\rblambda, \rbmu \in \R^n$. Let $T>0$ and $K = \{k_1, \ldots, k_{|K|}\} \subset \{1, \ldots, n\}$ such that the processes 
  \begin{equation*}
    (\phi_{k_1}[\rblambda](\rx;t), \ldots, \phi_{k_{|K|}}[\rblambda](\rx;t))_{t \in [0,T]}
  \end{equation*}
  and 
  \begin{equation*}
    (\phi_{k_1}[\rbmu](\ry;t), \ldots, \phi_{k_{|K|}}[\rbmu](\ry;t))_{t \in [0,T]}
  \end{equation*}
  follow the Local Sticky Particle Dynamics on $[0,T]$, with respective initial velocity vectors $\rblambda_K$ and $\rbmu_K$ defined as above.
 
  \begin{enumerate}[label=(\roman*), ref=\roman*] 
    \item\label{it:contractspd:1} For all $t \in [0,T]$,
    \begin{equation*}
      \sum_{k \in K} |\phi_k[\rblambda](\rx;T) - \phi_k[\rbmu](\ry;T)| \leq \sum_{k \in K} |\phi_k[\rblambda](\rx;t) - \phi_k[\rbmu](\ry;t)| + (T-t) \sum_{k \in K}|\barlambda_k-\bmu_k|,
    \end{equation*}
    \item\label{it:contractspd:2} In the case $\rblambda = \rbmu$, then for all $t \in [0,T]$, for all $p \in [1,+\infty)$,
    \begin{equation*}
      \sum_{k \in K} |\phi_k[\rblambda](\rx;T) - \phi_k[\rblambda](\ry;T)|^p \leq \sum_{k \in K} |\phi_k[\rblambda](\rx;t) - \phi_k[\rblambda](\ry;t)|^p,
    \end{equation*}
    and
    \begin{equation*}
      \sup_{k \in K} |\phi_k[\rblambda](\rx;T) - \phi_k[\rblambda](\ry;T)| \leq \sup_{k \in K} |\phi_k[\rblambda](\rx;t) - \phi_k[\rblambda](\ry;t)|.
    \end{equation*}
  \end{enumerate}
\end{prop}
\begin{proof}
  Without loss of generality, we assume that $K = \{1, \ldots, n\}$, so that $\rblambda_K = \rblambda$ and $\rbmu_K = \rbmu$. Now, by~\eqref{eq:vspd}, for all $p \in [1, +\infty)$,
  \begin{equation*}
    \begin{aligned}
      & \sum_{k=1}^n |\phi_k[\rblambda](\rx;T) - \phi_k[\rbmu](\ry;T)|^p = \sum_{k=1}^n |\phi_k[\rblambda](\rx;t) - \phi_k[\rbmu](\ry;t)|^p\\
      & + \sum_{k=1}^n \int_{s=t}^T p |\phi_k[\rblambda](\rx;s) - \phi_k[\rbmu](\ry;s)|^{p-2} (\phi_k[\rblambda](\rx;s) - \phi_k[\rbmu](\ry;s)) \left\{v_k[\rblambda](\rx;s) - v_k[\rbmu](\ry;s)\right\}\dd s,
    \end{aligned}
  \end{equation*}
  where we take the convention that $|z|^{p-2}z=0$ for $p \in [1,2]$.
  
  Using Lemma~\ref{lem:cluSPD}, we write, for all $k \in \{1, \ldots, n\}$,
  \begin{equation*}
    v_k[\rblambda](\rx;s) - v_k[\rbmu](\ry;s) = \barlambda_k - \bmu_k + \gamma_k[\rblambda](\rx;s) - \gamma_{k+1}[\rblambda](\rx;s) - \gamma_k[\rbmu](\ry;s) + \gamma_{k+1}[\rbmu](\ry;s).
  \end{equation*}
  We shall prove below that, for all $s \in (t,T]$,
  \begin{equation}\label{eq:pf:contractspd:1}
    \sum_{k=1}^n |\phi_k[\rblambda](\rx;s) - \phi_k[\rbmu](\ry;s)|^{p-2} (\phi_k[\rblambda](\rx;s) - \phi_k[\rbmu](\ry;s))\left\{\gamma_k[\rblambda](\rx;s) - \gamma_{k+1}[\rblambda](\rx;s)\right\} \leq 0;
  \end{equation}
  then, by symmetry, the contribution of $-\left\{\gamma_k[\rbmu](\ry;s) - \gamma_{k+1}[\rbmu](\ry;s)\right\}$ is also nonpositive, so that we obtain
  \begin{equation*}
    \begin{aligned}
      & \sum_{k=1}^n |\phi_k[\rblambda](\rx;T) - \phi_k[\rbmu](\ry;T)|^p \leq \sum_{k=1}^n |\phi_k[\rblambda](\rx;t) - \phi_k[\rbmu](\ry;t)|^p\\
      & + \sum_{k=1}^n \left\{\barlambda_k - \bmu_k\right\} \int_{s=t}^T p |\phi_k[\rblambda](\rx;s) - \phi_k[\rbmu](\ry;s)|^{p-2} (\phi_k[\rblambda](\rx;s) - \phi_k[\rbmu](\ry;s))\dd s,
    \end{aligned}
  \end{equation*}
  from which~\eqref{it:contractspd:1} and the first part of~\eqref{it:contractspd:2} easily follow. We derive the second part of~\eqref{it:contractspd:2} by letting $p$ grow to infinity after having taken the power $1/p$ of both sides of the inequality above.
  
  Let us now prove~\eqref{eq:pf:contractspd:1}. To this aim, we fix $s \in (t,T]$ and perform an Abel transform to write 
  \begin{equation*}
    \begin{aligned}
      & \sum_{k=1}^n |\phi_k[\rblambda](\rx;s) - \phi_k[\rbmu](\ry;s)|^{p-2} (\phi_k[\rblambda](\rx;s) - \phi_k[\rbmu](\ry;s))\left\{\gamma_k[\rblambda](\rx;s) - \gamma_{k+1}[\rblambda](\rx;s)\right\}\\
      & = \sum_{k=2}^n \gamma_k[\rblambda](\rx;s) \vartheta(\phi_{k-1}[\rblambda](\rx;s), \phi_{k-1}[\rbmu](\ry;s), \phi_k[\rblambda](\rx;s), \phi_k[\rbmu](\ry;s)),
    \end{aligned}
  \end{equation*}
  where
  \begin{equation*}
    \vartheta(\xi', \zeta', \xi, \zeta) := |\xi-\zeta|^{p-2}(\xi-\zeta) - |\xi'-\zeta'|^{p-2}(\xi'-\zeta'),
  \end{equation*}
  and we have applied Lemma~\ref{lem:cluSPD} to remove $\gamma_1[\rblambda](\rx;s)$ and $\gamma_{n+1}[\rblambda](\rx;s)$. Using Lemma~\ref{lem:cluSPD} again, we recall that $\gamma_k[\rblambda](\rx;s) \geq 0$ and if $\gamma_k[\rblambda](\rx;s) > 0$, then $\phi_{k-1}[\rblambda](\rx;s) = \phi_k[\rblambda](\rx;s)$, while we still have $\phi_{k-1}[\rbmu](\ry;s) \leq \phi_k[\rbmu](\ry;s)$. The conclusion of the proof now follows from the elementary observation that if $\xi'=\xi$ and $\zeta' \leq \zeta$, then $\vartheta(\xi', \zeta', \xi, \zeta) \leq 0$.
\end{proof}


\subsection{Definition of the MSPD}\label{ss:mspd} Let us now give a proper construction of the MSPD. First, in order to define the initial velocities of the particles, we encode the global ordering of a configuration $\x \in \Dnd$ in the set $\Rb(\x)$ defined by
\begin{equation*}
  \Rb(\x) := \{(\alpha:i, \beta:j) \in (\Part)^2 : \alpha < \beta, x_i^{\alpha} < x_j^{\beta}\},
\end{equation*}
and we let $\Nb(\x)$ refer to the cardinality of $\Rb(\x)$. 

Let us fix $\gamma:k \in \Part$ and, for all $\gamma' \not= \gamma$, define $\omega_{\gamma:k}^{\gamma'}(\x) \in [0,1]$ by
\begin{equation*}
  \omega^{\gamma'}_{\gamma:k}(\x) := \left\{\begin{aligned}
    & \frac{1}{n} \sum_{k'=1}^n \ind{(\gamma':k',\gamma:k) \in \Rb(\x)} & \text{if $\gamma' < \gamma$},\\
    & \frac{1}{n} \sum_{k'=1}^n \ind{(\gamma:k,\gamma':k') \not\in \Rb(\x)} & \text{if $\gamma' > \gamma$}. 
  \end{aligned}\right.
\end{equation*}

Under Assumption~\eqref{ass:C}, we can now define the initial velocity of the particle $\gamma:k$ in the MSPD started at $\x$ by
\begin{equation}\label{eq:vitesses}
  \tlambda_k^{\gamma}(\x) := n \int_{w=(k-1)/n}^{k/n} \lambda^{\gamma}\left(\omega_{\gamma:k}^1(\x), \ldots, \omega_{\gamma:k}^{\gamma-1}(\x), w, \omega_{\gamma:k}^{\gamma+1}(\x), \ldots, \omega_{\gamma:k}^d(\x)\right)\dd w,
\end{equation}
and we denote 
\begin{equation}\label{eq:tblambda}
  \tlambda^{\gamma}(\x) := (\tlambda_1^{\gamma}(\x), \ldots, \tlambda_n^{\gamma}(\x)) \in \R^n, \qquad \tblambda(\x) := (\tlambda^1(\x), \ldots, \tlambda^d(\x)) \in (\R^n)^d.
\end{equation}

For all $\x \in \Dnd$, we define the {\em Multitype Sticky Particle Dynamics started at $\x$}, and denote by $(\Phi(\x;t))_{t \geq 0}$, the continuous process taking its values in $\Dnd$ and constructed as follows: as long as there is no collision between particles of different types, each system evolves according to the Sticky Particle Dynamics with initial velocities given by~\eqref{eq:vitesses} above. When particles or clusters of different types collide, say at time $t^* > 0$, then the initial velocity of the particle $\gamma:k$ is updated to the value $\tlambda_k^{\gamma}(\Phi(\x;t^*))$.

Under Assumption~\eqref{ass:USH}, and whatever the composition of the clusters in each system, the velocity of a cluster of type $\alpha$ is always larger than the velocity of a cluster of type $\beta$ if $\alpha < \beta$. Therefore, the set $\Rb(\x)$ contains the pairs of particles $(\alpha:i, \beta:j)$ that will collide at a positive and finite time in the MSPD started at $\x$. At the first collision, say at time $t^* > 0$, between clusters of different types, then the fastest clusters cross the slowest clusters and the systems restart with initial velocities determined by the set $\Rb(\x)$ from which the pairs of particles $(\alpha:i, \beta:j)$ involved in the collision have been removed.

The outline of this subsection is as follows: in~\S\ref{sss:tspd}, we introduce and state a few properties of the {\em Typewise Sticky Particle Dynamics}, which simply describes the joint evolution of $d$ systems of sticky particles, that do not interact with each other. A proper construction of the actual MSPD is made in~\S\ref{sss:mspd}. Continuity properties of this dynamics are stated in~\S\ref{sss:mspdprop} and a peculiar formalism to describe collisions is introduced in~\S\ref{sss:mspdcoll}. Finally, we emphasise the fact that interactions remain local in the MSPD in~\S\ref{sss:mspdloc}. 


\subsubsection{The Typewise Sticky Particle Dynamics}\label{sss:tspd} This paragraph is dedicated to the study of the {\em Typewise Sticky Particle Dynamcis}, which is defined as follows.

\begin{defi}[Typewise Sticky Particle Dynamics]
  Let $\bblambda = (\rblambda^1, \ldots, \rblambda^d)$ be a family of $d$ vectors 
  \begin{equation*}
    \rblambda^{\gamma} = (\barlambda^{\gamma}_1, \ldots, \barlambda^{\gamma}_n) \in \R^n.
  \end{equation*}
  The {\em Typewise Sticky Particle Dynamics} with initial velocity vector $\bblambda$ is the flow $(\tPhi[\bblambda](\cdot;t))_{t \geq 0}$ defined on $\Dnd$ by, for all $\x=(\rx^1, \ldots, \rx^d) \in \Dnd$,
  \begin{equation*}
    \forall t \geq 0, \qquad \tPhi[\bblambda](\x; t) = (\phi[\rblambda^1](\rx^1;t), \ldots, \phi[\rblambda^d](\rx^d;t)).
  \end{equation*}
\end{defi}
In other words, $(\tPhi[\bblambda](\cdot;t))_{t \geq 0}$ describes the joint evolution of $d$ systems of $n$ particles, where the system of particles of type $\gamma$ follows the Sticky Particle Dynamics in $\Dn$ with initial position vector $\rx^{\gamma} := (x^{\gamma}_1, \ldots, x^{\gamma}_n) \in \Dn$ and initial velocity vector $\rblambda^{\gamma} \in \R^n$, independently of the other systems.

\sk
Applying~\eqref{it:contractspd:1} in Proposition~\ref{prop:contractspd} with $K=\{1, \ldots, n\}$ to each system already yields the following contraction property for the Typewise Sticky Particle Dynamics. Let us recall that $||\cdot||_1$ refers to the (normalised) $\Ls^1$ distance in $\Dnd$, see~\eqref{eq:L1distDnd}.

\begin{lem}[$\Ls^1$ contraction]\label{lem:contracttPhi}
  For all $\bblambda, \bbmu \in (\R^n)^d$, for all $\x, \y \in \Dnd$, for all $s,t \geq 0$ such that $s \leq t$,
  \begin{equation*}
    ||\tPhi[\bblambda](\x;t) - \tPhi[\bbmu](\y;t)||_1 \leq ||\tPhi[\bblambda](\x;s) - \tPhi[\bbmu](\y;s)||_1 + \frac{t-s}{n}\sum_{\gamma:k \in \Part} |\barlambda^{\gamma}_k - \barmu^{\gamma}_k|.
  \end{equation*}
\end{lem}

Let $\x \in \Dnd$. In order to define the MSPD started at $\x$ in~\S\ref{sss:mspd} below, we shall of course be concerned with the Typewise Sticky Particle Dynamics with initial velocity vector $\tblambda(\x)$ given by~\eqref{eq:tblambda}, up to the first collision between particles of different types. Therefore, we introduce the {\em collision time} $\ttinter_{\alpha:i, \beta:j}(\x)$ associated with a pair $(\alpha:i, \beta:j) \in \Rb(\x)$ as the time at which the particles $\alpha:i$ and $\beta:j$ collide in the Typewise Sticky Particle Dynamics started at $\x$. The following lemma is a straightforward consequence of Assumption~\eqref{ass:USH} combined with~\eqref{eq:vitesses}, and we do not give a proof.

\begin{lem}[Collision times]\label{lem:ttinter}
  Under Assumptions~\eqref{ass:C} and~\eqref{ass:USH}, let $\x \in \Dnd$ and $(\alpha:i, \beta:j) \in (\Part)^2$ such that $\alpha < \beta$.
  \begin{enumerate}[label=(\roman*), ref=\roman*]
    \item\label{it:ttinter:1} If $(\alpha:i, \beta:j) \not\in \Rb(\x)$, then, for all $t \geq 0$, 
    \begin{equation*}
      \tPhi_i^{\alpha}[\tblambda(\x)](\x;t) \geq  \tPhi_j^{\beta}[\tblambda(\x)](\x;t) + \ConstUSH t.
    \end{equation*}

    \item\label{it:ttinter:2} If $(\alpha:i, \beta:j) \in \Rb(\x)$, then there exists a unique $t =: \ttinter_{\alpha:i, \beta:j}(\x) > 0$ such that 
    \begin{equation*}
      \tPhi_i^{\alpha}[\tblambda(\x)](\x;t) = \tPhi_j^{\beta}[\tblambda(\x)](\x;t).
    \end{equation*}
    Then, for all $s \in [0,\ttinter_{\alpha:i, \beta:j}(\x)]$, 
    \begin{equation*}
      \tPhi_j^{\beta}[\tblambda(\x)](\x;s) - \tPhi_i^{\alpha}[\tblambda(\x)](\x;s) \geq \ConstUSH(\ttinter_{\alpha:i, \beta:j}(\x)-s),
    \end{equation*}
    while, for all $s \geq \ttinter_{\alpha:i, \beta:j}(\x)$, 
    \begin{equation*}
      \tPhi_i^{\alpha}[\tblambda(\x)](\x;s) - \tPhi_j^{\beta}[\tblambda(\x)](\x;s) \geq \ConstUSH(s-\ttinter_{\alpha:i, \beta:j}(\x)).
    \end{equation*}
  \end{enumerate}
\end{lem}

For all $\x \in \Dnd$, we now define $t^*(\x)$ by
\begin{equation}\label{eq:tstar}
  t^*(\x) := \left\{\begin{aligned}
    & +\infty & \text{if $\Nb(\x) = 0$},\\
    & \min\{\ttinter_{\alpha:i, \beta:j}(\x), (\alpha:i, \beta:j) \in \Rb(\x)\} \in (0,+\infty) & \text{otherwise}.
  \end{aligned}\right.
\end{equation}
For all $\x \in \Dnd$ such that $\Nb(\x) \geq 1$, we let $\x^* := \tPhi[\tblambda(\x)](\x; t^*(\x))$. The following corollary of Lemma~\ref{lem:ttinter} is a straightforward consequence of the flow property and the continuity of the trajectories for the Typewise Sticky Particle Dynamics, therefore we do not give a proof.

\begin{cor}[Evolution up to $t^*(\x)$]\label{cor:ttinter}
  Under the assumptions of Lemma~\ref{lem:ttinter}, let $\x \in \Dnd$, $t < t^*(\x)$ and let us denote $\x' := \tPhi[\tblambda(\x)](\x;t)$. Then $\Rb(\x') = \Rb(\x)$, $\tblambda(\x') = \tblambda(\x)$ and $t^*(\x') = t^*(\x) - t$. In addition, if $\Nb(\x) \geq 1$, then $\x'^* = \x^*$ and $\Rb(\x^*)$ is a strict subset of $\Rb(\x)$, so that $\Nb(\x^*) < \Nb(\x)$.
\end{cor}

\subsubsection{Construction of the MSPD}\label{sss:mspd} We are now ready to define the MSPD started at $\x \in \Dnd$.

\begin{defi}[Multitype Sticky Particle Dynamics]\label{defi:mspd}
  Under Assumptions~\eqref{ass:C} and~\eqref{ass:USH}, for all $\x \in \Dnd$, the {\em Multitype Sticky Particle Dynamics} started at $\x$ is the process $(\Phi(\x;t))_{t \geq 0}$, with values in $\Dnd$, defined by
  \begin{equation*}
    \forall t \geq 0, \qquad \Phi(\x;t) := \left\{\begin{aligned}
      & \tPhi[\tblambda(\x)](\x;t) & \text{if $t < t^*(\x)$},\\
      & \Phi(\x^*; t-t^*(\x)) & \text{if $t \geq t^*(\x)$}.
    \end{aligned}\right.
  \end{equation*}
\end{defi}
Since $\Nb(\x)$ is finite and Corollary~\ref{cor:ttinter} asserts that, for all $\x \in \Dnd$ such that $t^*(\x) < +\infty$, $\Nb(\x^*) < \Nb(\x)$, then the process $(\Phi(\x;t))_{t \geq 0}$ is well defined on $[0,+\infty)$.

\sk
Let us recall that, for the Sticky Particle Dynamics with initial position vector $\rx \in \Dn$ and initial velocity vector $\rblambda \in \R^n$, for all $k \in \{1, \ldots, n\}$, the process $(v_k[\rblambda](\rx;s))_{s \geq 0}$ satisfies
\begin{equation*}
  \forall t \geq 0, \qquad \phi_k[\rblambda](\rx;t) = x_k + \int_{s=0}^t v_k[\rblambda](\rx;s)\dd s,
\end{equation*}
see Definition~\ref{defi:cluSPD}. Now, for all $\x \in \Dnd$, for all $\gamma:k \in \Part$, we define the process $(v_k^{\gamma}(\x;s))_{s \geq 0}$ by 
\begin{equation}\label{eq:defvmspd}
  v_k^{\gamma}(\x; s) := \left\{\begin{aligned}
    & v_k[\tlambda^{\gamma}(\x)](\rx^{\gamma};s) & \text{if $s < t^*(\x)$},\\
    & v_k^{\gamma}(\x^*; s-t^*(\x)) & \text{if $s \geq t^*(\x)$},
  \end{aligned}\right.
\end{equation}
so that
\begin{equation*}
  \forall t \geq 0, \qquad \Phi_k^{\gamma}(\x;t) = x_k^{\gamma} + \int_{s=0}^t v_k^{\gamma}(\x;s) \dd s.
\end{equation*}
We easily deduce from this definition and~\eqref{eq:encadrelambda}-\eqref{eq:vitesses} that, for all $\x \in \Dnd$, for all $t \geq 0$,
\begin{equation}\label{eq:typeencadrelambda}
  \inf_{\bu \in [0,1]^d} \lambda^{\gamma}(\bu) \leq v_k^{\gamma}(\x; t) \leq \sup_{\bu \in [0,1]^d} \lambda^{\gamma}(\bu).
\end{equation}

We are now willing to define the {\em cluster} of a particle in the MSPD started at $\x$, similarly to Definition~\ref{defi:cluSPD} above. In this purpose, we first introduce the notion of {\em generical cluster}.

\begin{defi}[Generical clusters]
  A {\em generical cluster} is a pair $(\gamma, \{\uk, \ldots, \ok\})$, where $\gamma \in \{1, \ldots, d\}$ is the {\em type} of the generical cluster and $\{\uk, \ldots, \ok\}$ is a set of consecutive indices in $\{1, \ldots, n\}$. To refer to the generical cluster $c := (\gamma, \{\uk, \ldots, \ok\})$, we shall rather use the notation $c = \gamma : \uk \cdots \ok$.
\end{defi} 
Let us give a few rules to manipulate generical clusters.
\begin{itemize}
  \item The type of a generical cluster $c$ is denoted by $\type(c) \in \{1, \ldots, d\}$.
  \item The cardinality of a generical cluster $c = \gamma : \uk \cdots \ok$ is denoted by $|c|$ and worth $\ok-\uk+1$.
  \item For $\gamma':k' \in \Part$ and $c = \gamma : \uk \cdots \ok$, we shall write
  \begin{equation*}
    \gamma':k' \in c
  \end{equation*}
  if and only if $\gamma'=\gamma$ and $k' \in \{\uk, \ldots, \ok\}$. This set membership relation allows us to define the inclusion relation $a \subset b$ between generical clusters $a$ and $b$ as well as the union set $a \cup b$ and the Cartesian product $a \times b$ of two generical clusters $a$ and $b$.
  \item A generical cluster $\gamma:k \cdots k$ with a single element $\gamma:k$ shall rather be denoted by $\gamma:k$. It will always be clear from the context whether the notation $\gamma:k$ refers to a particle (that is, an element of $\Part$) or to a cluster containing a single particle.
\end{itemize}

We can now define the {\em cluster} of a particle in the MSPD started at $\x \in \Dnd$.

\begin{defi}[Cluster]
  The cluster of the particle $\gamma:k$ in the configuration $\Phi(\x;t)$ is the generical cluster defined by
  \begin{equation*}
    \clu_k^{\gamma}(\x; t) := \left\{\begin{aligned}
      & \gamma:\clu_k[\tlambda^{\gamma}(\x)](\rx^{\gamma};t) & \text{if $t < t^*(\x)$},\\
      & \clu_k^{\gamma}(\x^*; t-t^*(\x)) & \text{if $t \geq t^*(\x)$},
    \end{aligned}\right.
  \end{equation*}
  where we recall that $\clu_k[\tlambda^{\gamma}(\x)](\rx^{\gamma};t)$ was defined in Definition~\ref{defi:cluSPD}.
\end{defi}


\subsubsection{Continuity properties of the MSPD}\label{sss:mspdprop} In this paragraph, we state some continuity properties for the MSPD in Propositions~\ref{prop:mspd} and~\ref{prop:continuity}, the proofs of which are postponed to Subsection~\ref{app:pf:cont} in Appendix~\ref{app:proofs}.

\begin{prop}[Time continuity and flow]\label{prop:mspd}
  For all $\x \in \Dnd$, the process $(\Phi(\x;t))_{t \geq 0}$ has continuous trajectories in $\Dnd$. Besides, $(\Phi(\cdot;t))_{t \geq 0}$ defines a flow in $\Dnd$.
\end{prop}

For $p \in [1, +\infty]$, we recall the Definition~\ref{defi:Lspdist} of the (normalised) $\Ls^p$ distance on $\Dnd$, and denote
\begin{equation*}
  B_p(\x, \delta) := \{\y \in \Dnd : ||\x-\y||_p < \delta\}, \qquad \barB_p(\x, \delta) := \{\y \in \Dnd : ||\x-\y||_p \leq \delta\}.
\end{equation*}

\begin{prop}[Continuity with respect to the initial configuration]\label{prop:continuity}
  Let $\x \in \Dnd$. Then, for all $\epsilon > 0$, there exists $\delta > 0$ such that, for all $\y \in \barB_1(\x, \delta)$,
  \begin{equation*}
    \sup_{t \geq 0} || \Phi(\x;t) - \Phi(\y;t) ||_1 \leq \epsilon.
  \end{equation*}
\end{prop}


\subsubsection{Collision times}\label{sss:mspdcoll} For all $\x \in \Dnd$, for all $(\alpha:i, \beta:j) \in (\Part)^2$ such that $\alpha < \beta$, let us define
\begin{equation*}
  \tinter_{\alpha:i, \beta:j}(\x) := \inf\{t \geq 0 : \Phi_i^{\alpha}(\x;t) \geq \Phi_j^{\beta}(\x;t)\}.
\end{equation*}
Certainly, Assumption~\eqref{ass:USH} ensures that $\tinter_{\alpha:i, \beta:j}(\x) < +\infty$; while $\tinter_{\alpha:i, \beta:j}(\x) > 0$ if and only if $(\alpha:i, \beta:j) \in \Rb(\x)$. Besides, it is easily checked that 
\begin{equation*}
  t^*(\x) = \left\{\begin{aligned}
    & +\infty & \text{if $\Nb(\x)=0$},\\
    & \min\{\tinter_{\alpha:i, \beta:j}(\x), (\alpha:i, \beta:j) \in \Rb(\x)\} & \text{if $\Nb(\x) \geq 1$}.
  \end{aligned}\right.
\end{equation*}

For all $(\alpha:i, \beta:j) \in \Rb(\x)$, $\tinter_{\alpha:i, \beta:j}(\x)$ is nothing but the time at which the particles $\alpha:i$ and $\beta:j$ collide in the MSPD started at $\x$. On the contrary, if $(\alpha:i, \beta:j) \not\in \Rb(\x)$, then $\tinter_{\alpha:i,\beta:j}(\x)=0$, which is somehow consistant with the intuitive idea that the collision between $\alpha:i$ and $\beta:j$ happened `before the origin of times', which we shall refer to as the {\em virtual past}.

Assumption~\eqref{ass:USH} implies that the collision times $\tinter_{\alpha:i,\beta:j}(\x)$ have properties similar to those described in Lemma~\ref{lem:ttinter} for the collision times $\ttinter_{\alpha:i,\beta:j}(\x)$ in the Typewise Sticky Particle Dynamics. As a consequence, we state the following lemma without a demonstration.

\begin{lem}[Collision times in the MSPD]\label{lem:tinter}
  Let $\x \in \Dnd$ and $(\alpha:i, \beta:j) \in \Rb(\x)$. Then $\tinter_{\alpha:i,\beta:j}(\x) > 0$, and:
  \begin{itemize}
    \item for all $s \in [0,\tinter_{\alpha:i,\beta:j}(\x)]$, $\Phi_j^{\beta}(\x;s) - \Phi_i^{\alpha}(\x;s) \geq \ConstUSH (\tinter_{\alpha:i,\beta:j}(\x)-s)$,
    \item for all $s \geq \tinter_{\alpha:i,\beta:j}(\x)$, $\Phi_i^{\alpha}(\x;s) - \Phi_j^{\beta}(\x;s) \geq \ConstUSH (s-\tinter_{\alpha:i,\beta:j}(\x)).$
  \end{itemize}
\end{lem}

\subsubsection{Local interactions}\label{sss:mspdloc} We finally explain why the interactions in the MSPD remain local, in the sense of~\S\ref{sss:locspd}. Indeed, according to Definition~\ref{defi:mspd}, if $\Nb(\x) \geq 1$, then at the first instant $t^*(\x)$ of a collision between two particles of different types, the whole system restarts with new initial velocities determined by $\tblambda(\x^*)$. Therefore, the velocities of all the particles could be modified.

The following lemma ensures that only the velocities of the particles involved in a collision with particles of another type at time $t^*(\x)$ are actually modified. It is first useful to define the set 
\begin{equation}\label{eq:tau}
  \Ttau_{\gamma:k}(\x) := \{\tinter_{\alpha:i, \beta:j}(\x) : (\alpha:i, \beta:j) \in \Rb(\x), \gamma:k \in \{\alpha:i, \beta:j\}\}
\end{equation}
of instants at which the particle $\gamma:k$ collides with particles of different types in the MSPD started at $\x$. For all $T \geq 0$, we also let
\begin{equation}\label{eq:T-wedgetau}
  T^- \wedge \Ttau_{\gamma:k}(\x) := \left\{\begin{aligned}
    & \text{$0$ if the set $\Ttau_{\gamma:k}(\x) \cap [0,T)$ is empty},\\
    & \text{$\max(\Ttau_{\gamma:k}(\x) \cap [0,T))$ otherwise}.
  \end{aligned}\right.
\end{equation}
Note that $0 \leq T^- \wedge \Ttau_{\gamma:k}(\x) < T$.

\begin{lem}[Locality of the interactions in the MSPD]\label{lem:locintmspd}
  Let $\Ttau_{\gamma:k}(\x)$ be defined as above. 
  \begin{enumerate}[label=(\roman*), ref=\roman*]
    \item\label{it:locintmspd:1} For all $\gamma:k \in \Part$, if $t^*(\x) \not\in \Ttau_{\gamma:k}(\x)$, then 
    \begin{equation*}
      \tlambda_k^{\gamma}(\x^*) = \tlambda_k^{\gamma}(\x).
    \end{equation*}
    \item\label{it:locintmspd:2} For all $T > 0$, for all $\gamma \in \{1, \ldots, d\}$, if $K \subset \{1, \ldots, n\}$ is such that, for all $k \in K$,
    \begin{equation*}
      \clu_k^{\gamma}(\x;T) \subset \gamma:K
    \end{equation*}
    (with an obvious notation for $\gamma:K$), then the process $\{\Phi_k^{\gamma}(\x;t) : k \in K\}$ follows the Local Sticky Particle Dynamics, in the sense of Definition~\ref{defi:locspd}, on the interval $[t_0,T]$ with
    \begin{equation*}
      t_0 := \max_{k \in K}~ T^- \wedge \Ttau_{\gamma:k}(\x),
    \end{equation*}
    with initial velocity vector $\rblambda_K := (\barlambda_k)_{k \in K}$ defined by
    \begin{equation*}
      \forall k \in K, \qquad \barlambda_k := \tlambda_k^{\gamma}(\Phi(\x;t_0)).
    \end{equation*}
  \end{enumerate}
\end{lem}
\begin{proof}
  We first address~\eqref{it:locintmspd:1} and let $\gamma:k \in \Part$ such that $t^*(\x) \not\in \Ttau_{\gamma:k}(\x)$. Then, due to the definition of $\tlambda_k^{\gamma}(\x^*)$, it suffices to check that, for all $\gamma' \not= \gamma$, 
  \begin{equation*}
    \omega_{\gamma:k}^{\gamma'}(\x^*) = \omega_{\gamma:k}^{\gamma'}(\x).
  \end{equation*}
  We describe the case $\gamma' < \gamma$, the reverse case is symmetric. The equality above holds if and only if, for all $k' \in \{1, \ldots, n\}$, 
  \begin{equation*}
    (\gamma':k', \gamma:k) \in \Rb(\x) \qquad \text{if and only if} \qquad (\gamma':k', \gamma:k) \in \Rb(\x^*),
  \end{equation*}
  that is to say
  \begin{equation*}
    x_{k'}^{\gamma'} < x_k^{\gamma} \qquad \text{if and only if} \qquad \Phi_{k'}^{\gamma'}(\x;t^*(\x)) < \Phi_k^{\gamma}(\x;t^*(\x)),
  \end{equation*}
  which obviously holds true since $t^*(\x) \not\in \Ttau_{\gamma:k}(\x)$ implies that the particle $\gamma:k$ does not collide with any particle $\gamma':k'$ on $[0, t^*(\x)]$.
  
  The point~\eqref{it:locintmspd:2} is now an easy consequence of the choice of $t_0$, which ensures that, for all $k \in K$, the particle $\gamma:k$ does not collide with a particle of another type in the time interval $(t_0, T)$.
\end{proof}


\section{Construction of probabilistic solutions by approximation}\label{s:existence}

In this section, we detail the proof of Theorem~\ref{theo:existence}, which in particular provides existence of probabilistic solutions to~\eqref{eq:syst} under Assumptions~\eqref{ass:C} and~\eqref{ass:USH}. In Subsection~\ref{ss:closedness}, we first state a closedness property on the set of probabilistic solutions to~\eqref{eq:syst}. In Subsection~\ref{ss:MSPDsol}, we show that, for all $\x \in \Dnd$, the vector of empirical CDFs of the MSPD is an {\em exact} probabilistic solution to the system~\eqref{eq:syst}, but with {\em discrete} initial data induced by $\x$. Taking a sequence of initial configurations $(\x(n))_{n \geq 1}$ approximating the actual initial data $(u^1_0, \ldots, u^d_0)$ of the system~\eqref{eq:syst}, we finally combine the closedness property of Subsection~\ref{ss:closedness} with a tightness argument to complete the proof of Theorem~\ref{theo:existence} in Subsection~\ref{ss:pfexistence}.


\subsection{Closedness of the set of probabilistic solutions}\label{ss:closedness} This subsection contains the statement of Proposition~\ref{prop:closedness}, the proof of which is postponed to Section~\ref{app:pf:closedness} in Appendix~\ref{app:proofs}. 

\begin{prop}[Closedness of the set of probabilistic solutions]\label{prop:closedness}
  Under Assumption~\eqref{ass:C}, let $(\bu_n)_{n \geq 1}$ be a sequence of functions
  \begin{equation*}
    \bu_n = (u_n^1, \ldots, u_n^d) : [0,+\infty) \times \R \to [0,1]^d
  \end{equation*}
  such that:
  \begin{itemize}
    \item for all $n \geq 1$, the function $\bu_n$ is a probabilistic solution to the system~\eqref{eq:syst} with initial data $(u_{0,n}^1, \ldots, u_{0,n}^d)$,
    \item for all $t \geq 0$, for all $\gamma \in \{1, \ldots, d\}$, there exists a CDF $u^{\gamma}(t,\cdot)$ on the real line such that, for all $x \in \R$ for which $\Delta_x u^{\gamma}(t,x) = 0$, 
    \begin{equation*}
      \lim_{n \to +\infty} u_n^{\gamma}(t,x) = u^{\gamma}(t,x),
    \end{equation*}
    \item for all $\gamma, \gamma' \in \{1, \ldots, d\}$ such that $\gamma \not= \gamma'$, 
    \begin{equation}\label{eq:mutgammutgamp}
      \text{$\dd t$-almost everywhere,} \qquad \forall x \in \R, \quad \Delta_x u^{\gamma}(t,x)\Delta_x u^{\gamma'}(t,x) = 0.
    \end{equation}
  \end{itemize}
  Then the function $\bu = (u^1, \ldots, u^d) : [0,+\infty) \times \R \to [0,1]^d$ is a probabilistic solution to the system~\eqref{eq:syst} with initial data $(u_0^1, \ldots, u_0^d)$ defined by $u_0^{\gamma}(x) = u^{\gamma}(0,x)$.
\end{prop}


\subsection{Empirical CDFs of the MSPD}\label{ss:MSPDsol} For all $\x \in \Dnd$, recall the Definition~\ref{defi:muMSPD} of the vector of empirical CDFs $\bu[\x]$ of the MSPD started at $\x$. Let us check that the trajectory $(\Phi_k^{\gamma}(\x;t))_{t \geq 0}$ is Lipschitz continuous, and satisfies the characteristic equation
\begin{equation}\label{eq:MSPD:diff}
  \forall \gamma:k \in \Part, \qquad \dot{\Phi}^{\gamma}_k(\x;t) = \lambda^{\gamma}\{\bu[\x]\}(t, \Phi_k^{\gamma}(\x;t)), \qquad \text{$\dd t$-almost everywhere}.
\end{equation}
To this aim, let us fix $t \geq 0$ outside of the finite set $\{\tinter_{\alpha:i, \beta:j}(\x), (\alpha:i, \beta:j) \in \Rb(\x)\}$. We claim that, for all $\gamma:k \in \Part$,
\begin{equation}\label{eq:pf:MSPDsol:1}
  \lambda^{\gamma}\{\bu[\x]\}(t, \Phi_k^{\gamma}(\x;t)) = v_k^{\gamma}(\x;t),
\end{equation}
where we recall the definition~\eqref{eq:defvmspd} of $v_k^{\gamma}(\x;t)$. Clearly,~\eqref{eq:pf:MSPDsol:1} implies the characteristic equation~\eqref{eq:MSPD:diff}. To obtain~\eqref{eq:pf:MSPDsol:1}, fix $\gamma:k \in \Part$ and write $x := \Phi_k^{\gamma}(\x,t)$, $\gamma:\uk \cdots \ok := \clu_k^{\gamma}(\x;t)$. Then
\begin{equation*}
  u^{\gamma}[\x](t,x^-) = \frac{\uk-1}{n}, \quad u^{\gamma}[\x](t,x) = \frac{\ok}{n} \quad \text{and} \quad \Delta_x u^{\gamma}[\x](t,x) = \frac{\ok-\uk+1}{n} > 0.
\end{equation*}
As a consequence,
\begin{equation*}
  \lambda^{\gamma}\{\bu[\x]\}(t,x) = \frac{n}{\ok-\uk+1} \int_{w=(\uk-1)/n}^{\ok/n} \lambda^{\gamma}\left(u^1[\x](t,x), \ldots, w, \ldots, u^d[\x](t,x)\right)\dd w.
\end{equation*}
The choice of $t$ implies that, for all $\gamma' \in \{1, \ldots, d\}$ such that $\gamma \not= \gamma'$,
\begin{equation*}
  \Delta_x u^{\gamma'}[\x](t,x) = 0,
\end{equation*}
therefore, for all $k' \in \{\uk, \ldots, \ok\}$,
\begin{equation*}
  u^{\gamma'}[\x](t,x) = \omega^{\gamma'}_{\gamma:k'}(\Phi(\x;t)).
\end{equation*}
As a conclusion,
\begin{equation*}
  \lambda^{\gamma}\{\bu[\x]\}(t,x) = \frac{1}{\ok-\uk+1} \sum_{k'=\uk}^{\ok} \tlambda_{k'}^{\gamma}(\Phi(\x;t)) = v_k^{\gamma}(\x;t),
\end{equation*}
hence~\eqref{eq:pf:MSPDsol:1}.

\sk
We deduce the following proposition.
\begin{prop}[The MSPD provides an exact solution to~\eqref{eq:syst}]\label{prop:MSPDsol}
  Under Assumptions~\eqref{ass:C} and~\eqref{ass:USH}, for all $\x \in \Dnd$, the vector of empirical CDFs $\bu[\x]$ defined by~\eqref{eq:bun} is a probabilistic solution to the system~\eqref{eq:syst}, with initial data $(u^1_0[\x], \ldots, u^d_0[\x])$ defined by~\eqref{eq:bun0}.
\end{prop}
\begin{proof}
  Let us fix $\x \in \Dnd$. By construction, for all $t \geq 0$, for all $\gamma \in \{1, \ldots, d\}$, $u^{\gamma}[\x](t,\cdot)$ is a CDF on the real line. In order to prove that it is a probabilistic solution to the system~\eqref{eq:syst}, we first check that, for all $\gamma \in \{1, \ldots, d\}$, the function $u^{\gamma}[\x]$ is measurable on $[0,+\infty) \times \R$. Then, we check that $\bu[\x]$ satisfies~\eqref{it:sol:3} in Definition~\ref{defi:sol}. 
  
  \sk
  \noindent{\em Proof of measurability.} Recall that $u^{\gamma}[\x](t,\cdot)$ writes $H*\upmu^{\gamma}_t[\x]$. In this definition, replace the Heaviside $H$ with its continuous approximation $H_l$ defined by, for all $l \geq 1$,
  \begin{equation*}
    H_l(x) = \left\{\begin{array}{ll}
      0 & \text{if $x \leq -1/l$},\\
      1+lx & \text{if $-1/l < x < 0$},\\
      1 & \text{if $x \geq 0$},
    \end{array}\right.
  \end{equation*}
  so as to define $u^{\gamma}_l[\x](t,\cdot) := H_l*\upmu^{\gamma}_t[\x]$. Then, on the one hand, for all $t \geq 0$, the function $x \mapsto u^{\gamma}_l[\x](t,x)$ is continuous and nondecreasing on $\R$, hence Dini's Theorem implies that $u^{\gamma}_l[\x]$ is continuous, and therefore measurable, on $[0,+\infty)\times\R$. On the other hand, $H_l(x)$ converges to $H(x)$ for all $x \in \R$, therefore $u^{\gamma}[\x]$ is the pointwise limit of $u^{\gamma}_l[\x]$, which completes the proof.
  
  \sk
  \noindent{\em Proof of~\eqref{it:sol:3} in Definition~\ref{defi:sol}.} Let us fix $\bvarphi = (\varphi^1, \ldots, \varphi^d) \in \Cs^{1,0}_{\mathrm{c}}([0,+\infty)\times\R, \R^d)$ and, for all $\gamma \in \{1, \ldots, d\}$, define $\psi^{\gamma}$ by
  \begin{equation*}
    \forall (t,x) \in [0,+\infty) \times \R, \qquad \psi^{\gamma}(t,x) := \int_{y=x}^{+\infty} \varphi^{\gamma}(t,y) \dd y.
  \end{equation*}
  Owing to~\eqref{eq:pf:MSPDsol:1}, the chain rule formula for functions of finite variation~\cite[Proposition~(4.6), p.~6]{revuz} yields, for all $T \geq 0$, for all $\gamma:k \in \Part$,
  \begin{equation*}
    \begin{aligned}
      \psi^{\gamma}(T, \Phi_k^{\gamma}(\x;T)) & = \psi^{\gamma}(0,x_k^{\gamma}) + \int_{t=0}^T \left( \partial_t\psi^{\gamma}(t, \Phi_k^{\gamma}(\x;t)) + \partial_x\psi^{\gamma}(t,  \Phi_k^{\gamma}(\x;t)) \lambda^{\gamma}\{\bu[\x]\}(t, \Phi_k^{\gamma}(\x;t)) \right) \dd t\\
      & = \psi^{\gamma}(0,x_k^{\gamma}) + \int_{t=0}^T \left( \partial_t\psi^{\gamma}(t,  \Phi_k^{\gamma}(\x;t)) - \varphi^{\gamma}(t,  \Phi_k^{\gamma}(\x;t)) \lambda^{\gamma}\{\bu[\x]\}(t, \Phi_k^{\gamma}(\x;t)) \right) \dd t.
    \end{aligned}
  \end{equation*}
  Since $\varphi^{\gamma}$ has a compact support, the left-hand side above vanishes when $T$ grows to infinity, and taking the average of both sides for $k \in \{1, \ldots, n\}$ yields
  \begin{equation*}
    0 = \int_{x \in \R} \psi^{\gamma}(0,x)\dd u^{\gamma}_0[\x](x) + \int_{t=0}^{+\infty} \int_{x \in \R} \left( \partial_t\psi^{\gamma}(t,x) - \varphi^{\gamma}(t, x) \lambda^{\gamma}\{\bu[\x]\}(t, x) \right) \dd_x u^{\gamma}[\x](t,x) \dd t.
  \end{equation*}
  By the Fubini Theorem,
  \begin{equation*}
    \int_{x \in \R} \psi^{\gamma}(0,x)\dd u^{\gamma}_0[\x](x) = \int_{(x,y) \in \R^2} \ind{x \leq y} \varphi^{\gamma}(0,y) \dd u^{\gamma}_0[\x](x)\dd y = \int_{y \in \R} \varphi^{\gamma}(0,y) u^{\gamma}_0[\x](y) \dd y,
  \end{equation*}
  and we similarly obtain, for all $t \geq 0$,
  \begin{equation*}
    \int_{x \in \R} \partial_t\psi^{\gamma}(t,x) \dd_x u^{\gamma}[\x](t,x) = \int_{y \in \R} \partial_t\varphi^{\gamma}(t,y) u^{\gamma}[\x](t,y) \dd y.
  \end{equation*}
  As a consequence,
  \begin{equation*}
    \begin{aligned}
      & \int_{t=0}^{+\infty} \int_{y \in \R} \partial_t\varphi^{\gamma}(t,y) u^{\gamma}[\x](t,y)\dd y \dd t +  \int_{y \in \R} \varphi^{\gamma}(0,y) u^{\gamma}_0[\x](y)\dd y\\
      & \qquad = \int_{t=0}^{+\infty} \int_{x \in \R} \varphi^{\gamma}(t, x) \lambda^{\gamma}\{\bu[\x]\}(t, x) \dd_x u^{\gamma}[\x](t,x) \dd t,
    \end{aligned}
  \end{equation*}
  and we complete the proof by taking the sum of both sides for $\gamma \in \{1, \ldots, d\}$.
\end{proof}

\begin{rk}
  Proposition~\ref{prop:MSPDsol} provides easy examples for which the uniqueness of probabilistic solutions to~\eqref{eq:syst} fails. Indeed, fix $\x \in \Dnd$ and define $\hat{\x} \in \Dndeuxd$ by, for all $\gamma \in \{1, \ldots, d\}$, for all $k \in \{1, \ldots, n\}$,
\begin{equation*}
  \hat{x}_{2k-1}^{\gamma} = \hat{x}_{2k}^{\gamma} := x_k^{\gamma}.
\end{equation*}
Then, for all $\gamma \in \{1, \ldots, d\}$, the empirical distributions
\begin{equation*}
  \upmu^{\gamma}_0[\x] := \frac{1}{n} \sum_{k=1}^n \delta_{x_k^{\gamma}} \qquad \text{and} \qquad \upmu^{\gamma}_0[\hat{\x}] := \frac{1}{2n} \sum_{k=1}^{2n} \delta_{\hat{x}_k^{\gamma}}
\end{equation*}
coincide in $\Ps(\R)$. As a consequence, by Proposition~\ref{prop:MSPDsol}, the vectors of empirical CDFs $\bu[\x]$ and $\bu[\hat{\x}]$ are probabilistic solutions to the system~\eqref{eq:syst} with {\em the same} initial data.

But let us assume that there exists $\gamma \in \{1, \ldots, d\}$ such that $u \mapsto \lambda^{\gamma}(u^1, \ldots, u^{\gamma-1}, u, u^{\gamma+1}, \ldots, u^d)$ be increasing, for all $(u^1, \ldots, u^{\gamma-1}, u^{\gamma+1}, \ldots, u^d) \in [0,1]^{d-1}$. Then, in the MSPD started at $\hat{\x}$, the particles of type $\gamma$ instantaneously drift away from each other. As a consequence, for all $t \in (0, t^*(\hat{\x}))$, the marginal distribution $\upmu^{\gamma}_t[\hat{\x}]$ has exactly $2n$ atoms, while the marginal distribution $\upmu^{\gamma}_t[\x]$ possesses at most $n$ atoms. Therefore, the corresponding solutions to the system~\eqref{eq:syst} do not coincide.
\end{rk}


\subsection{Proof of Theorem~\ref{theo:existence}}\label{ss:pfexistence} The proof of Theorem~\ref{theo:existence} is based on a tightness argument for the empirical distribution of the MSPD. We recall that a sequence of probability measures $(\mu_n)_{n \geq 1}$ on some metric space $E$ is said to be {\em tight} if, for all $\epsilon > 0$, there exists a compact subset $K$ of $E$ such that $\mu_n(K) \geq 1-\epsilon$ for all $n \geq 1$~\cite[p.~8]{billingsley}. If $(\mu_n)_{n \geq 1}$ is tight, then Prohorov's Theorem~\cite[Theorem~5.1, p.~59]{billingsley} asserts that from each subsequence of $(\mu_n)_{n \geq 1}$, one can extract a further subsequence weakly converging to some $\mu \in \Ps(E)$. Conversely, if $E$ is complete and separable, then any sequence of probability measures $(\mu_n)_{n \geq 1}$ on $E$ of which every subsequence contains a weakly converging further subsequence is tight~\cite[Theorem~5.2, p.~60]{billingsley}. We finally recall that the set $\Cs([0,+\infty), \R^d)$, endowed with the topology of the uniform convergence on the compact sets of $[0,+\infty)$, is complete and separable; this follows from a slight adaptation of~\cite[Example~1.3, p.~11]{billingsley}.

\begin{prop}[Convergence of the MSPD]\label{prop:tightness}
  Under the assumptions of Theorem~\ref{theo:existence}, the sequence $(\upmu[\x(n)])_{n \geq 1}$ is tight. Besides, if $\bar{\upmu} \in \Ms$ refers to the limit of a converging subsequence, then for all $\gamma, \gamma' \in \{1, \ldots, d\}$ such that $\gamma \not= \gamma'$, the marginal probability measures $\bar{\upmu}^{\gamma}_t$ and $\bar{\upmu}^{\gamma'}_t$ have distinct atoms $\dd t$-almost everywhere.
\end{prop}
\begin{proof}
  Let us fix $T > 0$ and denote 
  \begin{equation*}
    \upmu_{[0,T]}[\x(n)] := \frac{1}{n} \sum_{k=1}^n \delta_{(\Phi_k^{\gamma}(\x(n);t))_{t \in [0,T]}} \in \Ps(\Cs([0,T],\R^d))
  \end{equation*}
  the empirical distribution of the restriction of the MSPD started at $\x(n)$ to $[0,T]$. We first prove that the sequence $(\upmu_{[0,T]}[\x(n)])_{n \geq 1}$ is tight on $\Cs([0,T],\R^d)$, using~\cite[Theorem~7.3, p.~82]{billingsley}, which is a consequence of the Arzelà-Ascoli Theorem. To apply this theorem, we need (i) to prove that the sequence of marginal distributions $\upmu_0[\x(n)] \in \Ps(\R^d)$ is tight, and (ii) to exhibit a uniform (in $n$) control on the modulus of continuity of the sample-paths of the MSPD started at $\x(n)$. 
  
  The point~(i) is obtained as follows: by the assumptions on the sequence $(\x(n))_{n \geq 1}$, the marginal distributions $\upmu^1_0[\x(n)], \ldots \upmu^d_0[\x(n)] \in \Ps(\R)$ of $\upmu_0[\x(n)] \in \Ps(\R^d)$ are weakly converging. Since $\R$ is complete and separable, we deduce that these marginal distributions are tight, which, by an easy adaptation of~\cite[Exercise~5.9, p.~65]{billingsley}, implies that the sequence $(\upmu_0[\x(n)])_{n \geq 1}$ itself is tight.
  
  The point~(ii) follows from the fact that, by~\eqref{eq:typeencadrelambda}, for all $n \geq 1$, for all $k \in \{1, \ldots, n\}$, the process
  \begin{equation*}
    \left(\Phi_k^1(\x(n);t), \ldots, \Phi_k^d(\x(n);t)\right)_{t \in [0,T]}
  \end{equation*}
  satisfies the Lipschitz continuity condition
  \begin{equation*}
    \sum_{\gamma=1}^d |\Phi_k^{\gamma}(\x(n);t) - \Phi_k^{\gamma}(\x(n);s)| \leq |t-s| \ConstBound{1},
  \end{equation*}
  with a constant $\ConstBound{1}$ that does not depend on $n$.

  Let us fix a subsequence of $(\upmu[\x(n)])_{n \geq 1}$, that we still index by $n$ for convenience. Then, by the argument above, the sequence $(\upmu_{[0,T]}[\x(n)])_{n \geq 1}$ is tight, and therefore, owing to the Prohorov Theorem, we can extract a further subsequence converging weakly to some probability measure $\bar{\upmu}_{[0,T]}$ on $\Cs([0,T], \R^d)$. Letting $T$ grow to infinity along some countable set and using a diagonal extraction procedure, we deduce that there exists an increasing sequence of integers $(n_{\ell})_{\ell \geq 1}$ and $\bar{\upmu} \in \Ms$ such that $\upmu[\x(n_{\ell})]$ converges weakly to $\bar{\upmu} \in \Ms$.
  
  \sk
  Let us now check that, for all $\gamma, \gamma' \in \{1, \ldots, d\}$ such that $\gamma \not= \gamma'$, $\dd t$-almost everywhere, the probability measures $\bar{\upmu}^{\gamma}_t$ and $\bar{\upmu}^{\gamma'}_t$ have distinct atoms. We note that this amounts to proving that
  \begin{equation*}
    \int_{t=0}^{+\infty} \bar{\upmu}^{\gamma}_t \otimes \bar{\upmu}^{\gamma'}_t(\{(x,x') \in \R^2 : x=x'\}) \dd t = 0,
  \end{equation*}
  where $\bar{\upmu}^{\gamma}_t \otimes \bar{\upmu}^{\gamma'}_t$ denotes the product measure of $\bar{\upmu}^{\gamma}_t$ and $\bar{\upmu}^{\gamma'}_t$ on $\R^2$. Following~\cite[(ii), Theorem~2.8, p.~23]{billingsley}, for all $t \geq 0$, the probability measure $\upmu_t^{\gamma}[\x(n_{\ell})] \otimes \upmu_t^{\gamma'}[\x(n_{\ell})]$ converges weakly to $\bar{\upmu}^{\gamma}_t \otimes \bar{\upmu}^{\gamma'}_t$ on $\R^2$. Hence, for all $\epsilon > 0$, the Portmanteau Theorem~\cite[(iv), Theorem~2.1, p.~16]{billingsley} yields
  \begin{equation*}
    \begin{aligned}
      & \bar{\upmu}^{\gamma}_t \otimes \bar{\upmu}^{\gamma'}_t(\{(x,x') \in \R^2 : |x-x'| < \epsilon\})\\
      & \qquad \leq \liminf_{\ell \to +\infty} \upmu_t^{\gamma}[\x(n_{\ell})] \otimes \upmu_t^{\gamma'}[\x(n_{\ell})](\{(x,x') \in \R^2 : |x-x'| < \epsilon\}),
    \end{aligned}
  \end{equation*}
  therefore by the Fatou lemma,
  \begin{equation*}
    \begin{aligned}
      & \int_{t=0}^{+\infty} \bar{\upmu}^{\gamma}_t \otimes \bar{\upmu}^{\gamma'}_t(\{(x,x') \in \R^2 : |x-x'| < \epsilon\})\dd t\\
      & \qquad \leq \liminf_{\ell \to +\infty} \int_{t=0}^{+\infty} \upmu_t^{\gamma}[\x(n_{\ell})] \otimes \upmu_t^{\gamma'}[\x(n_{\ell})](\{(x,x') \in \R^2 : |x-x'| < \epsilon\})\dd t.
    \end{aligned}
  \end{equation*}
  Now, for all $\ell \geq 1$, by the Fubini Theorem,
  \begin{equation*}
    \begin{aligned}
      & \int_{t=0}^{+\infty} \upmu_t^{\gamma}[\x(n_{\ell})] \otimes \upmu_t^{\gamma'}[\x(n_{\ell})](\{(x,x') \in \R^2 : |x-x'| < \epsilon\})\dd t\\
      & \qquad = \int_{(x,x') \in \R^2} \int_{t=0}^{+\infty} \dd t\ind{|x-x'| < \epsilon} \upmu_t^{\gamma}[\x(n_{\ell})](\dd x) \upmu_t^{\gamma'}[\x(n_{\ell})](\dd x')\\
      & \qquad = \frac{1}{n^2}\sum_{k=1}^n\sum_{k'=1}^n \int_{t=0}^{+\infty} \ind{|\Phi_k^{\gamma}(\x(n);t)-\Phi_{k'}^{\gamma'}(\x(n);t)| < \epsilon} \dd t.
    \end{aligned}
  \end{equation*}
  By Lemma~\ref{lem:tinter}, for all $\gamma:k, \gamma':k' \in \Part$ with $\gamma \not= \gamma'$,
  \begin{equation*}
    \int_{t=0}^{+\infty} \ind{|\Phi_k^{\gamma}(\x(n);t)-\Phi_{k'}^{\gamma'}(\x(n);t)| < \epsilon} \dd t \leq \frac{2\epsilon}{\ConstUSH}.
  \end{equation*}
  As a consequence,
  \begin{equation*}
    \begin{aligned}
      \int_{t=0}^{+\infty} \bar{\upmu}^{\gamma}_t \otimes \bar{\upmu}^{\gamma'}_t(\{(x,x') \in \R^2 : x=x'\})\dd t & \leq \int_{t=0}^{+\infty} \bar{\upmu}^{\gamma}_t \otimes \bar{\upmu}^{\gamma'}_t(\{(x,x') \in \R^2 : |x-x'| < \epsilon\})\dd t\\
      & \leq \frac{2\epsilon}{\ConstUSH},
    \end{aligned}
  \end{equation*}
  and we complete the proof by letting $\epsilon$ vanish.
\end{proof}

The proof of Theorem~\ref{theo:existence} finally comes as a straightforward consequence of Proposition~\ref{prop:tightness}.

\begin{proof}[Proof of Theorem~\ref{theo:existence}]
  Under the assumptions of Theorem~\ref{theo:existence}, let us fix a subsequence of $(\upmu[\x(n)])_{n \geq 1}$, and let $(\upmu[\x(n_{\ell})])_{\ell \geq 1}$ denote a further subsequence weakly converging to some $\bar{\upmu} \in \Ms$ as is given by Proposition~\ref{prop:tightness}. Define the function $\bu = (u^1, \ldots, u^d) : [0,+\infty) \times \R \to [0,1]^d$ by
  \begin{equation*}
    \forall t \geq 0, \quad \forall \gamma \in \{1, \ldots, d\}, \qquad u^{\gamma}(t,x) := H*\bar{\upmu}^{\gamma}_t(x).
  \end{equation*}
  We first note that, by Proposition~\ref{prop:MSPDsol}, for all $\ell \geq 1$, the function $\bu[\x(n_{\ell})]$ is a probabilistic solution to the system~\eqref{eq:syst}. Furthermore, by Lemma~\ref{lem:cvCDF}, we have, for all $\gamma \in \{1, \ldots, d\}$, for all $t \geq 0$, 
  \begin{equation*}
    \lim_{\ell \to +\infty} u^{\gamma}[\x(n_{\ell})](t,x) = u^{\gamma}(t,x)
  \end{equation*}
  for all $x \in \R$ such that $\Delta_x u^{\gamma}(t,x) = 0$. Finally, by the second part of Proposition~\ref{prop:tightness}, the function $\bu$ satifies~\eqref{eq:mutgammutgamp} in Proposition~\ref{prop:closedness}. 
  
  As a consequence, we can apply Proposition~\ref{prop:closedness} and conclude that $\bu$ is a probabilistic solution to the system~\eqref{eq:syst}, with initial data $(u_0^1, \ldots, u_0^d)$ defined by $u_0^{\gamma} = H*\bar{\upmu}^{\gamma}_0 = H*m^{\gamma}$. The proof of Theorem~\ref{theo:existence} is completed.
\end{proof}


\section{Trajectories associated with probabilistic solutions}\label{s:traj}

In Section~\ref{s:existence}, we checked that the MSPD satisfies the differential relation~\eqref{eq:MSPD:diff}. In other words, the MSPD behaves like what one would expect to be the characteristics associated with the system of transport equations
\begin{equation*}
  \forall \gamma \in \{1, \ldots, d\}, \qquad \left\{\begin{aligned}
    & \partial_t u^{\gamma} + \lambda^{\gamma}\{\bu\} \partial_x u^{\gamma} = 0,\\
    & u^{\gamma}(0,x) = u^{\gamma}_0(x).
  \end{aligned}\right.
\end{equation*}
However, the value of $\bu[\x](t,\Phi^{\gamma}_k(\x;t))$ is only constant between collisions of particles.

More generally, one may wonder whether such a description in terms of {\em trajectories} of a process $(\bX(t))_{t \geq 0}$ in $\R^d$, may be generalized to {\em any} probabilistic solution $\bu$ to~\eqref{eq:syst} and whether these trajectories satisfy a differential relation similar to~\eqref{eq:MSPD:diff}. In the MSPD, the positions of the particles are given by the quantiles of order $k/n$ of the empirical CDF, therefore it is natural to define, for all $v \in (0,1)$, the process $(\bX_v(t))_{t \geq 0}$ by 
\begin{equation}\label{eq:defbX}
  \forall t \geq 0, \qquad \bX_v(t) = (X^1_v(t), \ldots, X^d_v(t)) \in \R^d, \qquad X^{\gamma}_v(t) := u^{\gamma}(t,\cdot)^{-1}(v).
\end{equation}
In Subsection~\ref{ss:reptraj}, we show that, for all $\gamma \in \{1, \ldots, d\}$, $\dd v$-almost everywhere, the trajectory of $(X^{\gamma}_v(t))_{t \geq 0}$ is Lipschitz continuous, with Lipschitz constants given by the lower and upper bounds of $\lambda^{\gamma}$. This allows us to provide a {\em probabilistic representation} of the solution $\bu$ in terms of a stochastic process $(\bbX^1(t), \ldots, \bbX^d(t))_{t \geq 0}$. Then in Subsection~\ref{ss:veltraj}, we show that the process $(X^{\gamma}_v(t))_{t \geq 0}$ satisfies a differential relation similar to~\eqref{eq:MSPD:diff} if and only $\bu$ is a {\em renormalised solution to~\eqref{eq:syst} in the $\gamma$-th coordinate} in the sense of DiPerna and Lions~\cite{diperna}.


\subsection{Probabilistic representation of probabilistic solutions}\label{ss:reptraj} Throughout the subsection, we shall work under Assumption~\eqref{ass:C} and denote, for all $\gamma \in \{1, \ldots, d\}$,
\begin{equation*}
  \ulambda^{\gamma} := \inf_{\bu \in [0,1]^d} \lambda^{\gamma}(\bu), \qquad \olambda^{\gamma} := \sup_{\bu \in [0,1]^d} \lambda^{\gamma}(\bu).
\end{equation*}

\begin{prop}[Lipschitz continuity of trajectories]\label{prop:traj}
  Under Assumption~\eqref{ass:C}, let $\bu$ be a probabilistic solution to~\eqref{eq:syst} such that $t\mapsto \bu(t,\cdot)$ is continuous in $\Ls^1_{\mathrm{loc}}(\R)^d$, and let $(\bX(t))_{t \geq 0}$ be defined by~\eqref{eq:defbX}. Then, for all $\gamma \in \{1, \ldots, d\}$, $\dd v$-almost everywhere, the trajectory of $(X^{\gamma}_v(t))_{t \geq 0}$ is Lipschitz continuous and
  \begin{equation}\label{eq:liptraj}
    \ulambda^{\gamma} \leq \dot{X}^{\gamma}_v(t) \leq \olambda^{\gamma}, \qquad \text{$\dd t$-almost everywhere}.
  \end{equation}
\end{prop}
\begin{proof}
  Let us fix $\gamma \in \{1, \ldots, d\}$. The proof of~\eqref{eq:liptraj} is detailed in the two steps below. 

  \sk
  \noindent {\em Step~1: using intermediate functions $\uu^{\gamma}$ and $\ou^{\gamma}$.} From the definition of $u^{\gamma}$, we note that $\partial_x u^{\gamma}$ is a nonnegative measure, and then $u^{\gamma}$ satisfies
  \begin{equation*}
    \partial_t u^{\gamma} + \ulambda^{\gamma} \partial_x u^{\gamma} \leq 0 \leq \partial_t u^{\gamma} + \olambda^{\gamma} \partial_x u^{\gamma}
  \end{equation*}
  in the distributional sense on $(0,+\infty)\times\R$. This means that
  \begin{equation*}
    \uu^{\gamma}(t,x) := u^{\gamma}(t,x+\ulambda^{\gamma} t) \qquad \text{and} \qquad \ou^{\gamma}(t,x) := u^{\gamma}(t,x+\olambda^{\gamma} t)
  \end{equation*}
  satisfy
  \begin{equation}\label{eq::3}
    \partial_t \uu^{\gamma} \leq 0 \leq \partial_t \ou^{\gamma}
  \end{equation}
  in the distributional sense on $(0,+\infty)\times\R$. We claim that this implies the existence of a Borel subset $\mathcal{T}$ of $(0,+\infty)$ with zero Lebesgue measure such that, for all $x \in \R$, for all $t_1, t_2 \in (0,+\infty) \setminus \mathcal{T}$ with $t_1 \leq t_2$,
  \begin{equation}\label{eq::5}
    \uu^{\gamma}(t_2,x) - \uu^{\gamma}(t_1,x) \leq 0 \leq \ou^{\gamma}(t_2,x) - \ou^{\gamma}(t_1,x).
  \end{equation}
  The proof of this claim is postponed to Step~2 below.
  
  We deduce that for all $y\in\R$,
  \begin{equation*}
    u^{\gamma}(t_2, y+\ulambda^{\gamma}(t_2-t_1)) \leq u^{\gamma}(t_1,y) \leq u^{\gamma}(t_2,y+\olambda^{\gamma}(t_2-t_1)).
  \end{equation*}
  Fixing $v \in (0,1)$, then choosing $y=X^{\gamma}_v(t_1)$ in the right-hand inequality and $y=X^{\gamma}_v(t_2)-\ulambda^{\gamma}(t_2-t_1)$ in the left-hand inequality, we deduce from Assertions~\eqref{it:pseudoinv:1} and~\eqref{it:pseudoinv:2} in Lemma~\ref{lem:pseudoinv} that
  \begin{equation*}
    X^{\gamma}_v(t_1) + \ulambda^{\gamma}(t_2-t_1) \leq X^{\gamma}_v(t_2) \leq X^{\gamma}_v(t_1) + \olambda^{\gamma}(t_2-t_1), 
  \end{equation*}
  which holds for all $t_1 \leq t_2$ in $(0,+\infty) \setminus \mathcal{T}$.
  
  For all $v \in (0,1)$, we deduce the existence of $t \mapsto \tilde{X}^{\gamma}_v(t)$ which coincides with $X^{\gamma}_v(t)$ on $(0,+\infty) \setminus \mathcal{T}$ and such that 
  \begin{equation}\label{eq:conty}
    \forall 0 \leq t_1 \leq t_2, \qquad \tilde{X}^{\gamma}_v(t_1) + \ulambda^{\gamma}(t_2-t_1) \leq \tilde{X}^{\gamma}_v(t_2) \leq \tilde{X}^{\gamma}_v(t_1)+ \olambda^{\gamma}(t_2-t_1).
  \end{equation}
  The continuity of $t \mapsto \tilde{X}^{\gamma}_v(t)$ for all $v\in(0,1)$ combined with Lemma~\ref{lem:cvCDF} ensure that the CDF of the image of $\Unif$ by $v \mapsto \tilde{X}^{\gamma}_v(t)$, which coincides with $u^{\gamma}(t,\cdot)$ on $(0,+\infty)\setminus\mathcal{T}$, is continuous in $\Ls^1_{\mathrm{loc}}(\R)$ as a function of $t \in [0,+\infty)$. Since $t \mapsto u^{\gamma}(t,\cdot)$ is also continuous in $\Ls^1_{\mathrm{loc}}(\R)$, we deduce that 
  \begin{equation}\label{eq:ucdfy}
   \forall (t,x) \in [0,+\infty)\times\R, \qquad u^{\gamma}(t,x) = \int_{v=0}^1 \ind{\tilde{X}^{\gamma}_v(t) \leq x}\dd v.
  \end{equation}
  From Assertion~\eqref{it:pseudoinv:0} in Lemma~\ref{lem:pseudoinv} and Lemma~\ref{lem:cvCDF}, $t \mapsto X^{\gamma}_v(t)$ is continuous on $[0,+\infty)$ as soon as $v$ is not in 
  \begin{equation*}
    {\mathcal V} := \{v\in(0,1) : \exists t_1 \geq 0,\;\exists x<y\in\R,\;u^{\gamma}(t_1,x)=u^{\gamma} (t_1,y)=v\}.
  \end{equation*}

  Let $v\in{\mathcal V}$ and $t_1\geq 0$, $x,y\in\R$ be such that $x<y$ and $u^{\gamma}(t_1,x)=u^{\gamma}(t_1,y)=v$. The monotonicity of $w\mapsto \tilde{X}^{\gamma}_w(t_1)$ and~\eqref{eq:ucdfy} ensure that
  \begin{equation*}
    \forall w \in (0,v), \quad \tilde{X}^{\gamma}_w(t_1)\leq x \qquad \text{and} \qquad \forall w \in (v,1), \quad \tilde{X}^{\gamma}_w(t_1)>y.
  \end{equation*}
  Now, by~\eqref{eq:conty}, for $t_2\geq t_1$, $\tilde{X}^{\gamma}_w(t_2)\leq x+\olambda^{\gamma}(t_2-t_1)$ when $w\in (0,v)$ and $\tilde{X}^{\gamma}_w(t_2)> y+\ulambda^{\gamma}(t_2-t_1)$ when $w\in (v,1)$. For $t_2\in (t_1,t_1+(y-x)/(\olambda^{\gamma}-\ulambda^{\gamma}))$, $x + \olambda^{\gamma}(t_2-t_1) < y + \ulambda^{\gamma}(t_2-t_1)$ and, by~\eqref{eq:ucdfy}, 
  \begin{equation*}
    u^{\gamma}(t_2,x+\olambda^{\gamma}(t_2-t_1)) = u^{\gamma}(t_2,y+\ulambda^{\gamma}(t_2-t_1))=v.
  \end{equation*}
  Hence 
  \begin{equation*}
    {\mathcal V}=\{v\in(0,1):\exists t_2\in{\mathbb Q}_+,\;\exists x<y\in\R,\;u^{\gamma}(t_2,x)=u^{\gamma}(t_2,y)=v\},
  \end{equation*}
  and ${\mathcal V}$ is countable as a countable union of countable sets. Since $t\mapsto X^{\gamma}_v(t)$ and $t\mapsto \tilde{X}^{\gamma}_v(t)$ coincide for $v\not\in{\mathcal V}$, and for all $t \geq 0$, $v\mapsto X^{\gamma}_v(t)$ is nondecreasing, the conclusion follows from~\eqref{eq:conty}.
  
  \sk
  \noindent {\em Step 2: proof of~\eqref{eq::5}.} The proof of~\eqref{eq::5} should be standard, but we do not know any reference, so we propose a proof below.
  
  Let $R>0$. Let us consider a $\Cs^{\infty}$ function $\psi : \R \to [0,+\infty)$ with $\mbox{supp}\ \psi \subset [-R,R]$ and let $\chi_\epsilon\geq 0$ be a nonnegative smooth approximation of the indicator function $\chi(t)=\ind{t \in [t_1,t_2]}$ with compact support in $(0,+\infty)$, where $0<t_1<t_2$ are Lebesgue points of the function $u^{\gamma}\in \Ls^\infty((0,+\infty), \Ls^1([-R,R])$. Let us define the function 
  \begin{equation*}
    \Phi_\epsilon(t,x)=\psi(x) \chi_\epsilon(t)\geq 0.
  \end{equation*}
  Taking the distributional bracket of inequality~\eqref{eq::3} with the test function $\Phi_\epsilon$, and integrating by parts in the sense of distributions, we get
  \begin{equation*}
    -\int_{t=0}^{\infty} \int_{x \in \R} \uu^{\gamma}(t,x) \psi(x)\partial_t \chi_\epsilon(t) \dd x \dd t \le 0 \le -\int_{t=0}^{\infty} \int_{x \in \R} \ou^{\gamma}(t,x) \psi(x)\partial_t \chi_\epsilon(t) \dd x \dd t.
  \end{equation*}
  Passing to the limit as $\epsilon$ goes to zero, we obtain
  \begin{equation*}
    \int_{x \in \R} \uu^{\gamma}(t_2,x) \psi(x)\dd x - \int_{x \in \R} \uu^{\gamma}(t_1,x) \psi(x)\dd x \leq 0 \leq \int_{x \in \R} \ou^{\gamma}(t_2,x) \psi(x)\dd x - \int_{x \in \R} \ou^{\gamma}(t_1,x) \psi(x)\dd x.
  \end{equation*}
  Since $R$ and $\psi$ are arbitrary, this implies
  \begin{equation}\label{eq::4}
    \uu^{\gamma}(t_2,x) - \uu^{\gamma}(t_1,x) \leq 0 \leq \ou^{\gamma}(t_2,x) - \ou^{\gamma}(t_1,x) \qquad \text{$\dd x$-almost everywhere}.
  \end{equation}
  Because of the right continuity of $u^{\gamma}(t,\cdot)$, we conclude that~\eqref{eq::4} holds true for every $x \in \R$, which shows~\eqref{eq::5}.
\end{proof}

An immediate consequence of Proposition~\ref{prop:traj} is that probabilistic solutions to~\eqref{eq:syst} have a finite speed of propagation.

\begin{cor}[Finite speed of propagation]\label{cor:speedprop}
  Under Assumption~\eqref{ass:C}, let $\bu$ be a probabilistic solution to~\eqref{eq:syst} satisfying the assumptions of Proposition~\ref{prop:traj}. For all $\gamma \in \{1, \ldots, d\}$, for all $\tau, t \geq 0$:
  \begin{enumerate}[label=(\roman*), ref=\roman*]
    \item\label{it:speedprop:1} for all $a \in \R$, $u^{\gamma}(\tau,a) \leq u^{\gamma}(\tau+t, a+\olambda^{\gamma}t)$,
    \item\label{it:speedprop:2} for all $b \in \R$, $u^{\gamma}(\tau,b^-) \geq u^{\gamma}(\tau+t, (b+\ulambda^{\gamma}t)^-)$.
  \end{enumerate}
\end{cor}
\begin{proof}
  Let $v=u^{\gamma}(\tau,a)$. By~\eqref{it:pseudoinv:2} in Lemma~\ref{lem:pseudoinv}, $X^{\gamma}_v(\tau) = u^{\gamma}(\tau, \cdot)^{-1}(v) \leq a$, so that Proposition~\ref{prop:traj} yields
  \begin{equation*}
    u^{\gamma}(\tau+t,\cdot)^{-1}(v) = X^{\gamma}_v(\tau+t) \leq X^{\gamma}_v(\tau) + \olambda^{\gamma}t \leq a + \olambda^{\gamma}t,
  \end{equation*}
  therefore by~\eqref{it:pseudoinv:2} in Lemma~\ref{lem:pseudoinv} again, $u^{\gamma}(\tau,a) = v \leq u^{\gamma}(\tau+t, a + \olambda^{\gamma}t)$, whence~\eqref{it:speedprop:1}.
  
  Let us now fix $\epsilon > 0$ and $v>u^{\gamma}(\tau,b-\epsilon)$. By~\eqref{it:pseudoinv:2} in Lemma~\ref{lem:pseudoinv}, $X^{\gamma}_v(\tau) > b-\epsilon$, and by Proposition~\ref{prop:traj},
  \begin{equation*}
    u^{\gamma}(\tau+t,\cdot)^{-1}(v) = X^{\gamma}_v(\tau+t) \geq X^{\gamma}_v(\tau) + \ulambda^{\gamma}t > b + \ulambda^{\gamma}t - \epsilon,  
  \end{equation*}
  so that, by~\eqref{it:pseudoinv:2} in Lemma~\ref{lem:pseudoinv} again, $v > u^{\gamma}(\tau+t, b + \ulambda^{\gamma}t - \epsilon)$. Since $v$ is arbitrarily close to $u^{\gamma}(\tau,b-\epsilon)$, we deduce that $u^{\gamma}(\tau,b-\epsilon) \geq u^{\gamma}(\tau+t, b + \ulambda^{\gamma}t - \epsilon)$, and obtain~\eqref{it:speedprop:2} by taking the limit of this inequality when $\epsilon$ vanishes.
\end{proof}

The proof of Theorem~\ref{theo:existence}, and in particular Proposition~\ref{prop:tightness}, shows that, for the probabilistic solutions $\bu$ obtained there, there exists a probability measure $\bar{\upmu} \in \Ms$ such that 
\begin{equation}\label{eq:reptraj}
  \forall \gamma \in \{1, \ldots, d\}, \quad \forall (t,x) \in [0,+\infty) \times \R, \qquad u^{\gamma}(t,x) = H*\bar{\upmu}^{\gamma}_t(x).
\end{equation}
It is therefore natural to wonder if, for any probabilistic solution $\bu$, there exists $\upmu \in \Ms$ such that~\eqref{eq:reptraj} holds. In other words, does there exist a stochastic process $(\bbX^1(t), \ldots, \bbX^d(t))_{t \geq 0}$ with continuous sample-paths in $\R^d$, such that for all $\gamma \in \{1, \ldots, d\}$, for all $t \geq 0$, the function $u^{\gamma}(t, \cdot)$ is the CDF of the random variable $\bbX^{\gamma}(t)$? Proposition~\ref{prop:traj} provides a constructive positive answer to this question.

\begin{cor}[Probabilistic representation of probabilistic solutions]\label{cor:upmu}
  Under Assumption~\eqref{ass:C}, let $\bu$ be a probabilistic solution to~\eqref{eq:syst} satisfying the assumptions of Proposition~\ref{prop:traj}. Let $\bbv$ be a uniform random variable on $(0,1)$, and let us define the stochastic process $(\bbX^1(t), \ldots, \bbX^d(t))_{t \geq 0}$ by
  \begin{equation*}
    \forall t \geq 0, \quad \forall \gamma \in \{1, \ldots, d\}, \qquad \bbX^{\gamma}(t) := X^{\gamma}_{\bbv}(t).
  \end{equation*}
  Then the sample-paths of $(\bbX^1(t), \ldots, \bbX^d(t))_{t \geq 0}$ are almost surely continuous, and the law $\upmu \in \Ms$ of this process satisfies~\eqref{eq:reptraj}.
\end{cor}
The fact that the sample-paths of $(\bbX^1(t), \ldots, \bbX^d(t))_{t \geq 0}$ are almost surely continuous is a straightforward consequence of~\eqref{eq:liptraj}, and it follows from the change of variable formula in Lemma~\ref{lem:CDFm1} that the image $\upmu \in \Ms$ of the Lebesgue measure $\Unif$ on $(0,1)$ by $v \mapsto (X^1_v(t), \ldots, X^d_v(t))_{t\geq 0}$ satisfies \eqref{eq:reptraj}.

\begin{rk}\label{rk:contL1loc}
  The condition in Proposition~\ref{prop:traj} that $t \mapsto \bu(t,\cdot)$ be continuous in $\Ls^1_{\mathrm{loc}}(\R)^d$ is automatically satisfied if there exists $\bar{\upmu} \in \Ms$ such that~\eqref{eq:reptraj} holds. Indeed, in this case the Dominated Convergence Theorem implies that the mapping $t \mapsto \bar{\upmu}^{\gamma}_t$ is weakly continuous in $\Ps(\R)$, and by Lemma~\ref{lem:cvCDF}, for all $t_0 \geq 0$, $u^{\gamma}(t,x)$ converges to $u^{\gamma}(t_0,x)$, $\dd x$-almost everywhere. Since these functions are uniformly bounded, then the convergence holds in $\Ls^1_{\mathrm{loc}}(\R)^d$.
\end{rk}

\begin{rk}\label{rk:unicitebmu}
  Let $\bu$ be a probabilistic solution to~\eqref{eq:syst} obtained in Theorem~\ref{theo:existence} as the limit of the empirical CDFs of the MSPD along some subsequence $(\bu[\x(n_{\ell})])_{\ell \geq 1}$. We {\em a priori} have two probabilistic representations for $\bu$: by the probability measure $\bar{\upmu}$ defined in Proposition~\ref{prop:tightness} as the weak limit of $\upmu[\x(n_{\ell})]$ in $\Ms$, and by the probability measure $\upmu$ provided by Corollary~\ref{cor:upmu}. Let us check that these two probability measures actually coincide. For any continuous and bounded function $f : (\R^d)^k \to \R$ and any $0 \leq t_1 \leq t_2 \leq \cdots \leq t_k$, we have, for all $n \geq 1$,
  \begin{equation*}
    \begin{split}
      \int_{(\R^d)^k} f \dd \upmu_{t_1,\ldots,t_k}[\x(n)] = \int_{v=0}^1 f\big(& u^1[\x(n)](t_1,\cdot)^{-1}(v), \ldots, u^d[\x(n)](t_1,\cdot)^{-1}(v);\\
                                                                                 & \vdots\\
                                                                                 & u^1[\x(n)](t_k,\cdot)^{-1}(v), \ldots, u^d[\x(n)](t_k,\cdot)^{-1}(v)\big)\dd v,
    \end{split}
  \end{equation*}
  where $\upmu_{t_1,\ldots,t_k}[\x(n)]$ denotes the finite-dimensional marginal distribution of the measure $\upmu[\x(n)]$ at times $t_1,\ldots,t_k$. By an easy adaptation of~\cite[Lemma 3.5]{jourdain:characteristics}, this equality is preserved by weak convergence in $\Ms$, so that
  \begin{equation*}
    \begin{split}
      \int_{(\R^d)^k} f \dd \bar{\upmu}_{t_1,\ldots,t_k} = \int_{v=0}^1 f\big(& u^1(t_1,\cdot)^{-1}(v), \ldots, u^d(t_1,\cdot)^{-1}(v);\\
                                                                           & \vdots\\
                                                                           & u^1(t_k,\cdot)^{-1}(v), \ldots, u^d(t_k,\cdot)^{-1}(v)\big)\dd v,
    \end{split}
  \end{equation*}
  Therefore $\bar{\upmu}$ has the same finite-dimensional marginals as $\upmu$. Since a probability measure in $\Ms$ is determined by its finite-dimensional marginals, $\bar{\upmu}=\upmu$.
\end{rk}


\subsection{Renormalised solutions and identification of the velocity}\label{ss:veltraj} Given a probabilistic solution $\bu$ to the system~\eqref{eq:syst} satisfying the assumptions of Proposition~\ref{prop:traj}, we now want to provide a dynamical description, similar to~\eqref{eq:MSPD:diff}, of the evolution of the trajectory $(\bX_v(t))_{t \geq 0}$ defined in~\eqref{eq:defbX}. To this aim, we first need to introduce the notion of a {\em renormalised solution to~\eqref{eq:syst} in the $\gamma$-th coordinate}, which is adapted from DiPerna and Lions~\cite{diperna}.

\begin{defi}[Renormalised solution to~\eqref{eq:syst}]
  Under Assumption~\eqref{ass:C}, a probabilistic solution $\bu$ to the system~\eqref{eq:syst} is said to be a {\em renormalised solution to~\eqref{eq:syst} in the $\gamma$-th coordinate} if, for all $\Cs^1$ increasing functions $\beta : [0,1] \to \R$ such that $\beta(0)=0$ and $\beta(1)=1$, for all test functions $\varphi \in \Cs^{1,0}_{\mathrm{c}}([0,+\infty)\times\R, \R)$,
  \begin{equation*}
    \begin{aligned}
      & \int_{t=0}^{+\infty} \int_{x \in \R} \partial_t\varphi(t,x) \beta(u^{\gamma}(t,x)) \dd x\dd t + \int_{x \in \R} \varphi(0,x) \beta(u_0^{\gamma}(x)) \dd x\\
      & \qquad = \int_{t=0}^{+\infty} \int_{x \in \R} \varphi(t,x) \lambda^{\gamma}\{\bu\}(t,x) \dd_x (\beta \circ u^{\gamma})(t,x) \dd t,
    \end{aligned}
  \end{equation*}
  where $\dd_x (\beta \circ u^{\gamma})(t,x)$ refers to the probability measure with CDF $\beta(u^{\gamma}(t, \cdot))$.
\end{defi}

\sk
Recall that, if $\bu$ is a probabilistic solution to~\eqref{eq:syst} satisfying the assumptions of Proposition~\ref{prop:traj}, then with the notations of Subsection~\ref{ss:reptraj}, $\dd v$-almost everywhere in $(0,1)$, the process $(\bX_v(t))_{t \geq 0}$ is Lipschitz continuous and, for all $\gamma \in \{1, \ldots, d\}$, $\ulambda^{\gamma} \leq \dot{X}^{\gamma}_v(t) \leq \olambda^{\gamma}$, $\dd t$-almost everywhere. For trajectories associated with renormalised solutions to~\eqref{eq:syst}, this description is strengthened as follows. 

\begin{prop}[Trajectories associated with renormalised solutions]\label{prop:vitesses}
  Under Assumption~\eqref{ass:C}, let $\bu = (u^1, \ldots, u^d)$ be a probabilistic solution to~\eqref{eq:syst} satisfying the assumptions of Proposition~\ref{prop:traj}. Then, for all $\gamma \in \{1, \ldots, d\}$, $\bu$ is a renormalised solution to~\eqref{eq:syst} in the $\gamma$-th coordinate if and only if, $\dd v$-almost everywhere in $(0,1)$, the process $(X^{\gamma}_v(t))_{t \geq 0}$ defined in~\eqref{eq:defbX} is Lipschitz continuous and
  \begin{equation}\label{eq:Xpoint}
    \dot{X}^{\gamma}_v(t) = \lambda^{\gamma}\{\bu\}(t, X^{\gamma}_v(t)) \qquad \text{$\dd t$-almost everywhere.}
  \end{equation}
\end{prop}
\begin{proof}[Proof of necessity]
  Let us first fix $\gamma \in \{1, \ldots, d\}$ and assume that $\bu$ is a renormalised solution in the $\gamma$-th coordinate. Let us also fix $v_0 \in (0,1)$. Let us prove that, for all functions $\psi \in \Cs^{1,1}_{\mathrm{c}}([0,+\infty) \times \R, \R)$, for all $t \geq 0$,
  \begin{equation*}
    \psi(0,X^{\gamma}_{v_0}(0)) + \int_{t=0}^{+\infty}\left(\partial_t \psi(t, X^{\gamma}_{v_0}(t)) + \lambda^{\gamma}\{\bu\}(t, X^{\gamma}_{v_0}(t)) \partial_x \psi(t, X^{\gamma}_{v_0}(t))\right)\dd t = 0,
  \end{equation*}
  so that, for all $0 \leq t_1 \leq t_2$, taking smooth and compactly supported approximations $\psi(t,x)$ of $x \ind{t_1 \leq t \leq t_2}$ yields
  \begin{equation*}
    X^{\gamma}_{v_0}(t_2) - X^{\gamma}_{v_0}(t_1) = \int_{t=t_1}^{t_2} \lambda^{\gamma}\{\bu\}(t, X^{\gamma}_{v_0}(t)) \dd t.
  \end{equation*}
  For such a function $\psi$, let $\varphi := -\partial_x \psi$. For all $\epsilon > 0$, let $\beta_{\epsilon} : [0,1] \to [0,1]$ be an increasing $\Cs^1$ function, such that $\beta_{\epsilon}(0)=0$, $\beta_{\epsilon}(1)=1$ and, for all $v \in (0,1)$,
  \begin{equation}\label{eq:limbeta}
    \lim_{\epsilon \dto 0} \beta_{\epsilon}(v) = \ind{v \geq v_0}.
  \end{equation}
  
  Since $\bu$ is a renormalised solution, we have, for all $\epsilon > 0$,
  \begin{equation*}
    \begin{aligned}
      & \int_{t=0}^{+\infty} \int_{x \in \R} \partial_t \varphi(t,x) \beta_{\epsilon}(u^{\gamma}(t,x)) \dd x \dd t + \int_{x \in \R} \varphi(0,x) \beta_{\epsilon}(u^{\gamma}_0(x))\dd x\\
      & \qquad = \int_{t=0}^{+\infty} \int_{x \in \R} \varphi(t,x) \lambda^{\gamma}\{\bu\}(t,x) \dd_x(\beta_{\epsilon} \circ u^{\gamma})(t,x) \dd t.
    \end{aligned}
  \end{equation*}
  On account of~\eqref{eq:limbeta}, the Dominated Convergence Theorem gives
  \begin{equation*}
    \begin{aligned}
      \lim_{\epsilon \dto 0} \int_{t=0}^{+\infty} \int_{x \in \R} \partial_t \varphi(t,x) \beta_{\epsilon}(u^{\gamma}(t,x)) \dd x \dd t & = \int_{t=0}^{+\infty} \int_{x \in \R} \partial_t \varphi(t,x) \ind{u^{\gamma}(t,x)) \geq v_0} \dd x \dd t\\
      & = \int_{t=0}^{+\infty} \int_{x=X^{\gamma}_{v_0}(t)}^{+\infty} \partial_t \varphi(t,x)\dd x \dd t\\
      & = \int_{t=0}^{+\infty} \partial_t \psi(t, X^{\gamma}_{v_0}(t)) \dd t ;
    \end{aligned}
  \end{equation*}
  likewise,
  \begin{equation*}
    \lim_{\epsilon \dto 0} \int_{x \in \R} \varphi(0,x) \beta_{\epsilon}(u^{\gamma}_0(x))\dd x = \psi(0, X^{\gamma}_{v_0}(0)).
  \end{equation*}
  However, passing to the limit in the right-hand side is more delicate as, for all $t \geq 0$, the probability measure with CDF $\beta_{\epsilon}(u^{\gamma}(t, \cdot))$ converges weakly to the Dirac distribution in $X^{\gamma}_{v_0}(t)$, and the function $\lambda^{\gamma}\{\bu\}(t,\cdot)$ may be discontinuous at this point. 
  
  To handle this issue, we first fix $t \geq 0$ and remark that the function $\lambda^{\gamma}\{\bu\}(t, \cdot)$ is continuous outside of the countable set
  \begin{equation*}
    \mathcal{X} := \{x \in \R: \exists \gamma' \in \{1, \ldots, d\}, \Delta_x u^{\gamma'}(t,x) > 0\}.
  \end{equation*}
  This fact is obtained by writing 
  \begin{equation*}
    \lambda^{\gamma}\{\bu\}(t,x) = \int_{\theta=0}^1 \lambda^{\gamma}\left(u^1(t,x), \ldots, (1-\theta)u^{\gamma}(t, x^-) + \theta u^{\gamma}(t, x), \ldots, u^d(t,x)\right)\dd \theta
  \end{equation*}
  and noting that, for all $\theta \in [0,1]$, the integrand is continuous on $\R \setminus \mathcal{X}$.
  
  We can now assert that $v_0$ is in exactly one of the three following cases:
  \begin{enumerate}
    \item\label{it:pfvitesses:1} $X^{\gamma}_{v_0}(t) \not\in \mathcal{X}$,
    \item\label{it:pfvitesses:2} $X^{\gamma}_{v_0}(t) \in \mathcal{X}$ and $\Delta_x u^{\gamma}(t, X^{\gamma}_{v_0}(t)) > 0$,
    \item\label{it:pfvitesses:3} $X^{\gamma}_{v_0}(t) \in \mathcal{X}$ and $\Delta_x u^{\gamma}(t, X^{\gamma}_{v_0}(t)) = 0$.
  \end{enumerate}
  In case~\eqref{it:pfvitesses:1}, we deduce from the discussion above that $\lambda^{\gamma}\{\bu\}(t, \cdot)$ is continuous at $X^{\gamma}_{v_0}(t)$, and therefore by~\cite[Exercise~2.10~(a)]{billingsley}, we have
  \begin{equation*}
    \lim_{\epsilon \dto 0} \int_{x \in \R} \varphi(t,x) \lambda^{\gamma}\{\bu\}(t,x) \dd_x(\beta_{\epsilon} \circ u^{\gamma})(t,x) = \varphi(t,X^{\gamma}_{v_0}(t)) \lambda^{\gamma}\{\bu\}(t,X^{\gamma}_{v_0}(t)).
  \end{equation*}
  In case~\eqref{it:pfvitesses:2}, we also have $\Delta_x (\beta_{\epsilon} \circ u^{\gamma})(t,X^{\gamma}_{v_0}(t)) > 0$ and 
  \begin{equation*}
    \begin{aligned}
      & \int_{x \in \R} \varphi(t,x) \lambda^{\gamma}\{\bu\}(t,x) \dd_x(\beta_{\epsilon} \circ u^{\gamma})(t,x)\\
      & \qquad = \int_{x \not= X^{\gamma}_{v_0}(t)} \varphi(t,x) \lambda^{\gamma}\{\bu\}(t,x) \dd_x(\beta_{\epsilon} \circ u^{\gamma})(t,x)\\
      & \qquad \quad + \varphi(t, X^{\gamma}_{v_0}(t)) \lambda^{\gamma}\{\bu\}(t,X^{\gamma}_{v_0}(t))\left(\beta_{\epsilon}(u^{\gamma}(t, X^{\gamma}_{v_0}(t))) - \beta_{\epsilon}(u^{\gamma}(t, X^{\gamma}_{v_0}(t)^-))\right).
    \end{aligned}
  \end{equation*}
  By~\eqref{eq:limbeta}, if $u^{\gamma}(t, X^{\gamma}_{v_0}(t)^-) < v_0$, then $\beta_{\epsilon}(u^{\gamma}(t, X^{\gamma}_{v_0}(t))) - \beta_{\epsilon}(u^{\gamma}(t, X^{\gamma}_{v_0}(t)^-))$ converges to $1$ when $\epsilon$ goes to $0$, while the integral over $\R \setminus \{X^{\gamma}_{v_0}(t)\}$ vanishes due to the boundedness of $\varphi$ and $\lambda^{\gamma}\{\bu\}$. On the other hand, the set $\mathcal{V}_1(t)$ of values of $v_0$ such that $u^{\gamma}(t, X^{\gamma}_{v_0}(t)^-) = v_0$, is countable. We finally prove that the set $\mathcal{V}_2(t)$ of values of $v_0$ corresponding to case~\eqref{it:pfvitesses:3} is also countable. Indeed, in the latter case, $X^{\gamma}_{v_0}(t)$ belongs to the countable set $\mathcal{X}$. Assuming that there exists $v_0' \not= v_0$ such that $X^{\gamma}_{v_0'}(t) = X^{\gamma}_{v_0}(t)$ implies that $\Delta_x u^{\gamma}(t, X^{\gamma}_{v_0}(t)) > 0$ and therefore is a contradiction with the fact that $v_0$ is in case~\eqref{it:pfvitesses:3}. As a consequence, one can associate each $x \in \mathcal{X}$ with at most one $v_0$ in case~\eqref{it:pfvitesses:3} such that $x=X^{\gamma}_{v_0}(t)$, and therefore the set $\mathcal{V}_2(t)$ is countable.
  
  As a conclusion, for all $t \geq 0$, we have constructed a countable set $\mathcal{V}(t) := \mathcal{V}_1(t) \cup \mathcal{V}_2(t)$ such that, for $v_0 \in (0,1) \setminus \mathcal{V}(t)$,
  \begin{equation*}
    \lim_{\epsilon \dto 0} \int_{x \in \R} \varphi(t,x) \lambda^{\gamma}\{\bu\}(t,x) \dd_x(\beta_{\epsilon} \circ u^{\gamma})(t,x) = \varphi(t,X^{\gamma}_{v_0}(t)) \lambda^{\gamma}\{\bu\}(t,X^{\gamma}_{v_0}(t)).
  \end{equation*}
  By the Fubini Theorem, there exists a negligible set $\mathcal{V} \subset (0,1)$ such that, for all $v_0 \not\in \mathcal{V}$, we have $v_0 \not\in \mathcal{V}(t)$, $\dd t$-almost everywhere. As a consequence, for $v_0 \in (0,1) \setminus \mathcal{V}$, the Dominated Convergence Theorem yields
  \begin{equation*}
    \lim_{\epsilon \dto 0} \int_{t=0}^{+\infty} \int_{x \in \R} \varphi(t,x) \lambda^{\gamma}\{\bu\}(t,x) \dd_x(\beta_{\epsilon} \circ u^{\gamma})(t,x)\dd t  = \int_{t=0}^{+\infty} \varphi(t,X^{\gamma}_{v_0}(t)) \lambda^{\gamma}\{\bu\}(t,X^{\gamma}_{v_0}(t)) \dd t,
  \end{equation*}
  which completes the proof.
  
  \sk
  \noindent {\em Proof of sufficiency.} We assume that, for all $v \in (0,1)$, the process $(X^{\gamma}_v(t))_{t \geq 0}$ is Lipschitz continuous and satisfies~\eqref{eq:Xpoint}. Let $\beta : [0,1] \to \R$ be a $\Cs^1$ increasing function such that $\beta(0)=0$ and $\beta(1)=1$, and let $\varphi \in \Cs^{1,0}_{\mathrm{c}}([0,+\infty)\times\R, \R)$. Let us define
  \begin{equation*}
    \psi(t,x) := \int_{y=x}^{+\infty} \varphi(t,y)\dd y.
  \end{equation*}
  For all $v \in (0,1)$, for all $T \geq 0$,
  \begin{equation*}
    \psi(T, X^{\gamma}_v(T)) = \psi(0, X^{\gamma}_v(0)) + \int_{t=0}^T \left(\partial_t \psi(t, X^{\gamma}_v(t)) + \partial_x \psi(t, X^{\gamma}_v(t)) \lambda^{\gamma}\{\bu\}(t, X^{\gamma}_v(t)) \right) \dd t.
  \end{equation*}
  Taking $T$ large enough to cancel the left-hand side, multiplying by $\beta'(v)$, integrating over $(0,1)$ and performing the change of variable $w = \beta(v)$, we obtain
  \begin{equation*}
    \begin{aligned}
      0 & = \int_{w=0}^1 \psi\left(0, X^{\gamma}_{\beta^{-1}(w)}(0)\right) \dd w\\
        & \quad + \int_{w=0}^1\int_{t=0}^{+\infty} \left(\partial_t \psi\left(t, X^{\gamma}_{\beta^{-1}(w)}(t)\right) + \partial_x \psi\left(t, X^{\gamma}_{\beta^{-1}(w)}(t)\right) \lambda^{\gamma}\{\bu\}(t, X^{\gamma}_{\beta^{-1}(w)}(t)) \right) \dd t \dd w\\
        & = \int_{x \in \R} \psi(0, x) \dd_x (\beta \circ u^{\gamma})(0,x)+ \int_{t=0}^{+\infty}\int_{x \in \R} \left(\partial_t \psi(t, x) + \partial_x \psi(t, x) \lambda^{\gamma}\{\bu\}(t, x) \right) \dd_x (\beta \circ u^{\gamma})(t,x)\dd t,
    \end{aligned}
  \end{equation*}
  thanks to Lemma~\ref{lem:CDFm1}.
  
  By the Fubini Theorem,
  \begin{equation*}
    \int_{x \in \R} \psi(0, x) \dd_x (\beta \circ u^{\gamma})(0,x) = \int_{y \in \R} \varphi(t,y) \beta(u^{\gamma}_0(y)) \dd y
  \end{equation*}
  and similarly,
  \begin{equation*}
    \int_{t=0}^{+\infty}\int_{x \in \R} \partial_t \psi(t, x) \dd_x (\beta \circ u^{\gamma})(t,x)\dd t = \int_{t=0}^{+\infty}\int_{y \in \R} \partial_t \varphi(t, y) \beta(u^{\gamma}(t,y)) \dd y\dd t.
  \end{equation*}
  On the other hand, it is straightforward that
  \begin{equation*}
    \begin{aligned}
      & \int_{t=0}^{+\infty}\int_{x \in \R} \partial_x \psi(t, x) \lambda^{\gamma}\{\bu\}(t, x) \dd_x (\beta \circ u^{\gamma})(t,x)\dd t\\
      & \qquad = - \int_{t=0}^{+\infty}\int_{x \in \R} \varphi(t, x) \lambda^{\gamma}\{\bu\}(t, x) \dd_x (\beta \circ u^{\gamma})(t,x)\dd t,
    \end{aligned}
  \end{equation*}
  which concludes the proof.  
\end{proof}

Combining~\eqref{eq:MSPD:diff} with Proposition~\ref{prop:vitesses}, we see that, for all $\x \in \Dnd$, the vector of empirical CDFs $\bu[\x]$ of the MSPD started at $\x$ is a renormalised solution to~\eqref{eq:syst} in all its coordinates. Note that it is also easy to give a direct proof of this fact, by replacing the weight $1/n$ of the particle $\gamma:k$ with $\beta(k/n)-\beta((k-1)/n)$ in the proof of Proposition~\ref{prop:MSPDsol} --- which actually amounts to mimicking the proof of sufficiency above. 

As a consequence, if the set of renormalised solutions enjoyed a closedness property of the same nature as Proposition~\ref{prop:closedness}, then one would expect the approximation procedure described in Section~\ref{s:existence} to imply that the probabilistic solutions constructed in Theorem~\ref{theo:existence} are also renormalised solutions in all their coordinates, and therefore that the corresponding trajectories $(\bX_v(t))_{t \geq 0}$ satisfy the characteristic equation~\eqref{eq:Xpoint}. However, it seems to us that the set of renormalised solutions does not enjoy such a closedness property, and therefore we do not know, in general, whether probabilistic solutions obtained by Theorem~\ref{theo:existence} are renormalised solutions. The following lemma describes a situation in which this is actually the case.

\begin{lem}[Renormalised solutions obtained from Theorem~\ref{theo:existence}]\label{lem:renormdtppcont}
  Under Assumptions~\eqref{ass:C} and~\eqref{ass:USH}, let $\bu$ be a probabilistic solution to~\eqref{eq:syst} obtained by Theorem~\ref{theo:existence}. For all $\gamma \in \{1, \ldots, d\}$, if $\dd t$-almost everywhere, the function $u^{\gamma}(t,\cdot)$ is continuous on the real line, then $\bu$ is a renormalised solution in the $\gamma$-th coordinate. 
\end{lem}
Monotonicity conditions on the function $\lambda^{\gamma}$ ensuring that, $\dd t$-almost everywhere, the function $u^{\gamma}(t,\cdot)$ is continuous on the real line, will be discussed in Section~\ref{s:rar}.
\begin{proof}
  Let $\bu = (u^1, \ldots, u^d)$ be a probabilistic solution to~\eqref{eq:syst} obtained by Theorem~\ref{theo:existence}, so that there exists a sequence $(\x(n_{\ell}))_{\ell \geq 1}$ of initial configurations such that the sequence of empirical measures $\upmu[\x(n_{\ell})]$ converges weakly, when $\ell$ grows to infinity, to some probability measure $\bar{\upmu} \in \Ms$ such that $u^{\gamma}(t,x) = H*\bar{\upmu}^{\gamma}_t(x)$. In the sequel of the proof we drop the index $\ell$ and assume for convenience that $\upmu[\x(n)]$ converges weakly to $\bar{\upmu}$ when $n$ grows to infinity. Recall that we denote by $\bu[\x(n)]$ the vector of empirical CDFs of the MSPD started at $\x(n)$. We furthermore assume that $\gamma \in \{1, \ldots, d\}$ is such that, $\dd t$-almost everywhere, the function $u^{\gamma}(t,\cdot)$ is continuous on the real line.
  
  Given a $\Cs^1$ increasing function $\beta : [0,1] \to \R$ such that $\beta(0)=0$ and $\beta(1)=1$ and a test function $\varphi \in \Cs^{1,0}_{\mathrm{c}}([0,+\infty)\times\R, \R)$, the discussion above yields
  \begin{equation*}
    \begin{aligned}
      & \int_{t=0}^{+\infty} \int_{x \in \R} \partial_t\varphi(t,x) \beta(u^{\gamma}[\x(n)](t,x)) \dd x\dd t + \int_{x \in \R} \varphi(0,x) \beta(u^{\gamma}[\x(n)](0,x)) \dd x\\
      & \qquad = \int_{t=0}^{+\infty} \int_{x \in \R} \varphi(t,x) \lambda^{\gamma}\{\bu[\x(n)]\}(t,x) \dd_x (\beta \circ u^{\gamma}[\x(n)])(t,x) \dd t,
    \end{aligned}
  \end{equation*}
  and to prove Lemma~\ref{lem:renormdtppcont}, we have to take the limit of this equality when $n$ grows to infinity. First, since by Corollary~\ref{cor:cvCDFs}, $u^{\gamma}[\x(n)](t,x)$ converges $\dd x$-almost everywhere to $u^{\gamma}(t,x)$, for all $t \geq 0$, the Dominated Convergence Theorem yields
  \begin{equation*}
    \begin{aligned}
      & \lim_{n \to +\infty} \int_{t=0}^{+\infty} \int_{x \in \R} \partial_t\varphi(t,x) \beta(u^{\gamma}[\x(n)](t,x)) \dd x\dd t + \int_{x \in \R} \varphi(0,x) \beta(u^{\gamma}[\x(n)](0,x)) \dd x\\
      & \qquad = \int_{t=0}^{+\infty} \int_{x \in \R} \partial_t\varphi(t,x) \beta(u^{\gamma}(t,x)) \dd x\dd t + \int_{x \in \R} \varphi(0,x) \beta(u^{\gamma}(0,x)) \dd x.
    \end{aligned}
  \end{equation*}
  The function $\beta$ being continuous and increasing, it admits a continuous and increasing inverse $\beta^{-1}$ and for any CDF $F$ on the real line, the CDF $\beta(F)$ is such that, for all $v\in(0,1)$, $(\beta(F))^{-1}(v)=F^{-1}(\beta^{-1}(v))$. Therefore, by Lemma~\ref{lem:CDFm1}, for any bounded and measurable function $f$ on the real line,
  \begin{equation*}
    \int_{x\in\R}f(x)\beta(F(x))\dd x=\int_{w=0}^1f(F^{-1}(\beta^{-1}(w)))\dd w=\int_{v=0}^1f(F^{-1}(v))\beta'(v)\dd v.
  \end{equation*}
  Therefore, to conclude the proof, it is enough to check that for any $t \geq 0$ such that $u^{\gamma}(t,\cdot)$ is continuous,
  \begin{equation*}
    \begin{aligned}
      & \lim_{n \to +\infty} \int_{v=0}^1 \varphi\left(t,u^{\gamma}[\x(n)](t,\cdot)^{-1}(v)\right) \lambda^{\gamma}\{\bu[\x(n)]\}\left(t,u^{\gamma}[\x(n)](t,\cdot)^{-1}(v)\right) \beta'(v)\dd v\\
      & \qquad = \int_{v=0}^1 \varphi\left(t,u^{\gamma}(t,\cdot)^{-1}(v)\right) \lambda^{\gamma}\{\bu\}\left(t,u^{\gamma}(t,\cdot)^{-1}(v)\right) \beta'(v)\dd v.
    \end{aligned}
  \end{equation*}
  Owing to Lemma~\ref{lem:cvCDF}, $\varphi(t,u^{\gamma}[\x(n)](t,\cdot)^{-1}(v))$ converges to $\varphi(t,u^{\gamma}(t,\cdot)^{-1}(v))$, $\dd v$-almost everywhere. Therefore, by the Dominated Convergence Theorem, it now suffices to show that, $\dd v$-almost everywhere, $\lambda^{\gamma}\{\bu[\x(n)]\}(t,u^{\gamma}[\x(n)](t,\cdot)^{-1}(v))$ converges to $\lambda^{\gamma}\{\bu\}(t,u^{\gamma}(t,\cdot)^{-1}(v))$. Since $u^{\gamma}(t,\cdot)$ is continuous, then Lemma~\ref{lem:FnGn} already yields, for all $\gamma' \not= \gamma$, 
  \begin{equation*}
    \lim_{n \to +\infty} u^{\gamma'}[\x(n)]\left(t, u^{\gamma}[\x(n)](t,\cdot)^{-1}(v)\right) = u^{\gamma'}\left(t, u^{\gamma}(t,\cdot)^{-1}(v)\right), \qquad \text{$\dd v$-almost everywhere.}
  \end{equation*}
  Besides, applying Lemma~\ref{lem:FnGn} with $F$ and $G$ both equal to the continuous function $u^{\gamma}(t,\cdot)$, one obtains that, $\dd v$-almost everywhere,
  \begin{equation*}
    \begin{aligned}
      & \lim_{n \to +\infty} u^{\gamma}[\x(n)]\left(t, u^{\gamma}[\x(n)](t,\cdot)^{-1}(v)^-\right) = u^{\gamma}\left(t, u^{\gamma}(t,\cdot)^{-1}(v)\right),\\
      & \lim_{n \to +\infty} u^{\gamma}[\x(n)]\left(t, u^{\gamma}[\x(n)](t,\cdot)^{-1}(v)\right) = u^{\gamma}\left(t, u^{\gamma}(t,\cdot)^{-1}(v)\right).
    \end{aligned}
  \end{equation*}
  As a consequence, we can now use the Dominated Convergence Theorem to pass to the limit in the definition~\eqref{eq:dlambda} of $\lambda^{\gamma}\{\bu[\x(n)]\}(t,u^{\gamma}[\x(n)](t,\cdot)^{-1}(v))$, and thereby complete the proof.
\end{proof}


\section{Continuity of rarefaction coordinates}\label{s:rar}

In this section, we discuss the continuity of probabilistic solutions to the system~\eqref{eq:syst} obtained by Theorem~\ref{theo:existence}, under the following diagonal monotonicity conditions on the function $\lambda^{\gamma}$: we shall say that a coordinate $\gamma \in \{1, \ldots, d\}$ is
\begin{itemize}
  \item a {\em rarefaction coordinate} if, for all $(u^1, \ldots, u^{\gamma-1}, u^{\gamma+1}, \ldots, u^d) \in [0,1]^{d-1}$, for all $\uu, \ou \in [0,1]$ with $\uu \leq \ou$, 
  \begin{equation*}
   \lambda^{\gamma}(u^1, \ldots, u^{\gamma-1}, \ou, u^{\gamma+1}, \ldots, u^d) \geq \lambda^{\gamma}(u^1, \ldots, u^{\gamma-1}, \uu, u^{\gamma+1}, \ldots, u^d),
  \end{equation*}
  
  \item a {\em strong rarefaction coordinate} if there exists a positive constant $c > 0$ such that, for all $(u^1, \ldots, u^{\gamma-1}, u^{\gamma+1}, \ldots, u^d) \in [0,1]^{d-1}$, for all $\uu, \ou \in [0,1]$ with $\uu \leq \ou$, 
  \begin{equation}\label{eq:src}
   \lambda^{\gamma}(u^1, \ldots, u^{\gamma-1}, \ou, u^{\gamma+1}, \ldots, u^d) - \lambda^{\gamma}(u^1, \ldots, u^{\gamma-1}, \uu, u^{\gamma+1}, \ldots, u^d) \geq c(\ou-\uu).
  \end{equation}
  Notice that this condition implies that for all $(u^1, \ldots, u^{\gamma-1}, u^{\gamma+1}, \ldots, u^d) \in [0,1]^{d-1}$, for all $0\leq \uu<\ou\leq \uv<\ov\leq 1$,
  \begin{equation}\label{eq:src2}
    \begin{aligned}
      & \frac{1}{\ov-\uv}\int_{w=\uv}^{\ov} \lambda^{\gamma}(u^1, \ldots, u^{\gamma-1},w, u^{\gamma+1}, \ldots, u^d)\dd w-\frac{1}{\ou-\uu}\int_{z=\uu}^{\ou} \lambda^{\gamma}(u^1, \ldots, u^{\gamma-1},z, u^{\gamma+1}, \ldots, u^d)\dd z\\
      & \qquad \geq \frac{1}{(\ov-\uv)(\ou-\uu)}\int_{w=\uv}^{\ov} \int_{z=\uu}^{\ou} c(w-z)\dd z\dd w = \frac{c}{2}\left(\ov+\uv-\ou-\uu\right).
    \end{aligned}
  \end{equation}
\end{itemize}
In Subsection~\ref{ss:rar}, we address rarefaction coordinates and obtain a control on the modulus of continuity of our solutions in terms of the initial data, which follows from a uniform estimate on the MSPD. In particular, we show that, if $\gamma$ is a rarefaction coordinate and $u^{\gamma}_0$ is continuous, then $u^{\gamma}$ is continuous on $[0,+\infty) \times \R$. In Subsection~\ref{ss:strongrar}, we prove that, if $\gamma \in \{1, \ldots, d\}$ is a strong rarefaction coordinate, then $u^{\gamma}$ is continuous on $(0,+\infty)\times\R$ even when $u^{\gamma}_0$ fails to be continuous. 


\subsection{Continuity of rarefaction coordinates}\label{ss:rar} For rarefaction coordinates, we first obtain the following uniform estimate on the MSPD.

\begin{prop}[Discrete estimate for rarefaction coordinates]\label{prop:modulus:discr}
  Under Assumptions~\eqref{ass:LC} and~\eqref{ass:USH}, let $\gamma \in \{1, \ldots, d\}$ be a rarefaction coordinate. Then, for all $n \geq 2$, for all $\x \in \Dnd$, for all $k \in \{1, \ldots, n-1\}$,
  \begin{equation*}
    \inf_{t \geq 0} (\Phi_{k+1}^{\gamma}(\x;t) - \Phi_k^{\gamma}(\x;t)) \geq \frac{1}{\ConstRar} (x_{k+1}^{\gamma}-x_k^{\gamma}),
  \end{equation*}
  where
  \begin{equation}\label{eq:ConstRar}
    \ConstRar := \exp\left((d-1)\frac{\ConstLip}{\ConstUSH}\right) \geq 1.
  \end{equation}
\end{prop} 
\begin{proof}
  Let us fix a rarefaction coordinate $\gamma \in \{1, \ldots, d\}$, and $n \geq 2$, $\x \in \Dnd$, and finally $k \in \{1, \ldots, n-1\}$. For all $\gamma':k' \in \Part$ such that $\gamma' \not= \gamma$, we denote by $[T^-(\gamma':k'), T^+(\gamma':k')]$ the time interval on which the particle $\gamma':k'$ lies between the particles $\gamma:k$ and $\gamma:(k+1)$. More precisely, if $\gamma'<\gamma$, then
  \begin{equation*}
    T^-(\gamma':k') := \begin{cases}
      \tinter_{\gamma':k', \gamma:k}(\x) & \text{if $(\gamma':k', \gamma:k) \in \Rb(\x)$,}\\
      0 & \text{otherwise,}
    \end{cases}
  \end{equation*}
  and
  \begin{equation*}
    T^+(\gamma':k') := \begin{cases}
      \tinter_{\gamma':k', \gamma:(k+1)}(\x) & \text{if $(\gamma':k', \gamma:(k+1)) \in \Rb(\x)$,}\\
      0 & \text{otherwise;}
    \end{cases}
  \end{equation*}
  while for $\gamma'>\gamma$,
  \begin{equation*}
    T^-(\gamma':k') := \begin{cases}
      \tinter_{\gamma:(k+1), \gamma':k'}(\x) & \text{if $(\gamma:(k+1), \gamma':k') \in \Rb(\x)$,}\\
      0 & \text{otherwise,}
    \end{cases}
  \end{equation*}
  and
  \begin{equation*}
    T^+(\gamma':k') := \begin{cases}
      \tinter_{\gamma:k, \gamma':k'}(\x) & \text{if $(\gamma:k, \gamma':k') \in \Rb(\x)$,}\\
      0 & \text{otherwise;}
    \end{cases}
  \end{equation*}
  so that we finally have, for all $t \geq 0$, 
  \begin{equation}\label{eq:TplusTmoins}
    T^-(\gamma':k') \leq t < T^+(\gamma':k') \qquad \text{if and only if} \qquad \Phi_k^{\gamma}(\x;t) \leq \Phi_{k'}^{\gamma'}(\x;t) < \Phi_{k+1}^{\gamma}(\x;t).
  \end{equation}
  
  We first prove the following estimate: for all $\gamma':k' \in \Part$ such that $\gamma' \not= \gamma$,
  \begin{equation}\label{eq:pfmodulus:diffT}
    \forall t \in [T^-(\gamma':k'), T^+(\gamma':k')], \qquad t - T^-(\gamma':k') \leq \frac{1}{\ConstUSH}\left(\Phi_{k+1}^{\gamma}(\x;t) - \Phi_k^{\gamma}(\x;t)\right).
  \end{equation}
  If $T^+(\gamma':k') = 0$ then the inequality is trivial. If $T^+(\gamma':k') > 0$, then assuming that $\gamma'<\gamma$ and using~\eqref{eq:TplusTmoins}, we obtain, for all $t \in [T^-(\gamma':k'), T^+(\gamma':k')]$,
  \begin{equation*}
    \begin{aligned}
      \Phi_{k+1}^{\gamma}(\x;t) & \geq \Phi_{k'}^{\gamma'}(\x; t)\\
      & = \Phi_{k'}^{\gamma'}(\x; T^-(\gamma':k')) + \int_{s=T^-(\gamma':k')}^t v_{k'}^{\gamma'}(\x;s)\dd s\\
      & \geq \Phi_k^{\gamma}(\x; T^-(\gamma':k')) + \int_{s=T^-(\gamma':k')}^t v_{k'}^{\gamma'}(\x;s)\dd s\\
      & = \Phi_k^{\gamma}(\x; t) - \int_{s=T^-(\gamma':k')}^t v_k^{\gamma}(\x;s)\dd s + \int_{s=T^-(\gamma':k')}^t v_{k'}^{\gamma'}(\x;s)\dd s,
    \end{aligned}
  \end{equation*}
  so that
  \begin{equation*}
    \Phi_{k+1}^{\gamma}(\x;t) - \Phi_k^{\gamma}(\x; t) \geq \int_{s=T^-(\gamma':k')}^t \left(v_{k'}^{\gamma'}(\x;s)-v_k^{\gamma}(\x;s)\right)\dd s \geq \ConstUSH(t-T^-(\gamma':k')),
  \end{equation*}
  which yields~\eqref{eq:pfmodulus:diffT}. The case $\gamma'>\gamma$ works similarly.
  
  Let us now fix $0 \leq t_1 \leq t_2$. Certainly,
  \begin{equation}\label{eq:pfmodulus:1}
    \Phi_{k+1}^{\gamma}(\x;t_2) - \Phi_k^{\gamma}(\x;t_2) = \Phi_{k+1}^{\gamma}(\x;t_1) - \Phi_k^{\gamma}(\x;t_1) + \int_{s=t_1}^{t_2} (v_{k+1}^{\gamma}(\x;s) - v_k^{\gamma}(\x;s))\dd s.
  \end{equation}
  For all $s \in [t_1, t_2]$, either $\Phi_k^{\gamma}(\x;s) = \Phi_{k+1}^{\gamma}(\x;s)$, in which case $v_k^{\gamma}(\x;s) = v_{k+1}^{\gamma}(\x;s)$; or there exist $\uk \leq k$ and $\ok \geq k+1$ such that $\clu_k^{\gamma}(\x;s) = \gamma:\uk\cdots k$, $\clu_{k+1}^{\gamma}(\x;s) = \gamma:(k+1)\cdots\ok$, and thanks to the fact that $\gamma$ is a rarefaction coordinate, we have
  \begin{equation*}
    \begin{aligned}
      v_k^{\gamma}(\x;s) & = \int_{\theta=0}^1 \lambda^{\gamma}\left(\omega^1_{\gamma:k}(\Phi(\x;s)), \ldots, (1-\theta)\frac{\uk-1}{n} + \theta \frac{k}{n}, \ldots, \omega^d_{\gamma:k}(\Phi(\x;s))\right)\dd \theta\\
      & \leq \lambda^{\gamma}\left(\omega^1_{\gamma:k}(\Phi(\x;s)), \ldots, \frac{k}{n}, \ldots, \omega^d_{\gamma:k}(\Phi(\x;s))\right),
    \end{aligned}
  \end{equation*}
  as well as
  \begin{equation*}
    \begin{aligned}
      v_{k+1}^{\gamma}(\x;s) & = \int_{\theta=0}^1 \lambda^{\gamma}\left(\omega^1_{\gamma:(k+1)}(\Phi(\x;s)), \ldots, (1-\theta)\frac{k}{n} + \theta \frac{\ok}{n}, \ldots, \omega^d_{\gamma:(k+1)}(\Phi(\x;s))\right)\dd \theta\\
      & \geq \lambda^{\gamma}\left(\omega^1_{\gamma:(k+1)}(\Phi(\x;s)), \ldots, \frac{k}{n}, \ldots, \omega^d_{\gamma:(k+1)}(\Phi(\x;s))\right).
    \end{aligned}
  \end{equation*}
  In both cases, we deduce that
  \begin{equation*}
    \begin{aligned}
      v_{k+1}^{\gamma}(\x;s) - v_k^{\gamma}(\x;s) & \geq \lambda^{\gamma}\left(\omega^1_{\gamma:(k+1)}(\Phi(\x;s)), \ldots, \frac{k}{n}, \ldots, \omega^d_{\gamma:(k+1)}(\Phi(\x;s))\right)\\
      & \quad - \lambda^{\gamma}\left(\omega^1_{\gamma:k}(\Phi(\x;s)), \ldots, \frac{k}{n}, \ldots, \omega^d_{\gamma:k}(\Phi(\x;s))\right)\\
      & \geq -\ConstLip \sum_{\gamma' \not= \gamma} \left|\omega^{\gamma'}_{\gamma:(k+1)}(\Phi(\x;s))-\omega^{\gamma'}_{\gamma:k}(\Phi(\x;s))\right|
    \end{aligned}
  \end{equation*}
  owing to Assumption~\eqref{ass:LC}. Besides, it follows from~\eqref{eq:TplusTmoins} that, for all $\gamma' \not= \gamma$,
  \begin{equation*}
    \left|\omega^{\gamma'}_{\gamma:(k+1)}(\Phi(\x;s))-\omega^{\gamma'}_{\gamma:k}(\Phi(\x;s))\right| = \frac{1}{n} \sum_{k'=1}^n \ind{T^-(\gamma':k') \leq s < T^+(\gamma':k')},
  \end{equation*}
  which we plug into~\eqref{eq:pfmodulus:1} in order to get
  \begin{equation}\label{eq:pfmodulus:2}
    \begin{aligned}
      & \left(\Phi_{k+1}^{\gamma}(\x;t_1) - \Phi_k^{\gamma}(\x;t_1)\right) - \left(\Phi_{k+1}^{\gamma}(\x;t_2) - \Phi_k^{\gamma}(\x;t_2)\right)\\
      & \qquad \leq \frac{\ConstLip}{n} \sum_{\gamma':k' \in \Part, \gamma' \not= \gamma} \int_{s=t_1}^{t_2} \ind{T^-(\gamma':k') \leq s < T^+(\gamma':k')}\dd s\\
      & \qquad \leq \frac{\ConstLip}{n} \sum_{\substack{\gamma':k' \in \Part, \gamma' \not= \gamma\\ T^-(\gamma':k') < t_2, T^+(\gamma':k') > t_1}} \left( T^+(\gamma':k') \wedge t_2 - T^-(\gamma':k') \right)\\
      & \qquad \leq \frac{\ConstLip}{n\ConstUSH} \sum_{\substack{\gamma':k' \in \Part, \gamma' \not= \gamma\\ T^-(\gamma':k') < t_2, T^+(\gamma':k') > t_1}} \left( \Phi_{k+1}^{\gamma}(\x;T^+(\gamma':k') \wedge t_2) - \Phi_k^{\gamma}(\x;T^+(\gamma':k') \wedge t_2) \right),
    \end{aligned}
  \end{equation}
  where the last inequality follows from~\eqref{eq:pfmodulus:diffT}.
  
  Let $M \in \{0, \ldots, n(d-1)\}$ refer to the number of particles $\gamma':k' \in \Part$ such that $\gamma' \not= \gamma$ and $T^-(\gamma':k') < t_2$. Let $T_1 \geq T_2 \geq \cdots \geq T_M$ refer to the nonincreasing reordering of the corresponding quantities $T^+(\gamma':k') \wedge t_2$, and let us define $T_0 := t_2 \geq T_1$ and $T_{M+1} := 0 \leq T_M$. For all $m \in \{1, \ldots, M+1\}$, applying~\eqref{eq:pfmodulus:2} with $t_1 = T_m$ yields
  \begin{equation*}
    \begin{aligned}
      & \Phi_{k+1}^{\gamma}(\x;T_m) - \Phi_k^{\gamma}(\x;T_m) \\
      & \qquad \leq \Phi_{k+1}^{\gamma}(\x;T_0) - \Phi_k^{\gamma}(\x;T_0) + \frac{\ConstLip}{n\ConstUSH} \sum_{m' : T_{m'} > T_m} \left( \Phi_{k+1}^{\gamma}(\x;T_{m'}) - \Phi_k^{\gamma}(\x;T_{m'}) \right)\\
      & \qquad \leq \Phi_{k+1}^{\gamma}(\x;T_0) - \Phi_k^{\gamma}(\x;T_0) + \frac{\ConstLip}{n\ConstUSH} \sum_{m'=1}^{m-1} \left( \Phi_{k+1}^{\gamma}(\x;T_{m'}) - \Phi_k^{\gamma}(\x;T_{m'}) \right),
    \end{aligned}
  \end{equation*}
  which yields, for all $m \in \{1, \ldots, M+1\}$,
  \begin{equation*}
    \Phi_{k+1}^{\gamma}(\x;T_m) - \Phi_k^{\gamma}(\x;T_m) \leq \left(1 + \frac{\ConstLip}{n\ConstUSH}\right)^{m-1} \left(\Phi_{k+1}^{\gamma}(\x;T_0) - \Phi_k^{\gamma}(\x;T_0)\right).
  \end{equation*}
  In particular, for $m=M+1$,
  \begin{equation*}
    \begin{aligned}
      x_{k+1}^{\gamma} - x_k^{\gamma} = \Phi_{k+1}^{\gamma}(\x;0) - \Phi_k^{\gamma}(\x;0) & \leq \left(1 + \frac{\ConstLip}{n\ConstUSH}\right)^M \left(\Phi_{k+1}^{\gamma}(\x;t_2) - \Phi_k^{\gamma}(\x;t_2)\right)\\
      & \leq \ConstRar \left(\Phi_{k+1}^{\gamma}(\x;t_2) - \Phi_k^{\gamma}(\x;t_2)\right).
    \end{aligned}
  \end{equation*}
  Since $t_2 \geq 0$ is arbitrary, the proof is completed.
\end{proof}

Let us recall that, given a bounded function $F : \R \to \R$, the {\em modulus of continuity} $\omega_F$ of $F$ is defined by
\begin{equation*}
  \forall \delta > 0, \qquad \omega_F(\delta) := \sup_{x,y \in \R : |x-y| \leq \delta} |F(x)-F(y)|,
\end{equation*}
see~\cite[p.~80]{billingsley}. In particular, if $F$ is the CDF of the probability measure $m$ on $\R$, then
\begin{equation*}
  \omega_F(\delta) = \sup_{x \in \R} F(x+\delta)-F(x) = \sup_{x \in \R} m((x,x+\delta]).
\end{equation*}
Proposition~\ref{prop:modulus:discr} yields the following control of the modulus of continuity for rarefaction coordinates.

\begin{cor}[Control of the modulus of continuity for rarefaction coordinates]\label{cor:modulus}
  Under the assumptions of Proposition~\ref{prop:modulus:discr}, let $\gamma \in \{1, \ldots, d\}$ be a rarefaction coordinate. Let $\bu$ be a probabilistic solution to the system~\eqref{eq:syst} obtained by Theorem~\ref{theo:existence}, and let $(\bX_v(t))_{t \geq 0}$, $v \in (0,1)$ be the trajectories associated with $\bu$ defined by \eqref{eq:defbX}.
  \begin{enumerate}[label=(\roman*), ref=\roman*]
    \item\label{it:modulus:1} For all $s\geq 0$ and all $\uv, \ov \in (0,1)$ such that $\uv \leq \ov$,
    \begin{equation*}
      \inf_{t \geq s} \left(X^{\gamma}_{\ov}(t) - X^{\gamma}_{\uv}(t)\right) \geq \frac{1}{\ConstRar}\left(X^{\gamma}_{\ov}(s) - X^{\gamma}_{\uv}(s)\right),
    \end{equation*}
    where we recall the definition~\eqref{eq:ConstRar} of $\ConstRar$.
    
    \item\label{it:modulus:2} If for some $s\geq 0$, $u^{\gamma}(s, \cdot)$ is continuous on $\R$, then $u^{\gamma}$ is continuous on $[s,+\infty) \times \R$.
    
    \item\label{it:modulus:3} For all $\delta > 0$, for all $t \geq s\geq 0$, $\omega_{u^{\gamma}(t,\cdot)}(\delta) \leq \omega_{u^{\gamma}(s,\cdot)}(\ConstRar \delta)$.
  \end{enumerate}
\end{cor}

\begin{proof}
  In the proof, for notational simplicity, we do not consider subsequences and assume that $\upmu[\x(n)]$ converges weakly to $\bar{\upmu}$ such that $u^{\gamma}(t,x) = H*\bar{\upmu}^{\gamma}_t(x)$ when $n$ grows to infinity.
  
  \sk
  \noindent {\em Proof of~\eqref{it:modulus:1}.} Let us fix $\uv, \ov \in (0,1)$ with $\uv \leq \ov$, and let $n$ be large enough to ensure that $\lfloor n \uv \rfloor \geq 1$. By Propositions~\ref{prop:mspd} and \ref{prop:modulus:discr}, we have, for all $t \geq s\geq 0$,
  \begin{equation*}
    \Phi^{\gamma}_{\lfloor n \ov \rfloor}(\x(n);t) - \Phi^{\gamma}_{\lfloor n \uv \rfloor}(\x(n);t) \geq \frac{1}{\ConstRar} \left(\Phi^{\gamma}_{\lfloor n \ov \rfloor}(\x(n);s) - \Phi^{\gamma}_{\lfloor n \uv \rfloor}(\x(n);s)\right),
  \end{equation*}
  that is to say
  \begin{equation*}
    \begin{aligned}
      & u^{\gamma}[\x(n)](t,\cdot)^{-1}\left(\frac{\lfloor n \ov \rfloor}{n}\right) - u^{\gamma}[\x(n)](t,\cdot)^{-1}\left(\frac{\lfloor n \uv \rfloor}{n}\right)\\
      & \qquad \geq \frac{1}{\ConstRar} \left(u^{\gamma}[\x(n)](s,\cdot)^{-1}\left(\frac{\lfloor n \ov \rfloor}{n}\right) - u^{\gamma}[\x(n)](s,\cdot)^{-1}\left(\frac{\lfloor n \uv \rfloor}{n}\right)\right).
    \end{aligned}
  \end{equation*}
  Let us fix $t \geq s \geq 0$, $v \in (0,1)$ and $\epsilon > 0$ such that both $u^{\gamma}(s,\cdot)^{-1}$ and $u^{\gamma}(t,\cdot)^{-1}$ are continuous at $v$ and $v-\epsilon$. Then for $n$ large enough,
  \begin{equation*}
    v - \epsilon \leq \frac{\lfloor n v \rfloor}{n} \leq v,
  \end{equation*}
  so that by Lemma~\ref{lem:cvCDF} and the monotonicity of $u^{\gamma}(s,\cdot)^{-1}$,
  \begin{equation*}
    \begin{aligned}
      u^{\gamma}(s,\cdot)^{-1}(v - \epsilon) & \leq \liminf_{n \to +\infty} u^{\gamma}[\x(n)](s,\cdot)^{-1}\left(\frac{\lfloor n v \rfloor}{n}\right)\\
      & \leq \limsup_{n \to +\infty} u^{\gamma}[\x(n)](s,\cdot)^{-1}\left(\frac{\lfloor n v \rfloor}{n}\right) \leq u^{\gamma}(s,\cdot)^{-1}(v),
    \end{aligned}
  \end{equation*}
  and the same inequality holds at time $t$. Letting $\epsilon$ vanish but keeping $u^{\gamma}(s,\cdot)^{-1}$ and $u^{\gamma}(t,\cdot)^{-1}$ continuous at $v-\epsilon$, we deduce that
  \begin{equation*}
    \begin{aligned}
      & \lim_{n \to +\infty} u^{\gamma}[\x(n)](s,\cdot)^{-1}\left(\frac{\lfloor n v \rfloor}{n}\right) = u^{\gamma}(s,\cdot)^{-1}(v) = X^{\gamma}_v(s),\\
      & \lim_{n \to +\infty} u^{\gamma}[\x(n)](t,\cdot)^{-1}\left(\frac{\lfloor n v \rfloor}{n}\right) = u^{\gamma}(t,\cdot)^{-1}(v) = X^{\gamma}_v(t).
    \end{aligned}
  \end{equation*}
  We deduce that $\dd\uv\dd\ov$-almost everywhere on $\{\uv \leq \ov\}$,
  \begin{equation*}
    X^{\gamma}_{\ov}(t) - X^{\gamma}_{\uv}(t) \geq \frac{1}{\ConstRar}\left(X^{\gamma}_{\ov}(s) - X^{\gamma}_{\uv}(s)\right),
  \end{equation*}
  and since $v \mapsto (X^{\gamma}_v(s),X^{\gamma}_v(t))$ is left continuous, this inequality actually holds for all $\uv,\ov \in (0,1)$ with $\uv \leq \ov$.
  
  \sk
  \noindent {\em Proof of~\eqref{it:modulus:2}.} It follows from the definition of the pseudo-inverse of a CDF $F$ that $F$ is continuous if and only if $F^{-1}$ is increasing. As a consequence, if $u^{\gamma}(s,\cdot)$ is continuous, then $v \mapsto X^{\gamma}_v(s)$ is increasing, so that by~\eqref{it:modulus:1}, $v \mapsto X^{\gamma}_v(t)$ is increasing, and therefore $u^{\gamma}(t, \cdot)$ is continuous on $\R$, for all $t \geq s$. By the Dini Theorem, we conclude that $u^{\gamma}$ is continuous on $[s,+\infty) \times \R$.

  \sk
  \noindent {\em Proof of~\eqref{it:modulus:3}.} Let us fix $\delta > 0$ and $t \geq s\geq 0$. For all $x \in \R$ such that $u^{\gamma}(t,x) < u^{\gamma}(t,x+\delta)$, let $\uv, \ov \in (0,1)$ such that $u^{\gamma}(t,x) < \uv \leq \ov =u^{\gamma}(t, x+\delta)$. By~\eqref{it:pseudoinv:2} in Lemma~\ref{lem:pseudoinv}, $X^{\gamma}_{\uv}(t) > x$ and $X^{\gamma}_{\ov}(t) \leq x +\delta$, which, by~\eqref{it:modulus:1}, implies
  \begin{equation*}
    X^{\gamma}_{\ov}(s) - X^{\gamma}_{\uv}(s) < \ConstRar \delta,
  \end{equation*}
  and therefore $u^{\gamma}(s,X^{\gamma}_{\ov}(s)) - u^{\gamma}(s,X^{\gamma}_{\uv}(s)^-) \leq \omega_{u^{\gamma}(s,\cdot)}(\ConstRar \delta)$ so that, by~\eqref{it:pseudoinv:1} in Lemma~\ref{lem:pseudoinv},
  \begin{equation*}
    u^{\gamma}(t, x+\delta) -\uv  \leq \omega_{u^{\gamma}(s,\cdot)}(\ConstRar \delta).
  \end{equation*}
  Taking $\uv$ arbitrarily close to $u^{\gamma}(t,x)$, we deduce that
  \begin{equation*}
    u^{\gamma}(t, x+\delta) - u^{\gamma}(t, x) \leq \omega_{u^{\gamma}(s,\cdot)}(\ConstRar \delta),
  \end{equation*}
  which finally yields
  \begin{equation*}
    \omega_{u^{\gamma}(t, \cdot)}(\delta) \leq \omega_{u^{\gamma}(s,\cdot)}(\ConstRar \delta)
  \end{equation*}
  since $x$ is arbitrary.
\end{proof}


\subsection{Strong rarefaction coordinates}\label{ss:strongrar} We now address strong rarefaction coordinates. A key point in the proof of Proposition~\ref{prop:strongrar} below is the remark that, if $\gamma \in \{1, \ldots, d\}$ is a strong rarefaction coordinate, then, for all $n \geq 2$, for all $\x \in \Dnd$, for all $k \in \{1, \ldots, n-1\}$, the particles $\gamma:k$ and $\gamma:(k+1)$ never meet at positive times in the MSPD started at $\x$. Indeed, these particles have distinct positions just after the initial time and if there existed $t > 0$ such that $\Phi^{\gamma}_k(\x;t) = \Phi^{\gamma}_{k+1}(\x;t)$, then this would imply that there is a particle $\gamma':k'$ of another type colliding with $\gamma:k$ and $\gamma:(k+1)$ at the same time, and such that 
\begin{equation*}
  \Phi^{\gamma}_k(\x;s) < \Phi^{\gamma'}_{k'}(\x;s) < \Phi^{\gamma}_{k+1}(\x;s)
\end{equation*}
shortly before the collision. This is a contradiction with Assumption~\eqref{ass:USH}.

\begin{prop}[Continuity of strong rarefaction coordinates]\label{prop:strongrar}
  Under Assumptions~\eqref{ass:LC} and~\eqref{ass:USH}, let $\gamma \in \{1, \ldots, d\}$ be a strong rarefaction coordinate. Let $\bu$ be a probabilistic solution to~\eqref{eq:syst} obtained by Theorem~\ref{theo:existence}. Then $u^{\gamma}$ is continuous on $(0,+\infty) \times \R$, and if $u^{\gamma}_0$ is continuous on $\R$, then $u^{\gamma}$ is actually continuous on $[0,+\infty) \times \R$.
\end{prop}
\begin{proof}
  In the proof, for notational simplicity, we do not consider subsequences and assume that $\upmu[\x(n)]$ converges weakly to $\bar{\upmu}$ such that $u^{\gamma}(t,x) = H*\bar{\upmu}^{\gamma}_t(x)$ when $n$ grows to infinity.
  
  Let $\gamma \in \{1, \ldots, d\}$ be a strong rarefaction coordinate and $c$ denote the constant in~\eqref{eq:src}. By the Dini Theorem, it is clear that Proposition~\ref{prop:strongrar} follows if we show that, for all $t>0$, $u^{\gamma}(t,\cdot)$ is continuous on the real line. The point~\eqref{it:modulus:2} in Corollary~\ref{cor:modulus} ensures that it is enough to prove that $u^{\gamma}(t,\cdot)$ is continuous $\dd t$-almost everywhere. Let us check this continuity property by using the MSPD. To this aim, we recall that, by Proposition~\ref{prop:tightness}, $\dd t$-almost everywhere, for all $\gamma' \in \{1, \ldots, d\}$, the jumps of $u^{\gamma}(t,\cdot)$ and $u^{\gamma'}(t,\cdot)$ occur at distinct positions, and fix such a $t > 0$. Let us assume that $u^{\gamma}(t,\cdot)$ is discontinuous, \ie that there exist $\uv, \ov \in (0,1)$, with $\uv < \ov$, such that
  \begin{equation*}
    X^{\gamma}_{\uv}(t) = X^{\gamma}_{\ov}(t) =: y.
  \end{equation*}
  By the choice of $t$, there exists $\eta > 0$ such that
  \begin{equation*}
    \sum_{\gamma' \not= \gamma} |u^{\gamma'}(t,y+\eta) - u^{\gamma'}(t,(y-\eta)^-)| \leq \frac{c(\uv-\ov)}{6\ConstLip},
  \end{equation*}
  and by the Portmanteau Theorem~\cite[Theorem~2.1, p.~16]{billingsley}, there exists $n_1 \geq 1$ such that, for all $n \geq n_1$,
  \begin{equation*}
    \sum_{\gamma' \not= \gamma} |u^{\gamma'}[\x(n)](t,y+\eta) - u^{\gamma'}[\x(n)](t,(y-\eta)^-)| \leq \frac{c(\uv-\ov)}{3\ConstLip}.
  \end{equation*}
  On the other hand, the left continuous function $v \mapsto X^{\gamma}_v(t)$ is constant, and therefore continuous, on $[\uv,\ov)$. Up to replacing $\ov$ with $(\uv+\ov)/2$, we may assume that $v \mapsto X^{\gamma}_v(t)$ is continuous on $[\uv,\ov]$. Defining, for all $v \in (0,1)$, for all $s \geq 0$,
  \begin{equation*}
    X^{\gamma,n}_v(s) := u^{\gamma}[\x(n)](t,\cdot)^{-1}(v),
  \end{equation*}
  we deduce from Lemma~\ref{lem:cvCDF} that
  \begin{equation*}
    \lim_{n \to +\infty} X^{\gamma,n}_{\uv}(t) = \lim_{n \to +\infty} X^{\gamma,n}_{\ov}(t) = y,
  \end{equation*}
  so that there exists $n_2 \geq 1$ such that, for all $n \geq n_2$,
  \begin{equation*}
    y-\frac{\eta}{2} \leq X^{\gamma,n}_{\uv}(t) \leq X^{\gamma,n}_{\ov}(t) \leq y + \frac{\eta}{2}.
  \end{equation*}
  As a consequence, we deduce from Corollary~\ref{cor:speedprop} that
  \begin{equation*}
    \begin{aligned}
      & \sum_{\gamma' \not= \gamma} |u^{\gamma'}[\x(n)](s,X^{\gamma,n}_{\ov}(s)) - u^{\gamma'}[\x(n)](s,X^{\gamma,n}_{\uv}(s))|\\
      & \qquad \leq \sum_{\gamma' \not= \gamma} |u^{\gamma'}[\x(n)](t,y+\eta) - u^{\gamma'}[\x(n)](t,(y-\eta)^-)| \leq \frac{c(\uv-\ov)}{3\ConstLip},
    \end{aligned}
  \end{equation*}
  as soon as $n \geq n_1 \vee n_2$ and $s \leq t$ is such that $t-s \leq \eta / (4\ConstBound{\infty})$. Now if $n$ is large enough to ensure that $\ov-\uv > 1/n$ (say $n \geq n_3$), then the processes $(X^{\gamma,n}_{\uv}(s))_{s \geq 0}$ and $(X^{\gamma,n}_{\ov}(s))_{s \geq 0}$ describe the motion of two distinct particles in the MSPD started at $\x(n)$. In particular, according to the discussion at the beginning of this subsection, for all $s>0$, we have $X^{\gamma,n}_{\uv}(s)<X^{\gamma,n}_{\ov}(s)$ so that $u^{\gamma}[\x(n)](s,X^{\gamma,n}_{\ov}(s)^-)\geq u^{\gamma}[\x(n)](s,X^{\gamma,n}_{\uv}(s))$.
  
  For all $s \in [0,t]$ such that $t-s \leq \eta / (4\ConstBound{\infty})$ and $n \geq n_1 \vee n_2 \vee n_3$, we now recall the definition~\eqref{eq:dlambda} of $\lambda^{\gamma}\{\bu[\x(n)]\}(s,X^{\gamma,n}_v(s))$ and~\eqref{eq:src2} to write
  \begin{equation*}
    \begin{aligned}
      & \lambda^{\gamma}\{\bu[\x(n)]\}(s,X^{\gamma,n}_{\ov}(s)) - \lambda^{\gamma}\{\bu[\x(n)]\}(s,X^{\gamma,n}_{\uv}(s))\\
      & \qquad \geq \frac{c}{2} \left(u^{\gamma}[\x(n)](s,X^{\gamma,n}_{\ov}(s)^-) - u^{\gamma}[\x(n)](s,X^{\gamma,n}_{\uv}(s)^-)\right)\\
      & \qquad \quad + \frac{c}{2} \left(u^{\gamma}[\x(n)](s,X^{\gamma,n}_{\ov}(s)) - u^{\gamma}[\x(n)](s,X^{\gamma,n}_{\uv}(s))\right)\\
      & \qquad \quad - \ConstLip \sum_{\gamma' \not= \gamma} \left|u^{\gamma'}[\x(n)](s,X^{\gamma,n}_{\ov}(s)) - u^{\gamma'}[\x(n)](s,X^{\gamma,n}_{\uv}(s))\right|.
    \end{aligned}
  \end{equation*}
  The last term of the right-hand side is larger than $-c(\uv-\ov)/3$, while Lemma~\ref{lem:pseudoinv} allows to bound the sum of the first two terms by
  \begin{equation*}
    \begin{aligned}
      & \frac{c}{2} \left(u^{\gamma}[\x(n)](s,X^{\gamma,n}_{\ov}(s)^-) - u^{\gamma}[\x(n)](s,X^{\gamma,n}_{\uv}(s)^-) + u^{\gamma}[\x(n)](s,X^{\gamma,n}_{\ov}(s)) - u^{\gamma}[\x(n)](s,X^{\gamma,n}_{\uv}(s))\right)\\
      & \qquad \geq \frac{c}{2} \left(u^{\gamma}[\x(n)](s,X^{\gamma,n}_{\ov}(s)^-) - \uv + \ov - u^{\gamma}[\x(n)](s,X^{\gamma,n}_{\uv}(s))\right) \geq \frac{c}{2} \left(\ov-\uv\right).
    \end{aligned} 
  \end{equation*}
  As a conclusion,
  \begin{equation*}
    \lambda^{\gamma}\{\bu[\x(n)]\}(s,X^{\gamma,n}_{\ov}(s)) - \lambda^{\gamma}\{\bu[\x(n)]\}(s,X^{\gamma,n}_{\uv}(s)) \geq \frac{c}{6} \left(\ov-\uv\right),
  \end{equation*}
  so that, fixing $s_0 \in [0,t)$ such that $t-s_0 \leq \eta / (4\ConstBound{\infty})$ and using~\eqref{eq:MSPD:diff}, we obtain
  \begin{equation*}
    X^{\gamma,n}_{\ov}(t) - X^{\gamma,n}_{\uv}(t) \geq X^{\gamma,n}_{\ov}(s_0) - X^{\gamma,n}_{\uv}(s_0) + (t-s_0) \frac{c}{6} \left(\ov-\uv\right),
  \end{equation*}
  which is a contradiction with the fact that $\lim_{n \to +\infty} X^{\gamma,n}_{\ov}(t) - X^{\gamma,n}_{\uv}(t) = 0$. As a consequence, $u^{\gamma}(t,\cdot)$ is continuous and the proof is completed.
\end{proof}



\part{Stability estimates and construction of semigroup solutions}\label{part:2}

\section{Uniform \texorpdfstring{$\Ls^p$}{Lp} stability estimates on the MSPD}\label{s:stab}

This section is dedicated to the proof of Theorem~\ref{theo:stabMSPD}. In the scalar case, the latter result immediately follows from Proposition~\ref{prop:contractspd}, with $\ConstStab_p = 1$ for all $p \in [1,+\infty]$, and holds under Assumption~\eqref{ass:C} instead of the stronger Assumption~\eqref{ass:LC}.

Throughout the section, we therefore always implicitely assume that $d \geq 2$. The heart of the proof consists in establishing the following $\Ls^1$ and $\Ls^{\infty}$ stability estimates: for all $\x, \y \in \Dnd$,
\begin{equation}\label{eq:estim}
  \begin{aligned}
    & \sup_{t \geq 0} ||\Phi(\x;t) - \Phi(\y;t)||_1 \leq \ConstStab_1 ||\x-\y||_1,\\
    & \sup_{t \geq 0} ||\Phi(\x;t) - \Phi(\y;t)||_{\infty} \leq \ConstStab_{\infty} ||\x-\y||_{\infty},
  \end{aligned}
\end{equation}
for some constants $\ConstStab_1$ and $\ConstStab_{\infty}$ that do not depend on $n$. 

We shall assume first that $\x$ and $\y$ satisfy the following conditions:
\begin{itemize}
  \item they belong to the set of {\em good configurations}, which is introduced in Subsection~\ref{ss:good} and implies that the topology of the trajectories of the associated MSPD can be encoded by elementary algebraic structures,
  \item they are {\em locally homeomorphic} in the sense that the trajectories of the associated MSPD are described by the same algebraic structures. 
\end{itemize}
Under these conditions, we translate the problem of estimating $||\Phi(\x;t) - \Phi(\y;t)||_1$ and $||\Phi(\x;t) - \Phi(\y;t)||_{\infty}$ into a purely algebraic problem, that we solve in Subsection~\ref{ss:locstab} to obtain a local version of~\eqref{eq:estim}. 

We then extend this result to a global estimate by constructing paths joining arbitrary configurations $\x$ and $\y$ in $\Dnd$ that can be decomposed into small portions, on which our local estimate can be applied and then integrated along the path. This requires a detailed analysis of the geometry of the trajectories of the MSPD, that we carry out in Subsection~\ref{ss:interpolation}. 

We finally derive Theorem~\ref{theo:stabMSPD} from~\eqref{eq:estim} using the boundedness of the velocities for the temporal estimate, and a classical interpolation argument to obtain stability in all the $\Ls^p$ distances.


\subsection{Collisions, self-interactions and good configurations}\label{ss:good} This subsection is dedicated to the introduction of a few notions that shall allow us to describe the trajectories of the MSPD. Following the construction made in Section~\ref{s:mspd}, in the MSPD, the velocity of a particle is likely to be modified by two types of events: collisions with particles or clusters of the same type, to which we shall refer as {\em self-interactions}, and collisions with particles or clusters of a different type, to which we shall refer as {\em collisions}. 


\subsubsection{Collisions and self-interactions}\label{sss:collisions} Let $\x \in \Dnd$, with $\Nb(\x) \geq 1$. Recall that, for all $(\alpha:i, \beta:j) \in \Rb(\x)$, the collision time $\tinter_{\alpha:i, \beta:j}(\x) \in (0,+\infty)$ was defined in~\S\ref{sss:mspdcoll}. We now define the associated {\em space-time point of collision}.

\begin{defi}[Space-time point of collision]
  Let $\x \in \Dnd$ with $\Nb(\x) \geq 1$. For all $(\alpha:i, \beta:j) \in \Rb(\x)$, we denote by 
  \begin{equation*}
    \Xiinter_{\alpha:i, \beta:j}(\x) := \left(\xiinter_{\alpha:i, \beta:j}(\x), \tinter_{\alpha:i, \beta:j}(\x)\right) \in \R \times (0,+\infty)
  \end{equation*}
  the space-time point of collision between the particles $\alpha:i$ and $\beta:j$ in the MSPD started at $\x$, where
  \begin{equation*}
    \xiinter_{\alpha:i, \beta:j}(\x) := \Phi_i^{\alpha}(\x; \tinter_{\alpha:i, \beta:j}(\x)) = \Phi_j^{\beta}(\x; \tinter_{\alpha:i, \beta:j}(\x)) \in \R.
  \end{equation*}
\end{defi}

For all $\x \in \Dnd$, we denote by
\begin{equation*}
  \Ibinter(\x) := \{\Xiinter_{\alpha:i, \beta:j}(\x) : (\alpha:i,\beta:j) \in \Rb(\x)\}
\end{equation*}
the set of space-time points of collisions in the MSPD started at $\x$. Of course, $\Ibinter(\x)$ is the empty set if $\Nb(\x)=0$.

We now define space-time points of self-interactions as the space-time points at which two particles of the same type collide with each other. Our definition relies on the notion of {\em left limit} of a cluster.
\begin{defi}[Left limit of clusters]\label{defi:clutm}
  Let $\x \in \Dnd$ and $\gamma:k \in \Part$. For all $t>0$, let 
  \begin{equation*}
    t_0 := \inf\{s \in [0,t) : \forall r \in [s,t), \Nb(\Phi(\x;r)) = \Nb(\Phi(\x;s))\}.
  \end{equation*}
  Then we define the {\em left limit} in $t$ of the cluster $\clu_k^{\gamma}(\x;t)$ by
  \begin{equation*}
    \clu_k^{\gamma}(\x;t^-) := \bigcup_{s \in [t_0,t)} \clu_k^{\gamma}(\x;s).
  \end{equation*}
\end{defi}

The fact that, at time $t > 0$, two particles $\gamma:k$ and $\gamma:k'$ of different types collide with each other is exactly described by the conditions
\begin{equation*}
  \Phi_k^{\gamma}(\x;t) = \Phi_{k'}^{\gamma}(\x;t) =: \xi \qquad \text{and} \qquad \clu_k^{\gamma}(\x;t^-) \not= \clu_{k'}^{\gamma}(\x;t^-),
\end{equation*}
and we shall say that $(\xi, t)$ is a space-time point of self-interaction for $\gamma:k$ and $\gamma:k'$. Let us underline the fact that, while Assumption~\eqref{ass:USH} ensures that two particles of different types can collide at most once, it is generically possible that two particles of the same type stick together into a cluster, then that this cluster be split by a collision with a cluster of another type, and that the two particles collide again with each other. 

\begin{defi}[Space-time points of self-interactions]
  Let $\x \in \Dnd$. For all $\gamma \in \{1, \ldots, d\}$, for all $k,k' \in \{1, \ldots, n\}$, we define $\Ibself_{\gamma:k, \gamma:k'}(\x)$ as the set of space-time points $(\xi,t)$ such that
  \begin{equation*}
    \Phi_k^{\gamma}(\x;t) = \Phi_{k'}^{\gamma}(\x;t) = \xi \qquad \text{and} \qquad \clu_k^{\gamma}(\x;t^-) \not= \clu_{k'}^{\gamma}(\x;t^-).
  \end{equation*}
\end{defi}

Although the set $\Ibself_{\gamma:k, \gamma:k'}(\x)$ may contain more than one element, the particles $\gamma:k$ and $\gamma:k'$ cannot collide more than once between each collision with particles of other types. Since there is only a finite number of such collisions, it is clear that the set $\Ibself_{\gamma:k, \gamma:k'}(\x)$ always contains a finite number of elements.

We finally define the set of space-time points of self-interactions in the MSPD started at $\x$ as
\begin{equation*}
  \Ibself(\x) := \bigcup_{\gamma=1}^d \bigcup_{k,k'=1}^n \Ibself_{\gamma:k, \gamma:k'}(\x).
\end{equation*}


\subsubsection{Configurations with no collision at initial time}\label{sss:defDrond} We define the subset $\Drnd$ of $\Dnd$ as follows.

\begin{defi}[Configurations with no collision at initial time]\label{defi:Drnd}
  The set $\Drnd$ is the set of configurations $\x \in \Dnd$ such that, for all $(\alpha:i, \beta:j) \in (\Part)^2$ with $\alpha < \beta$, then $x_i^{\alpha} \not= x_j^{\beta}$.
\end{defi}
Certainly, $\Drnd$ is a dense open subset of $\Dnd$. Further properties of the set $\Drnd$ are discussed in Lemma~\ref{lem:Drondnd} in Appendix~\ref{app:proofs}.


\subsubsection{Good configurations}\label{sss:goodconf} We now define the set $\Good$ of {\em good configurations} as follows.
\begin{defi}[Good configurations]\label{defi:BinColl}
  The set of {\em good configurations} $\Good \subset \Dnd$ is defined by $\x \in \Good$ if and only if $\x \in \Drnd$ and either $\Nb(\x)=0$, or $\Nb(\x) \geq 1$ and:
  \begin{enumerate}[label=(\roman*), ref=\roman*]
    \item\label{it:BinColl:1} for all $(\alpha:i, \beta:j), (\alpha':i', \beta':j') \in \Rb(\x)$, $\Xiinter_{\alpha:i, \beta:j}(\x) = \Xiinter_{\alpha':i', \beta':j'}(\x)$ implies $\alpha'=\alpha$ and $\beta'=\beta$,
    \item\label{it:BinColl:2} the sets $\Ibinter(\x)$ and $\Ibself(\x)$ are disjoint.
  \end{enumerate}
\end{defi}
The point~\eqref{it:BinColl:1} expresses the fact that collisions are {\em binary}, \ie they never involve particles of more than two types. The point~\eqref{it:BinColl:2} means that two clusters of the same type cannot collide with each other at the same time as they collide with a cluster of a different type: self-interactions are separated from collisions, see Figure~\ref{fig:good}.

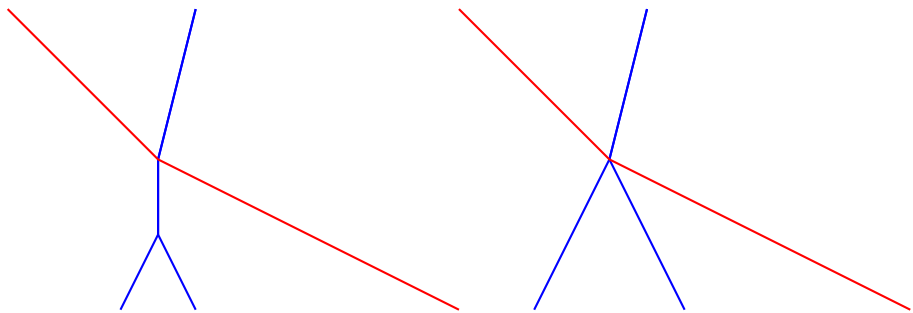
\begin{figure}[ht]
  \begin{pspicture}(12,4)
    \psline[linecolor=blue](1.5,0)(2,1)(2,2)(2.5,4)
    \psline[linecolor=blue](2.5,0)(2,1)(2,2)(2.5,4)
    \psline[linecolor=red](6,0)(2,2)(0,4)
      
    \psline[linecolor=blue](7,0)(8,2)(8.5,4)
    \psline[linecolor=blue](9,0)(8,2)(8.5,4)
    \psline[linecolor=red](12,0)(8,2)(6,4)
      
  \end{pspicture}
  \caption{The left-hand side of the picture shows the trajectory of the MSPD started at a {\em good configuration}, since self-interaction space-time points are separated from collisions. On the contrary, the right-hand side of the picture shows the trajectory of the MSPD started at a configuration that cannot be good, since two distinct clusters of the same type have a self-interaction at the same time as they collide with a cluster of another type.}
  \label{fig:good}
\end{figure}

Subsection~\ref{ss:locstab} provides detailed topological properties of the trajectories of the MSPD started at a good configuration, while Subsection~\ref{ss:interpolation} rather addresses the geometric properties of these trajectories.


\subsection{Local stability estimates}\label{ss:locstab} In this subsection, we establish the estimates~\eqref{eq:estim} for initial configurations $\x$ and $\y$ satisfying particular properties. In order to formulate these properties, we first introduce in~\S\ref{sss:BinColl} some algebraic structures encoding the topology of the trajectory of the MSPD started at good configurations. In particular, we define the {\em collision graph} of a good configuration as the oriented graph describing the order of collisions of each particle in the associated MSPD.

In~\S\ref{sss:locstab}, we say that two good configurations satisfy the Local Homeomorphic condition~{\rm (\hyperref[cond:C]{LHM})} if they have the same collision graph and also satisfy a few more technical properties. For such a choice of inital configurations $\x$ and $\y$, we are able to derive in~\S\ref{sss:prelim} a system of recursive inequations, indexed by the collision graph, on the distances $||\Phi(\x;t)-\Phi(\y;t)||_1$ and $||\Phi(\x;t)-\Phi(\y;t)||_{\infty}$ at the instants of collisions. The transcription of this system into a purely algebraic problem is made in~\S\ref{sss:coupling}, and the latter problem is solved in~\S\ref{sss:totmass}.


\subsubsection{Trajectories of the MSPD started at good configurations}\label{sss:BinColl} We first introduce a few notions to describe the topology of the trajectory of the MSPD started at good configurations.

\sk\paragraph{\bf Collisions} Let $\x \in \Good$, with $\Nb(\x) \geq 1$. We define the equivalence relation $\sim$ on $\Rb(\x)$ by, for all $(\alpha:i, \beta:j), (\alpha':i', \beta':j') \in \Rb(\x)$, 
\begin{equation*}
  (\alpha:i, \beta:j) \sim (\alpha':i', \beta':j') \qquad \text{if and only if} \qquad \Xiinter_{\alpha:i, \beta:j}(\x) = \Xiinter_{\alpha':i', \beta':j'}(\x).
\end{equation*}
Let $\Classe(\x) := \Rb(\x)/\sim$ refer to the set of equivalence classes and $\Mb(\x) \geq 1$ denote the cardinality of $\Classe(\x)$. Each equivalence class $\classe \in \Classe(\x)$ is naturally associated with a space-time point 
\begin{equation*}
  \Xi(\x;\classe) = (\xi(\x;\classe), T(\x;\classe)) \in \R \times (0,+\infty),
\end{equation*}
defined by 
\begin{equation*}
  \Xi(\x;\classe) := \Xiinter_{\alpha:i, \beta:j}(\x) \qquad \text{for any $(\alpha:i, \beta:j) \in \classe$}. 
\end{equation*}

In addition, the point~\eqref{it:BinColl:1} of Definition~\ref{defi:BinColl} implies that, for all $\classe \in \Classe(\x)$, there exist $\alpha, \beta \in \{1, \ldots, d\}$ such that $\alpha < \beta$ and, for all $(\alpha':i', \beta':j') \in \classe$, $\alpha'=\alpha$ and $\beta'=\beta$. Letting
\begin{equation*}
  \begin{aligned}
    & a := \{\alpha:i \in \Part : \exists \beta:j \in \Part, (\alpha:i, \beta:j) \in \classe\},\\
    & b := \{\beta:j \in \Part : \exists \alpha:i \in \Part, (\alpha:i, \beta:j) \in \classe\},
  \end{aligned}
\end{equation*}
it is easily checked that $\classe = a \times b$. Note that, due to the point~\eqref{it:BinColl:2} of Definition~\ref{defi:BinColl}, for all $(\alpha:i, \beta:j) \in a \times b$, $\clu_i^{\alpha}(\x;T(\x;\classe)^-) = a$ and $\clu_j^{\beta}(\x;T(\x;\classe)^-) = b$. However, the clusters $a$ and $b$ can be splitted at the collision if the velocities of the particles after the collision do not satisfy the stability condition~\eqref{eq:stab}, therefore we generally only have $\clu_i^{\alpha}(\x;T(\x;\classe)) \subset a$ and $\clu_j^{\beta}(\x;T(\x;\classe)) \subset b$.

In the sequel, we shall simply refer to the equivalence classes as {\em collisions}, and say that a generical cluster $c$ is {\em involved} in the collision $\classe = a \times b$ if $c = a$ or $c = b$.

If $\x \in \Good$ and $\Nb(\x) = 0$, we simply define $\Mb(\x)=0$.

\sk\paragraph{\bf Collision graph} Let $\x \in \Good$. For all $\gamma:k \in \Part$, we denote by $\Classe_{\gamma:k}(\x)$ the subset of $\Classe(\x)$ composed by the collisions $\classe = a \times b$ such that $\gamma:k \in a \cup b$. Note that $\Classe_{\gamma:k}(\x)$ is empty if the particle $\gamma:k$ does not collide with a particle of another type in the MSPD started at $\x$. Clearly, two distinct collisions $\classe', \classe \in \Classe_{\gamma:k}(\x)$ have distinct instants of collision $T(\x; \classe') \not= T(\x; \classe)$, since two distinct collisions involving the same particle $\gamma:k$ cannot occur at the same time. As a consequence, the increasing order of instants of collisions induces a total order on the set $\Classe_{\gamma:k}(\x)$, to which we shall only refer as the {\em order of collisions}.

For all $\gamma \in \{1, \ldots, d\}$, for all $\classe', \classe \in \Classe(\x)$, we shall write
\begin{equation*}
  \classe' \lto{\gamma} \classe
\end{equation*}
whenever there exists $k \in \{1, \ldots, n\}$ such that $\classe', \classe \in \Classe_{\gamma:k}(\x)$ and $\classe$ is the next element after $\classe'$ for the order of collisions on $\Classe_{\gamma:k}(\x)$. The {\em collision graph} of a good configuration $\x$ is now defined as the oriented graph with set of vertices $\Classe(\x)$, and set of arcs induced by the relations $\classe' \stackrel{\gamma}{\to} \classe$. If $\Nb(\x)=0$ then the collision graph of $\x$ is nothing but the empty graph.

By construction, an arc is naturally associated with at least a type $\gamma \in \{1, \ldots, d\}$, and since Assumption~\eqref{ass:USH} ensures that two particles of distinct type can only collide once, each arc actually has a unique type. Besides, since $\classe' \lto{\gamma} \classe$ implies that $T(\x;\classe') < T(\x; \classe)$, there is no oriented cycle in the collision graph. 

\sk\paragraph{\bf Numbering the collisions} Let us now explain how to number the collisions $\classe \in \Classe(\x)$ in a consistant fashion with the partial order induced by the orientation of the collision graph.

\begin{lem}[Numbering the collisions]\label{lem:numbering}
  Under Assumptions~\eqref{ass:C} and~\eqref{ass:USH}, let $\x \in \Good$, with $M := \Mb(\x) \geq 1$. Then the set of collisions $\Classe(\x)$ can be numbered in such a fashion $\classe_1, \ldots, \classe_M$ that, for all $m', m \in \{1, \ldots, M\}$ satisfying
  \begin{equation*}
    \classe_{m'} \lto{\gamma} \classe_m
  \end{equation*}
  for some $\gamma \in \{1, \ldots, d\}$, then $m' < m$.
\end{lem}
\begin{proof}
  Let us call {\em leaves} the collisions $\classe \in \Classe(\x)$ such that there is no $\classe' \in \Classe(\x)$ pointing toward $\classe$ in the collision graph. Clearly, $\classe$ is a leaf if and only if, for all $\gamma:k \in \Part$ such that $\classe \in \Classe_{\gamma:k}(\x)$, $\classe$ is the minimal element of $\Classe_{\gamma:k}(\x)$ for the order of collisions. Since there is no oriented cycle in the collision graph, the set of leaves is nonempty, and this property remains true for all nonempty subgraphs of the collision graph obtained by removing a leaf and its adjacent arcs.
  
  We now proceed as follows: we choose one leaf, call it $\classe_1$, remove it from the graph together with all the adjacent arcs, and restart the construction as long as the graph is nonempty. At the $m$-th step, the selected collision $\classe_m$ is minimal, for the order of collisions, among the remaining elements of all the sets $C_{\gamma:k}$ to which it belongs. This ensures that the numbering is consistent with the partial order induced by the orientation of the collision graph.
\end{proof}

\begin{rk}
  An effective way to proceed as in the proof of Lemma~\ref{lem:numbering} is to number the collisions in the increasing order of collision times. If several distinct collisions have the same collision time, then they cannot involve the same particle; therefore, any local ordering of these collisions leads to a numbering satisfying the conclusion of Lemma~\ref{lem:numbering}.
\end{rk}

\sk\paragraph{\bf Last collision time} For all $\gamma:k \in \Part$, we finally define $\Tbmax_{\gamma:k}(\x)$ by 
\begin{equation*}
  \Tbmax_{\gamma:k}(\x) := 0
\end{equation*}
if $\Classe_{\gamma:k}(\x)$ is empty, and
\begin{equation*}
  \Tbmax_{\gamma:k}(\x) := \max_{\classe \in \Classe_{\gamma:k}(\x)} T(\x;\classe)
\end{equation*}
otherwise.


\subsubsection{Statement of the local stability estimates}\label{sss:locstab} Two configurations $\x, \y \in \Dnd$ are said to satisfy the Local Homeomorphic condition~{\rm (\hyperref[cond:C]{LHM})} if: 
\begin{enumerate}[label=(LHM-\arabic*), ref=LHM-\arabic*]
  \item\label{cond:C}\label{cond:C1} $\x, \y \in \Good$ and $\Rb(\x)=\Rb(\y)=:R$,
  \item\label{cond:C2} $\x$ and $\y$ have the same collision graph, which in particular implies $\Classe(\x)=\Classe(\y)=:C$,
  \item\label{cond:C3} for all $\classe \in C$, letting $T^-(\classe) := T(\x;\classe) \wedge T(\y;\classe)$ and $T^+(\classe) := T(\x;\classe) \vee T(\y;\classe)$,
  \begin{enumerate}[label=(\alph*), ref=LHM-3\alph*]
    \item\label{cond:C3a} for all arcs $\classe' \stackrel{\gamma}{\to} \classe$, $T^+(\classe') < T^-(\classe)$,
    \item\label{cond:C3b} if $T^-(\classe) = T(\x;\classe) < T(\y; \classe) = T^+(\classe)$, then for all $(\alpha:i, \beta:j) \in \classe = a \times b$,
    \begin{equation*}
      \forall t \in [T(\x;\classe), T(\y;\classe)], \qquad \left\{\begin{aligned}
        & \clu_i^{\alpha}(\x;t) = \clu_i^{\alpha}(\x;T(\x;\classe)),\\
        & \clu_j^{\beta}(\x;t) = \clu_j^{\beta}(\x;T(\x;\classe)),
      \end{aligned}\right.
    \end{equation*}
    \begin{equation*}
      \forall t \in [T(\x;\classe), T(\y;\classe)), \qquad \left\{\begin{aligned}
        & \clu_i^{\alpha}(\y;t) = \clu_i^{\alpha}(\y;T(\x;\classe)),\\
        & \clu_j^{\beta}(\y;t) = \clu_j^{\beta}(\y;T(\x;\classe)),
      \end{aligned}\right.
    \end{equation*}
    and a symmetric statement holds in the case $T^-(\classe) = T(\y;\classe) < T(\x; \classe) = T^+(\classe)$.
  \end{enumerate}
  The time intervals $[T^-(\classe), T^+(\classe)]$ shall be referred to as {\em collision intervals}.
\end{enumerate}
Condition~\eqref{cond:C3b} only expresses the fact that no self-interaction occurs on collision intervals. 

We are now able to state our local stability estimates.

\begin{prop}[Local stability estimates]\label{prop:locstab}
  Under Assumptions~\eqref{ass:LC} and \eqref{ass:USH}, for all $\x,\y \in \Dnd$ satisfying Condition~{\rm (\hyperref[cond:C]{LHM})},
  \begin{equation*}
    \begin{aligned}
      & \sup_{t \geq 0} || \Phi(\x; t) - \Phi(\y; t) ||_1 \leq \ConstStab_1 ||\x-\y||_1,\\
      & \sup_{t \geq 0} || \Phi(\x; t) - \Phi(\y; t) ||_{\infty} \leq \ConstStab_{\infty} ||\x-\y||_{\infty},
    \end{aligned}
  \end{equation*}
  where $\ConstStab_1$ and $\ConstStab_{\infty}$ are defined in~\eqref{eq:ConstStab}.
\end{prop}

The proof of Proposition~\ref{prop:locstab} is detailed in~\S\ref{sss:prelim}, \S\ref{sss:coupling} and~\S\ref{sss:totmass} below. Throughout these paragraphs, we fix $\x, \y \in \Dnd$ satisfying Condition~{\rm (\hyperref[cond:C]{LHM})} and adopt the notations of Condition~{\rm (\hyperref[cond:C]{LHM})} by denoting by $R$ the set $\Rb(\x)=\Rb(\y)$, by $N=\Nb(\x)=\Nb(\y)$ its cardinality, by $C$ the set of collisions $\Classe(\x)=\Classe(\y)$ and by $M=\Mb(\x)=\Mb(\y)$ its cardinality. Besides, Condition~\eqref{cond:C2} ensures that, for all $\gamma:k \in \Part$, the sets $\Classe_{\gamma:k}(\x)$ and $\Classe_{\gamma:k}(\y)$ are the same, with the same order of collisions. These sets are denoted by $C_{\gamma:k}$. We finally denote
\begin{equation*}
  \Tmax_{\gamma:k} := \Tbmax_{\gamma:k}(\x) \vee \Tbmax_{\gamma:k}(\y).
\end{equation*}

For all $t \geq 0$ and $\gamma:k \in \Part$, we define
\begin{equation*}
  d_{\gamma:k}(t) := |\Phi_k^{\gamma}(\x;t) - \Phi_k^{\gamma}(\y;t)|,
\end{equation*}
so that 
\begin{equation*}
  ||\Phi(\x;t) - \Phi(\y;t)||_1 = \frac{1}{n}\sum_{\gamma:k \in \Part} d_{\gamma:k}(t), \qquad ||\Phi(\x;t) - \Phi(\y;t)||_{\infty} = \sup_{\gamma:k \in \Part} d_{\gamma:k}(t).
\end{equation*}

In~\S\ref{sss:prelim} we provide local (in time) estimates on the growth of $d_{\gamma:k}(t)$ inside and outside collision intervals. In~\S\ref{sss:coupling}, we introduce an {\em auxiliary system} that shall allow us to integrate these estimates along the whole sequence of collisions, and we explain how this auxiliary system can be coupled with the family of processes $\{(d_{\gamma:k}(t))_{t \geq 0}, \gamma:k \in \Part\}$. In~\S\ref{sss:totmass}, we obtain a bound on the auxiliary system that is transferred to the original processes $||\Phi(\x;t)-\Phi(\y;t)||_1$ and $||\Phi(\x;t) - \Phi(\y;t)||_{\infty}$ thanks to the coupling argument developed in~\S\ref{sss:coupling}. 


\subsubsection{Preliminary estimates}\label{sss:prelim} Let us first collect the following preliminary estimates on the joint evolution of the family of processes $\{(d_{\gamma:k}(t))_{t \geq 0}, \gamma:k \in \Part\}$.

\begin{lem}[Preliminary estimates]\label{lem:prelim}
  Let the assumptions of Proposition~\ref{prop:locstab} hold.
  \begin{enumerate}[label=(\roman*), ref=\roman*]
    \item\label{it:prelim:1} For all $\classe = a \times b \in C$, for all $t \in [T^-(\classe), T^+(\classe)]$,
    \begin{equation*}
      \begin{aligned}
        & \max_{\alpha:i \in a} d_{\alpha:i}(t) \leq \left(1 + \frac{\Ratio}{n} |b|\right) \frac{1}{|a|}\sum_{\alpha:i \in a} d_{\alpha:i}(T^-(\classe)) + \frac{\Ratio}{n} \sum_{\beta:j \in b} d_{\beta:j}(T^-(\classe)),\\
        & \max_{\beta:j \in b} d_{\beta:j}(t) \leq \left(1 + \frac{\Ratio}{n} |a|\right) \frac{1}{|b|} \sum_{\beta:j \in b} d_{\beta:j}(T^-(\classe)) + \frac{\Ratio}{n} \sum_{\alpha:i \in a} d_{\alpha:i}(T^-(\classe)),
      \end{aligned}
    \end{equation*}
    where we recall that $\Ratio = 3\ConstLip/\ConstUSH$.
    
    \item\label{it:prelim:2} Let $\classe = a \times b \in C$, $c \in \{a, b\}$ and $\gamma:= \type(c)$. For all $\gamma:k \in c$, let us define $t_{\gamma:k}' := T^+(\classe')$ if there exists $\classe' \in C_{\gamma:k}$ such that $\classe' \lto{\gamma} \classe$, and $t_{\gamma:k}' := 0$ otherwise. Then, for all $t \leq T^-(\classe)$,
    \begin{equation*}
      \begin{aligned}
        \sum_{\gamma:k \in c} \ind{t > t_{\gamma:k}'} d_{\gamma:k}(t) & \leq \sum_{\gamma:k \in c} \ind{t > t_{\gamma:k}'} d_{\gamma:k}(t_{\gamma:k}'),\\
        \sup_{\gamma:k \in c} \ind{t > t_{\gamma:k}'} d_{\gamma:k}(t) & \leq \sup_{\gamma:k \in c} \ind{t > t_{\gamma:k}'} d_{\gamma:k}(t_{\gamma:k}').
      \end{aligned}
    \end{equation*}
    \item\label{it:prelim:3} For all $t \geq 0$, for all $\gamma$ in $\{1, \ldots, d\}$,
    \begin{equation*}
      \begin{aligned}
        \sum_{k=1}^n \ind{t > \Tmax_{\gamma:k}} d_{\gamma:k}(t) & \leq \sum_{k=1}^n \ind{t > \Tmax_{\gamma:k}} d_{\gamma:k}(\Tmax_{\gamma:k}),\\
        \sup_{1 \leq k \leq n} \ind{t > \Tmax_{\gamma:k}} d_{\gamma:k}(t) & \leq \sup_{1 \leq k \leq n} \ind{t > \Tmax_{\gamma:k}} d_{\gamma:k}(\Tmax_{\gamma:k}).
      \end{aligned}
    \end{equation*}
  \end{enumerate}
\end{lem}

Let us highlight the fact that $t'_{\gamma:k}$ and $\Tmax_{\gamma:k}$ play similar roles in the respective cases~\eqref{it:prelim:2} and~\eqref{it:prelim:3}. 

\begin{proof}[Proof of Lemma~\ref{lem:prelim}]
  We first address~\eqref{it:prelim:1} and fix $\classe = a \times b \in C$. We assume that $T^-(\classe) = T(\x;\classe) \leq T(\y; \classe) = T^+(\classe)$, the opposite case is symmetric. Let us denote $\x' := \Phi(\x;T(\x;\classe))$ and $\y' := \Phi(\y;T(\x;\classe))$. For all $t \in [T(\x;\classe), T(\y;\classe)]$, we first remark that the value of
  \begin{equation*}
    \xi_i^{\alpha}(t) := \Phi_i^{\alpha}(\x;T(\x;\classe)) + (t-T(\x;\classe))v_i^{\alpha}(\y;T(\x;\classe))
  \end{equation*}
  does not depend on the choice of $\alpha:i \in a$. Indeed, $\Phi_i^{\alpha}(\x;T(\x;\classe))$ is the location of the collision $\classe$ in the MSPD started at $\x$, while Conditions~\eqref{cond:C3a} and~\eqref{cond:C3b} ensure that, for all $\alpha:i \in a$, 
  \begin{equation*}
    v_i^{\alpha}(\y;T(\x;\classe)) = \frac{1}{|a|} \sum_{\alpha:i' \in a} \tlambda_{i'}^{\alpha}(\y').
  \end{equation*}
    
  We shall use the following facts, the proofs of which are postponed below.
  
  {\em Fact~1:} the processes $\{\Phi_i^{\alpha}(\x;t) : \alpha:i \in a\}$ and $\{\xi_i^{\alpha}(t) : \alpha:i \in a\}$ follow the Local Sticky Particle Dynamics on $[T(\x;\classe), T(\y;\classe)]$, with respective initial velocity vectors $(\tlambda_i^{\alpha}(\x'))_{\alpha:i \in a}$ and $(\tlambda_i^{\alpha}(\y'))_{\alpha:i \in a}$.
  
  {\em Fact~2:} for all $\alpha:i \in a$,
  \begin{equation*}
    \left|\tlambda_i^{\alpha}(\x') - \tlambda_i^{\alpha}(\y')\right| \leq \frac{\ConstLip}{n}|b|.
  \end{equation*}
  
  {\em Fact~3:} for all $\alpha:i, \alpha:i' \in a$, for all $t \in [T(\x;\classe), T(\y;\classe)]$,
  \begin{equation*}
    |\Phi_i^{\alpha}(\x;t) - \Phi_{i'}^{\alpha}(\x;t)| \leq 2(T(\y;\classe)-T(\x;\classe)) \frac{\ConstLip}{n}|b|.
  \end{equation*}
  
  {\em Fact~4:} the nonnegative quantity $T(\y;\classe)-T(\x;\classe)$ satisfies
  \begin{equation*}
    T(\y;\classe)-T(\x;\classe) \leq \frac{1}{\ConstUSH}\left(\frac{1}{|b|} \sum_{\beta:j \in b} d_{\beta:j}(T(\x;\classe)) + \frac{1}{|a|} \sum_{\alpha:i \in a}d_{\alpha:i}(T(\x;\classe))\right).
  \end{equation*}
  
  Taking these facts for granted, we now fix $\alpha:i \in a$ and write, for all $t \in [T(\x;\classe), T(\y;\classe)]$,
  \begin{equation}\label{eq:pfprelim:01}
    d_{\alpha:i}(t) \leq |\Phi_i^{\alpha}(\x;t)-\xi_i^{\alpha}(t)| + |\Phi_i^{\alpha}(\y;t)-\xi_i^{\alpha}(t)|.
  \end{equation}
  On the one hand, it is clear from Conditions~\eqref{cond:C3a} and~\eqref{cond:C3b} that the value of $\Phi_i^{\alpha}(\y;t)$ does not depend on the choice of $\alpha:i \in a$, and that $\Phi_i^{\alpha}(\y;t)$ and $\xi_i^{\alpha}(t)$ evolve at the same velocity, so that
  \begin{equation}\label{eq:pfprelim:02}
    |\Phi_i^{\alpha}(\y;t)-\xi_i^{\alpha}(t)| = |\Phi_i^{\alpha}(\y;T(\x;\classe))-\xi_i^{\alpha}(T(\x;\classe))| = d_{\alpha:i}(T(\x;\classe)) = \frac{1}{|a|}\sum_{\alpha:i' \in a} d_{\alpha:i'}(T(\x;\classe)).
  \end{equation}
  On the other hand, 
  \begin{equation*}
    |\Phi_i^{\alpha}(\x;t)-\xi_i^{\alpha}(t)| \leq \frac{1}{|a|} \sum_{\alpha:i' \in a} \left\{|\Phi_i^{\alpha}(\x;t)-\Phi_{i'}^{\alpha}(\x;t)| + |\Phi_{i'}^{\alpha}(\x;t)-\xi_{i'}^{\alpha}(t)|\right\},
  \end{equation*}
  and combining Facts~1 and~2 with~\eqref{it:contractspd:1} in Proposition~\ref{prop:contractspd} yields
  \begin{equation*}
    \begin{aligned}
      & \sum_{\alpha:i' \in a} |\Phi_{i'}^{\alpha}(\x;t)-\xi_{i'}^{\alpha}(t)|\\
      & \qquad \leq \sum_{\alpha:i' \in a} |\Phi_{i'}^{\alpha}(\x;T(\x;\classe))-\xi_{i'}^{\alpha}(T(\x;\classe))| + (t-T(\x;\classe))\sum_{\alpha:i' \in a} \left|\tlambda_{i'}^{\alpha}(\x') - \tlambda_{i'}^{\alpha}(\y')\right|\\
      & \qquad \leq \frac{\ConstLip}{n}|a||b|(t-T(\x;\classe)),
    \end{aligned}
  \end{equation*}
  while Fact~3 gives
  \begin{equation*}
    \sum_{\alpha:i' \in a} |\Phi_i^{\alpha}(\x;t)-\Phi_{i'}^{\alpha}(\x;t)| \leq 2\frac{\ConstLip}{n}|a||b|(t-T(\x;\classe)).
  \end{equation*}
  As a consequence of the two previous inequalities,
  \begin{equation*}
    \begin{aligned}
      |\Phi_i^{\alpha}(\x;t)-\xi_i^{\alpha}(t)| & \leq 3\frac{\ConstLip}{n}|b|(t-T(\x;\classe)) \\
      & \leq \frac{\Ratio}{n}\left(\sum_{\beta:j \in b} d_{\beta:j}(T(\x;\classe)) + \frac{|b|}{|a|} \sum_{\alpha:i' \in a}d_{\alpha:i'}(T(\x;\classe))\right),
    \end{aligned}
  \end{equation*}
  where we have used Fact~4 at the second line. Then the conclusion is obtained by plugging this inequality and~\eqref{eq:pfprelim:02} into~\eqref{eq:pfprelim:01}, and the uniform bound on $d_{\beta:j}(t)$, $\beta:j \in b$ in~\eqref{it:prelim:1} follows similarly.
  
  \sk
  We now prove the Facts~1, 2, 3 and~4 used above.
  
  \noindent {\em Proof of Fact~1:} The process $\{\xi_i^{\alpha}(t) : \alpha:i \in a\}$ follows the Local Sticky Particle Dynamics on $[T(\x;\classe), T(\y;\classe)]$, with initial velocity vector $(\tlambda_i^{\alpha}(\y'))_{\alpha:i \in a}$, as a straightforward consequence of its definition. Let us use~\eqref{it:locintmspd:2} in Lemma~\ref{lem:locintmspd} to prove that the process $\{\Phi_i^{\alpha}(\x;t) : \alpha:i \in a\}$ follows the Local Sticky Particle Dynamics on $[T(\x;\classe), T(\y;\classe)]$. By Condition~\eqref{cond:C3b}, for all $\alpha:i \in a$, $\clu_i^{\alpha}(\x;T(\y;\classe)) \subset a$, and the set $\Ttau_{\alpha:i}(\x)$ as is defined in~\eqref{eq:tau} has an empty intersection with $(T(\x;\classe), T(\y;\classe))$. As a consequence, Lemma~\ref{lem:locintmspd} asserts that the process $\{\Phi_i^{\alpha}(\x;t) : \alpha:i \in a\}$ follows the Local Sticky Particle Dynamics on $[T(\x;\classe), T(\y;\classe)]$, with initial velocity vector $(\tlambda_i^{\alpha}(\x'))_{\alpha:i \in a}$.
  
  \noindent {\em Proof of Fact~2:} Let us first check that, for all $\alpha:i \in a$,
  \begin{itemize}
    \item for all $\gamma \not\in \{\alpha, \beta\}$, $\omega^{\gamma}_{\alpha:i}(\x') = \omega^{\gamma}_{\alpha:i}(\y')$,
    \item $|\omega^{\beta}_{\alpha:i}(\x')-\omega^{\beta}_{\alpha:i}(\y')| \leq |b|/n$.
  \end{itemize}
  By the definition of $\omega^{\gamma}_{\alpha:i}(\x')$ and $\omega^{\gamma}_{\alpha:i}(\y')$, the first point above easily follows if we prove that, for all $\gamma:k \in \Part$ such that $\gamma \not\in \{\alpha, \beta\}$ (say $\gamma < \alpha$),
  \begin{equation*}
    x'^{\gamma}_k < x'^{\alpha}_i \qquad \text{if and only if} \qquad y'^{\gamma}_k < y'^{\alpha}_i.
  \end{equation*}
  But let us assume for instance that $x'^{\gamma}_k < x'^{\alpha}_i$ and $y'^{\gamma}_k \geq y'^{\alpha}_i$. Then by Condition~\eqref{cond:C3a}, the collision with $\gamma:k$ comes after $\classe$ in $\Classe_{\alpha:i}(\x)$, while it is either not in $\Classe_{\alpha:i}(\y)$, or it comes before $\classe$. This is a contradiction with Condition~\eqref{cond:C2}. As far as the second point above is concerned, the same argument shows that the particles $\beta:j$ that do not belong to $b$ have the same contribution in 
  \begin{equation*}
    \omega^{\beta}_{\alpha:i}(\x') = \frac{1}{n}\sum_{j=1}^n \ind{x'^{\alpha}_i \geq x'^{\beta}_j}
  \end{equation*}
  and in 
  \begin{equation*}
    \omega^{\beta}_{\alpha:i}(\y') = \frac{1}{n}\sum_{j=1}^n \ind{y'^{\alpha}_i \geq y'^{\beta}_j},
  \end{equation*}
  which is enough to ensure that $|\omega^{\beta}_{\alpha:i}(\x')-\omega^{\beta}_{\alpha:i}(\y')| \leq |b|/n$. As a consequence, it follows from the definition of $\tlambda$ and Assumption~\eqref{ass:LC} that, for all $\alpha:i \in a$,
  \begin{equation*}
    \left|\tlambda_i^{\alpha}(\x') - \tlambda_i^{\alpha}(\y')\right| \leq \frac{\ConstLip}{n}|b|,
  \end{equation*}
  which completes the proof of Fact~2.

  \noindent {\em Proof of Fact~3:} Let us write $a = \alpha:\ui\cdots\oi$ and first remark that, for all $\alpha:i, \alpha:i' \in a$, for all $t \in [T(\x;\classe), T(\y;\classe)]$,
  \begin{equation*}
    |\Phi_i^{\alpha}(\x;t)-\Phi_{i'}^{\alpha}(\x;t)| \leq \Phi_{\oi}^{\alpha}(\x;t) - \Phi_{\ui}^{\alpha}(\x;t),
  \end{equation*}
  and, by Conditions~\eqref{cond:C3a} and~\eqref{cond:C3b},
  \begin{equation*}
    \begin{aligned}
      \Phi_{\oi}^{\alpha}(\x;t) - \Phi_{\ui}^{\alpha}(\x;t) & = (t-T(\x;\classe))\left(v_{\oi}^{\alpha}(\x;T(\x;\classe))-v_{\ui}^{\alpha}(\x;T(\x;\classe))\right)\\
      & \leq (T(\y;\classe)-T(\x;\classe))\left(v_{\oi}^{\alpha}(\x;T(\x;\classe))-v_{\ui}^{\alpha}(\x;T(\x;\classe))\right).
    \end{aligned}
  \end{equation*}
  If $\clu_{\ui}^{\alpha}(\x;T(\x;\classe)) = \clu_{\oi}^{\alpha}(\x;T(\x;\classe))$, then $v_{\oi}^{\alpha}(\x;T(\x;\classe))=v_{\ui}^{\alpha}(\x;T(\x;\classe))$ and Fact~3 is trivial. Otherwise, let us write $\clu_{\ui}^{\alpha}(\x;T(\x;\classe)) = \alpha:\ui\cdots\ui'$ and $\clu_{\oi}^{\alpha}(\x;T(\x;\classe)) = \alpha:\oi'\cdots\oi$, with $\ui \leq \ui' < \oi' \leq \oi$. Then
  \begin{equation*}
    \begin{aligned}
      0 & \leq v_{\oi}^{\alpha}(\x;T(\x;\classe))-v_{\ui}^{\alpha}(\x;T(\x;\classe))\\
      & \leq \frac{1}{\ui'-\ui+1}\sum_{i=\ui}^{\ui'} \tlambda_i^{\alpha}(\x') - \frac{1}{\oi-\oi'+1}\sum_{i=\oi'}^{\oi} \tlambda_i^{\alpha}(\x')\\
      & = \frac{1}{\ui'-\ui+1}\sum_{i=\ui}^{\ui'} \tlambda_i^{\alpha}(\x') - \frac{1}{\ui'-\ui+1}\sum_{i=\ui}^{\ui'} \tlambda_i^{\alpha}(\y')\\
      & \quad + \frac{1}{\ui'-\ui+1}\sum_{i=\ui}^{\ui'} \tlambda_i^{\alpha}(\y') - \frac{1}{\oi-\oi'+1}\sum_{i=\oi'}^{\oi} \tlambda_i^{\alpha}(\y')\\
      & \quad + \frac{1}{\oi-\oi'+1}\sum_{i=\oi'}^{\oi} \tlambda_i^{\alpha}(\y') - \frac{1}{\oi-\oi'+1}\sum_{i=\oi'}^{\oi} \tlambda_i^{\alpha}(\x')\\
      & \leq \frac{1}{\ui'-\ui+1}\sum_{i=\ui}^{\ui'} |\tlambda_i^{\alpha}(\x')-\tlambda_i^{\alpha}(\y')| + \frac{1}{\oi-\oi'+1}\sum_{i=\oi'}^{\oi} |\tlambda_i^{\alpha}(\y') - \tlambda_i^{\alpha}(\x')|,
    \end{aligned}
  \end{equation*}
  where Conditions~\eqref{cond:C3a} and~\eqref{cond:C3b} allow us to use Lemma~\ref{lem:extstab} and get
  \begin{equation*}
    \frac{1}{\ui'-\ui+1}\sum_{i=\ui}^{\ui'} \tlambda_i^{\alpha}(\y') - \frac{1}{\oi-\oi'+1}\sum_{i=\oi'}^{\oi} \tlambda_i^{\alpha}(\y') \leq 0.
  \end{equation*}
  We now deduce from Fact~2 that each sum in the right-hand side above is lower than $2|b|\ConstLip/n$, which completes the proof of Fact~3.
  
  \noindent {\em Proof of Fact~4:} Note that $\Phi_i^{\alpha}(\y;T(\y;\classe)) = \Phi_j^{\beta}(\y; T(\y;\classe))$, which rewrites
  \begin{equation*}
    \Phi_j^{\beta}(\y;T(\x;\classe)) - \Phi_i^{\alpha}(\y;T(\x;\classe)) = \int_{s=T(\x;\classe)}^{T(\y;\classe)} (v_i^{\alpha}(\y;s) - v_j^{\beta}(\y;s))\dd s
  \end{equation*}
  owing to~\eqref{eq:vspd}. On account of~\eqref{eq:typeencadrelambda}, the right-hand side above is larger than $\ConstUSH(T(\y;\classe)-T(\x;\classe))$, so that
  \begin{equation*}
    \begin{aligned}
      T(\y;\classe)-T(\x;\classe) & \leq \frac{1}{\ConstUSH}\left(\Phi_j^{\beta}(\y;T(\x;\classe)) - \Phi_i^{\alpha}(\y;T(\x;\classe))\right)\\
      & = \frac{1}{\ConstUSH}\left(\Phi_j^{\beta}(\y;T(\x;\classe))-\Phi_j^{\beta}(\x;T(\x;\classe))+\Phi_i^{\alpha}(\x;T(\x;\classe))- \Phi_i^{\alpha}(\y;T(\x;\classe))\right)\\
      & \leq \frac{1}{\ConstUSH}\left(|\Phi_j^{\beta}(\y;T(\x;\classe))-\Phi_j^{\beta}(\x;T(\x;\classe))|+|\Phi_i^{\alpha}(\x;T(\x;\classe))- \Phi_i^{\alpha}(\y;T(\x;\classe))|\right)\\
      & = \frac{1}{\ConstUSH}\left(d_{\beta:j}(T(\x;\classe)) + d_{\alpha:i}(T(\x;\classe))\right),
    \end{aligned}
  \end{equation*}
  where we have used the fact that $\Phi_j^{\beta}(\x;T(\x;\classe))=\Phi_i^{\alpha}(\x;T(\x;\classe))$. Taking the sum of both sides on $(\alpha:i, \beta:j) \in a \times b$ and then dividing by $|a||b|$, we obtain
  \begin{equation*}
    T(\y;\classe)-T(\x;\classe) \leq \frac{1}{\ConstUSH}\left(\frac{1}{|b|} \sum_{\beta:j \in b} d_{\beta:j}(T(\x;\classe)) + \frac{1}{|a|} \sum_{\alpha:i \in a}d_{\alpha:i}(T(\x;\classe))\right),
  \end{equation*}
  which completes the proof of Fact~4 and~\eqref{it:prelim:1} at the same time.
  
  \sk
  \noindent {\em Proof of~\eqref{it:prelim:2} and~\eqref{it:prelim:3}.} Let us fix $\classe = a \times b \in C$, $c \in \{a, b\}$ and $\gamma:= \type(c)$. As a preliminary step, let us point out the fact that, for all $\gamma:k \in c$, the quantity $t'_{\gamma:k}$ defined above easily rewrites
  \begin{equation*}
    t'_{\gamma:k} = \max\{(T^-(\classe))^- \wedge \Ttau_{\gamma:k}(\x), (T^-(\classe))^- \wedge \Ttau_{\gamma:k}(\y)\},
  \end{equation*}
  where we recall the definition~\eqref{eq:T-wedgetau} of $T^- \wedge \Ttau_{\gamma:k}(\x)$ and $T^- \wedge \Ttau_{\gamma:k}(\y)$. As a consequence, on the time interval $(t'_{\gamma:k}, T^-(\classe))$, the particle $\gamma:k$ does not collide with any particle of another type, neither in the MSPD started at $\x$ nor in the MSPD started at $\y$.
  
  Let us denote by $t'_1 < \cdots < t'_r$ the ordered elements of the set $\{t'_{\gamma:k}, \gamma:k \in c\}$. For all $l \in \{1, \ldots, r\}$, we denote by $c_l$ the set of particles $\gamma:k$ such that $t'_{\gamma:k} = t'_l$. We also define $t'_{r+1} := T^-(\classe) > t'_r$. Thanks to Condition~\eqref{cond:C3b}, for all $l \in \{1, \ldots, r\}$, the processes
  \begin{equation*}
    \{\Phi_k^{\gamma}(\x;t) : \gamma:k \in c_1 \sqcup \cdots \sqcup c_l\} \qquad \text{and} \qquad \{\Phi_k^{\gamma}(\y;t) : \gamma:k \in c_1 \sqcup \cdots \sqcup c_l\}
  \end{equation*}
  follow the Local Sticky Particle Dynamics on $[t'_l, t'_{l+1}]$, with the same initial velocity vectors. As a consequence,~\eqref{it:contractspd:1} in Proposition~\ref{prop:contractspd} yields, for all $t \in (t'_l, t'_{l+1}]$,
  \begin{equation*}
    \begin{aligned}
      \sum_{\gamma:k \in c} \ind{t > t_{\gamma:k}'} d_{\gamma:k}(t) & = \sum_{\gamma:k \in c_1 \sqcup \cdots \sqcup c_l} d_{\gamma:k}(t)\\
      & \leq \sum_{\gamma:k \in c_1 \sqcup \cdots \sqcup c_l} d_{\gamma:k}(t'_l) = \sum_{\gamma:k \in c_1 \sqcup \cdots \sqcup c_{l-1}} d_{\gamma:k}(t'_l) + \sum_{\gamma:k \in c_l} d_{\gamma:k}(t'_{\gamma:k}),
    \end{aligned}
  \end{equation*}
  therefore we obtain by induction that, for all $t \leq T^-(\classe)$,
  \begin{equation*}
    \sum_{\gamma:k \in c} \ind{t > t_{\gamma:k}'} d_{\gamma:k}(t) \leq \sum_{l=1}^r \sum_{\gamma:k \in c_l} d_{\gamma:k}(t'_{\gamma:k}) = \sum_{\gamma:k \in c} d_{\gamma:k}(t'_{\gamma:k}).
  \end{equation*}
  Applying~\eqref{it:contractspd:2} in Proposition~\ref{prop:contractspd} instead of~\eqref{it:contractspd:1}, we similarly obtain
  \begin{equation*}
    \sup_{\gamma:k \in c} \ind{t > t_{\gamma:k}'} d_{\gamma:k}(t) \leq \sup_{\gamma:k \in c} d_{\gamma:k}(t'_{\gamma:k}).
  \end{equation*}
     
  Finally,~\eqref{it:prelim:3} is obtained by the same arguments as~\eqref{it:prelim:2}: fixing $\gamma \in \{1, \ldots, d\}$ and denoting by $T_1 < \cdots < T_r$ the ordered elements of the set $\{\Tmax_{\gamma:k}, k \in \{1, \ldots, n\}\}$, we obtain that, for all $l \in \{1, \ldots, r\}$, the processes $\{\Phi_k^{\gamma}(\x;t) : \Tmax_{\gamma:k} < T_l\}$ and $\{\Phi_k^{\gamma}(\y;t) : \Tmax_{\gamma:k} < T_l\}$ follow the Local Sticky Particle Dynamics on $[T_l, T_{l+1})$ (where we take the convention that $T_{r+1} = +\infty$), with the same initial velocity vector. The conclusion follows in the same fashion as for~\eqref{it:prelim:2}.
\end{proof}


\subsubsection{Coupling with an auxiliary system}\label{sss:coupling} Let us fix a numbering of the collisions $\classe_1, \ldots, \classe_M$ as is provided by Lemma~\ref{lem:numbering}. Following the estimations of Lemma~\ref{lem:prelim}, for all $m \in \{1, \ldots, M\}$, for all $\gamma:k \in a_m \cup b_m$, $d_{\gamma:k}(T^+(\classe_m))$ is intuitively expected to be bounded by the quantity $e_m(\gamma:k)$ defined as follows: for all $\gamma:k \in \Part$, $e_0(\gamma:k) := d_{\gamma:k}(0)$; while, for all $m \in \{1, \ldots, M\}$,
\begin{itemize}
  \item for all $\alpha:i \in a_m$,
  \begin{equation*}
    e_m(\alpha:i) := \left(1+\frac{\Ratio}{n}|b_m|\right)\frac{1}{|a_m|}\sum_{\alpha:i' \in a_m} e_{m-1}(\alpha:i') + \frac{\Ratio}{n}\sum_{\beta:j \in b_m} e_{m-1}(\beta:j),
  \end{equation*}
  \item for all $\beta:j \in b_m$,
  \begin{equation*}
    e_m(\beta:j) := \left(1+\frac{\Ratio}{n}|a_m|\right)\frac{1}{|b_m|}\sum_{\beta:j' \in b_m} e_{m-1}(\beta:j') + \frac{\Ratio}{n}\sum_{\alpha:i \in a_m} e_{m-1}(\alpha:i),
  \end{equation*}
  \item for all $\gamma:k \not\in a_m \cup b_m$,
  \begin{equation*}
    e_m(\gamma:k) := e_{m-1}(\gamma:k).
  \end{equation*}
\end{itemize}
The sequence of functions $(e_m)_{0 \leq m \leq M}$ on $\Part$ is called the {\em auxiliary system}. Let us note that we have the following monotonicity relation: for all $m \in \{1, \ldots, M\}$,
\begin{equation}\label{eq:EmEmm1:new}
  \sum_{\alpha:i \in a_m} e_m(\alpha:i) \geq \sum_{\alpha:i \in a_m} e_{m-1}(\alpha:i), \qquad \sum_{\beta:j \in b_m} e_m(\beta:j) \geq \sum_{\beta:j \in b_m} e_{m-1}(\beta:j).
\end{equation}
The {\em total mass} of the auxiliary system is the nondecreasing sequence $(\totmass_m)_{0 \leq m \leq M}$ defined by
\begin{equation*}
  \totmass_m := \sum_{\gamma:k \in \Part} e_m(\gamma:k);
\end{equation*}
in particular,
\begin{equation}\label{eq:totmass0}
  ||\x-\y||_1 = \frac{\totmass_0}{n}.
\end{equation}

The coupling between the auxiliary system and the family of processes $\{(d_{\gamma:k}(t))_{t \geq 0}, \gamma:k \in \Part\}$ works as follows.

\begin{lem}[Coupling with the auxiliary system]\label{lem:coupling}
  Let us assume that the conditions of Proposition~\ref{prop:locstab} hold.
  \begin{enumerate}[label=(\roman*), ref=\roman*]
    \item\label{it:coupling:1} For all $m \in \{1, \ldots, M\}$, for all $t \in [T^-(\classe_m), T^+(\classe_m)]$,
    \begin{equation*}
      \begin{aligned}
        & \forall \alpha:i \in a_m, \qquad d_{\alpha:i}(t) \leq e_m(\alpha:i),\\
        & \forall \beta:j \in b_m, \qquad d_{\beta:j}(t) \leq e_m(\beta:j).
      \end{aligned}
    \end{equation*}
    
    \item\label{it:coupling:2} For all $t \geq 0$, 
    \begin{equation*}
      ||\Phi(\x;t)-\Phi(\y;t)||_1 = \frac{1}{n}\sum_{\gamma:k\in\Part} d_{\gamma:k}(t) \leq \frac{\totmass_M}{n}.
    \end{equation*}
    
    \item\label{it:coupling:3} For all $t \geq 0$, 
    \begin{equation*}
      ||\Phi(\x;t)-\Phi(\y;t)||_{\infty} = \sup_{\gamma:k\in\Part} d_{\gamma:k}(t) \leq \sup_{0 \leq m \leq M} \sup_{\gamma:k \in \Part} e_m(\gamma:k).
    \end{equation*}
  \end{enumerate}
\end{lem}
\begin{proof}
  The proof of~\eqref{it:coupling:1} works by induction on $m \in \{1, \ldots, M\}$. Let $m \in \{1, \ldots, M\}$ such that, if $m \geq 2$, then for all $m' \in \{1, \ldots, m-1\}$, for all $t \in [T^-(\classe_{m'}), T^+(\classe_{m'})]$, 
  \begin{equation*}
      \begin{aligned}
        & \forall \alpha:i \in a_{m'}, \qquad d_{\alpha:i}(t) \leq e_{m'}(\alpha:i),\\
        & \forall \beta:j \in b_{m'}, \qquad d_{\beta:j}(t) \leq e_{m'}(\beta:j).
      \end{aligned}
  \end{equation*}
  Let us fix $\alpha:i \in a_m$. By~\eqref{it:prelim:1} in Lemma~\ref{lem:prelim}, for all $t \in [T^-(\classe_m), T^+(\classe_m)]$,
  \begin{equation*}
    d_{\alpha:i}(t) \leq \left(1 + \frac{\Ratio}{n} |b_m|\right) \frac{1}{|a_m|}\sum_{\alpha:i' \in a_m} d_{\alpha:i'}(T^-(\classe_m)) + \frac{\Ratio}{n} \sum_{\beta:j \in b_m} d_{\beta:j}(T^-(\classe_m)),
  \end{equation*}
  and by~\eqref{it:prelim:2} in Lemma~\ref{lem:prelim},
  \begin{equation*}
    \sum_{\alpha:i' \in a_m} d_{\alpha:i'}(T^-(\classe_m)) \leq \sum_{\alpha:i' \in a_m} d_{\alpha:i'}(t'_{\alpha:i'}), \qquad \sum_{\beta:j \in b_m} d_{\beta:j}(T^-(\classe_m)) \leq \sum_{\beta:j \in b_m} d_{\beta:j}(t'_{\beta:j}),
  \end{equation*}
  where $t'_{\alpha:i'}$ is $T^+(\classe')$ if there exists $\classe' \in C_{\alpha:i'}$ such that $\classe' \lto{\alpha} \classe_m$ and $0$ otherwise; $t'_{\beta:j}$ in the second inequality is defined similarly.
  
  Let $m_1, \ldots, m_K \leq m-1$ be the indices of all the collisions $\classe'$ such that $\classe' \lto{\alpha} \classe_m$, and for all $k \in \{1, \ldots, K\}$, let us denote by $a'_{m_k}$ the cluster of type $\alpha$ involved in the collision $\classe_{m_k}$. Then
  \begin{equation*}
    \sum_{\alpha:i' \in a_m} d_{\alpha:i'}(t'_{\alpha:i'}) = \sum_{k=1}^K \sum_{\alpha:i' \in a'_{m_k}} d_{\alpha:i'}(T^+(\classe_{m_k})) + \sum_{\alpha:i' \in a_m : t'_{\alpha:i'}=0} d_{\alpha:i'}(0).
  \end{equation*}
  For all $k \in \{1, \ldots, K\}$, $m_k \leq m-1$ so that
  \begin{equation*}
    \sum_{\alpha:i' \in a'_{m_k}} d_{\alpha:i'}(T^+(\classe_{m_k})) \leq \sum_{\alpha:i' \in a'_{m_k}} e_{m_k}(\alpha:i') = \sum_{\alpha:i' \in a'_{m_k}} e_{m_k+1}(\alpha:i') = \cdots = \sum_{\alpha:i' \in a'_{m_k}} e_{m-1}(\alpha:i'),
  \end{equation*}
  while, for all $\alpha:i'$ such that $t'_{\alpha:i'}=0$,
  \begin{equation*}
    d_{\alpha:i'}(0) = e_0(\alpha:i') = e_1(\alpha:i') = \cdots = e_{m-1}(\alpha:i').
  \end{equation*}
  As a conclusion, for all $t \in [T^-(\classe_m), T^+(\classe_m)]$,
  \begin{equation*}
    d_{\alpha:i}(t) \leq \left(1 + \frac{\Ratio}{n} |b_m|\right) \frac{1}{|a_m|}\sum_{\alpha:i' \in a_m} e_{m-1}(\alpha:i') + \frac{\Ratio}{n} \sum_{\beta:j \in b_m} e_{m-1}(\beta:j) = e_m(\alpha:i),
  \end{equation*}
  and the bound on $d_{\beta:j}(t)$, $\beta:j \in b_m$, $t \in [T^-(\classe_m), T^+(\classe_m)]$, follows from the same arguments.
  
  We now address~\eqref{it:coupling:2} and~\eqref{it:coupling:3}. Let us fix $t \geq 0$ and note that, at time $t$, a particle $\gamma:k \in \Part$ is in exactly one of the following cases:
  \begin{enumerate}
    \item\label{it:pf:coupl:1} there exists $\classe \in C_{\gamma:k}$ such that $T^-(\classe) \leq t \leq T^+(\classe)$,
    \item\label{it:pf:coupl:2} $t \leq \Tmax_{\gamma:k}$ and, for all $\classe \in C_{\gamma:k}$, $t \not\in [T^-(\classe), T^+(\classe)]$.
    \item\label{it:pf:coupl:3} $t > \Tmax_{\gamma:k}$.
  \end{enumerate}
  
  If the particle $\gamma:k$ is in case~\eqref{it:pf:coupl:1}, let us note that, by Condition~\eqref{cond:C3a}, there is only one $\classe \in C_{\gamma:k}$ such that $T^-(\classe) \leq t \leq T^+(\classe)$. Let $\mu_t(\gamma:k)$ denote the number of the collision $\classe$. By~\eqref{it:coupling:1}, $d_{\gamma:k}(t) \leq e_{\mu_t(\gamma:k)}(\gamma:k)$. 
  
  In case~\eqref{it:pf:coupl:2}, let us denote by $\classe$ the first collision $\classe \in C_{\gamma:k}$ (for the order of collisions) such that $t < T^-(\classe)$. Let $t'_{\gamma:k}$ be defined as in Lemma~\ref{lem:prelim}. If $t'_{\gamma:k}=0$, then we let $\mu_t(\gamma:k)=0$ so that $d_{\gamma:k}(t'_{\gamma:k}) = e_{\mu_t(\gamma:k)}(\gamma:k)$. Otherwise, there exists $\classe' \in C_{\gamma:k}$ such that $\classe' \lto{\gamma} \classe$, in this case we let $\mu_t(\gamma:k)$ refer to the number of the collision $\classe'$ and~\eqref{it:coupling:1} yields $d_{\gamma:k}(t'_{\gamma:k}) \leq e_{\mu_t(\gamma:k)}(\gamma:k)$. Let us now note that, calling $c$ the generical cluster of type $\gamma$ involved in the collision $\classe$, for all $\gamma:k' \in c$ such that $t'_{\gamma:k'} < t$, then $\gamma:k'$ is also in case~\eqref{it:pf:coupl:2}, and~\eqref{it:prelim:2} in Lemma~\ref{lem:prelim} yields
  \begin{equation*}
    \begin{aligned}
      \sum_{\gamma:k' \in c} \ind{t > t_{\gamma:k'}'} d_{\gamma:k'}(t) & \leq \sum_{\gamma:k' \in c} \ind{t > t_{\gamma:k'}'} e_{\mu_t(\gamma:k')}(\gamma:k'),\\
      \sup_{\gamma:k' \in c} \ind{t > t_{\gamma:k'}'} d_{\gamma:k'}(t) & \leq \sup_{\gamma:k' \in c} \ind{t > t_{\gamma:k'}'} e_{\mu_t(\gamma:k')}(\gamma:k').
    \end{aligned}
  \end{equation*}

  In case~\eqref{it:pf:coupl:3}, if $C_{\gamma:k}$ is empty, we define $\mu_t(\gamma:k)=0$ so that $d_{\gamma:k}(\Tmax_{\gamma:k}) = e_{\mu_t(\gamma:k)}(\gamma:k)$, otherwise, we denote by $\mu_t(\gamma:k)$ the number of the last collision in $C_{\gamma:k}$, and then~\eqref{it:coupling:1} yields $d_{\gamma:k}(\Tmax_{\gamma:k}) \leq e_{\mu_t(\gamma:k)}(\gamma:k)$. Similarly,~\eqref{it:prelim:3} in Lemma~\ref{lem:prelim} yields
  \begin{equation*}
    \begin{aligned}
      \sum_{k'=1}^n \ind{t > \Tmax_{\gamma:k'}} d_{\gamma:k'}(t) & \leq \sum_{k'=1}^n \ind{t > \Tmax_{\gamma:k'}} e_{\mu_t(\gamma:k')}(\gamma:k'),\\
      \sup_{1 \leq k' \leq n} \ind{t > \Tmax_{\gamma:k'}} d_{\gamma:k'}(t) & \leq \sup_{1 \leq k' \leq n} \ind{t > \Tmax_{\gamma:k'}} e_{\mu_t(\gamma:k')}(\gamma:k').
    \end{aligned}
  \end{equation*}
  
  As a consequence, we have constructed a function $\mu_t : \Part \to \{0, \ldots, M\}$ such that
  \begin{equation}\label{eq:pfcoupling:ineq}
    \sum_{\gamma:k \in \Part} d_{\gamma:k}(t) \leq \sum_{\gamma:k \in \Part} e_{\mu_t(\gamma:k)}(\gamma:k), \qquad \sup_{\gamma:k \in \Part} d_{\gamma:k}(t) \leq \sup_{\gamma:k \in \Part} e_{\mu_t(\gamma:k)}(\gamma:k).
  \end{equation}
  The point~\eqref{it:coupling:3} easily follows from the second inequality of~\eqref{eq:pfcoupling:ineq}. We also obtain 
  \begin{equation*}
    ||\Phi(\x;t)-\Phi(\y;t)||_1 \leq \frac{1}{n} \sum_{\gamma:k \in \Part} \sup_{0 \leq m \leq M} e_m(\gamma:k)
  \end{equation*}
  as a straightforward consequence of the first inequality, but the sum and supremum are in the reverse order compared with~\eqref{it:coupling:2}. Hence we need to work a little more on the first inequality.  
  
  Let us define $\bar{\mu}_t(\gamma:k)$ as the index of the first collision in $C_{\gamma:k}$ with number strictly larger than $\mu_t(\gamma:k)$, or $M+1$ if there is no such collision. We now check that the function $\mu_t$ satisfies the following two conditions, which will enable us to conclude thanks to Lemma~\ref{lem:Mnu} below.
  \begin{enumerate}[label=(*\arabic*), ref=*\arabic*]
    \item\label{cond:pfcoupling:star1} for all $\gamma:k \in \Part$ such that $\mu_t(\gamma:k) \in \{1, \ldots, M\}$, then $\classe_{\mu_t(\gamma:k)} \in C_{\gamma:k}$,
    \item\label{cond:pfcoupling:star2} for all $\ell \geq 0$, there is no path
    \begin{equation*}
      \classe_{m_0} \lto{\gamma_1} \classe_{m_1} \lto{\gamma_2} \cdots \lto{\gamma_{\ell}} \classe_{m_{\ell}}
    \end{equation*}
    such that 
    \begin{equation*}
      m_0 \in \{\bar{\mu}_t(\gamma':k'), \gamma':k' \in \Part\}, \qquad m_{\ell} \in \{\mu_t(\gamma:k), \gamma:k \in \Part\}
    \end{equation*}
    in the collision graph, where the case $\ell=0$ is understood as the condition that the two sets above be disjoint.
  \end{enumerate}
  It follows from the construction of $\mu_t$ that the latter satisfies~\eqref{cond:pfcoupling:star1} as well as the property that, for all $\gamma:k \in \Part$,
  \begin{itemize}
    \item if $\mu_t(\gamma:k) \in \{1, \ldots, M\}$, then $T^-(\classe_{\mu_t(\gamma:k)}) \leq t$,
    \item if $\bar{\mu}_t(\gamma:k) \in \{1, \ldots, M\}$, then $T^-(\classe_{\bar{\mu}_t(\gamma:k)}) > t$.
  \end{itemize}
  As a consequence, if there exists a path
  \begin{equation*}
    \classe_{m_0} \lto{\gamma_1} \classe_{m_1} \lto{\gamma_2} \cdots \lto{\gamma_{\ell}} \classe_{m_{\ell}}
  \end{equation*}
  in the collision graph, with $m_0 = \bar{\mu}_t(\gamma':k')$, $m_{\ell} = \mu_t(\gamma:k)$, for some $\gamma':k', \gamma:k \in \Part$, then either $\ell=0$ in which case $t < T^-(\classe_{m_0}) \leq t$ is absurd, or Condition~\eqref{cond:C3a} yields $t < T^-(\classe_{m_0}) < T^-(\classe_{m_{\ell}}) \leq t$, which is also absurd. As a conclusion, there is no such path in the graph, and $\mu_t$ satisfies~\eqref{cond:pfcoupling:star2}.
  
  Following Lemma~\ref{lem:Mnu} below, Conditions~\eqref{cond:pfcoupling:star1} and~\eqref{cond:pfcoupling:star2} imply that
  \begin{equation*}
    \sum_{\gamma:k \in \Part} e_{\mu_t(\gamma:k)}(\gamma:k) \leq \totmass_M,
  \end{equation*}
  which allows us to complete the proof by injecting this inequality into the first part of~\eqref{eq:pfcoupling:ineq}.
\end{proof}
  
The proof of~\eqref{it:coupling:2} in Lemma~\ref{lem:coupling} relies on Lemma~\ref{lem:Mnu} below. Before stating the latter, we first introduce a few notions. For all functions $\mu : \Part \to \{0, \ldots, M\}$, we define $\bar{\mu} : \Part \to \{1, \ldots, M+1\}$ by, for all $\gamma:k \in \Part$,
\begin{equation*}
  \bar{\mu}(\gamma:k) := \min\left(\{\mu(\gamma:k)+1, \ldots, M\} \cap \{m : \classe_m \in C_{\gamma:k}\}\right)
\end{equation*}
if the set in the right-hand side is nonempty, and
\begin{equation*}
  \bar{\mu}(\gamma:k) := M+1
\end{equation*}
otherwise. Note that, for all $\gamma:k \in \Part$, $\bar{\mu}(\gamma:k) > \mu(\gamma:k)$ and
\begin{equation}\label{eq:pfcoupling:A}
  e_{\mu(\gamma:k)}(\gamma:k) = e_{\mu(\gamma:k)+1}(\gamma:k) = \cdots = e_{\bar{\mu}(\gamma:k)-1}(\gamma:k).
\end{equation}

Let us also denote by $\Mnu$ the set of functions $\mu : \Part \to \{0, \ldots, M\}$ satisfying the conditions~\eqref{cond:pfcoupling:star1} and~\eqref{cond:pfcoupling:star2} introduced at the end of the proof of Lemma~\ref{lem:coupling} above. For $\mu \in \Mnu$, combining~\eqref{eq:pfcoupling:A} with~\eqref{eq:EmEmm1:new}, we now remark that the group of particles $\gamma:k$ such that $\bar{\mu}(\gamma:k)$ is minimal satisfies
\begin{equation*}
  \sum e_{\mu(\gamma:k)}(\gamma:k) \leq \sum e_{\bar{\mu}(\gamma:k)}(\gamma:k),
\end{equation*}
so that one obtains an upper bound on $\sum_{\gamma:k \in \Part} e_{\mu(\gamma:k)}(\gamma:k)$ if one replaces $\mu(\gamma:k)$ with $\bar{\mu}(\gamma:k)$ for those particles. Iterating the argument until all the quantities $\bar{\mu}(\gamma:k)$ reach the maximum value $M+1$, we finally obtain the expected upper bound $\totmass_M$. The rigorous formulation of this iterative argument is detailed in Lemma~\ref{lem:Mnu}.
  
\begin{lem}[Property of the set $\Mnu$]\label{lem:Mnu}
  For all functions $\mu : \Part \to \{0, \ldots, M\}$ in the set $\Mnu$ introduced above, we have
  \begin{equation*}
    \sum_{\gamma:k \in \Part} e_{\mu(\gamma:k)}(\gamma:k) \leq \totmass_M.
  \end{equation*}
\end{lem}
\begin{proof}
  For all functions $\mu : \Part \to \{0, \ldots, M\}$, let us define
  \begin{equation*}
    \bar{\mu}_* := \min_{\gamma:k \in \Part} \bar{\mu}(\gamma:k) \in \{1, \ldots, M+1\},
  \end{equation*}
  and let us denote by $\Mnu_*$ the set of functions $\mu \in \Mnu$ such that $\bar{\mu}_* \leq M$. Then we have the following property: for all $\mu \in \Mnu_*$, for all $\gamma:k \in \Part$, 
  \begin{equation}\label{eq:pfcoupling:B}
    \bar{\mu}(\gamma:k) = \bar{\mu}_* \qquad \text{if and only if} \qquad \classe_{\bar{\mu}_*} \in C_{\gamma:k}.
  \end{equation}
  Indeed, the direct implication is a straightforward consequence of the definition of $\bar{\mu}$. The reverse implication stems from the following argument: if $\gamma:k \in \Part$ is such that $\classe_{\bar{\mu}_*} \in C_{\gamma:k}$, then the minimality of $\bar{\mu}_*$ implies that $\bar{\mu}(\gamma:k) \geq \bar{\mu}_*$. Assume that $\bar{\mu}(\gamma:k) > \bar{\mu}_*$, then by~\eqref{cond:pfcoupling:star1} and the definition of $\bar{\mu}$, we have that $\bar{\mu}_* \leq \mu(\gamma:k)$. As a consequence, there exists a path
  \begin{equation*}
    \classe_{\bar{\mu}_*} \lto{\gamma} \cdots \lto{\gamma} \classe_{\mu(\gamma:k)}
  \end{equation*}
  in the collision graph, which is a contradiction with~\eqref{cond:pfcoupling:star2}.
  
  For all $\mu \in \Mnu_*$, we now define $\trans\mu : \Part \to \{0, \ldots, M\}$ by, for all $\gamma:k \in \Part$,
  \begin{equation*}
    \tau\mu(\gamma:k) := \begin{cases}
      \mu(\gamma:k) & \text{if $\bar{\mu}(\gamma:k) > \bar{\mu}_*$,}\\
      \bar{\mu}_* & \text{if $\bar{\mu}(\gamma:k) = \bar{\mu}_*$.}
    \end{cases}
  \end{equation*}
  Let us note that, for all $\gamma:k \in \Part$ such that $\bar{\mu}(\gamma:k) > \bar{\mu}_*$, then $\bar{\trans\mu}(\gamma:k) = \bar{\mu}(\gamma:k)$ and as a consequence,
  \begin{equation}\label{eq:pfcoupling:C}
    \bar{\trans\mu}_* > \bar{\mu}_*.
  \end{equation}
  
  We now prove that $\trans\mu \in \Mnu$. The fact that $\trans\mu$ satisfies~\eqref{cond:pfcoupling:star1} easily follows from~\eqref{eq:pfcoupling:B} combined with the fact that $\mu$ satisfies~\eqref{cond:pfcoupling:star1}. As far as~\eqref{cond:pfcoupling:star2} is concerned, let us assume by contradiction that there exists a path
  \begin{equation*}
    \classe_{m_0} \lto{\gamma_1} \classe_{m_1} \lto{\gamma_2} \cdots \lto{\gamma_{\ell}} \classe_{m_{\ell}}
  \end{equation*}
  in the collision graph, with $m_0 = \bar{\trans\mu}(\gamma':k')$, $m_{\ell} = \trans\mu(\gamma:k)$, for some $\gamma':k', \gamma:k \in \Part$. Then
  \begin{equation*}
    \trans\mu(\gamma:k) = m_{\ell} \geq m_0 = \bar{\trans\mu}(\gamma':k') \geq \bar{\trans\mu}_* > \bar{\mu}_*,
  \end{equation*}
  where the last inequality follows from~\eqref{eq:pfcoupling:C}. As a consequence, we deduce from the construction of $\trans\mu$ that $m_{\ell} = \mu(\gamma:k)$, while $m_0$ is either $\bar{\mu}(\gamma':k')$ or such that $\classe_{\bar{\mu}_*} \lto{\gamma'} \classe_{m_0}$. In both cases, there is a contradiction with the fact that $\mu$ satisfies~\eqref{cond:pfcoupling:star2}.
  
  The introduction of the operator $\trans$ allows to obtain the following key property: for all $\mu \in \Mnu_*$,
  \begin{equation*}
    \sum_{\gamma:k \in \Part} e_{\mu(\gamma:k)}(\gamma:k) \leq \sum_{\gamma:k \in \Part} e_{\trans\mu(\gamma:k)}(\gamma:k).
  \end{equation*}
  To prove this inequality, it suffices to check that
  \begin{equation*}
    \sum_{\substack{\gamma:k \in \Part\\ \bar{\mu}(\gamma:k) = \bar{\mu}_*}} e_{\mu(\gamma:k)}(\gamma:k) \leq \sum_{\substack{\gamma:k \in \Part\\ \bar{\mu}(\gamma:k) = \bar{\mu}_*}} e_{\bar{\mu}_*}(\gamma:k),
  \end{equation*}
  which follows from the sequence of assertions
  \begin{equation*}
    \begin{aligned}
      \sum_{\substack{\gamma:k \in \Part\\ \bar{\mu}(\gamma:k) = \bar{\mu}_*}} e_{\mu(\gamma:k)}(\gamma:k) & = \sum_{\substack{\gamma:k \in \Part\\ \bar{\mu}(\gamma:k) = \bar{\mu}_*}} e_{\bar{\mu}_*-1}(\gamma:k)\\
      & = \sum_{\alpha:i \in a_{\bar{\mu}}} e_{\bar{\mu}_*-1}(\alpha:i) + \sum_{\beta:j \in b_{\bar{\mu}}} e_{\bar{\mu}_*-1}(\beta:j)\\
      & \leq \sum_{\alpha:i \in a_{\bar{\mu}}} e_{\bar{\mu}_*}(\alpha:i) + \sum_{\beta:j \in b_{\bar{\mu}}} e_{\bar{\mu}_*}(\beta:j),
    \end{aligned}
  \end{equation*}
  where we have used~\eqref{eq:pfcoupling:A} at the first line,~\eqref{eq:pfcoupling:B} at the second line, and~\eqref{eq:EmEmm1:new} at the third line.
  
  As a consequence, for all $\mu \in \Mnu$, either $\bar{\mu}_* = M+1$ in which case~\eqref{eq:pfcoupling:A} yields
  \begin{equation*}
    \sum_{\gamma:k \in \Part} e_{\mu(\gamma:k)}(\gamma:k) = \sum_{\gamma:k \in \Part} e_{\bar{\mu}(\gamma:k)-1}(\gamma:k) = \sum_{\gamma:k \in \Part} e_M(\gamma:k) = \totmass_M,
  \end{equation*}
  or $\mu \in \Mnu_*$ and by~\eqref{eq:pfcoupling:C}, the operator $\trans$ can be applied a finite number $r$ of times to obtain $\bar{\tau^r \mu}_* = M+1$, in which case we recover
  \begin{equation*}
    \sum_{\gamma:k \in \Part} e_{\mu(\gamma:k)}(\gamma:k) \leq \sum_{\gamma:k \in \Part} e_{\trans^r\mu(\gamma:k)}(\gamma:k) = \totmass_M,
  \end{equation*}
  which completes the proof.
\end{proof}


\subsubsection{Bounding the total mass}\label{sss:totmass} As a consequence of Lemma~\ref{lem:coupling}, the local stability estimates of Proposition~\ref{prop:locstab} are derived from the following estimation on the total mass of the auxiliary system.

\begin{lem}[Estimation on the total mass]\label{lem:totmass}
  Under the assumptions of Proposition~\ref{prop:locstab}, the total mass of the auxiliary system satisfies
  \begin{equation*}
    \totmass_M \leq \ConstStab_1 \totmass_0,
  \end{equation*}
  where $\ConstStab_1$ is defined by~\eqref{eq:ConstStab}. Besides,
  \begin{equation*}
    \sup_{0 \leq m \leq M} \sup_{\gamma:k \in \Part} e_m(\gamma:k) \leq \ConstStab_{\infty} \sup_{\gamma:k \in \Part} e_0(\gamma:k),
  \end{equation*}
  where $\ConstStab_{\infty}$ is defined by~\eqref{eq:ConstStab}.
\end{lem}
The conclusion of Proposition~\ref{prop:locstab} easily follows from the combination of Lemmas~\ref{lem:coupling} and~\ref{lem:totmass} with~\eqref{eq:totmass0}.

In order to prove Lemma~\ref{lem:totmass}, let us introduce a few notions and notations. Given a sequence of particles $g = (\gamma:k_1, \ldots, \gamma:k_L) \in (\Part)^L$, with $L \geq 1$, and a particle $\gamma:k_{L+1} \in \Part$, we denote by $g::(\gamma:k_{L+1})$ the sequence $(\gamma:k_1, \ldots, \gamma:k_{L+1}) \in (\Part)^{L+1}$. 

For all $m \in \{0, \ldots, M\}$ and for all $\gamma:k \in \Part$, we first define the set $\Gamma^-_m(\gamma:k)$ of sequences of particles as follows:
\begin{itemize}
  \item for all $\gamma:k \in \Part$, the set $\Gamma^-_0(\gamma:k)$ contains the single element $(\gamma:k)$,
  \item for all $m \in \{1, \ldots, M\}$, 
  \begin{equation*}
    \begin{aligned}
      & \forall \alpha:i \in a_m, \qquad \Gamma^-_m(\alpha:i) := \bigcup_{\alpha:i' \in a_m} \{g'::(\alpha:i), g' \in \Gamma^-_{m-1}(\alpha:i')\},\\
      & \forall \beta:j \in b_m, \qquad \Gamma^-_m(\beta:j) := \bigcup_{\beta:j' \in b_m} \{g'::(\beta:j), g' \in \Gamma^-_{m-1}(\beta:j')\},\\
      & \forall \gamma:k \not\in a_m \cup b_m, \qquad \Gamma^-_m(\gamma:k) := \Gamma^-_{m-1}(\gamma:k)\}.
    \end{aligned}
  \end{equation*} 
\end{itemize}
In other words, $\Gamma^-_m(\gamma:k)$ contains the set of sequences $g = (\gamma:k_0, \ldots, \gamma:k_L)$, such that there exists a sequence of collisions $(\classe_{m_1}, \ldots, \classe_{m_L})$ satisfying:
\begin{itemize}
  \item $\classe_{m_1}$ is the first element of $C_{\gamma:k_0}$,
  \item for all $l \in \{1, \ldots, L-1\}$, $\classe_{m_l}$ and $\classe_{m_{l+1}}$ are consecutive elements of $C_{\gamma:k_l}$,
  \item $\gamma:k_L = \gamma:k$ and $\classe_{m_L}$ is the last element of $C_{\gamma:k}$ with number lower than $m$.
\end{itemize}
Note that the sequence of collisions $(\classe_{m_1}, \ldots, \classe_{m_L})$ is uniquely determined by the conditions above. 

An element $g$ of some set $\Gamma^-_m(\gamma:k)$ shall be called a {\em type path}, as it describes an oriented path in the collision graph with all edges having the same type. For all $g \in \Gamma^-_m(\gamma:k)$, we denote $F(g) := \gamma:k_0$. Besides, for all $m' \in \{0, \ldots, m\}$, we define $c_{m'}(g)$ as follows: if there exists $l \in \{1, \ldots, L\}$ such that $m' = m_l$, then $c_{m'}(g)$ is the generical cluster of type $\gamma$ involved in the collision $\classe_{m_l}$; while otherwise, $c_{m'}(g) = \gamma:k_l$, where $l$ is the largest index in $\{0, \ldots, L\}$ such that $m' > m_l$ (we take the convention that $m_0 = 0$). We finally define the weight of a type path $g \in \Gamma^-_m(\gamma:k)$ by
\begin{equation*}
  w^-_m(g) := \prod_{m'=1}^m \frac{1}{|c_{m'}(g)|}.
\end{equation*}
This quantity has the following interpretation: given $m \in \{1, \ldots, M\}$ and $\gamma:k \in \Part$, let $\classe$ be the last collision in $C_{\gamma:k}$ with number lower than $m$. Select a particle $\gamma:k'$ uniformly at random among the particles of type $\gamma$ involved in the collision $\classe$. If it exists, call $\classe'$ the collision preceding $\classe$ in $C_{\gamma:k'}$, and move from $\classe$ to $\classe'$. This motion is backward with respect to the orientation of the collision graph. Now repeat the random selection and backward motion as long as possible. Then, the sequence $(\gamma:k, \gamma:k', \ldots)$ of selected particles at successive collisions is the reverse of a type path $g \in \Gamma^-_m(\gamma:k)$, and its {\em weight} $w^-_m(g)$ is the probability of selecting this path. In particular, we deduce that, for all $m \in \{0, \ldots, M\}$, for all $\gamma:k \in \Part$,
\begin{equation}\label{eq:sumwg}
  \sum_{g \in \Gamma^-_m(\gamma:k)} w^-_m(g) = 1.
\end{equation}

We now define the {\em history} of a type path $g \in \Gamma^-_m(\gamma:k)$ as follows. In the case $m=0$, we let
\begin{equation*}
  H_0(\gamma:k) := \left(\bigcup_{\gamma' < \gamma} \{(0, \gamma':k'), (\gamma':k',\gamma:k) \not\in R\}\right) \cup \left(\bigcup_{\gamma' > \gamma} \{(0, \gamma':k'), (\gamma:k,\gamma':k') \not\in R\}\right);
\end{equation*}
in other words, $H_0(\gamma:k)$ contains all the pairs $(0,\gamma':k')$ where $\gamma':k'$ is a particle that will never cross the particle $\gamma:k$. Now for all $m \in \{1, \dots, M\}$,
\begin{itemize}
  \item for all $\alpha:i \in a_m$, for all $g \in \Gamma^-_m(\alpha:i)$, 
  \begin{equation*}
    H_m(g) := H_{m-1}(g') \cup \{(m,\beta:j), \beta:j \in b_m\},
  \end{equation*}
  where $g' \in \cup_{\alpha:i' \in a_m}(\alpha:i')$ is such that $g = g'::(\alpha:i)$; 
  \item for all $\beta:j \in b_m$, for all $g \in \Gamma^-_m(\beta:j)$, 
  \begin{equation*}
    H_m(g) := H_{m-1}(g') \cup \{(m,\alpha:i), \alpha:i \in a_m\},
  \end{equation*}
  where $g' \in \cup_{\beta:j' \in b_m} \Gamma^-_{m-1}(\beta:j')$ is such that $g = g'::(\beta:j)$;
  \item for all $\gamma:k \not\in a_m \cup b_m$, then for all $g \in \Gamma^-_m(\gamma:k) = \Gamma^-_{m-1}(\gamma:k)$,
  \begin{equation*}
    H_m(g) = H_{m-1}(g).
  \end{equation*}
\end{itemize}
In other words, $H_m(g)$ records the pairs $(m',\gamma':k')$ such that at the collision $\classe_{m'}$, the particle contained in the path $g$ has crossed the particle $\gamma':k'$.

The sets $H_m(g)$ have the following properties.

\begin{lem}[Properties of $H_m(g)$]\label{lem:propHmg}
  Let $m \in \{0, \ldots, M\}$, $\gamma:k \in \Part$ and $g \in \Gamma^-_m(\gamma:k)$.
  \begin{enumerate}[label=(\roman*), ref=\roman*]
    \item\label{it:propHmg:1} For all $(m',\gamma':k') \in H_m(g)$, we have $\gamma' \not= \gamma$.
    \item\label{it:propHmg:2} For all $(m',\gamma':k'), (m'',\gamma'':k'') \in H_m(g)$ such that $m' \not= m''$, we have $\gamma':k' \not= \gamma'':k''$.
    \item\label{it:propHmg:3} For all $\bar{m} \in \{1, \ldots, m\}$, for all $c \in \{a_{\bar{m}}, b_{\bar{m}}\}$, if there exists $(m',\gamma':k') \in H_m(g)$ such that $m' \leq \bar{m}$ and $\gamma':k' \in c$, then for all $\gamma':k'' \in c$, there exists $m'' \in \{0, \ldots, m\}$ such that $m'' \leq \bar{m}$ and $(m'', \gamma':k'') \in H_m(g)$.
  \end{enumerate}
\end{lem}

Despite its seemingly technical formulation, Lemma~\ref{lem:propHmg} is quite intuitive, and instead of detailing its proof, we rather give a formal explanation of each result. Coming back to the MSPD started at $\x$ (or indifferently $\y$), one can associate $g = (\gamma:k_0, \ldots, \gamma:k_L)$ with the continuous path $(G(t))_{t \geq 0}$ starting from $x_{k_0}^{\gamma}$, then joining the space-time points of collisions $\Xi(\x;\classe_{m_l})$ and $\Xi(\x;\classe_{m_{l+1}})$ following the trajectory of the particle $\gamma:k_l$, and such that $G(t) = \Phi_{k_L}^{\gamma}(\x;t)$ for $t \geq T(\x;\classe_{m_L})$.

Then $H_m(g)$ is the set of pairs $(m',\gamma':k')$ such that the particle $\gamma':k'$ have crossed the path $G$ at the collision $\classe_{m'}$, with $m'$ lower than $m$, or in the virtual past for $m'=0$. The point~\eqref{it:propHmg:1} is therefore obvious. The point~\eqref{it:propHmg:2} means that, along the path $G$, one cannot cross the same particle twice; since $G$ remains supported by the trajectories of particles of the same type $\gamma$, this is a straightforward consequence of Assumption~\eqref{ass:USH}. Finally, the point~\eqref{it:propHmg:3} expresses the fact that if two particles of the same type are involved in a collision $\classe_{\bar{m}}$ such that $\bar{m} \leq m$, which is not located along the path $G$, and at least one of these particles has crossed $G$ (possibly in the virtual past), then the other one has necessarily crossed $G$ too. This is a consequence of the continuity of $G$ combined with the properties of the numbering of the collisions.

\sk
The proof of Lemma~\ref{lem:totmass} relies on the intermediary Lemmas~\ref{lem:keyestim} and~\ref{lem:L1Linf}.

\begin{lem}[Integration along paths]\label{lem:keyestim}
  Under the assumptions of Lemma~\ref{lem:totmass}, for all $m \in \{0, \ldots, M\}$, for all $\gamma:k \in \Part$,
  \begin{equation}\label{eq:Emc}
    e_m(\gamma:k) \leq \sum_{g \in \Gamma^-_m(\gamma:k)} \left(1 + \frac{\Ratio}{n}\right)^{|H_m(g)|} w^-_m(g)\left\{e_0(F(g)) + \frac{\Ratio}{n} \sum_{(m',\gamma':k') \in H_m(g)} e_{m'-1}(\gamma':k')\right\},
  \end{equation}
  where we take the convention that, for all $\gamma':k' \in \Part$, $e_{-1}(\gamma':k') = 0$.
\end{lem}
\begin{proof}
  The proof works by induction on $m \in \{0, \ldots, M\}$. For $m=0$, the inequality is trivial. Now let $m \in \{1, \ldots, M\}$ such that~\eqref{eq:Emc} holds true up to $m-1$. Then for all $\gamma:k \not\in a_m \cup b_m$,
  \begin{equation*}
    \begin{aligned}
      & e_m(\gamma:k) = e_{m-1}(\gamma:k)\\
      & \leq \sum_{g \in \Gamma^-_{m-1}(\gamma:k)} \left(1 + \frac{\Ratio}{n}\right)^{|H_{m-1}(g)|} w^-_{m-1}(g) \left\{e_0(F(g)) + \frac{\Ratio}{n} \sum_{(m',\gamma':k') \in H_{m-1}(g)} e_{m'-1}(\gamma':k')\right\}\\
      & = \sum_{g \in \Gamma^-_m(\gamma:k)} \left(1 + \frac{\Ratio}{n}\right)^{|H_m(g)|} w^-_m(g) \left\{e_0(F(g)) + \frac{\Ratio}{n} \sum_{(m',\gamma':k') \in H_m(g)} e_{m'-1}(\gamma':k')\right\},
    \end{aligned}
  \end{equation*}
  as $\Gamma^-_{m-1}(\gamma:k) = \Gamma^-_m(\gamma:k)$ and, for all $g \in \Gamma^-_{m-1}(\gamma:k)$, we have $H_{m-1}(g) = H_m(g)$ and $w^-_{m-1}(g) = w^-_m(g)$. Now for all $\alpha:i \in a_m$,
  \begin{equation*}
    \begin{aligned}
      e_m(\alpha:i) & = \left(1+\frac{\Ratio}{n}|b_m|\right) \frac{1}{|a_m|} \sum_{\alpha:i' \in a_m} e_{m-1}(\alpha:i') + \frac{\Ratio}{n} \sum_{\beta:j \in b_m} e_{m-1}(\beta:j)\\
      & \leq \sum_{\alpha:i' \in a_m} \sum_{g' \in \Gamma^-_{m-1}(\alpha:i')} \left(1+\frac{\Ratio}{n}\right)^{|b_m| + |H_{m-1}(g')|} \\
      & \quad \times \frac{1}{|a_m|}w^-_{m-1}(g')\left\{e_0(F(g')) + \frac{\Ratio}{n} \sum_{(m',\gamma':k') \in H_{m-1}(g')} e_{m'-1}(\gamma':k')\right\}\\
      & \quad + \frac{\Ratio}{n} \sum_{\beta:j \in b_m} e_{m-1}(\beta:j),
    \end{aligned}
  \end{equation*}
  where we have used the elementary inequality
  \begin{equation}\label{eq:elementary}
    \forall x \geq 0, \quad \forall k \geq 1, \qquad 1+kx \leq (1+x)^k.
  \end{equation}
  Let us remark that each type path $g \in \Gamma^-_m(\alpha:i)$ writes $g'::(\alpha:i)$ with $g' \in \sqcup_{\alpha:i' \in a_m} \Gamma^-_{m-1}(\alpha:i')$, and that $|H_m(g)| = |H_{m-1}(g')| + |b_m|$, $w^-_m(g) = w^-_{m-1}(g') / |a_m|$, and $F(g)=F(g')$. We deduce that
  \begin{equation*}
    \begin{aligned}
      & \sum_{\alpha:i' \in a_m} \sum_{g' \in \Gamma^-_{m-1}(\alpha:i')} \left(1+\frac{\Ratio}{n}\right)^{|b_m| + |H_{m-1}(g')|} \times \frac{1}{|a_m|}w^-_{m-1}(g')e_0(F(g'))\\
      & \qquad = \sum_{g \in \Gamma^-_m(\alpha:i)} \left(1+\frac{\Ratio}{n}\right)^{|H_m(g)|} w^-_m(g)e_0(F(g)),
    \end{aligned}
  \end{equation*}
  while, for all $\alpha:i' \in a_m$,~\eqref{eq:sumwg} yields
  \begin{equation*}
    1 = \sum_{g' \in \Gamma^-_{m-1}(\alpha:i')} w^-_{m-1}(g') \leq \sum_{g' \in \Gamma^-_{m-1}(\alpha:i')} w^-_{m-1}(g') \left(1+\frac{\Ratio}{n}\right)^{|b_m| + |H_{m-1}(g')|}
  \end{equation*}
  so that
  \begin{equation*}
    1 \leq \sum_{\alpha:i' \in a_m} \sum_{g' \in \Gamma^-_{m-1}(\alpha:i')}\left(1+\frac{\Ratio}{n}\right)^{|b_m| + |H_{m-1}(g')|} \frac{1}{|a_m|}w^-_{m-1}(g')  
  \end{equation*}
  and therefore
  \begin{equation*}
    \begin{aligned}
      & \sum_{\alpha:i' \in a_m} \sum_{g' \in \Gamma^-_{m-1}(\alpha:i')} \left(1+\frac{\Ratio}{n}\right)^{|b_m| + |H_{m-1}(g')|} \frac{1}{|a_m|}w^-_{m-1}(g')\frac{\Ratio}{n} \sum_{(m',\gamma':k') \in H_{m-1}(g')} e_{m'-1}(\gamma':k')\\
      & \quad + \frac{\Ratio}{n} \sum_{\beta:j \in b_m} e_{m-1}(\beta:j)\\
      & \leq \sum_{g \in \Gamma^-_m(\alpha:i)} \left(1+\frac{\Ratio}{n}\right)^{|H_m(g)|} w^-_m(g) \frac{\Ratio}{n} \sum_{(m',\gamma':k') \in H_m(g)} e_{m'-1}(\gamma':k'),
    \end{aligned}
  \end{equation*}
  which completes the proof.
\end{proof}

\begin{lem}[The $\Ls^{\infty}-\Ls^1$ estimate]\label{lem:L1Linf}
  Under the assumptions of Lemma~\ref{lem:totmass}, we have the $\Ls^{\infty}-\Ls^1$ estimate: for all $m \in \{0, \ldots, M\}$, for all $\gamma:k \in \Part$,
  \begin{equation}\label{eq:Emc2}
    e_m(\gamma:k) \leq \exp(\Ratio(d-1))\left\{\sum_{k'=1}^n e_0(\gamma:k')\sum_{g \in \Gamma^-_m(\gamma:k)} \ind{F(g)=\gamma:k'}w^-_m(g) + \frac{\Ratio}{n}\totmass_m\right\}.
  \end{equation}
\end{lem}
\begin{proof}
  Let us note that, for all $m \in \{0, \ldots, M\}$, for all $\gamma:k \in \Part$, the points~\eqref{it:propHmg:1} and~\eqref{it:propHmg:2} of Lemma~\ref{lem:propHmg} yield, for all $g \in \Gamma^-_m(\gamma:k)$,
  \begin{equation*}
    |H_m(g)| \leq n(d-1)
  \end{equation*}
  and therefore
  \begin{equation*}
    \left(1+\frac{\Ratio}{n}\right)^{|H_m(g)|} \leq \exp(\Ratio(d-1)).
  \end{equation*}
  Furthermore, we rewrite
  \begin{equation*}
    \sum_{g \in \Gamma^-_m(\gamma:k)} w^-_m(g) e_0(F(g)) = \sum_{k'=1}^n e_0(\gamma:k') \sum_{g \in \Gamma^-_m(\gamma:k)} \ind{F(g)=\gamma:k'}w^-_m(g).
  \end{equation*}
  We shall now prove that, for all $g \in \Gamma^-_m(\gamma:k)$,
  \begin{equation}\label{eq:sumHm}
    \sum_{(m',\gamma':k') \in H_m(g)} e_{m'-1}(\gamma':k') \leq \totmass_m,
  \end{equation}
  which leads to the expected $\Ls^{\infty}-\Ls^1$ estimate~\eqref{eq:Emc2} when combined with~\eqref{eq:Emc}.
  
  Let us fix $m \in \{0, \ldots, M\}$, $\gamma:k \in \Part$ and $g \in \Gamma^-_m(\gamma:k)$. We first prove by induction on $\bar{m} \in \{0, \ldots, m\}$ that
  \begin{equation}\label{eq:pftotmass0}
    \sum_{\substack{(m',\gamma':k') \in H_m(g)\\ m' \leq \bar{m}}} e_{m'-1}(\gamma':k') \leq \sum_{\substack{(m',\gamma':k') \in H_m(g)\\ m' \leq \bar{m}}} e_{\bar{m}}(\gamma':k').
  \end{equation}
  The case $\bar{m}=0$ follows from the convention that $e_{-1}(\gamma':k')=0$, see Lemma~\ref{lem:keyestim}. Now let $\bar{m} \in \{1, \ldots, M\}$ such that the inequality above holds true for $\bar{m}-1$. Then
  \begin{equation*}
    \begin{aligned}
      \sum_{\substack{(m',\gamma':k') \in H_m(g)\\ m' \leq \bar{m}}} e_{m'-1}(\gamma':k') & = \sum_{\substack{(m',\gamma':k') \in H_m(g)\\ m' \leq \bar{m}-1}} e_{m'-1}(\gamma':k') + \sum_{\substack{(m',\gamma':k') \in H_m(g)\\ m' = \bar{m}}} e_{m'-1}(\gamma':k')\\
      & \leq \sum_{\substack{(m',\gamma':k') \in H_m(g)\\ m' \leq \bar{m}}} e_{\bar{m}-1}(\gamma':k'),
    \end{aligned}
  \end{equation*}
  and we just have to check that
  \begin{equation}\label{eq:pftotmass1}
    \sum_{\substack{(m',\gamma':k') \in H_m(g)\\ m' \leq \bar{m}}} e_{\bar{m}-1}(\gamma':k') \leq \sum_{\substack{(m',\gamma':k') \in H_m(g)\\ m' \leq \bar{m}}} e_{\bar{m}}(\gamma':k').
  \end{equation}
  To this aim, we note that, for all $\gamma':k' \in \Part$:
  \begin{itemize}
    \item either $\gamma':k' \not\in a_{\bar{m}} \cup b_{\bar{m}}$, in which case $e_{\bar{m}-1}(\gamma':k') = e_{\bar{m}}(\gamma':k')$,
    \item or there exists $c \in \{a_{\bar{m}}, b_{\bar{m}}\}$ such that $\gamma':k' \in c$, in which case the point~\eqref{it:propHmg:3} of Lemma~\ref{lem:propHmg} ensures that all the quantities $e_{\bar{m}-1}(\gamma':k'')$, for $\gamma':k'' \in c$, appear in the sum at the left-hand side of the inequality~\eqref{eq:pftotmass1}. But by~\eqref{eq:EmEmm1:new},
    \begin{equation*}
      \sum_{\gamma':k'' \in c} e_{\bar{m}-1}(\gamma':k'') \leq \sum_{\gamma':k'' \in c} e_{\bar{m}}(\gamma':k'').
    \end{equation*}
  \end{itemize}
  The inequality~\eqref{eq:pftotmass1} follows immediately, and the proof of~\eqref{eq:pftotmass0} is completed. Applying the latter inequality with $\bar{m}=m$ and using the point~\eqref{it:propHmg:2} of Lemma~\ref{lem:propHmg}, we conclude that
  \begin{equation*}
    \sum_{(m',\gamma':k') \in H_m(g)} e_{m'-1}(\gamma':k') \leq \sum_{(m',\gamma':k') \in H_m(g)} e_{m}(\gamma':k') \leq \totmass_m,
  \end{equation*}
  and thereby obtain~\eqref{eq:sumHm}.
\end{proof}

We are now ready to complete the proof of Lemma~\ref{lem:totmass}. We first address the $\Ls^1$ estimate.
  
\begin{proof}[Derivation of the $\Ls^1$ estimate in Lemma~\ref{lem:totmass}] 
  We use our $\Ls^{\infty}-\Ls^1$ estimate~\eqref{eq:Emc2} to obtain a bound on $\totmass_M$. By the definition of the auxiliary system, for all $m \in \{1, \ldots, M\}$,
  \begin{equation*}
    \begin{aligned}
      \totmass_m & = \totmass_{m-1} + \frac{2\Ratio}{n}\left(|b_m| \sum_{\alpha:i \in a_m} e_{m-1}(\alpha:i) + |a_m| \sum_{\beta:j \in b_m} e_{m-1}(\beta:j)\right)\\
      & \leq \left(1 + \frac{4\Ratio^2}{n^2} |a_m||b_m| \exp(\Ratio(d-1))\right) \totmass_{m-1}\\
      & \quad + \frac{2\Ratio}{n}\exp(\Ratio(d-1)) |b_m| \sum_{i'=1}^n e_0(\alpha:i') \sum_{\alpha:i \in a_m} \sum_{g \in \Gamma^-_{m-1}(\alpha:i)} \ind{F(g)=\alpha:i'} w^-_{m-1}(g)\\
      & \quad + \frac{2\Ratio}{n}\exp(\Ratio(d-1)) |a_m| \sum_{j'=1}^n e_0(\beta:j') \sum_{\beta:j \in b_m} \sum_{g \in \Gamma^-_{m-1}(\beta:j)} \ind{F(g)=\beta:j'} w^-_{m-1}(g),
    \end{aligned}
  \end{equation*}
  where we have used~\eqref{eq:Emc2} for the inequality. Using the elementary inequality~\eqref{eq:elementary} again, we obtain
  \begin{equation*}
    \totmass_M \leq \left(1 + \frac{4\Ratio^2}{n^2}\exp(\Ratio(d-1))\right)^{\sum_{m=1}^M |a_m||b_m|} \left\{\totmass_0 + \frac{2\Ratio}{n}\exp(\Ratio(d-1)) (A_M+B_M)\right\},
  \end{equation*}
  where
  \begin{equation*}
    \begin{aligned}
      A_M & := \sum_{m=1}^M |b_m| \sum_{i'=1}^n e_0(\alpha:i') \sum_{\alpha:i \in a_m} \sum_{g \in \Gamma^-_{m-1}(\alpha:i)} \ind{F(g)=\alpha:i'} w^-_{m-1}(g),\\
      B_M & := \sum_{m=1}^M |a_m| \sum_{j'=1}^n e_0(\beta:j') \sum_{\beta:j \in b_m} \sum_{g \in \Gamma^-_{m-1}(\beta:j)} \ind{F(g)=\beta:j'} w^-_{m-1}(g).
    \end{aligned}
  \end{equation*}

  For all $m \in \{1, \ldots, M\}$, $a_m \times b_m$ is a subset of $R$ with cardinality $|a_m||b_m|$, and for $m' < m$, the subsets $a_{m'} \times b_{m'} = \classe_{m'}$ and $a_m \times b_m = \classe_m$ are disjoint. As a consequence, for all $m \in \{1, \ldots, M\}$,
  \begin{equation*}
    \sum_{m=1}^M |a_m||b_m| \leq |R| \leq |\{(\alpha:i, \beta:j) \in (\Part)^2: \alpha < \beta\}| = n^2 \frac{d(d-1)}{2},
  \end{equation*}
  therefore
  \begin{equation*}
    \left(1+\frac{4\Ratio^2}{n^2}\exp\left(\Ratio(d-1)\right)\right)^{\sum_{m=1}^M|a_m||b_m|} \leq \exp\left(2\Ratio^2d(d-1)\exp\left(\Ratio(d-1)\right)\right).
  \end{equation*}
  It now remains to obtain estimates on the quantities $A_M$ and $B_M$. To this aim, we rewrite
  \begin{equation*}
    A_M = \sum_{\alpha:i' \in \Part} e_0(\alpha:i') I_{\alpha:i'},
  \end{equation*}
  where, for all $\alpha:i' \in \Part$,
  \begin{equation*}
    I_{\alpha:i'} := \sum_{m=1}^M |b_m| \sum_{\alpha:i \in a_m} \sum_{g \in \Gamma^-_{m-1}(\alpha:i)} \ind{F(g)=\alpha:i'} w^-_{m-1}(g).
  \end{equation*}
  
  Let us fix $\alpha:i' \in \Part$ and obtain an estimate on $I_{\alpha:i'}$. We first note that, for all $m \in \{1, \ldots, M\}$, for all $\alpha:\bar{i} \in a_m$, the mapping $g \mapsto g::(\alpha:\bar{i})$ establishes a one-to-one correspondance between the sets
  \begin{equation*}
    \bigsqcup_{\alpha:i \in a_m} \{g \in \Gamma^-_{m-1}(\alpha:i) : F(g) = \alpha:i'\}
  \end{equation*}
  and
  \begin{equation*}
    \{g \in \Gamma^-_m(\alpha:\bar{i}) : F(g) = \alpha:i'\},
  \end{equation*}
  and that, in addition, for all $g$ in the first set above,
  \begin{equation*}
    w^-_m(g::(\alpha:\bar{i})) = \frac{1}{|a_m|} w^-_{m-1}(g),
  \end{equation*}
  so that
  \begin{equation*}
    \begin{aligned}
      \sum_{\alpha:i \in a_m} \sum_{g \in \Gamma^-_{m-1}(\alpha:i)} \ind{F(g)=\alpha:i'} w^-_{m-1}(g) & =  \sum_{g \in \Gamma^-_m(\alpha:\bar{i})} \ind{F(g)=\alpha:i'} |a_m|w^-_m(g)\\
      & = \sum_{\alpha:\bar{i} \in a_m} \sum_{g \in \Gamma^-_m(\alpha:\bar{i})} \ind{F(g)=\alpha:i'} w^-_m(g).
    \end{aligned}
  \end{equation*}
  As a consequence, $I_{\alpha:i'}$ rewrites
  \begin{equation*}
    I_{\alpha:i'} = \sum_{m=1}^M |b_m| \sum_{\alpha:i \in a_m} \sum_{g \in \Gamma^-_m(\alpha:i)} \ind{F(g)=\alpha:i'} w^-_m(g).
  \end{equation*}
  
  We now define the set $\bar{\Gamma}(\alpha:i')$ by
  \begin{equation*}
    \bar{\Gamma}(\alpha:i') := \bigsqcup_{i=1}^n \{\bar{g} \in \Gamma^-_M(\alpha:i) : F(\bar{g})=\alpha:i'\}.
  \end{equation*}
  A type path $\bar{g} = (\alpha:i_0, \ldots, \alpha:i_{\bar{L}}) \in \bar{\Gamma}(\alpha:i')$ is associated with a sequence of collisions $(\classe_{m_1}, \ldots, \classe_{m_{\bar{L}}})$ having the property that $\classe_{m_{\bar{L}}}$ is the last element of $C_{\alpha:i_{\bar{L}}}$. The total weight $\bar{w}(\bar{g}) := w^-_M(\bar{g})$ of the type path $\bar{g}$ has the following interpretation: start from the particle $\alpha:i'$ and move to the first collision $\classe_{m_1}$ in $C_{\alpha:i'}$ if it exists. This motion is forward with respect to the orientation of the collision graph. Now select a particle uniformly at random among the particles of type $\alpha$ involved in the collision $\classe_{m_1}$, and repeat the motion forward and random selection as long as possible. Then $\bar{w}(\bar{g})$ is the probability of selecting the type path $\bar{g}$; therefore,
  \begin{equation}\label{eq:sumbwg}
    \sum_{\bar{g} \in \bar{\Gamma}(\alpha:i')} \bar{w}(\bar{g}) = 1.
  \end{equation}
  Besides, we have the identity, for all $m \in \{1, \ldots, M\}$, for all $\alpha:i \in a_m$,
  \begin{equation*}
    \sum_{g \in \Gamma^-_m(\alpha:i)} \ind{F(g)=\alpha:i'} w^-_m(g) = \frac{1}{|a_m|}\sum_{\bar{g} \in \bar{\Gamma}(\alpha:i')} \ind{\bar{g} \in \mathcal{A}_m} \bar{w}(\bar{g})
  \end{equation*}
  where $\mathcal{A}_m$ is the set of type paths $\bar{g} = (\alpha:i_0, \ldots, \alpha:i_{\bar{L}})$, associated with the sequence of collisions $(\classe_{m_1}, \ldots, \classe_{m_{\bar{L}}})$, such that there exists $L \in \{1, \ldots, \bar{L}\}$ such that $m = m_L$ and $\alpha:i_L \in a_m$. We deduce that
  \begin{equation*}
    I_{\alpha:i'} = \sum_{\bar{g} \in \bar{\Gamma}(\alpha:i')} \bar{w}(\bar{g}) \sum_{m=1}^M |b_m| \sum_{\alpha:i \in a_m}  \ind{\bar{g} \in \mathcal{A}_m} \leq \sum_{\bar{g} \in \bar{\Gamma}(\alpha:i')} \bar{w}(\bar{g}) |H_M(\bar{g})| \leq n(d-1),
  \end{equation*}
  where we have used~\eqref{it:propHmg:2} of Lemma~\ref{lem:propHmg} as well as~\eqref{eq:sumbwg} in the last inequality. We conclude that
  \begin{equation*}
    A_M \leq n(d-1) \sum_{\alpha:i' \in \Part} e_0(\alpha:i') = n(d-1)\totmass_0
  \end{equation*}
  and, similarly, 
  \begin{equation*}
    B_M \leq n(d-1) \totmass_0.
  \end{equation*}
  As a consequence,
  \begin{equation*}
    \totmass_M \leq \totmass_0 \left(1 + 4\Ratio(d-1)\exp(\Ratio(d-1))\right)\exp\left(2\Ratio^2d(d-1)\exp\left(\Ratio(d-1)\right)\right),
  \end{equation*}
  which is the $\Ls^1$ estimate $\totmass_M \leq \ConstStab_1 \totmass_0$ where $\ConstStab_1$ is given by~\eqref{eq:ConstStab}.
\end{proof}

\begin{proof}[Derivation of the $\Ls^{\infty}$ estimate in Lemma~\ref{lem:totmass}] 
  Injecting the $\Ls^1$ estimate above into~\eqref{eq:Emc2}, we obtain, for all $m \in \{0, \ldots, M\}$, for all $\gamma:k \in \Part$,
  \begin{equation*}
    \begin{aligned}
      e_m(\gamma:k) & \leq \exp(\Ratio(d-1))\left\{\sum_{k'=1}^n e_0(\gamma:k') \sum_{g \in \Gamma^-_m(\gamma:k)} \ind{F(g)=\gamma:k'} w^-_m(g) + \frac{\Ratio}{n} \ConstStab_1\totmass_0\right\}\\
      & \leq \exp(\Ratio(d-1)) \left\{\sum_{g \in \Gamma^-_m(\gamma:k)} w^-_m(g)+d\Ratio \ConstStab_1\right\} \sup_{\gamma':k' \in \Part} e_0(\gamma':k')\\
      & = \exp(\Ratio(d-1)) \left\{1+d\Ratio \ConstStab_1\right\} \sup_{\gamma':k' \in \Part} e_0(\gamma':k'),
    \end{aligned}
  \end{equation*}
  thanks to~\eqref{eq:sumwg}, whence the $\Ls^{\infty}$ estimate 
  \begin{equation*}
    \sup_{0 \leq m \leq M} \sup_{\gamma:k \in \Part} e_m(\gamma:k) \leq \ConstStab_{\infty} \sup_{\gamma:k \in \Part} e_0(\gamma:k)
  \end{equation*}
  with $\ConstStab_{\infty}$ given by~\eqref{eq:ConstStab}.
\end{proof}


\subsection{From local to global stability estimates}\label{ss:interpolation} In this subsection, we explain how to remove Condition~{\rm (\hyperref[cond:C]{LHM})} from Proposition~\ref{prop:locstab}; namely, we prove the following result.

\begin{prop}[Global stability estimate]\label{prop:globstab}
  Under Assumptions~\eqref{ass:LC} and \eqref{ass:USH}, for all $\x, \y \in \Dnd$,
  \begin{equation*}
    \begin{aligned}
      & \sup_{t \geq 0} ||\Phi(\x;t)-\Phi(\y;t)||_1 \leq \ConstStab_1 ||\x-\y||_1,\\
      & \sup_{t \geq 0} ||\Phi(\x;t)-\Phi(\y;t)||_{\infty} \leq \ConstStab_{\infty} ||\x-\y||_{\infty},
    \end{aligned}
  \end{equation*}
  where $\ConstStab_1$ and $\ConstStab_{\infty}$ are given in Proposition~\ref{prop:locstab}.
\end{prop}

The subsection is organised as follows. Proposition~\ref{prop:globstab} is derived from the local stability estimates of Proposition~\ref{prop:locstab} by integrating the latter along a continuous path joining arbitrary initial configurations, that can be decomposed into small portions on which Proposition~\ref{prop:locstab} applies. Geometrical tools allowing the construction of such a path are introduced in~\S\ref{sss:radial}, and the global interpolation procedure is described in~\S\ref{sss:interpolation}. The whole argument relies on the nondegeneracy condition~\eqref{cond:ND} introduced in~\S\ref{sss:ND}, and an approximation procedure of degenerate characteristic fields by nondegenerate ones is detailed in~\S\ref{sss:approxND}.


\subsubsection{The nondegeneracy condition}\label{sss:ND} Let us introduce the following {\em nondegeneracy} condition on the functions $\lambda^1, \ldots, \lambda^d$.
\begin{enumerate}[label=(ND), ref=ND]
  \item\label{cond:ND} For all $\x \in \Dnd$, for all $\gamma \in \{1, \ldots, d\}$, for all $\uk < \ok$ in $\{1, \ldots, n\}$ such that
  \begin{equation*}
    \forall \gamma' \not= \gamma, \quad \forall k \in \{\uk, \ldots, \ok\}, \qquad \omega_{\gamma:k}^{\gamma'}(\x) = \omega_{\gamma:\uk}^{\gamma'}(\x),
  \end{equation*}
  we have
  \begin{equation*}
    \forall k \in \{\uk, \ldots, \ok-1\}, \qquad \frac{1}{k-\uk+1} \sum_{k'=\uk}^k \tlambda_{k'}^{\gamma}(\x) \not= \frac{1}{\ok-k} \sum_{k'=k+1}^{\ok} \tlambda_{k'}^{\gamma}(\x).
  \end{equation*}
\end{enumerate}

This condition expresses the fact that two clusters of the same type with no particle between them cannot have the same velocity. Note that the condition is written for a fixed value of $n$ and therefore only depends on the finite number of values of $\tlambda_k^{\gamma}(\x)$, $\x \in \Dnd$ and $\gamma:k \in \Part$. We will use the following consequence of Condition~\eqref{cond:ND}.

\begin{lem}[Continuity of the composition of clusters]\label{lem:sGNLclu}
  Under Assumptions~\eqref{ass:C} and~\eqref{ass:USH}, and Condition~\eqref{cond:ND}, for all $\x \in \Dnd$, for all $t \in (0, t^*(\x))$ such that 
  \begin{equation*}
    \forall \gamma:k \in \Part, \qquad \clu_k^{\gamma}(\x;t^-) = \clu_k^{\gamma}(\x;t),
  \end{equation*}
  there exists $\eta > 0$ such that, for all $\y \in \barB_1(\x,\eta)$,
  \begin{equation*}
    \forall \gamma:k \in \Part, \qquad \clu_k^{\gamma}(\y;t) = \clu_k^{\gamma}(\x;t).
  \end{equation*}  
\end{lem}
\begin{proof}
  Let $\x \in \Dnd$ and $t \in (0, t^*(\x))$ satisfying the properties above. Let us first fix $t' \in (0,t)$ such that, for all $s \in [t',t]$, for all $\gamma:k \in \Part$, $\clu_k^{\gamma}(\x;s)=\clu_k^{\gamma}(\x;t)$; in other words, there is no self-interaction in the MSPD started at $\x$ on the time interval $[t',t]$. We shall denote $\x' := \Phi(\x;t')$.
  
  Let us fix $\delta > 0$ small enough to ensure that, for all $\gamma:k$ and $\gamma':k'$ such that $\clu_k^{\gamma}(\x;t) \not= \clu_{k'}^{\gamma'}(\x;t)$, 
  \begin{equation*}
    \forall s \in [t',t], \qquad [\Phi_k^{\gamma}(\x;s)-\delta,\Phi_k^{\gamma}(\x;s)+\delta] \cap [\Phi_{k'}^{\gamma'}(\x;s)-\delta,\Phi_{k'}^{\gamma'}(\x;s)+\delta] = \emptyset;
  \end{equation*}
  
  On the other hand, by Lemma~\ref{lem:Drondnd}, one can choose $\eta'$ small enough to ensure that, for all $\y' \in \barB_1(\x',\eta')$, then $\y' \in \Drnd$, $\Rb(\y') = \Rb(\x')$ and $t^*(\y') > t'-t$. By Lemma~\ref{lem:contracttPhi} combined with the flow property of Proposition~\ref{prop:mspd}, these conditions imply that, for all $\y' \in \barB_1(\x',\eta')$,
  \begin{equation*}
    \forall s \in [t',t], \qquad ||\Phi(\y';s-t')-\Phi(\x;s)||_1 \leq ||\y' - \x'||_1 \leq \eta'.
  \end{equation*}
  We now want to fix $\eta'$ small enough to satisfy the conditions above, and such that, for all $\gamma:k \in \Part$, if $\y' \in \barB_1(\x',\eta')$, then $\clu_k^{\gamma}(\y',t-t')=\clu_k^{\gamma}(\x;t)$. 
  
  We first require that $\eta' \leq \delta/n$, so that if $\y' \in \barB_1(\x',\eta')$, then for all $\gamma:k \in \Part$, for all $s \in [t',t]$, 
  \begin{equation*}
    |\Phi_k^{\gamma}(\y';s-t')-\Phi_k^{\gamma}(\x;s)| \leq n||\Phi(\y';s-t')-\Phi(\x;s)||_1 \leq \delta,
  \end{equation*}
  and therefore $\Phi_k^{\gamma}(\y';s-t')=\Phi_{k'}^{\gamma'}(\y';s-t')$ only if $\clu_k^{\gamma}(\x;t)=\clu_{k'}^{\gamma'}(\x;t)$.
  
  Let us now fix $\gamma:k \in \Part$. If $\clu_k^{\gamma}(\x;t) = \gamma:k$, then for all $\y' \in \barB_1(\x',\eta')$, the assertion above implies that $\clu_k^{\gamma}(\y;s-t') = \gamma:k$ for all $s \in [t',t]$. On the contrary, if $\clu_k^{\gamma}(\x;t) = \gamma:\uk\cdots\ok$ with $\uk < \ok$, then the stability condition~\eqref{eq:stab} combined with Condition~\eqref{cond:ND} yield, for all $\uk \leq k < \ok$,
  \begin{equation}\label{eq:pf:sGNLclu:1}
    \frac{1}{k-\uk+1} \sum_{k'=\uk}^k \tlambda_{k'}^{\gamma}(\x') > \frac{1}{\uk-k} \sum_{k'=k+1}^{\ok} \tlambda_{k'}^{\gamma}(\x'),
  \end{equation}
  and the same inequality holds if one replaces $\x'$ with $\y'$ since $\Rb(\x') = \Rb(\y')$. By the same arguments as above, we have $y'^{\gamma}_{\ok} - y'^{\gamma}_{\uk} \leq 2n \eta' \leq 2 \delta$. Let us write, for all $s \in [t',t]$,
  \begin{equation*}
    \begin{aligned}
      & \Phi_{\uk}^{\gamma}(\y'; s-t') = y'^{\gamma}_{\uk} + \int_{r=0}^{s-t'} v_{\uk}^{\gamma}(\y';r)\dd r,\\
      & \Phi_{\ok}^{\gamma}(\y'; s-t') = y'^{\gamma}_{\ok} + \int_{r=0}^{s-t'} v_{\ok}^{\gamma}(\y';r)\dd r.
    \end{aligned}
  \end{equation*}
  Let us fix $s \in [t',t]$. If $\Phi_{\uk}^{\gamma}(\y'; s-t') = \Phi_{\ok}^{\gamma}(\y'; s-t')$, then $\clu_k^{\gamma}(\y';s-t') = \gamma:\uk\cdots\ok$ and this remains the case up to time $t-t'$. Otherwise, we have, for all $r \in [0,s-t']$, $\Phi_{\uk}^{\gamma}(\y';r) < \Phi_{\ok}^{\gamma}(\y';r)$ and therefore
  \begin{equation*}
    \clu_{\uk}^{\gamma}(\y';r) = \gamma:\uk\cdots\uk', \qquad \clu_{\ok}^{\gamma}(\y';r) = \gamma:\ok'\cdots\ok,
  \end{equation*}
  for some $\uk \leq \uk' < \ok' \leq \ok$. Arguing as in the proof of Lemma~\ref{lem:extstab}, but where the stability condition~\eqref{eq:stab} is replaced with the stronger condition~\eqref{eq:pf:sGNLclu:1}, we get $v_{\uk}^{\gamma}(\y';r) > v_{\ok}^{\gamma}(\y';r)$. Since $\uk'$ and $\ok'$ can only take a finite number of values, we deduce that there exists $\theta > 0$ such that, for all $r \in [0,s-t']$, $v_{\uk}^{\gamma}(\y';r) - v_{\ok}^{\gamma}(\y';r) \geq \theta$. As a consequence, if $\Phi_{\uk}^{\gamma}(\y'; s-t') < \Phi_{\ok}^{\gamma}(\y'; s-t')$ then we necessarily have
  \begin{equation*}
    \theta(s-t') \leq \int_{r=0}^{s-t'} \left(v_{\uk}^{\gamma}(\y';r) - v_{\ok}^{\gamma}(\y';r)\right)\dd r < y'^{\gamma}_{\ok} - y'^{\gamma}_{\uk} \leq 2n\eta'.
  \end{equation*}
  By contraposition, we deduce that if we choose $\eta' < \theta(t-t')/(2n)$, then the self-interaction between the particles $\gamma:\uk$ and $\gamma:\ok$ in the MSPD started at $\y'$ occurs before the time $t-t'$, which implies $\clu_k^{\gamma}(\y';t-t') = \gamma:\uk\cdots\ok$.
  
  Taking the minimum of such admissible $\eta'$ on all the particles $\gamma:k \in \Part$, we conclude that, for all $\y \in \Dnd$ such that $\Phi(\y;t') \in \barB_1(\x',\eta')$, we have $\clu_k^{\gamma}(\y;t) = \clu_k^{\gamma}(\x;t)$, for all $\gamma:k \in \Part$. By Proposition~\ref{prop:continuity}, there exists $\eta > 0$ such that, for all $\y \in \barB_1(\x,\eta)$, $\Phi(\y;t') \in \barB_1(\Phi(\x;t'),\eta')$; which completes the proof.
\end{proof}

The nondegeneracy condition~\eqref{cond:ND} implies that the set of good configurations is dense in $\Dnd$.

\begin{lem}[Density of $\Good$]\label{lem:gooddense}
  Under Assumptions~\eqref{ass:C} and~\eqref{ass:USH}, and Condition~\eqref{cond:ND}, the set $\Good$ is dense in $\Dnd$.
\end{lem}

The proof of Lemma~\ref{lem:gooddense} is postponed to Subsection~\ref{app:pf:gooddense} in Appendix~\ref{app:proofs}.


\subsubsection{Radial blow-up of singularities}\label{sss:radial} Given a configuration $\x \in \Dnd$ and a good configuration $\y$ in the neighbourhood of $\x$, we now want to construct a path joining $\x$ to $\y$ that can be decomposed into small portions on which Proposition~\ref{prop:locstab} can be applied. To this aim, we call {\em singularity} a space-time point at which a non binary collision, or both a collision and a self-interaction, occur in the MSPD started at $\x$. Note that a configuration $\y \in \Drnd$ is good if and only there is no singularity in the MSPD started at $\y$. Then we remark that, if $\y \in \Good$ is close enough to $\x$, singularities in the MSPD started at $\x$ are {\em radially blown up} in the MSPD started at $\y$, in the sense that if one shrinks the the trajectory of the MSPD started at $\y$ around the singularity, one obtains the trajectory of the MSPD started at $\x$.

In this paragraph, we first give a proper definition of the notion of {\em locally homothetic configurations} $\x$ and $\y$ corresponding to the description above, then we use the radial blow-up of singularities property to construct paths joining $\x$ to $\y$ with the expected properties. 

For all space-time points $\Xi = (\xi_0, \tau_0) \in \R \times (0,+\infty)$, for all $\dxi \in \R$, $\dtau \in (0,\tau_0)$, we shall denote by
\begin{equation*}
  \Xi^{\dxi, \dtau} := [\xi_0-\dxi, \xi_0+\dxi] \times [\tau_0-\dtau,\tau_0+\dtau] \subset \R \times (0,+\infty)
\end{equation*}
the {\em $(\dxi,\dtau)$-box} around $\Xi$. The open segments $(\xi_0-\dxi,\xi_0+\dxi)\times\{\tau_0-\dtau\}$ and $(\xi_0-\dxi,\xi_0+\dxi)\times\{\tau_0+\dtau\}$ shall be referred to as the {\em horizontal sides} of the box.

\begin{defi}[Proper covering of $\Ibinter(\x)$]
  Let $\x \in \Drnd$, with $\Nb(\x) \geq 1$. A {\em proper covering} of $\Ibinter(\x)$ is a pair $(\dxi,\dtau)$ such that:
  \begin{itemize}
    \item $\dxi > 0$, $\dtau \in (0, t^*(\x))$,
    \item for all $\Xi, \Xi' \in \Ibinter(\x)$ such that $\Xi \not= \Xi'$, then the intersection $\Xi^{\dxi, \dtau} \cap \Xi'^{\dxi, \dtau}$ of the $(\dxi,\dtau)$-boxes around $\Xi$ and $\Xi'$ is empty,
    \item for all $\Xi = (\xi_0, \tau_0) \in \Ibinter(\x)$,
    \begin{itemize}
      \item for all $\gamma:k \in \Part$ such that there exists $t \in [\tau_0-\dtau, \tau_0+\dtau]$ such that $\Phi_k^{\gamma}(\x;t) \in [\xi_0-\dxi, \xi_0+\dxi]$, then
      \begin{equation*}
        \Phi_k^{\gamma}(\x;\tau_0) = \xi_0,
      \end{equation*}
      \ie all the particles passing in the box are involved in the collision associated with $\Xi$,
      \item for all particles $\gamma:k$ in the box, 
      \begin{equation*}
        \Phi_k^{\gamma}(\x;\tau_0-\dtau) \in (\xi_0-\dxi, \xi_0+\dxi) \quad \text{and} \quad \Phi_k^{\gamma}(\x;\tau_0+\dtau) \in (\xi_0-\dxi, \xi_0+\dxi),
      \end{equation*}
      \ie the particle enters and exits the box by the horizontal side; besides,
      \begin{equation*}
        \forall t \in [\tau_0-\dtau, \tau_0), \qquad \clu_k^{\gamma}(\x;t) = \clu_k^{\gamma}(\x;(\tau_0-\dtau)^-)
      \end{equation*}
      and
      \begin{equation*}
        \forall t \in [\tau_0, \tau_0+\dtau], \qquad \clu_k^{\gamma}(\x;t) = \clu_k^{\gamma}(\x;\tau_0),
      \end{equation*}
      \ie self-interactions in the box can only occur at the space-time point $\Xi$.
    \end{itemize}
  \end{itemize}
\end{defi}

Given a proper covering $(\dxi,\dtau)$ of $\Ibinter(\x)$, the set of $(\dxi,\dtau)$-boxes around the points of $\Ibinter(\x)$ is drawn on Figure~\ref{fig:covering}. Examples of boxes around space-time points of collisions, with dimensions that do not define a proper covering, are shown on Figure~\ref{fig:badboxes}.

\begin{figure}[ht]
  \begin{pspicture}(10,5)
    \psline[linecolor=blue](0,0)(5,2)(6.5,3)(10,5)
    \psline[linecolor=blue](2,0)(5,2)(6.5,3)(10,5)
    
    \psline[linecolor=red](4.5,0)(5,1)(5,2)(4.5,4)(4.5,5)
    \psline[linecolor=red](5.5,0)(5,1)(5,2)(4.5,4)(4.25,5)
    
    \psline[linecolor=green](8.5,0)(5,2)(0,4.5)
    \psline[linecolor=green](9.5,0)(6.5,3)(4.5,4)(3.5,5)
    
    \pspolygon[linecolor=black](4.2,1.8)(5.8,1.8)(5.8,2.2)(4.2,2.2)
    \pspolygon[linecolor=black](5.7,2.8)(7.3,2.8)(7.3,3.2)(5.7,3.2)
    \pspolygon[linecolor=black](3.7,3.8)(5.3,3.8)(5.3,4.2)(3.7,4.2)
    
    \psline[linecolor=black]{<->}(5.5,3.8)(5.5,4.2)
    \rput(5.9,4){\textcolor{black}{$2\dtau$}}
    \psline[linecolor=black]{<->}(3.7,4.4)(5.3,4.4)
    \rput(5,4.7){\textcolor{black}{$2\dxi$}}
  \end{pspicture}
  \caption{An example of set of $(\dxi,\dtau)$-boxes around the points of $\Ibinter(\x)$.}
  \label{fig:covering}
\end{figure}
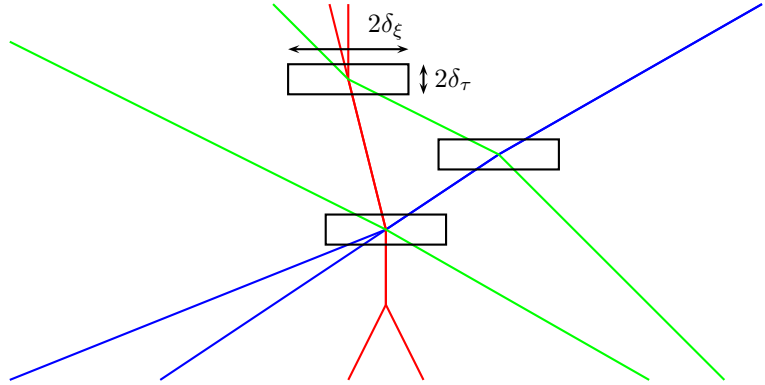
  
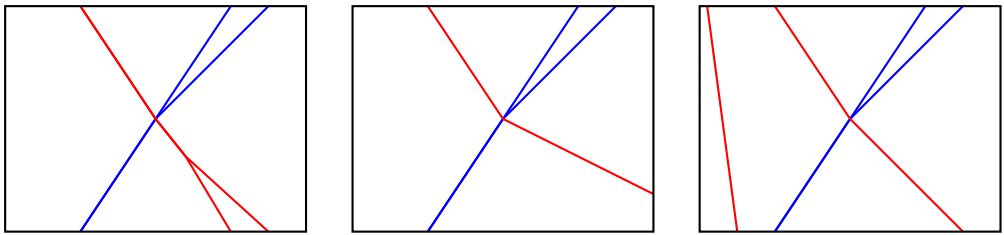
\begin{figure}[ht]
  \begin{pspicture}(4,3)
    \psline[linecolor=blue](1,0)(2,1.5)(3.5,3)
    \psline[linecolor=blue](1,0)(2,1.5)(3,3)
    
    \psline[linecolor=red](3,0)(2.4,1)(2,1.5)(1,3)
    \psline[linecolor=red](3.5,0)(2.4,1)(2,1.5)(1,3)
  
    \pspolygon[linecolor=black](0,0)(4,0)(4,3)(0,3)
  \end{pspicture}
  \hskip 5mm
  \begin{pspicture}(4,3)
    \psline[linecolor=blue](1,0)(2,1.5)(3.5,3)
    \psline[linecolor=blue](1,0)(2,1.5)(3,3)
    
    \psline[linecolor=red](4,.5)(2,1.5)(1,3)
  
    \pspolygon[linecolor=black](0,0)(4,0)(4,3)(0,3)
  \end{pspicture}
  \hskip 5mm
  \begin{pspicture}(4,3)
    \psline[linecolor=blue](1,0)(2,1.5)(3.5,3)
    \psline[linecolor=blue](1,0)(2,1.5)(3,3)
    
    \psline[linecolor=red](3.5,0)(2,1.5)(1,3)
    \psline[linecolor=red](.5,0)(.1,3)
 
    \pspolygon[linecolor=black](0,0)(4,0)(4,3)(0,3)
  \end{pspicture}
  \caption{The box on the left-hand figure contains a self-interaction at a distinct space-time point from the collision. On the central figure, a particle enters the box by a vertical side. The box on the right-hand figure is crossed by a particle that is not involved in the collision.}
  \label{fig:badboxes}
\end{figure}

Let us note that a proper covering of $\Ibinter(\x)$ always exists. Indeed, since the set $\Ibself(\x)$ is finite, one can construct $\dtau \in (0,t^*(\x))$ small enough to ensure that, for all $\Xi=(\xi_0,\tau_0) \in \Ibinter(\x)$, the particles involved in the collision associated with $\Xi$ do not have self-interactions on the time interval $[\tau_0-\dtau, \tau_0+\dtau]$ (except possibly at time $\tau_0$). Besides, since the velocities are bounded by $\ConstBound{\infty}$, given a choice of $\dtau$, any choice of $\dxi$ such that
\begin{equation}\label{eq:dimbox}
  \dxi > \dtau \ConstBound{\infty}
\end{equation}
ensures that particles enter and leave the box by the horizontal sides. Finally, one can shrink $\dtau$ and keep $\dxi$ satisfying~\eqref{eq:dimbox} accordingly to obtain boxes small enough for being disjoint and not being crossed by particles not involved in the corresponding collision.

We can now give a definition of locally homothetic configurations.

\begin{defi}[Locally homothetic configurations]
  Let $\x \in \Drnd$. A configuration $\y \in \Dnd$ is said to be {\em locally homothetic} to $\x$ if $\y \in \Drnd$ and either $\Nb(\x)=\Nb(\y)=0$, or $\Rb(\x)=\Rb(\y)$ and there exists a proper covering $(\dxi,\dtau)$ of $\Ibinter(\x)$ such that, for all $\Xi_0 = (\xi_0, \tau_0) \in \Ibinter(\x)$, 
  \begin{itemize}
    \item for all $\gamma:k \in \Part$ such that $\Phi_k^{\gamma}(\x;\tau_0) = \xi_0$,
    \begin{equation*}
      \begin{aligned}
        & \Phi_k^{\gamma}(\y;\tau_0-\dtau) \in (\xi_0-\dxi, \xi_0+\dxi), \qquad \clu_k^{\gamma}(\y;\tau_0-\dtau) = \clu_k^{\gamma}(\x;\tau_0-\dtau),\\
        & \Phi_k^{\gamma}(\y;\tau_0+\dtau) \in (\xi_0-\dxi, \xi_0+\dxi), \qquad \clu_k^{\gamma}(\y;\tau_0+\dtau) = \clu_k^{\gamma}(\x;\tau_0+\dtau),
      \end{aligned}
    \end{equation*}
    \item for all $(\alpha:i, \beta:j) \in \Rb(\x)$ such that $\Xiinter_{\alpha:i, \beta:j}(\x) = \Xi_0$, the space-time point of collision $\Xiinter_{\alpha:i, \beta:j}(\y)$ belongs to the $(\dxi,\dtau)$-box around $\Xi_0$, and for all $\rho \in [0,1]$,
    \begin{equation}\label{eq:lochom}
      \Xiinter_{\alpha:i, \beta:j}((1-\rho)\x + \rho\y) = (1-\rho)\Xi_0 + \rho \Xiinter_{\alpha:i, \beta:j}(\y),
    \end{equation}
    \item for all $\gamma \in \{1, \ldots, d\}$, for all $k, k' \in \{1, \ldots, n\}$ such that $\Xi_0 \in \Ibself_{\gamma:k, \gamma:k'}(\x)$, the intersection
    \begin{equation*}
      \Xi_0^{(\dxi,\dtau)} \cap \Ibself_{\gamma:k, \gamma:k'}(\y)
    \end{equation*}
    is either empty or contains a unique element $\Xi_{\gamma:k, \gamma:k'}(\y)$; in the latter case, for all $\rho \in (0,1]$, the intersection 
    \begin{equation*}
      \Xi_0^{(\dxi,\dtau)} \cap \Ibself_{\gamma:k, \gamma:k'}((1-\rho)\x+\rho\y)
    \end{equation*}
    contains a unique element $\Xi_{\gamma:k, \gamma:k'}((1-\rho)\x+\rho\y)$ and we have
    \begin{equation}\label{eq:lochom:self}
      \Xi_{\gamma:k, \gamma:k'}((1-\rho)\x+\rho\y) = (1-\rho)\Xi_0 + \rho\Xi_{\gamma:k, \gamma:k'}(\y).
    \end{equation}
  \end{itemize}
  We shall sometimes precise that $\y$ locally homothetic to $\x$ {\em with respect to the proper covering $(\dxi, \dtau)$}.
\end{defi}

Let us remark that if $\Nb(\x)=0$ then any configuration $\y \in \Drnd$ such that $\Nb(\y)=0$ is locally homothetic to $\x$.

\begin{lem}[Radial blow-up of singularities]\label{lem:radial}
  Under Assumptions~\eqref{ass:C} and~\eqref{ass:USH}, and Condition~\eqref{cond:ND}, let $\x \in \Drnd$.
  \begin{enumerate}[label=(\roman*), ref=\roman*]
    \item\label{it:radial:1} If $\Nb(\x)=0$, there exists $\kappa > 0$ such that, for all $\y \in B_1(\x, \kappa)$, $\y \in \Drnd$ and $\Nb(\y)=0$ so that $\y$ is locally homothetic to $\x$.
    \item\label{it:radial:2} If $\Nb(\x) \geq 1$, then for all proper coverings $(\dtau, \dxi)$ of $\Ibinter(\x)$, there exists $\kappa > 0$ such that, for all $\y \in B_1(\x, \kappa)$, $\y$ is locally homothetic to $\y$ with respect to $(\dtau, \dxi)$.
  \end{enumerate}
\end{lem}
\begin{proof}
  The point~\eqref{it:radial:1} is a straightforward consequence of~\eqref{it:Drondnd:1} in Lemma~\ref{lem:Drondnd}. 
  
  The proof of~\eqref{it:radial:2} works by induction on $\Nb(\x) \geq 1$. Let us fix $N \geq 0$ such that the lemma is satisfied for all $\x \in \Drnd$ such that $\Nb(\x) \leq N$. Let $\x \in \Drnd$ with $\Nb(\x) = N+1$; in particular, $t^*(\x) < +\infty$. Let $(\dxi, \dtau)$ be a proper covering of $\Ibinter(\x)$.
  
  Using Lemma~\ref{lem:Drondnd} again, we first obtain $\kappa_1 > 0$ such that, for all $\y \in B_1(\x, \kappa_1)$, $\y \in \Drnd$ and $\Rb(\x) = \Rb(\y)$.
  
  Without loss of generality, let us assume that $\dtau$ is small enough to satisfy
  \begin{equation*}
    t' := t^*(\x) + \dtau < t^*(\x)+t^*(\x^*) - \dtau,
  \end{equation*}
  and take $\dxi$ small enough to satisfy~\eqref{eq:dimbox}, so that $(\dxi, \dtau)$ remains a proper covering of $\Ibinter(\x)$. Then, on the time interval $[0,t^*(\x)+\dtau]$, the only collisions in the MSPD started at $\x$ occur at time $t^*(\x)$, possibly at different locations. Besides, $\Phi(\x;t') \in \Drnd$, $\Nb(\Phi(\x;t')) \leq N$, and if $\Nb(\Phi(\x;t')) \geq 1$, then $(\dxi, \dtau)$ remains a proper covering of $\Ibinter(\Phi(\x;t'))$. As a consequence, there exists $\kappa' > 0$ such that, for all $\y' \in \barB_1(\Phi(\x;t'), \kappa')$, then $\y'$ is locally homothetic to $\Phi(\x;t')$ (with respect to $(\dxi,\dtau)$ if $\Nb(\Phi(\x;t')) \geq 1$). By Proposition~\ref{prop:continuity}, there exists $\kappa_2>0$ such that, for all $\y \in B_1(\x, \kappa_2)$, $\Phi(\y; t') \in B_1(\Phi(\x;t'), \kappa')$.
  
  Combining Proposition~\ref{prop:continuity} and Lemma~\ref{lem:Drondnd}, we obtain $\kappa_3>0$ such that, for all $\y \in \barB_1(\x,\kappa_3)$, 
  \begin{itemize}
    \item $\Phi(\y;t^*(\x)-\dtau) \in \Drnd$ and $\Rb(\Phi(\y;t^*(\x)-\dtau)) = \Rb(\Phi(\x;t^*(\x)-\dtau))$,
    \item $\Phi(\y;t^*(\x)+\dtau) \in \Drnd$ and $\Rb(\Phi(\y;t^*(\x)+\dtau)) = \Rb(\Phi(\x;t^*(\x)+\dtau))$,
  \end{itemize}
  and, for all $\gamma:k \in \Part$,
  \begin{itemize}
    \item if the particle $\gamma:k$ is involved in a collision at the space-time point $(\xi_0,t^*(\x))$ in the MSPD started at $\x$, then for all $t \in [t^*(\x)-\dtau,t^*(\x)+\dtau]$, $\Phi_k^{\gamma}(\y;t) \in (\xi_0-\dxi,\xi_0+\dxi)$,
    \item if the particle $\gamma:k$ is not involved in a collision at time $t^*(\x)$ in the MSPD started at $\x$, then in the MSPD started at $\y$, the particle $\gamma:k$ does not cross any of the $(\dxi,\dtau)$-boxes around points of $\Ibinter(\x)$ on the time interval $[0,t']$.
  \end{itemize}
  These conditions ensure that, for all particles $\gamma:k$ involved in a collision at time $t^*(\x)$ in the MSPD started at $\x$, the corresponding particle enters and exits the $(\dxi, \dtau)$-box around $(\Phi_k^{\gamma}(\x;t^*(\x)), t^*(\x))$ by horizontal sides in the MSPD started at $\y$; besides, all the collision and self-interaction space-time points in which it is involved remain in the box.
  
  Combining Proposition~\ref{prop:continuity}, Lemma~\ref{lem:Drondnd} and Lemma~\ref{lem:sGNLclu}, we finally construct $\kappa_4 > 0$ such that, for all $\y \in \barB_1(\x,\kappa_4)$, for all $\gamma:k \in \Part$,
  \begin{equation*}
    \clu_k^{\gamma}(\y;t^*(\x)-\dtau) = \clu_k^{\gamma}(\x;t^*(\x)-\dtau), \qquad \clu_k^{\gamma}(\y;t^*(\x)+\dtau) = \clu_k^{\gamma}(\x;t^*(\x)+\dtau).
  \end{equation*}
  Note that, on account of Condition~\eqref{cond:ND}, Lemma~\ref{lem:sGNLclu} can be applied since the fact that $(\dxi,\dtau)$ is a proper covering of $\Ibinter(\x)$ implies that, on the time interval $(t^*(\x), t^*(\x)+\dtau]$, there is no self-interaction in the MSPD started at $\x$.
  
  We can now define $\kappa := \min\{\kappa_1, \ldots, \kappa_4\}$ and fix $\y \in B_1(\x,\kappa)$ and $\rho \in [0,1]$. To complete the proof, we have to check that the homothetic relations~\eqref{eq:lochom} and~\eqref{eq:lochom:self} are satisfied for all $\Xi_0 = (\xi_0, \tau_0) \in \Ibinter(\x)$. We address the cases $\tau_0=t^*(\x)$ and $\tau_0 > t^*(\x)$ separately, and shall proceed in three steps. In Step~1, we prove that
  \begin{equation*}
    \Phi((1-\rho)\x + \rho\y; t^*(\x)-\dtau) = (1-\rho)\Phi(\x;t^*(\x)-\dtau) + \rho \Phi(\y;t^*(\x)-\dtau).
  \end{equation*}
  In Step~2, we establish the homothetic relations~\eqref{eq:lochom} and~\eqref{eq:lochom:self} for $\tau_0=t^*(\x)$, and we check that
  \begin{equation}\label{eq:pf:radial:tprime}
    \Phi((1-\rho)\x+\rho\y;t^*(\x)+\dtau) = (1-\rho)\Phi(\x;t^*(\x)+\dtau) + \rho\Phi(\y;t^*(\x)+\dtau).
  \end{equation}
  Finally, we apply an inductive argument to address the case $\tau_0 > t^*(\x)$ in Step~3.  
  
  \sk
  \noindent {\em Step~1.} Since $t^*(\y) > t^*(\x)-\dtau$, then for all $t \in [0, t^*(\x)-\dtau]$, $\Phi(\x;t) = \tPhi[\tblambda(\x)](\x;t)$ and $\Phi(\y;t) = \tPhi[\tblambda(\y)](\y;t)$. Besides, $\Rb(\x) = \Rb(\y)$ so that $\tblambda(\x) = \tblambda(\y)$. Let $\gamma:k \in \Part$ and let us denote 
  \begin{equation*}
    c := \clu_k^{\gamma}(\x; t^*(\x)-\dtau) = \clu_k^{\gamma}(\y; t^*(\x)-\dtau).
  \end{equation*}
  Note that $||\x-((1-\rho)\x+\rho\y)||_1 = \rho||\x-\y||_1 \leq \kappa_4$, therefore $c = \clu_k^{\gamma}((1-\rho)\x+\rho\y;t^*(\x)-\dtau)$.
  
  Let us now remark that the processes $\{\Phi_k^{\gamma}(\x;t) : \gamma:k \in c \}$, $\{\Phi_k^{\gamma}(\y;t) : \gamma:k \in c \}$ and $\{\Phi_k^{\gamma}((1-\rho)\x+\rho\y;t) : \gamma:k \in c \}$ follow the Local Sticky Particle Dynamics on $[0, t^*(\x)-\dtau]$, with the same initial velocity vector. As a consequence, the centre of masses
  \begin{equation*}
    \frac{1}{|c|} \sum_{\gamma:k \in c} \Phi_k^{\gamma}(\x;t), \quad \frac{1}{|c|} \sum_{\gamma:k \in c} \Phi_k^{\gamma}(\y;t), \quad \frac{1}{|c|} \sum_{\gamma:k \in c} \Phi_k^{\gamma}((1-\rho)\x+\rho\y;t),
  \end{equation*}
  travel at the same constant velocity
  \begin{equation*}
    \frac{1}{|c|} \sum_{\gamma:k \in c} \tblambda_k^{\gamma}(\x)
  \end{equation*}
  on $[0, t^*(\x)-\dtau]$. Thus,
  \begin{equation*}
    \begin{aligned}
      & \frac{1}{|c|} \sum_{\gamma:k \in c} \Phi_k^{\gamma}((1-\rho)\x+\rho\y;t^*(\x)-\dtau) = \frac{1}{|c|} \sum_{\gamma:k \in c}\left( (1-\rho)x_k^{\gamma}+\rho y_k^{\gamma} + (t^*(\x)-\dtau) \tblambda_k^{\gamma}(\x)\right)\\
      & \qquad = (1-\rho)\frac{1}{|c|} \sum_{\gamma:k \in c} \Phi_k^{\gamma}(\x;t^*(\x)-\dtau) + \rho \frac{1}{|c|} \sum_{\gamma:k \in c} \Phi_k^{\gamma}(\y;t^*(\x)-\dtau),
    \end{aligned}
  \end{equation*}
  which of courses rewrites, for all $\gamma:k \in c$,
  \begin{equation*}
    \Phi_k^{\gamma}((1-\rho)\x+\rho\y;t^*(\x)-\dtau) = (1-\rho) \Phi_k^{\gamma}(\x;t^*(\x)-\dtau) + \rho \Phi_k^{\gamma}(\y;t^*(\x)-\dtau)
  \end{equation*}
  and completes Step~1.
  
  \sk
  \noindent {\em Step~2.} Let $\gamma:k \in \Part$. If the particle $\gamma:k$ does not collide with a particle of another type between times $t^*(\x)-\dtau$ and $t^*(\x)+\dtau =: t'$ in the MSPD started at $\x$ (or equivalently $\y$ or $(1-\rho)\x+\y$), then the same arguments as in Step~1 using the Local Sticky Particle Dynamics ensure that
  \begin{equation*}
    \Phi_k^{\gamma}((1-\rho)\x+\rho\y;t') = (1-\rho) \Phi_k^{\gamma}(\x;t') + \rho \Phi_k^{\gamma}(\y;t').
  \end{equation*}
  
  Otherwise, there exists a unique space-time point
  \begin{equation*}
    \Xi_0 \in \{\Xiinter_{\alpha:i, \beta:j}(\x) : (\alpha:i, \beta:j) \in \Rb(\x), \tinter_{\alpha:i,\beta:j}(\x) \in [t^*(\x)-\dtau, t^*(\x)+\dtau]\},
  \end{equation*}
  such that all the collisions with particles of another type and all the self-interactions of the particle $\gamma:k$ between times $t^*(\x)-\dtau$ and $t^*(\x)+\dtau$ in the MSPD started at $\x$ occur at the space-time point $\Xi_0$. By the definition of $\kappa$, the particle $\gamma:k$ collides with the same particles of another type and have the same self-interactions in the MSPD started at $\y$, and the corresponding space-time points of collisions and self-interactions belong to the $(\dxi, \dtau)$-box around $\Xi_0$; but of course, they can be distinct. Let us denote by $\Xi_{(1)}, \ldots, \Xi_{(L)}$ the sequence of these distinct space-time points of collisions and self-interactions, ranked by the increasing order of the times of collisions or self-interactions. For all $l \in \{1, \ldots, L\}$, we write $\Xi_{(l)} = (\xi_{(l)}, \tau_{(l)})$, so that 
  \begin{equation*}
    t^*(\x)-\dtau < \tau_{(1)} < \cdots < \tau_{(L)} < t^*(\x)+\dtau.
  \end{equation*}
  For all $l \in \{1, \ldots, L\}$, we finally denote by $S_{l,l+1}$ the space-time segment connecting $\Xi_{(l)}$ to $\Xi_{(l+1)}$, and let $S_{0,1}$ refer to the space-time segment connecting $(\Phi_k^{\gamma}(\y;t^*(\x)-\dtau), t^*(\x)-\dtau)$ to $\Xi_{(1)}$, and $S_{L,L+1}$ refer to the space-time segment connecting $\Xi_{(L)}$ to $(\Phi_k^{\gamma}(\y;t^*(\x)+\dtau), t^*(\x)+\dtau)$.
  
  We now define, for all $l \in \{1, \ldots, L\}$,
  \begin{equation*}
    \Xi'_{(l)} = (\xi'_{(l)}, \tau'_{(l)}) := (1-\rho)\Xi_0 + \rho\Xi_{(l)},
  \end{equation*}
  and similarly denote by $S'_{l,l+1}$ the space-time segment connecting $\Xi'_{(l)}$ to $\Xi'_{(l+1)}$ while $S'_{0,1}$ refers to the space-time segment connecting $((1-\rho)\Phi_k^{\gamma}(\x;t^*(\x)-\dtau)+\rho\Phi_k^{\gamma}(\y;t^*(\x)-\dtau), t^*(\x)-\dtau)$ to $\Xi'_{(1)}$ and $S'_{L,L+1}$ refers to the space-time segment connecting $\Xi'_{(L)}$ to $((1-\rho)\Phi_k^{\gamma}(\x;t^*(\x)+\dtau)+\rho\Phi_k^{\gamma}(\y;t^*(\x)+\dtau), t^*(\x)+\dtau)$.
  
  By the Intercept Theorem, if $\rho \in (0,1]$, then for all $l \in \{0, \ldots, L\}$, the segments $S_{l,l+1}$ and $S'_{l,l+1}$ are parallel. As a consequence, if $\rho \in (0,1]$, then the process $\Phi_k'^{\gamma}$ defined on $[t^*(\x)-\dtau, t^*(\x)+\dtau]$ by
  \begin{equation*}
    \forall l \in \{0, \ldots, L\}, \qquad S'_{l,l+1} = \{(\Phi_k'^{\gamma}(t), t), t \in [\tau'_{(l)}, \tau'_{(l+1)}]\}
  \end{equation*}
  (where $\tau'_{(0)} := t^*(\x)-\dtau$, $\tau'_{(L+1)} := t^*(\x)+\dtau$), has the same slope as the process $\Phi_k^{\gamma}(\y;\cdot)$ on each corresponding linear part, see Figure~\ref{fig:thales}. Besides, if two particles $\gamma:k$ and $\gamma:k'$ are in the same cluster on some linear part in the MSPD started at $\y$, then it is clear that the corresponding trajectories $\Phi_k'^{\gamma}$, $\Phi_{k'}'^{\gamma}$ coincide on the corresponding linear part.  
  
  \begin{figure}[ht]
    \begin{pspicture}(13,5)
      \psline[linecolor=black](0,0)(13,0)
      \psline[linecolor=black](0,4)(13,4)
      
      \rput(0,.5){\textcolor{black}{$t^*(\x)-\dtau$}}
      \rput(0,4.5){\textcolor{black}{$t^*(\x)+\dtau$}}
      
      \rput(1.5,2){$\Xi_0$}
      
      \psline[linecolor=blue](1,0)(2,2)(2.5,4)
      \psline[linecolor=blue](2,0)(2,2)(2.5,4)
      \psline[linecolor=red](6,0)(2,2)(0,4)
      \psline[linecolor=red](6,0)(2,2)(1,4)
      
      \psline[linecolor=blue](11.67,0)(11.67,.67)(11.67,3)(11.92,4)
      \psline[linecolor=blue](10,0)(10.5,1)(11,2)(11.67,3)(11.92,4)
      \psline[linecolor=red](13,0)(11.67,.67)(10.5,1)(7.5,4)
      \psline[linecolor=red](13,0)(11.67,.67)(11,2)(10,4)
      
      \psline[linecolor=black, linestyle=dotted](2,2)(11.67,.67)
      \psline[linecolor=black, linestyle=dotted](2,2)(11.67,3)
      \psline[linecolor=black, linestyle=dotted](2,2)(10.5,1)
      \psline[linecolor=black, linestyle=dotted](2,2)(11,2)
      
      \psline[linecolor=blue, linestyle=dashed](8.45,0)(8.45,1.11)(8.45,2.67)(8.78,4)
      \psline[linecolor=blue, linestyle=dashed](7,0)(7.67,1.33)(8.45,2.67)(8.78,4)
      \psline[linecolor=red, linestyle=dashed](10.67,0)(8.45,1.11)(7.67,1.33)(5,4)
      \psline[linecolor=red, linestyle=dashed](10.67,0)(8.45,1.11)(8.1,2)(7.1,4)
      
    \end{pspicture}
    \caption{The trajectory of the MSPD started at $\x$ is plotted on the left-hand side of the figure, while the trajectory of the MSPD started at $\y$ is plotted on the right-hand side. The trajectory of the process $\Phi'$ is plotted in dashed lines. Each linear part is parallel to the corresponding part in the trajectory of the MSPD started at $\y$. The black lines represent the horizontal sides of the box.}
    \label{fig:thales}
  \end{figure}
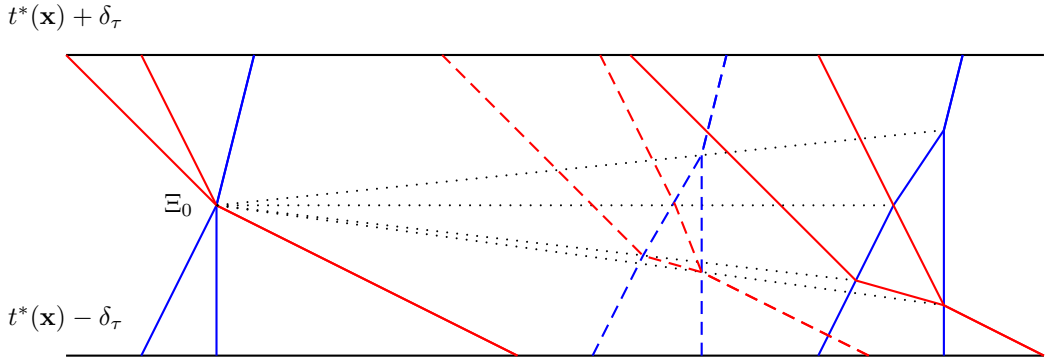
  
  As a conclusion, the processes $\Phi_k'^{\gamma}(t-(t^*(\x)-\dtau))$, $t \in [t^*(\x)-\dtau, t^*(\x)+\dtau]$, for all $\gamma:k$ such that 
  \begin{equation*}
    (\Phi_k^{\gamma}(\x; t^*(\x)), t^*(\x)) = \Xi_0,
  \end{equation*}
  exactly describe the motion of the particles in the MSPD started at $(1-\rho)\Phi(\x;t^*(\x)-\dtau) + \rho)\Phi(\y;t^*(\x)-\dtau)$. Thanks to Step~1, we conclude that
  \begin{equation*}
    \forall t \in [t^*(\x)-\dtau, t^*(\x)+\dtau], \qquad \Phi_k'^{\gamma}(t) = \Phi_k^{\gamma}((1-\rho)\x+\rho\y; t),
  \end{equation*}
  which yields~\eqref{eq:lochom}, \eqref{eq:lochom:self} for all the collision and self-interaction space-time points for the particle $\gamma:k$ on the time interval $[0,t']$; besides,
  \begin{equation*}
    \Phi_k^{\gamma}((1-\rho)\x+\rho\y;t') = \Phi_k'^{\gamma}(t') = (1-\rho)\Phi_k^{\gamma}(\x;t')+\rho\Phi_k^{\gamma}(\y;t').
  \end{equation*}
  This completes the proof of Step~2.
  
  \sk
  \noindent {\em Step~3.} Let $(\alpha:i, \beta:j) \in \Rb(\Phi(\y;t')) = \Rb(\Phi(\x;t'))$, so that
  \begin{equation*}
    \tinter_{\alpha:i, \beta:j}(\x), \tinter_{\alpha:i, \beta:j}(\y), \tinter_{\alpha:i, \beta:j}((1-\rho)\x+\rho\y) > t'.
  \end{equation*}
  Then, by the flow property of the MSPD,
  \begin{equation*}
    \begin{aligned}
      \xiinter_{\alpha:i, \beta:j}((1-\rho)\x + \rho\y) & = \xiinter_{\alpha:i, \beta:j}(\Phi((1-\rho)\x + \rho\y; t'))\\
      & = \xiinter_{\alpha:i, \beta:j}((1-\rho)\Phi(\x; t')+\rho\Phi(\y; t'))\\
      & = (1-\rho)\xiinter_{\alpha:i, \beta:j}(\Phi(\x; t')) + \rho\xiinter_{\alpha:i, \beta:j}(\Phi(\y; t')),
    \end{aligned}
  \end{equation*}
  where we used Step~2 at the second line and the fact that $\Phi(\y; t') \in B_1(\Phi(\x; t'), \kappa')$ at the third line. Using the flow property for the MSPD again, we conclude that the right-hand side above rewrites $(1-\rho)\xiinter_{\alpha:i, \beta:j}(\x) + \rho\xiinter_{\alpha:i, \beta:j}(\y)$. The very same arguments allow to address self-interactions as well, and also yield
  \begin{equation*}
    \begin{aligned}
      \tinter_{\alpha:i, \beta:j}((1-\rho)\x + \rho\y) & = \tinter_{\alpha:i, \beta:j}(\Phi((1-\rho)\x + \rho\y; t')) - t'\\
      & = (1-\rho)\left(\tinter_{\alpha:i, \beta:j}(\Phi(\x; t'))- t'\right)+ \rho\left(\tinter_{\alpha:i, \beta:j}(\Phi(\y; t'))- t'\right)\\
      & = (1-\rho)\tinter_{\alpha:i, \beta:j}(\x) + \rho\tinter_{\alpha:i, \beta:j}(\y),
    \end{aligned}
  \end{equation*}
  which completes the proof.
\end{proof}

We now explain how to construct a path joining a configuration $\x$ to a good configuration $\y$ close to $\x$, along which pairs of configurations satisfy the Local Homeomorphic Condition~{\rm (\hyperref[cond:C]{LHM})}. For the sake of understandability, we first describe the case $\x \in \Good$ in Lemma~\ref{lem:CondCloc:G} below. 
Then, the situation is actually very simple as, for $\y$ close enough to $\x$, the locally homothetic property implies that $\y \in \Good$ and $\x, \y$ satisfy Condition~{\rm (\hyperref[cond:C]{LHM})}. The case of an arbitrary configuration $\x \in \Drnd$ is addressed in Lemma~\ref{lem:CondCloc}.

\begin{lem}[Construction of locally homeomorphic configurations, good case]\label{lem:CondCloc:G}
  Under the assumptions of Lemma~\ref{lem:radial}, let $\x \in \Good$, and if $\Nb(\x) \geq 1$, let $(\dxi, \dtau)$ be a proper covering of $\Ibinter(\x)$. Let $\kappa > 0$ be given by Lemma~\ref{lem:radial}. For all $\y \in \barB_1(\x,\kappa)$, the configuration $\y$ belongs to the set $\Good$ and the configurations $\x$ and $\y$ satisfy Condition~{\rm (\hyperref[cond:C]{LHM})}.
\end{lem}
\begin{proof}
  If $\Nb(\x)=0$, then there is nothing to prove. Let us assume that $\Nb(\x) \geq 1$, let $(\dxi, \dtau)$ be a proper covering of $\Ibinter(\x)$ and let $\kappa > 0$ be given by Lemma~\ref{lem:radial}, so that $\y$ is locally homothetic to $\x$ with respect to $(\dxi, \dtau)$. In particular, $\Rb(\x)=\Rb(\y)$ and if $(\alpha:i, \beta:j), (\alpha':i', \beta':j') \in \Rb(\y)$ are such that
  \begin{equation*}
    \Xiinter_{\alpha:i, \beta:j}(\y) = \Xiinter_{\alpha':i', \beta':j'}(\y),
  \end{equation*}
  then it necessarily holds 
  \begin{equation*}
    \Xiinter_{\alpha:i, \beta:j}(\x) = \Xiinter_{\alpha':i', \beta':j'}(\x).
  \end{equation*}
  Since $\x \in \Good$, this implies that $\y \in \Good$. Besides, on account of the definitions of proper coverings and good configurations, in the MSPD started at $\x$, there is no self-interaction in the $(\dxi,\dtau)$-boxes around space-time points of collisions. Since the clusters at entry and exit of these boxes have the same composition in the MSPD started at $\y$, we deduce that self-interactions are separated from collisions in the MSPD started at $\y$. As a consequence, $\y \in \Good$.
  
  We have already checked that $\x$ and $\y$ satisfy Condition~\eqref{cond:C1}. Condition~\eqref{cond:C2}, which asserts that $\x$ and $\y$ have the same collision graph, is an easy consequence of the equality of clusters at entry and exit of boxes. Now if two collisions $\classe'$ and $\classe$ are such that $\classe' \lto{\gamma} \classe$, then the fact that
  \begin{equation*}
    (\Xi(\x;\classe'))^{\dxi, \dtau} \cap (\Xi(\x;\classe))^{\dxi, \dtau} = \emptyset, \quad \Xi(\y;\classe') \in (\Xi(\x;\classe'))^{\dxi, \dtau}, \quad \Xi(\y;\classe) \in (\Xi(\x;\classe))^{\dxi, \dtau},
  \end{equation*}
  implies that
  \begin{equation*}
    T^+(\classe') = T(\x;\classe') \vee T(\y;\classe') \leq T(\x;\classe') + \dtau < T(\x;\classe) - \dtau < T(\x;\classe) \wedge T(\y;\classe) = T^-(\classe), 
  \end{equation*}
  which yields Condition~\eqref{cond:C3a}. Finally, Condition~\eqref{cond:C3b} is also a consequence of the identity of the compositions of of clusters at entry and exit of boxes.
\end{proof}

When $\x$ is not a good configuration, one can obviously not expect Condition~{\rm (\hyperref[cond:C]{LHM})} to hold for $\x$ and $\y$ chosen as in Lemma~\ref{lem:CondCloc:G}. As is plotted on Figure~\ref{fig:geom}, singularities can lead this condition to fail even for the locally homothetic good configurations $\y$ and $(1-\rho)\x + \rho\y$ when $\rho$ is too far from $1$. However, based on the radial blow-up of singularities property described in~\S\ref{sss:radial}, we prove in Lemma~\ref{lem:CondCloc} below that, for $\rho_* < 1$ with $\rho_*$ close to $1$, $\y$ and $(1-\rho_*)\x + \rho_*\y$ actually satisfy the Local Homeomorphic Condition~{\rm (\hyperref[cond:C]{LHM})}. Iterating the argument starting from $(1-\rho_*)\x + \rho_*\y$ instead of $\y$, we obtain that the geometric sequence $(\rho_*^m)_{m \geq 0}$ has the property that, for all $m \geq 1$, the configurations $(1-\rho_*^{m-1})\x + \rho_*^{m-1}\y$ and $(1-\rho_*^m)\x + \rho_*^m\y$ satisfy Condition~{\rm (\hyperref[cond:C]{LHM})}.

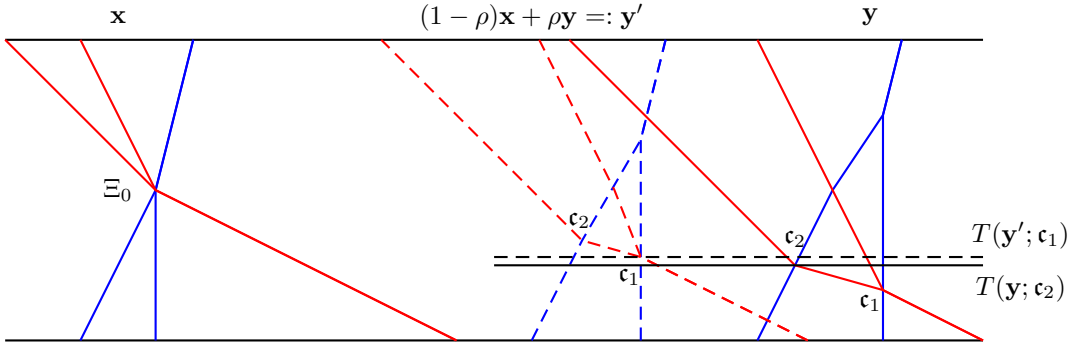
\begin{figure}[ht]
  \begin{pspicture}(14,5)
    \psline[linecolor=black](0,0)(13,0)
    \psline[linecolor=black](0,4)(13,4)
    
      
    \rput(1.5,2){$\Xi_0$}
      
      \psline[linecolor=blue](1,0)(2,2)(2.5,4)
      \psline[linecolor=blue](2,0)(2,2)(2.5,4)
      \psline[linecolor=red](6,0)(2,2)(0,4)
      \psline[linecolor=red](6,0)(2,2)(1,4)
      
      \psline[linecolor=blue](11.67,0)(11.67,.67)(11.67,3)(11.92,4)
      \psline[linecolor=blue](10,0)(10.5,1)(11,2)(11.67,3)(11.92,4)
      \psline[linecolor=red](13,0)(11.67,.67)(10.5,1)(7.5,4)
      \psline[linecolor=red](13,0)(11.67,.67)(11,2)(10,4)
      
      
      \psline[linecolor=blue, linestyle=dashed](8.45,0)(8.45,1.11)(8.45,2.67)(8.78,4)
      \psline[linecolor=blue, linestyle=dashed](7,0)(7.67,1.33)(8.45,2.67)(8.78,4)
      \psline[linecolor=red, linestyle=dashed](10.67,0)(8.45,1.11)(7.67,1.33)(5,4)
      \psline[linecolor=red, linestyle=dashed](10.67,0)(8.45,1.11)(8.1,2)(7.1,4)
    
    \rput(1.5,4.3){$\x$}
    \rput(7,4.3){$(1-\rho)\x + \rho\y =: \y'$}
    \rput(11.5,4.3){$\y$}
    
    \psline[linecolor=black](6.5,1)(13,1)
    \psline[linecolor=black, linestyle=dashed](6.5,1.11)(13,1.11)
    
    \rput(11.5,.5){$\classe_1$}
    \rput(10.5,1.35){$\classe_2$}
    \rput(8.3,.85){$\classe_1$}
    \rput(7.63,1.6){$\classe_2$}
    
    \rput(13.5,.7){\textcolor{black}{$T(\y;\classe_2)$}}
    \rput(13.5,1.4){\textcolor{black}{$T(\y';\classe_1)$}}
      
  \end{pspicture}
  \caption{The configurations $\y$ and $\y':=(1-\rho)\x+\rho\y$ are both good configurations and they are locally homothetic to $\x$. In their collision graph, $\classe_1 \to \classe_2$; however, for the choice of $\rho$ on the figure, $T(\y';\classe_1) > T(\y;\classe_2)$, therefore Condition~\eqref{cond:C3a} is not satisfied by the pair $\y, \y'$.}
  \label{fig:geom}
\end{figure}

\begin{lem}[Construction of locally homeomorphic configurations, bad case]\label{lem:CondCloc}
  Under the assumptions of Lemma~\ref{lem:radial}, let $\x \in \Drnd$, and if $\Nb(\x) \geq 1$, let $(\dxi, \dtau)$ be a proper covering of $\Ibinter(\x)$. Let $\kappa > 0$ be given by Lemma~\ref{lem:radial}. For all $\y \in B_1(\x, \kappa) \cap \Good$, there exists $\rho_* \in (0,1)$ such that, for all $m \geq 1$, the configurations $(1-\rho_*^{m-1})\y + \rho_*^{m-1}\x$ and $(1-\rho_*^m)\y + \rho_*^m\x$ satisfy Condition~{\rm (\hyperref[cond:C]{LHM})}.
\end{lem}
\begin{proof}
  Let $\y \in B_1(\x, \kappa) \cap \Good$. For all $\rho \in (0,1]$, it follows from Lemma~\ref{lem:radial} that the collisions locally look alike in the MSPD started at $\y$ and at $(1-\rho)\x+\rho\y$. This implies that $(1-\rho)\x+\rho\y \in \Good$; and, for all $\rho, \rho' \in (0,1]$, $\Rb((1-\rho)\x+\rho\y) = \Rb((1-\rho')\x+\rho'\y)$ and $(1-\rho)\x+\rho\y$, $(1-\rho')\x+\rho'\y$ have the same collision graph, so that they satisfy Conditions~\eqref{cond:C1} and~\eqref{cond:C2}. 
  
  Let us now explain how to construct $\rho_* \in (0,1)$ in such a way that, for all $m \geq 1$, the configurations $(1-\rho_*^{m-1})\y + \rho_*^{m-1}\x$ and $(1-\rho_*^m)\y + \rho_*^m\x$ satisfy Conditions~\eqref{cond:C3a} and~\eqref{cond:C3b}. Let us denote $C := \Classe(\y)$. For all $\classe \in C$, it follows from Lemma~\ref{lem:radial} that there exists a space-time point $\Xi_0(\classe)$ such that
  \begin{equation*}
    \forall \rho \in (0,1], \qquad \Xi((1-\rho)\x+\rho\y;\classe) = (1-\rho)\Xi_0(\classe) + \rho\Xi(\y;\classe), 
  \end{equation*}
  and in particular, the collision times satisfy
  \begin{equation*}
    \forall \rho \in (0,1], \qquad T((1-\rho)\x+\rho\y;\classe) = (1-\rho)T_0(\classe) + \rho T(\y;\classe),     
  \end{equation*}
  where we denote $\Xi_0(\classe) = (\xi_0(\classe), T_0(\classe))$. Therefore, for all $\rho \in (0,1]$, $(1-\rho)\x + \rho\y$ and $\y$ satisfy Condition~\eqref{cond:C3a} as soon as, for all $\classe', \classe \in C$ such that $\classe' \lto{\gamma} \classe$,
  \begin{equation*}
    \left((1-\rho)T_0(\classe') + \rho T(\y;\classe')\right) \vee T(\y;\classe') < \left((1-\rho)T_0(\classe) + \rho T(\y;\classe)\right) \wedge T(\y;\classe),
  \end{equation*}
  which is always the case if $\Xi_0(\classe') \not= \Xi_0(\classe)$ and reduces to
  \begin{equation*}
    \rho > \frac{T_0(\classe)-T(\y;\classe)}{T_0(\classe)-T(\y;\classe')}
  \end{equation*}
  if $\Xi_0(\classe') = \Xi_0(\classe)$ and either $T(\y;\classe') < T(\y;\classe) < T_0(\classe)$ or $T_0(\classe) < T(\y;\classe') < T(\y;\classe)$. We denote by $\rho_{*,1}$ the infimum of the set of $\rho \in (0,1)$ satisfying these conditions; then, for all $\rho > \rho_{*,1}$, $(1-\rho)\x + \rho\y$ and $\y$ satisfy Condition~\eqref{cond:C3a}. Very similar arguments combined with the fact that $\y \in \Good$ allow us to construct $\rho_{*,2} \in (0,1)$ such that, for all $\rho > \rho_{*,2}$, $(1-\rho)\x+\rho\y$ and $\y$ satisfy Condition~\eqref{cond:C3b}.

  As a conclusion, let us define $\rho_*$ to be any number such that 
  \begin{equation*}
    \rho_{*,1} \vee \rho_{*,2} < \rho_* < 1.
  \end{equation*}
  Then we have proved that the pair of configurations $\y$ and $(1-\rho_*)\x+\rho_*\y$ satifies Condition~{\rm (\hyperref[cond:C]{LHM})}. To complete the proof, we apply the same arguments starting from $(1-\rho_*)\x + \rho_*\y$ instead of $\y$. We obtain that, for all $\rho \in (0,1]$, the configurations
  \begin{equation*}
    (1-\rho)\x + \rho((1-\rho_*)\x+\rho_*\y) = (1-\rho\rho_*)\x + \rho\rho_*\y \qquad \text{and} \qquad (1-\rho_*)\x+\rho_*\y
  \end{equation*}
  satisfy Condition~\eqref{cond:C1} and~\eqref{cond:C2}. Besides, Condition~\eqref{cond:C3a} holds as soon as
  \begin{equation*}
    \rho > \frac{T_0(\classe)-T((1-\rho_*)\x+\rho_*\y;\classe)}{T_0(\classe)-T((1-\rho_*)\x+\rho_*\y;\classe')} = \frac{T_0(\classe)-T(\y;\classe)}{T_0(\classe)-T(\y;\classe')}
  \end{equation*}
  for all $\classe', \classe \in C((1-\rho_*)\x+\rho_*\y) = C(\y)$ such that $\classe' \lto{\gamma} \classe$, $\Xi_0(\classe')=\Xi_0(\classe)$ and either 
  \begin{equation*}
    T((1-\rho_*)\x+\rho_*\y;\classe') < T((1-\rho_*)\x+\rho_*\y;\classe) < T_0(\classe),
  \end{equation*}
  which reduces to $T(\y;\classe') < T(\y;\classe) < T_0(\classe)$, or 
  \begin{equation*}
    T_0(\classe) < T((1-\rho_*)\x+\rho_*\y;\classe') < T((1-\rho_*)\x+\rho_*\y;\classe),
  \end{equation*}
  which reduces to $T_0(\classe) < T(\y;\classe') < T(\y;\classe)$. As a consequence, the conditions on $\rho$ are the same as above and taking the infimum over the admissible values of $\rho$ yields the same quantity $\rho_{*,1}$. Likewise, to ensure that $(1-\rho_*^2)\x + \rho_*^2\y$ and $(1-\rho_*)\x+\rho_*\y$ satisfy Condition~\eqref{cond:C3b}, we obtain the same quantity $\rho_{*,2}$ as above, therefore taking $\rho=\rho_*$ again, we conclude that the configurations $(1-\rho_*^2)\x + \rho_*^2\y$ and $(1-\rho_*)\x+\rho_*\y$ satisfy Condition~{\rm (\hyperref[cond:C]{LHM})}. The proof is completed by induction.
\end{proof}


\subsubsection{Interpolation procedure}\label{sss:interpolation} In this paragraph, we describe the interpolation procedure allowing to derive global stability estimates from the local stability estimates of Proposition~\ref{prop:locstab}, under Condition~\eqref{cond:ND}. The latter condition is removed in the next subsection.

\begin{lem}[Global stability estimate under Condition~\eqref{cond:ND}]\label{lem:globstab:ND}
  Under Assumptions~\eqref{ass:LC} and \eqref{ass:USH}, and Condition~\eqref{cond:ND}, for all $\x, \y \in \Dnd$,
  \begin{equation*}
    \begin{aligned}
      & \sup_{t \geq 0} ||\Phi(\x;t)-\Phi(\y;t)||_1 \leq \ConstStab_1 ||\x-\y||_1,\\
      & \sup_{t \geq 0} ||\Phi(\x;t)-\Phi(\y;t)||_{\infty} \leq \ConstStab_{\infty} ||\x-\y||_{\infty},
    \end{aligned}
  \end{equation*}
  where $\ConstStab_1$ and $\ConstStab_{\infty}$ are given in Proposition~\ref{prop:locstab}.
\end{lem} 
\begin{proof}
  Let us begin by mentioning that the arguments of the proof do not depend on the choice of the distance; in particular, continuity and density results are valid whatever the choice of the distance since these distances are equivalent. Therefore, the notation $||\cdot||$ shall indifferently refer to $||\cdot||_1$ or $||\cdot||_{\infty}$. The corresponding stability constant shall simply be denoted $\ConstStab$.
  
  We first recall that $\Drnd$ is dense in $\Dnd$ and, by Proposition~\ref{prop:continuity}, for all $t \geq 0$, the mapping $(\x,\y) \mapsto ||\Phi(\x; t) - \Phi(\y; t)||$ is continuous on $(\Dnd)^2$. As a consequence, it suffices to prove that, for all $t \geq 0$, for all $(\x,\y) \in \Drnd^2$, $||\Phi(\x; t) - \Phi(\y; t)|| \leq \ConstStab ||\x-\y||$. 
  
  We fix $\x, \y \in \Drnd$ and proceed by interpolation as follows. In Step~1, we split the segment 
  \begin{equation}\label{eq:pf:globalstab:S}
    S := \{(1-s)\x + s\y, s \in [0,1]\}
  \end{equation}
  into a finite number of segments
  \begin{equation}\label{eq:pf:globalstab:Sk}
    S_k := \{(1-s)\x + s\y, s \in [s_k,s_{k+1}]\}, \qquad k \in \{0, \ldots, K\},
  \end{equation}
  where $0 =: s_0 < s_1 < \cdots < s_K < s_{K+1} := 1$ are such that, for all $k \in \{0, \ldots, K\}$, for all $s \in (s_k, s_{k+1})$, $(1-s)\x + s\y \in \Drnd$. In Step~2, for all $k \in \{0, \ldots, K\}$ and $\epsilon>0$ small enough, we define the segment $S_k^{\epsilon}$ by
  \begin{equation}\label{eq:pf:globalstab:Skeps}
    S_k^{\epsilon} := \{(1-s)\x + s\y, s \in [s_k+\epsilon,s_{k+1}-\epsilon]\},
  \end{equation}
  and construct a piecewise linear and continuous path joining the extreme points of $S_k^{\epsilon}$, with length arbitrarily close to the length of $S_k^{\epsilon}$, and allowing to apply Lemma~\ref{lem:CondCloc} on a finite number of linear parts of the path in Step~3. We let $\epsilon$ vanish and complete the interpolation procedure in Step~4.
  
  \sk
  \noindent {\em Step~1.} Let $S$ be defined by~\eqref{eq:pf:globalstab:S}. For all $s \in [0,1]$, $(1-s)\x + s\y \not\in \Drnd$ if and only if there exists $(\alpha:i, \beta:j) \in (\Part)^2$ such that $\alpha < \beta$ and
  \begin{equation*}
    (1-s)x_i^{\alpha} + sy_i^{\alpha} = (1-s)x_j^{\beta} + sy_j^{\beta},
  \end{equation*}
  which rewrites
  \begin{equation*}
    s(x_j^{\beta}-x_i^{\alpha} + y_i^{\alpha}-y_j^{\beta}) = x_j^{\beta}-x_i^{\alpha},
  \end{equation*}
  where we recall that $x_j^{\beta}-x_i^{\alpha} \not=0$ since $\x \in \Drnd$. As a consequence, either $x_j^{\beta}-x_i^{\alpha} + y_i^{\alpha}-y_j^{\beta} \not=0$ in which case there is at most one solution $s \in [0,1]$ to the equation above, or $x_j^{\beta}-x_i^{\alpha} + y_i^{\alpha}-y_j^{\beta} =0$ in which case there is no solution. We deduce that there is a finite number $K \geq 0$ of points $s \in [0,1]$ such that $(1-s)\x + s\y \not\in \Drnd$ and we index these points by their increasing ordering: $0 < s_1 < \cdots < s_K < 1$. For the convenience of notation in the sequel of the proof, we define $s_0 := 0$ and $s_{K+1} := 1$, so that for all $k \in \{0, \ldots, K\}$, for all $s \in (s_k, s_{k+1})$, $(1-s)\x + s\y \in \Drnd$. We finally define the segments $(S_k)_{0 \leq k \leq K}$ as in~\eqref{eq:pf:globalstab:Sk}.
  
  \sk
  \noindent {\em Step~2.} In this step we fix $k \in \{0, \ldots, K\}$ and $\epsilon > 0$ such that $s_k + \epsilon < s_{k+1} - \epsilon$. Then, the segment $S_k^{\epsilon}$ defined by~\eqref{eq:pf:globalstab:Skeps} is a compact subset of $\Drnd$. Its length is worth
  \begin{equation*}
    ||(1-(s_{k+1}-\epsilon))\x + (s_{k+1}-\epsilon)\y - (1-(s_k+\epsilon))\x - (s_k+\epsilon)\y|| = (s_{k+1}-s_k-2\epsilon)||\x-\y||.
  \end{equation*}
  Let us write
  \begin{equation*}
    S_k^{\epsilon} \subset \bigcup_{\z \in S_k^{\epsilon}} B_1(\z, \kappa(\z)),
  \end{equation*}
  where, for all $\z \in S_k^{\epsilon}$, we fix a proper covering of $\Ibinter(\z)$ if $\Nb(\z) \geq 1$ and let $\kappa(\z)$ be given by Lemma~\ref{lem:radial}. Let us extract a finite subcover $B_1(\z_1, \kappa(\z_1)), \ldots, B_1(\z_L, \kappa(\z_L))$ of $S_k^{\epsilon}$ where, for all $l \in \{1, \ldots, L\}$, $\z_l \in S_k^{\epsilon}$ writes $(1-\sigma_l)\x + \sigma_l\y$ with $s_k+\epsilon \leq \sigma_1 < \cdots < \sigma_L \leq s_{k+1}-\epsilon$. We also define $\sigma_0 := s_k+\epsilon$, $\sigma_{L+1} := s_{k+1}-\epsilon$ and $\z_0 := (1-\sigma_0)\x + \sigma_0\y$, $\z_{L+1} := (1-\sigma_{L+1})\x + \sigma_{L+1}\y$. Note that, for all $l \in \{0, \ldots, L\}$, the intersection of $B_1(\z_l, \kappa(\z_l))$ and $B_1(\z_{l+1}, \kappa(\z_{l+1}))$ is nonempty and contains the set 
  \begin{equation*}
    \{(1-s)\x + s\y, s \in (\sigma_l+\kappa(\z_l), \sigma_{l+1}-\kappa(\z_{l+1}))\}.
  \end{equation*}
  
  We finally fix $\eta > 0$ and use the density of the set $\Good$ (see Lemma~\ref{lem:gooddense}) to construct 
  \begin{equation*}
    \z'_{0,1}, \ldots, \z'_{L,L+1} \in \Good
  \end{equation*}
  such that, for all $l \in \{0, \ldots, L\}$, $\z'_{l, l+1} \in B_1(\z_l, \kappa(\z_l)) \cap B_1(\z_{l+1}, \kappa(\z_{l+1}))$, and in addition,
  \begin{equation*}
    \sum_{l=0}^L ||\z_l - \z'_{l,l+1}|| + ||\z'_{l,l+1} - \z_{l+1}|| \leq (s_{k+1}-s_k-2\epsilon)||\x-\y|| + \eta.
  \end{equation*}
  
  The quantities introduced in Step~2 are summarised on Figure~\ref{fig:pf:globalstab:Step2}.
  
  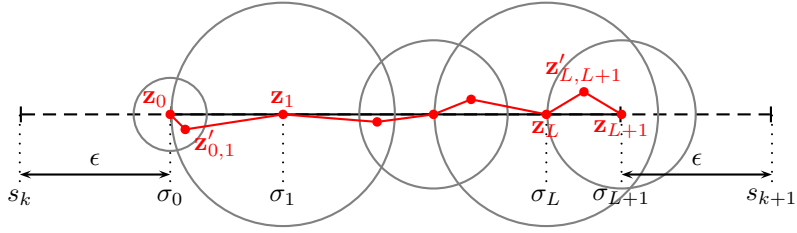
\begin{figure}[ht]
    \begin{pspicture}(10,3)
      \psline[linestyle=dashed]{|-|}(0,1.5)(10,1.5)
      \psline{|-|}(2,1.5)(8,1.5)

      \pscircle[linecolor=gray](2,1.5){.5}
      \pscircle[linecolor=gray](3.5,1.5){1.5}
      \pscircle[linecolor=gray](5.5,1.5){1}
      \pscircle[linecolor=gray](7,1.5){1.5}
      \pscircle[linecolor=gray](8,1.5){1}

      \psline[linestyle=dotted](0,1.5)(0,.6)
      \psline[linestyle=dotted](2,1.5)(2,.6)
      \psline[linestyle=dotted](3.5,1.5)(3.5,.6)
      \psline[linestyle=dotted](7,1.5)(7,.6)
      \psline[linestyle=dotted](8,1.5)(8,.6)
      \psline[linestyle=dotted](10,1.5)(10,.6)

      \psline[linecolor=red]{*-}(2,1.5)(2.2,1.3)
      \psline[linecolor=red]{*-}(2.2,1.3)(3.5,1.5)
      \psline[linecolor=red]{*-}(3.5,1.5)(4.75,1.4)
      \psline[linecolor=red]{*-}(4.75,1.4)(5.5,1.5)
      \psline[linecolor=red]{*-}(5.5,1.5)(6,1.7)
      \psline[linecolor=red]{*-}(6,1.7)(7,1.5)
      \psline[linecolor=red]{*-}(7,1.5)(7.5,1.8)
      \psline[linecolor=red]{*-*}(7.5,1.8)(8,1.5)
      
      \rput(0,.4){$s_k$}
      \rput(2,.4){$\sigma_0$}
      \rput(3.5,.4){$\sigma_1$}
      \rput(7,.4){$\sigma_L$}
      \rput(8,.4){$\sigma_{L+1}$}
      \rput(10,.4){$s_{k+1}$}
      
      \rput(1.8,1.7){\textcolor{red}{$\z_0$}}
      \rput(2.6,1.1){\textcolor{red}{$\z'_{0,1}$}}
      \rput(3.5,1.7){\textcolor{red}{$\z_1$}}
      \rput(7,1.3){\textcolor{red}{$\z_L$}}
      \rput(7.5,2.1){\textcolor{red}{$\z'_{L,L+1}$}}
      \rput(8,1.3){\textcolor{red}{$\z_{L+1}$}}
      
      \psline{<->}(0,.7)(2,.7)
      \rput(1,.9){$\epsilon$}
      \psline{<->}(8,.7)(10,.7)
      \rput(9,.9){$\epsilon$}
      
    \end{pspicture}
    \caption{The segment $S_k$ is drawn in dashed line, while the segment $S_k^{\epsilon}$ is drawn in solid line. Gray circles stand for the open balls $B_1(\z_l, \kappa(\z_l))$. The points $\z'_{0,1}, \ldots, \z'_{L,L+1}$ are chosen in the dense set $\Good$ in order to ensure that the difference between the length of the red path and the length $(s_{k+1}-s_k-2\epsilon)||\x-\y||$ of $S_k^{\epsilon}$ be smaller than $\eta$.}
    \label{fig:pf:globalstab:Step2}
  \end{figure}
  
  \sk
  \noindent {\em Step~3.} As a continuation of Step~2, let us fix $l \in \{0, \ldots, L\}$. We now prove
  \begin{equation*}
    \sup_{t \geq 0} ||\Phi(\z_l;t) - \Phi(\z'_{l,l+1};t)|| \leq \ConstStab ||\z_l-\z'_{l,l+1}||,
  \end{equation*}
  and similar arguments shall also yield
  \begin{equation*}
    \sup_{t \geq 0} ||\Phi(\z_{l+1};t) - \Phi(\z'_{l,l+1};t)|| \leq \ConstStab ||\z_{l+1}-\z'_{l,l+1}||.
  \end{equation*}
  By Step~2, $\z_l \in \Drnd$ and $\z'_{l,l+1} \in B_1(\z_l, \kappa(\z_l)) \cap \Good$. As a consequence, Lemma~\ref{lem:CondCloc} implies that there exists $\rho_* \in (0,1)$ such that, for all $m \geq 1$, $(1-\rho_*^{m-1})\z_l + \rho_*^{m-1}\z'_{l,l+1}$ and $(1-\rho_*^m)\z_l + \rho_*^m\z'_{l,l+1}$ satisfy Condition~{\rm (\hyperref[cond:C]{LHM})}. Therefore, for all $m \geq 1$, Proposition~\ref{prop:locstab} yields, for all $t \geq 0$,
  \begin{equation*}
    ||\Phi((1-\rho_*^m)\z_l + \rho_*^m\z'_{l,l+1}; t) - \Phi((1-\rho_*^{m-1})\z_l + \rho_*^{m-1}\z'_{l,l+1}; t)|| \leq \ConstStab (\rho_*^{m-1}-\rho_*^m) ||\z_l - \z'_{l,l+1}||.
  \end{equation*}
  
  We finally deduce from the triangle inequality that, for all $M \geq 1$,
  \begin{equation*}
    \begin{aligned}
      & ||\Phi((1-\rho_*^M)\z_l+\rho_*^M\z'_{l,l+1};t) - \Phi(\z'_{l,l+1};t)||\\
      & \qquad \leq \sum_{m=1}^M ||\Phi((1-\rho_*^m)\z_l + \rho_*^m\z'_{l,l+1}; t) - \Phi((1-\rho_*^{m-1})\z_l + \rho_*^{m-1}\z'_{l,l+1}; t)||\\
      & \qquad \leq \sum_{m=1}^M \ConstStab (\rho_*^{m-1}-\rho_*^m) ||\z_l - \z'_{l,l+1}|| = \ConstStab (1-\rho_*^M)||\z_l - \z'_{l,l+1}||,
    \end{aligned}
  \end{equation*}
  and use Proposition~\ref{prop:continuity} to conclude that
  \begin{equation*}
    \sup_{t \geq 0} ||\Phi(\z_l;t) - \Phi(\z'_{l,l+1};t)|| \leq \ConstStab ||\z_l - \z'_{l,l+1}||.
  \end{equation*}
  
  \sk
  \noindent {\em Step~4.} We finally complete the interpolation procedure described in the introduction of the proof. First, it follows from Step~3 that
  \begin{equation*}
    \begin{aligned}
      & \sup_{t \geq 0} ||\Phi(\z_0; t) - \Phi(\z_{L+1}; t)|| \leq \sum_{l=0}^L \sup_{t \geq 0} \left(||\Phi(\z_l; t) - \Phi(\z'_{l,l+1}; t)|| +  ||\Phi(\z'_{l,l+1}; t) - \Phi(\z_{l+1}; t)||\right)\\
      & \qquad \leq \ConstStab \sum_{l=0}^L ||\z_l-\z'_{l,l+1}|| +  ||\z'_{l,l+1}-\z_{l+1}||\\
      & \qquad \leq \ConstStab \left((s_{k+1}-s_k-2\epsilon)||\x-\y|| + \eta\right).
    \end{aligned}
  \end{equation*}
  Recalling that $\z_0 = (1-(s_k+\epsilon))\x + (s_k+\epsilon)\y$ and $\z_{L+1} = (1-(s_{k+1}-\epsilon))\x + (s_{k+1}-\epsilon)\y$, and letting $\eta$ vanish, we obtain
  \begin{equation*}
    \sup_{t \geq 0} ||\Phi((1-(s_k+\epsilon))\x + (s_k+\epsilon)\y; t) - \Phi((1-(s_{k+1}-\epsilon))\x + (s_{k+1}-\epsilon)\y; t)|| \leq \ConstStab (s_{k+1}-s_k-2\epsilon)||\x-\y||.
  \end{equation*}
  Taking the limit of both sides when $\epsilon$ vanishes and using Proposition~\ref{prop:continuity}, we finally write
  \begin{equation*}
    \sup_{t \geq 0} ||\Phi((1-s_k)\x + s_k\y; t) - \Phi((1-s_{k+1})\x + s_{k+1}\y; t)|| \leq \ConstStab (s_{k+1}-s_k)||\x-\y||
  \end{equation*}
  and complete the proof thanks to the triangle inequality again.
\end{proof}


\subsubsection{Approximation of degenerate characteristic fields}\label{sss:approxND} We now complete the proof of Proposition~\ref{prop:globstab} by removing Condition~\eqref{cond:ND} from the statement of Lemma~\ref{lem:globstab:ND}. We use the following approximation argument.

\begin{lem}[Nondegenerate approximation of degenerate characteristic fields]\label{lem:approxND}
  Let us assume that the function $\blambda = (\lambda^1, \ldots, \lambda^d)$ satisfies Assumptions~\eqref{ass:USH} and~\eqref{ass:LC}. Then, for all $n \geq 1$, there exists a sequence of functions $\blambda^{[q]} = (\lambda^{[q],1}, \ldots, \lambda^{[q],d})$, $q \geq 1$, satisfying Assumptions~\eqref{ass:USH} and~\eqref{ass:LC} as well as Condition~\eqref{cond:ND}, such that, when $q$ grows to infinity:
  \begin{enumerate}[label=(\roman*), ref=\roman*]
    \item\label{it:approxND:1} for all $\x \in \Dnd$, for all $\gamma:k \in \Part$, $(\tlambda^{[q]})_k^{\gamma}(\x)$ converges to $\tlambda_k^{\gamma}(\x)$,
    \item\label{it:approxND:2} for all $\gamma \in \{1, \ldots, d\}$, $\sup_{\bu \in [0,1]^d} |\lambda^{[q],\gamma}(\bu)|$ converges to $\sup_{\bu \in [0,1]^d} |\lambda^{\gamma}(\bu)|$, the Lipschitz continuity constant $\ConstLip^{[q]}$ of $\blambda^{[q]}$ converges to the Lipschitz continuity constant $\ConstLip$ of $\blambda$ and the uniform strict hyperbolicity constant $\ConstUSH^{[q]}$ of $\blambda^{[q]}$ converges to the uniform strict hyperbolicity constant $\ConstUSH$ of $\blambda$, 
    \item\label{it:approxND:3} for all $\x \in \Dnd$, for all $t \geq 0$, the configuration $\Phi^{[q]}(\x;t)$ at time $t$ of the MSPD started at $\x$ with velocity vectors determined by $\blambda^{[q]}$ converges to the configuration $\Phi(\x;t)$ at time $t$ of the MSPD started at $\x$ with velocity vectors determined by $\blambda$.
  \end{enumerate}
\end{lem}

The conclusion of the proof of Proposition~\ref{prop:globstab} is now straightforward: applying Lemma~\ref{lem:globstab:ND} to the MSPD with velocity vectors determined by $\blambda^{[q]}$, we obtain, for all $\x, \y \in \Dnd$ and for all $t \geq 0$,
  \begin{equation*}
    \begin{aligned}
      & ||\Phi^{[q]}(\x;t)-\Phi^{[q]}(\y;t)||_1 \leq \ConstStab_1^{[q]} ||\x-\y||_1,\\
      & ||\Phi^{[q]}(\x;t)-\Phi^{[q]}(\y;t)||_{\infty} \leq \ConstStab_{\infty}^{[q]} ||\x-\y||_{\infty},
    \end{aligned}
  \end{equation*}
where the meaning of $\ConstStab_1^{[q]}$ and $\ConstStab_{\infty}^{[q]}$ is obvious. Since these stability constants are continuous functions of $\ConstLip^{[q]}$ and $\ConstUSH^{[q]}$, there is no difficulty in taking the limit when $q$ grows to infinity of both inequalities and thus obtaining Proposition~\ref{prop:globstab}.

\begin{proof}[Proof of Lemma~\ref{lem:approxND}]
  The proof is decomposed into two independent parts: in the first part, we construct a particular sequence of functions $\blambda^{[q]}$ satisfying Condition~\eqref{cond:ND} as well as the points~\eqref{it:approxND:1} and~\eqref{it:approxND:2}. In the second part, we prove that any sequence of functions $\blambda^{[q]}$ satisfying the points~\eqref{it:approxND:1} and~\eqref{it:approxND:2} necessarily satisfies the point~\eqref{it:approxND:3}.
  
  \sk
  \noindent {\em Construction of $\blambda^{[q]}$.} Let us fix $\x \in \Part$, $\gamma \in \{1, \ldots, d\}$ and $\uk < \ok$ in $\{1, \ldots, n\}$, such that
  \begin{equation*}
    \forall \gamma' \not= \gamma, \quad \forall k \in \{\uk, \ldots, \ok\}, \qquad \omega_{\gamma:k}^{\gamma'}(\x) = \omega_{\gamma:\uk}^{\gamma'}(\x).
  \end{equation*}  
  Then, for all $k \in \{\uk, \ldots, \ok-1\}$, for all $\epsilon > 0$, we have
  \begin{equation*}
    \begin{aligned}
      & \frac{1}{k-\uk+1} \sum_{k'=\uk}^k n\int_{w=(k'-1)/n}^{k'/n} \left\{\lambda^{\gamma}\left(\omega^1_{\gamma:k'}(\x), \ldots, \omega^{\gamma-1}_{\gamma:k'}(\x), w, \omega^{\gamma+1}_{\gamma:k'}(\x), \ldots, \omega^d_{\gamma:k'}(\x)\right) - \epsilon w\right\}\dd w\\
      & \qquad = \frac{n}{k-\uk+1} \int_{w=(\uk-1)/n}^{k/n} \left\{\lambda^{\gamma}\left(\omega^1_{\gamma:k}(\x), \ldots, \omega^{\gamma-1}_{\gamma:k}(\x), w, \omega^{\gamma+1}_{\gamma:k}(\x), \ldots, \omega^d_{\gamma:k}(\x)\right) - \epsilon w\right\}\dd w\\
      & \qquad = \frac{1}{k-\uk+1} \sum_{k'=\uk}^k \tlambda_{k'}^{\gamma}(\x) - \frac{\epsilon}{2n} \frac{k^2-(\uk-1)^2}{k-\uk+1}\\
      & \qquad = \frac{1}{k-\uk+1} \sum_{k'=\uk}^k \tlambda_{k'}^{\gamma}(\x) - \frac{\epsilon}{2n} (k+\uk-1),
    \end{aligned}
  \end{equation*}
  and similarly
  \begin{equation*}
    \begin{aligned}
      & \frac{1}{\ok-k} \sum_{k'=k+1}^{\ok} n\int_{w=(k'-1)/n}^{k'/n} \left\{\lambda^{\gamma}\left(\omega^1_{\gamma:k'}(\x), \ldots, \omega^{\gamma-1}_{\gamma:k'}(\x), w, \omega^{\gamma+1}_{\gamma:k'}(\x), \ldots, \omega^d_{\gamma:k'}(\x)\right) - \epsilon w\right\}\dd w\\
      & \qquad = \frac{1}{\ok-k} \sum_{k'=k+1}^{\ok} \tlambda_{k'}^{\gamma}(\x) - \frac{\epsilon}{2n} (\ok+k).
    \end{aligned}
  \end{equation*}

  If
  \begin{equation*}
    \frac{1}{k-\uk+1} \sum_{k'=\uk}^k \tlambda_{k'}^{\gamma}(\x) \not= \frac{1}{\ok-k} \sum_{k'=k+1}^{\ok} \tlambda_{k'}^{\gamma}(\x),
  \end{equation*}
  then for $\epsilon$ small enough, we still have
  \begin{equation*}
    \frac{1}{k-\uk+1} \sum_{k'=\uk}^k \tlambda_{k'}^{\gamma}(\x) - \frac{\epsilon}{2n} (k+\uk-1) \not= \frac{1}{\ok-k} \sum_{k'=k+1}^{\ok} \tlambda_{k'}^{\gamma}(\x) - \frac{\epsilon}{2n} (\ok+k).
  \end{equation*}

  On the contrary, if
  \begin{equation*}
    \frac{1}{k-\uk+1} \sum_{k'=\uk}^k \tlambda_{k'}^{\gamma}(\x) = \frac{1}{\ok-k} \sum_{k'=k+1}^{\ok} \tlambda_{k'}^{\gamma}(\x),
  \end{equation*}
  then the fact that $\uk-1\not=\ok$ ensures that we still have
  \begin{equation*}
    \frac{1}{k-\uk+1} \sum_{k'=\uk}^k \tlambda_{k'}^{\gamma}(\x) - \frac{\epsilon}{2n} (k+\uk-1) \not= \frac{1}{\ok-k} \sum_{k'=k+1}^{\ok} \tlambda_{k'}^{\gamma}(\x) - \frac{\epsilon}{2n} (\ok+k).
  \end{equation*}
  
  For all $\gamma' \not= \gamma$, $\omega^{\gamma'}_{\gamma:k}(\x)$ can only take the values $0, 1/n, \ldots, 1$ when $\x$ varies in $\Part$. Taking the minimum of all admissible $\epsilon$ for all these possible values, and all the possible choices of $\gamma \in \{1, \ldots, d\}$ and $\uk \leq k < \ok$ in $\{1, \ldots, n\}$, we obtain $\epsilon_0 > 0$ such that, for all $q \geq 1$, the function $\blambda^{[q]} = (\lambda^{[q],1}, \ldots, \lambda^{[q],d})$ defined by, for all $\gamma \in \{1, \ldots, d\}$,
  \begin{equation*}
    \forall \bu \in [0,1]^d, \qquad \lambda^{[q], \gamma}(\bu) := \lambda^{\gamma}(\bu) - \frac{\epsilon_0}{q} u^{\gamma}
  \end{equation*}
  satisfies Condition~\eqref{cond:ND} and the point~\eqref{it:approxND:1}. Up to decreasing $\epsilon_0$ again, it is easy to prove that the functions $\blambda^{[q]}$ also satisfy Assumptions~\eqref{ass:USH} and~\eqref{ass:LC}, and that the associated constants satisfy the point~\eqref{it:approxND:2}.
  
  \sk
  \noindent {\em Proof of~\eqref{it:approxND:3}.} Let $\blambda^{[q]} = (\lambda^{[q],1}, \ldots, \lambda^{[q],d})$, $q \geq 1$, be a sequence of functions satisfying Assumptions~\eqref{ass:USH} and~\eqref{ass:LC}, Condition~\eqref{cond:ND} as well as the points~\eqref{it:approxND:1} and~\eqref{it:approxND:2}. Let $\Phi^{[q]}$ denote the MSPD flow associated with the velocity vectors determined by $\blambda^{[q]}$. We prove by induction on $\Nb(\x)$ that, for all $\x \in \Dnd$,
  \begin{equation}\label{eq:pf:approxND:iii}
    \forall t \geq 0, \qquad \lim_{q \to +\infty} \Phi^{[q]}(\x;t) = \Phi(\x;t).
  \end{equation}
  
  If $\Nb(\x)=0$, then we have, for all $t \geq 0$, $\Phi^{[q]}(\x;t) = \tPhi[\tblambda^{[q]}(\x)](\x;t)$, which converges to $\tPhi[\tblambda(\x)](\x;t)=\Phi(\x;t)$ on account of Lemma~\ref{lem:contracttPhi} combined with the point~\eqref{it:approxND:1}. Now let $N \geq 0$ such that~\eqref{eq:pf:approxND:iii} holds for all $\x \in \Dnd$ such that $\Nb(\x) \leq N$. Let us fix $\x \in \Dnd$ such that $\Nb(\x) = N+1$, and $T > 0$. For all $q \geq 1$, $\Phi^{[q]}(\x;0) = \x$ and the process $(\Phi^{[q]}(\x;t))_{t \in [0,T]}$ is Lipschitz continuous in $\Dnd$, and by~\eqref{it:approxND:2}, its Lipschitz norm is uniformly bounded with respect to $q$. As a consequence, it follows from the Arzelà-Ascoli Theorem that, along a subsequence that we still index by $q$ for convenience, $(\Phi^{[q]}(\x;t))_{t \in [0,T]}$ converges uniformly to a continuous process $(\varphi(t))_{t \in [0,T]}$ in $\Dnd$. The point~\eqref{it:approxND:3} follows if we identify this limit with $(\Phi(\x;t))_{t \in [0,T]}$.
  
  In this purpose, let us note that the sequence
  \begin{equation*}
    \{((\tinter_{\alpha:i, \beta:j})^{[q]}(\x) \wedge T)_{(\alpha:i, \beta:j) \in \Rb(\x)} : q \geq 1\}
  \end{equation*}
  is bounded in $[0,+\infty)^{\Nb(\x)}$, and therefore, up to extracting a further subsequence, we may assume that, for all $(\alpha:i, \beta:j) \in \Rb(\x)$,
  \begin{equation*}
    \lim_{q \to +\infty} (\tinter_{\alpha:i, \beta:j})^{[q]}(\x) \wedge T = \bar{\tau}_{\alpha:i, \beta:j} \in (0,T],
  \end{equation*}
  where the fact that $\bar{\tau}_{\alpha:i, \beta:j} > 0$ follows from Lemma~\ref{lem:tinter} and the point~\eqref{it:approxND:2}. As a consequence, $t^{*[q]}(\x) \wedge T$ converges to 
  \begin{equation*}
    \bar{\tau} := \min\{\bar{\tau}_{\alpha:i, \beta:j} : (\alpha:i, \beta:j) \in \Rb(\x)\} \in (0,T].
  \end{equation*}
  
  We first remark that, since $(\Phi^{[q]}(\x;t))_{t \in [0,T]}$ converges uniformly to $(\varphi(t))_{t \in [0,T]}$, then we have, for all $(\alpha:i, \beta:j) \in \Rb(\x)$ such that $\bar{\tau}_{\alpha:i, \beta:j} < T$,
  \begin{equation}\label{eq:pf:approxND:phiphi}
    \varphi_{\alpha:i}(\bar{\tau}_{\alpha:i, \beta:j}) = \varphi_{\beta:j}(\bar{\tau}_{\alpha:i, \beta:j}).
  \end{equation}
  We now fix $0 < \eta < \bar{\tau}$. Then, there exists $q_0 \geq 1$ such that, for all $q \geq q_0$, we have $t^{*[q]}(\x) \geq \bar{\tau}-\eta/2$, and then, for all $t \in [0, \bar{\tau}-\eta]$, $\Phi^{[q]}(\x;t) = \tPhi[\tblambda^{[q]}(\x)](\x;t)$ converges to $\tPhi[\tblambda(\x)](\x;t)=\varphi(t)$ thanks to Lemma~\ref{lem:contracttPhi} and the point~\eqref{it:approxND:1} again. Besides, by Lemma~\ref{lem:tinter}, we have, for all $(\alpha:i, \beta:j) \in \Rb(\x)$, for all $t \in [0,\bar{\tau}-\eta]$,
  \begin{equation*}
    \tPhi_j^{\beta}[\tblambda(\x)](\x; t) - \tPhi_i^{\alpha}[\tblambda(\x)](\x; t) \geq \frac{\eta}{2\ConstUSH},
  \end{equation*}
  so that $t^*(\x) \geq \bar{\tau}-\eta$ and therefore, for all $t \in [0,\bar{\tau}-\eta]$, $\tPhi[\tblambda(\x)](\x;t) = \Phi(\x;t)$. Letting $\eta$ vanish, we deduce that:
  \begin{itemize}
    \item for all $t \in [0, \bar{\tau})$, $\varphi(t) = \Phi(\x;t)$,
    \item $t^*(\x) \geq \bar{\tau}$.
  \end{itemize}
  Since both $\varphi$ and the MSPD are continuous, we also have $\varphi(\bar{\tau}) = \Phi(\x;\bar{\tau})$. If $\bar{\tau}=T$, then we have proved that
  \begin{equation*}
    \forall t \in [0,T], \qquad \lim_{q \to +\infty} \Phi^{[q]}(\x;t) = \Phi(\x;t).
  \end{equation*}
  
  Otherwise, there exists $(\alpha:i, \beta:j) \in \Rb(\x)$ such that $\bar{\tau}_{\alpha:i, \beta:j} < T$, and applying~\eqref{eq:pf:approxND:phiphi} to any pair $(\alpha:i, \beta:j)$ such that $\bar{\tau}_{\alpha:i, \beta:j} = \bar{\tau}$, we first obtain $\bar{\tau} = t^*(\x)$. For all $t \in (t^*(\x), T]$, we now write
  \begin{equation*}
    \begin{aligned}
      ||\Phi^{[q]}(\x;t) - \Phi(\x;t)||_1 & = ||\Phi^{[q]}(\Phi^{[q]}(\x;t^*(\x));t-t^*(\x)) - \Phi(\Phi(\x;t^*(\x));t-t^*(\x))||_1\\
      & \leq ||\Phi^{[q]}(\Phi^{[q]}(\x;t^*(\x));t-t^*(\x)) - \Phi^{[q]}(\Phi(\x;t^*(\x));t-t^*(\x))||_1\\
      & \quad + ||\Phi^{[q]}(\Phi(\x;t^*(\x));t-t^*(\x)) - \Phi(\Phi(\x;t^*(\x));t-t^*(\x))||_1.
    \end{aligned}
  \end{equation*}
  On the one hand, $\Nb(\Phi(\x;t^*(\x))) \leq N$, and therefore
  \begin{equation*}
    \lim_{q \to +\infty} \Phi^{[q]}(\Phi(\x;t^*(\x));t-t^*(\x)) = \Phi(\Phi(\x;t^*(\x));t-t^*(\x)).
  \end{equation*}
  On the other hand, $\blambda^{[q]}$ satisfies Condition~\eqref{cond:ND}, so that Lemma~\ref{lem:globstab:ND} yields
  \begin{equation*}
    ||\Phi^{[q]}(\Phi^{[q]}(\x;t^*(\x));t-t^*(\x)) - \Phi^{[q]}(\Phi(\x;t^*(\x));t-t^*(\x))||_1 \leq \ConstStab^{[q]}_1 ||\Phi^{[q]}(\x;t^*(\x)) - \Phi(\x;t^*(\x))||_1.
  \end{equation*}
  By the first part of the argument, $\Phi^{[q]}(\x;t^*(\x))$ converges to $\Phi(\x;t^*(\x))$, while by the point~\eqref{it:approxND:2}, the value of $\ConstStab^{[q]}_1$ is uniformly bounded with respect to $q$. As a conclusion,
  \begin{equation*}
    \lim_{q \to +\infty} ||\Phi^{[q]}(\x;t) - \Phi(\x;t)||_1 = 0,
  \end{equation*}
  so that $\Phi^{[q]}(\x;t)$ converges to $\Phi(\x;t)$, for all $t \in [0,T]$. Since $T$ is arbitrary, the proof is completed.
\end{proof}


\subsection{Proof of Theorem~\ref{theo:stabMSPD}}\label{ss:pfstabMSPD} Theorem~\ref{theo:stabMSPD} is finally obtained by interpolating the $\Ls^1$ and $\Ls^{\infty}$ estimates of Proposition~\ref{prop:globstab} thanks to the Riesz-Thorin Theorem.

\begin{proof}[Proof of Theorem~\ref{theo:stabMSPD}]
  Let us fix $\x, \y \in \Dnd$ and $s,t \geq 0$. Then, for all $p \in [1,+\infty]$,
  \begin{equation*}
    ||\Phi(\x;s)-\Phi(\y;t)||_p \leq ||\Phi(\x;s)-\Phi(\y;s)||_p + ||\Phi(\y;s)-\Phi(\y;t)||_p,
  \end{equation*}
  and by~(\ref{eq:defvmspd}-\ref{eq:typeencadrelambda}), for all $p \in [1,+\infty)$,
  \begin{equation*}
    ||\Phi(\y;s)-\Phi(\y;t)||_p^p = \frac{1}{n} \sum_{\gamma=1}^d\sum_{k=1}^n \left|\int_{r=s}^t v_k^{\gamma}(\y;r)\dd r\right|^p \leq |t-s|^p (\ConstBound{p})^p;
  \end{equation*}
  similarly,
  \begin{equation*}
    ||\Phi(\y;s)-\Phi(\y;t)||_{\infty} \leq |t-s| \ConstBound{\infty}.
  \end{equation*}
  
  It now remains to prove that
  \begin{equation*}
    ||\Phi(\x;s)-\Phi(\y;s)||_p \leq \ConstStab_p ||\x-\y||_p,
  \end{equation*}
  for some $\ConstStab_p$ that depends neither on $n$ nor on $s$. By Proposition~\ref{prop:globstab}, this is already the case for $p \in \{1, +\infty\}$. 
  
  We first extend $\Phi(\cdot;s)$ into a nonlinear operator of the linear space $\R^{d\times n}$ by defining, for all $\bar{\x} \in \R^{d\times n}$,
  \begin{equation*}
    \bar{\Phi}(\bar{\x}) =: \Phi(\pi(\bar{\x});s),
  \end{equation*} 
  where 
  \begin{equation*}
    \pi : \left\{\begin{array}{ccc}
      \R^{d \times n} & \to & \Dnd\\
      (\bar{x}_j^{\gamma})_{1 \leq \gamma \leq d, 1 \leq j \leq n} & \mapsto & (\bar{x}_{(k)}^{\gamma})_{1 \leq \gamma \leq d, 1 \leq k \leq n}
    \end{array}\right.
  \end{equation*}
  and, for all $\gamma \in \{1, \ldots, d\}$, $\bar{x}_{(1)}^{\gamma} \leq \cdots \leq \bar{x}_{(n)}^{\gamma}$ refers to the increasing reordering of $\bar{x}_1^{\gamma}, \ldots, \bar{x}_n^{\gamma}$.
  
  Then, by Proposition~\ref{prop:globstab}, we have, for all $\bar{\x}, \bar{\y} \in \R^{d \times n}$,
  \begin{equation*}
    \begin{aligned}
      & ||\bar{\Phi}(\bar{\x}) - \bar{\Phi}(\bar{\y})||_{\ell^1} \leq \ConstStab_1 ||\pi(\bar{\x})-\pi(\bar{\y})||_{\ell^1} \leq ||\bar{\x}-\bar{\y}||_{\ell^1},\\
      & ||\bar{\Phi}(\bar{\x}) - \bar{\Phi}(\bar{\y})||_{\ell^{\infty}} \leq \ConstStab_{\infty} ||\pi(\bar{\x})-\pi(\bar{\y})||_{\ell^{\infty}} \leq ||\bar{\x}-\bar{\y}||_{\ell^{\infty}},
    \end{aligned}
  \end{equation*}
  where $||\cdot||_{\ell^1}$ and $||\cdot||_{\ell^{\infty}}$ refer to the usual $\ell^1$ and $\ell^{\infty}$ norms on the linear space $\R^{d \times n}$. The second inequality of both lines follows from the observation that, for all $\gamma \in \{1, \ldots, d\}$, if we define
  \begin{equation*}
    m := \frac{1}{n} \sum_{j=1}^n \delta_{\bar{x}_j^{\gamma}}, \quad m' := \frac{1}{n} \sum_{j=1}^n \delta_{\bar{y}_j^{\gamma}} \qquad \in \Ps(\R),
  \end{equation*}
  and
  \begin{equation*}
    \mathfrak{m} := \frac{1}{n} \sum_{j=1}^n \delta_{(\bar{x}_j^{\gamma}, \bar{y}_j^{\gamma})} \in \Ps(\R^2),
  \end{equation*}
  then, with the notations of Definition~\ref{defi:wass}, $\mathfrak{m} <^m_{m'}$ and
  \begin{equation*}
    \int_{(x,x') \in \R^2} |x-x'|^p \mathfrak{m}(\dd x\dd x') = \frac{1}{n} \sum_{j=1}^n |\bar{x}_j^{\gamma}-\bar{y}_j^{\gamma}|^p,
  \end{equation*}
  while Remark~\ref{rk:wassemp} yields
  \begin{equation*}
    \Ws_p(m,m') = \frac{1}{n} \sum_{k=1}^n |\bar{x}_{(k)}^{\gamma}-\bar{y}_{(k)}^{\gamma}|^p,
  \end{equation*}
  with the notations of the definition of $\pi$. The conclusion follows from the minimality of the Wasserstein distance.
  
  We deduce that
  \begin{equation}\label{eq:barPhixy}
    \bar{\Phi}(\bar{\x}) - \bar{\Phi}(\bar{\y}) = \int_{\theta=0}^1 \Diff\bar{\Phi}((1-\theta)\bar{\x} + \theta\bar{\y}) (\bar{\x}-\bar{\y}) \dd \theta,
  \end{equation}
  where the matrix $\Diff\bar{\Phi}(\bar{\z})$ is defined $\dd \bar{\z}$-almost everywhere and satisfies
  \begin{equation*}
    |||\Diff\bar{\Phi}(\bar{\z})|||_{\ell^1} \leq \ConstStab_1, \qquad |||\Diff\bar{\Phi}(\bar{\z})|||_{\ell^{\infty}} \leq \ConstStab_{\infty},
  \end{equation*}
  and $|||\cdot|||_{\ell^p}$ refers to the norm of operators on $(\R^{d \times n}, ||\cdot||_{\ell^p})$. Applying the Riesz-Thorin Theorem~\cite[Theorem~VI.10.11, p.~525]{DunSchI}, we obtain that, $\dd \bar{\z}$-almost everywhere,
  \begin{equation*}
    |||\Diff\bar{\Phi}(\bar{\z})|||_{\ell^p} \leq \ConstStab_p,
  \end{equation*}
  with $\ConstStab_p := \ConstStab_1^{1/p}\ConstStab_{\infty}^{1-1/p}$. Injecting this relation in~\eqref{eq:barPhixy}, we conclude that
  \begin{equation*}
    ||\bar{\Phi}(\bar{\x}) - \bar{\Phi}(\bar{\y})||_{\ell^p} \leq \ConstStab_p ||\bar{\x}-\bar{\y}||_{\ell^p}.
  \end{equation*}
  Taking $\bar{\x} = \x, \bar{\y} = \y$ in $\Dnd$, and $p \in (1,+\infty)$, we rewrite the inequality above as
  \begin{equation*}
    \sum_{\gamma=1}^d \sum_{k=1}^n |\Phi_k^{\gamma}(\x;s) - \Phi_k^{\gamma}(\y;s)|^p \leq (\ConstStab_p)^p \sum_{\gamma=1}^d \sum_{k=1}^n |x^{\gamma} - y_k^{\gamma}|^p,
  \end{equation*}
  and we conclude by dividing both parts of the inequality by $n$ and taking the power $1/p$.
\end{proof}


\section{Construction and identification of stable semigroup solutions}\label{s:sg}

This section is dedicated to the proof of Theorem~\ref{theo:sg}. 

In Subsection~\ref{ss:bS}, we explain how to pass to the large-scale limit in the discrete stability estimates of Theorem~\ref{theo:stabMSPD}, which naturally yields Wasserstein stability estimates on the solutions to~\eqref{eq:syst} obtained by Theorem~\ref{theo:existence}. As a byproduct of these stability estimates, we show that our solutions are semigroups on appropriate classes of vectors of probability measures.

In Subsection~\ref{ss:BB}, we introduce the uniqueness conditions of Bianchini and Bressan for the system~\eqref{eq:syst}, and show that our solutions satisfy these conditions. This allows us to identify all our semigroup solutions, and to finally complete the proof of Theorem~\ref{theo:sg} in Subsection~\ref{ss:pfsg}.


\subsection{Construction of stable semigroup solutions}\label{ss:bS} In this subsection, we combine Theorems~\ref{theo:existence} and~\ref{theo:stabMSPD} to construct stable semigroups solving~\eqref{eq:syst}. The main result of this subsection is Proposition~\ref{prop:cSt}.


\subsubsection{Further properties of the Wasserstein distance}\label{sss:wassfur} The convergence in Wasserstein distance of any order implies the weak convergence on $\Ps(\R)$~\cite[Theorem~6.9]{villani}. The converse is not true, but the Wasserstein distance however enjoys the following lower semicontinuity property.

\begin{lem}[Lower semicontinuity of the Wasserstein distance]\label{lem:lsc}
  Let $(m_n)_{n \geq 1}$ and $(m'_n)_{n \geq 1}$ be two sequences of probability measures on $\R$ converging weakly to the respective limits $m$ and $m'$ in $\Ps(\R)$. Then, for all $p \in [1, +\infty]$,
  \begin{equation*}
    \Ws_p(m, m') \leq \liminf_{n \to +\infty} \Ws_p(m_n, m'_n).
  \end{equation*} 
\end{lem}
Of course, both terms of the inequality above can be infinite.
\begin{proof}
  For $p \in [1, +\infty)$, the result is proved in~\cite[Remark~6.12]{villani}. If $p=+\infty$, then letting $F_n := H*m_n$, $G_n := H*m_n'$, $F := H*m$, $G := H*m'$, Lemma~\ref{lem:cvCDF} yields, for all continuity points $v$ of $|F^{-1}-G^{-1}|$,
  \begin{equation*}
    \begin{aligned}
      |F^{-1}(v) - G^{-1}(v)| & = \lim_{n \to +\infty} |F_n^{-1}(v) - G_n^{-1}(v)|\\
      & \leq \liminf_{n \to +\infty} \sup_{v' \in (0,1)} |F_n^{-1}(v') - G_n^{-1}(v')| = \liminf_{n \to +\infty} \Ws_{\infty}(m_n, m'_n).
    \end{aligned}
  \end{equation*}
  Since the function $|F^{-1}-G^{-1}|$ is left continuous with right limits, we deduce that the bound above holds for all $v \in (0,1)$, which implies the desired result.
\end{proof}

Throughout this section, the following notion of {\em $\Ws_1$ stability class} plays an important role.

\begin{defi}[$\Ws_1$ stability class]\label{defi:dP}
  For all $\bm^* \in \Ps(\R)^d$, we denote by $\dP_{\bm^*}$ the $\Ws_1$ stability class of $\bm^*$ defined as the set of $\bm \in \Ps(\R)^d$ such that
  \begin{equation*}
    \Ws^{(d)}_1(\bm^*, \bm) < +\infty,
  \end{equation*}
  where we recall the definition~\eqref{eq:Wpd} of the distance $\Ws^{(d)}_1$.
\end{defi}

The topology of $\Ws_1$ stability classes is described by Lemma~\ref{lem:dP}.

\begin{lem}[Properties of $\dP_{\bm^*}$]\label{lem:dP}
  For all $\bm^* \in \Ps(\R)^d$, the set $\dP_{\bm^*}$ is complete and separable for the $\Ws^{(d)}_1$ topology.
\end{lem}
\begin{proof}
  Let $\bm^* \in \Ps(\R)^d$ and recall the Definition~\ref{defi:dP} of the $\Ws_1$ stability class $\dP_{\bm^*}$. If $\bm^* = (m^{*,1}, \ldots, m^{*,d})$ is such that
  \begin{equation*}
    \sum_{\gamma=1}^d \int_{x \in \R} |x| m^{*,\gamma}(\dd x) < +\infty,
  \end{equation*}
  then $\dP_{\bm^*}$ is the space of all $\bm \in \Ps(\R)^d$ satisfying the same integrability condition, and Lemma~\ref{lem:dP} follows from~\cite[Theorem~6.18]{villani}.
  
  In the general case, it is clear from the definition of $\dP_{\bm^*}$ that there is no loss of generality in assuming that $d=1$, therefore we now fix $m^* \in \Ps(\R)$ and prove that the set $\dP_{m^*}$ of probability measures $m \in \Ps(\R)$ such that $\Ws_1(m,m^*) < +\infty$ is complete for the $\Ws_1$ topology and contains a countable and dense subset. If $(m_n)_{n\geq 1}$ is a Cauchy sequence in $\dP_{m^*}$ for the $\Ws_1$ topology, then, by the triangle inequality, $\sup_{n\geq 1} \Ws_1(m^*,m_n)<+\infty$. Now for $M>0$, $$m_n(\{x:|x|>M\})\leq m^*(\{x:|x|>M/2\})+\frac{2}{M}\Ws_1(m^*,m_n),$$ so that the sequence $(m_n)_{n\geq 1}$ is tight. One may extract a subsequence converging weakly to $m_\infty$. From the lower semicontinuity of $\Ws_1$ stated in Lemma \ref{lem:lsc}, one easily checks that $m_\infty\in\dP_{m^*}$ and $\Ws_1(m_n,m_\infty)$ tends to $0$ as $n$ grows to infinity.

  Let us address separability. For all integers $M \geq 2$, let us denote by $\dP_{m^*}^{0,M}$ the set of probability measures on $\R$ equal to the sum of the image $m^*_M$ of the Lebesgue measure on $[0,1/M]\cup [1-1/M,1]$ by $(H*m^*)^{-1}$ and a finite linear combination of Dirac masses at rational points with rational coefficients. 
We prove that the countable set
  \begin{equation*}
    \dP_{m^*}^0 := \bigcup_{M \geq 2} \dP_{m^*}^{0,M}
  \end{equation*}
  is dense in $\dP_{m^*}$.  To this aim, we fix $m \in \dP_{m^*}$ and $\epsilon > 0$. For $M$ large enough,
  \begin{equation*}
    \int_{u=0}^{1/M}|(H*m)^{-1}(u)-(H*m^*)^{-1}(u)|\dd u + \int_{u=1-1/M}^1|(H*m)^{-1}(u)-(H*m^*)^{-1}(u)|\dd u \leq \frac{\epsilon}{2}.
  \end{equation*}
  It follows from the proof of \cite[Theorem~6.18]{villani}, that the image $m^M$ of the uniform probability measure on $[1/M,1-1/M]$ by $(H*m)^{-1}$ may be approximated by a finite linear combination $\sum_{j=1}^J a_j\delta_{x_j}$ of Dirac masses at rational points with rational coefficients  so that $\Ws_1(m^M,\sum_{j=1}^J a_j\delta_{x_j})\leq \epsilon/2$. Now 
  \begin{equation*}
    \begin{aligned}
      \Ws_1\left(m,m^*_M+\sum_{j=1}^J \frac{(M-2)a_j}{M}\delta_{x_j}\right) & \leq \int_{u=0}^{1/M}|(H*m)^{-1}(u)-(H*m^*)^{-1}(u)|\dd u\\
      & \quad + \frac{M-2}{M}\Ws_1\left(m^M,\sum_{j=1}^J a_j\delta_{x_j}\right)\\
      & \quad + \int_{u=1-1/M}^1|(H*m)^{-1}(u)-(H*m^*)^{-1}(u)|\dd u\\
      & \leq \epsilon,
    \end{aligned}
  \end{equation*}
which concludes the proof.
\end{proof}

In order to work with a distance on $\Ps(\R)$ that can be compared with the Wasserstein distance of order $1$, but is weaker and only metrises weak convergence, it shall be useful to introduce the following {\em modified Wasserstein distance}.

\begin{defi}[Modified Wasserstein distance]\label{defi:tWs}
  For all $m, m' \in \Ps(\R)$, let us define the {\em modified Wasserstein distance} $\tilde{\Ws}_1(m,m')$ by
  \begin{equation*}
    \tilde{\Ws}_1(m, m') := \inf_{\mathfrak{m} <^m_{m'}} \int_{(x,x') \in \R^2} (|x-x'| \wedge 1) \mathfrak{m}(\dd x\dd x'),
  \end{equation*}
  with the same notations as in the Definition~\ref{defi:wass} of the Wasserstein distance.  
\end{defi}

It is clear that, for all $m,m' \in \Ps(\R)$, $\tilde{\Ws}_1(m,m') \leq \Ws_1(m,m')$. Besides, a sequence $(m_n)_{n \geq 1}$ converges weakly to $m \in \Ps(\R)$ if and only if $\tilde{\Ws}_1(m_n,m)$ converges to $0$; this follows from~\cite[Corollary~6.13]{villani} since the distances $|x-x'|$ and $|x-x'|\wedge 1$ induce the same topology on $\R$.


\subsubsection{The discretisation operator} Recall the Definition~\ref{defi:chi} of the discretisation operator. The convergence properties of this operator are addressed in Lemma~\ref{lem:cvci} for the weak convergence of marginals and Lemma~\ref{lem:cvciwass} for the Wasserstein distance.

\begin{lem}[Weak convergence of the initial discretisation]\label{lem:cvci}
  Let $\bm = (m^1, \ldots, m^d) \in \Ps(\R)^d$. For all $n \geq 1$, let us denote $\x(n) := \chi_n \bm$. Then, for all $\gamma \in \{1, \ldots, d\}$, the sequence of empirical distributions
  \begin{equation*}
    m^{\gamma}_n := \frac{1}{n} \sum_{k=1}^n \delta_{x_k^{\gamma}(n)} \in \Ps(\R)
  \end{equation*}
  converges weakly to the probability measure $m^{\gamma}$.
\end{lem}
\begin{proof}
  For all $n \geq 1$, for all $\gamma:k \in \Part$, let us define
  \begin{equation*}
    x_k^{\gamma,-}(n) := (H*m^{\gamma})^{-1} \left(\frac{2k-1}{2(n+1)}\right), \qquad x_k^{\gamma,+}(n) := (H*m^{\gamma})^{-1} \left(\frac{2k+1}{2(n+1)}\right),
  \end{equation*}
  so that $x_k^{\gamma,-}(n) \leq x_k^{\gamma}(n) \leq x_k^{\gamma,+}(n)$. Fix $\gamma \in \{1, \ldots, d\}$ and define the probability measures $m_n^{\gamma,-}$ and $m_n^{\gamma,+}$ on $\R$ by
  \begin{equation*}
    m_n^{\gamma,\pm} := \frac{1}{n} \sum_{k=1}^n \delta_{x_k^{\gamma,\pm}(n)} = \Unif_n^{\pm} \circ (H*m^{\gamma})^{-1},
  \end{equation*}
  where
  \begin{equation*}
    \Unif_n^{\pm} := \frac{1}{n}\sum_{k=1}^n \delta_{(2k \pm 1)/(2(n+1))} \in \Ps([0,1]).
  \end{equation*}
  By an elementary Riemann sum argument, both $\Unif_n^-$ and $\Unif_n^+$ converge weakly to the Lebesgue measure $\Unif$ on $[0,1]$. By the Mapping Theorem~\cite[Theorem~2.7, p.~21]{billingsley}, we deduce that both $m_n^{\gamma,-}$ and $m_n^{\gamma,+}$ converge weakly to $m^{\gamma}$. On the other hand, it follows from the definition of $x^{\gamma,-}_k(n)$ and $x^{\gamma,+}_k(n)$ that for all $x \in \R$,
  \begin{equation*}
    H*m_n^{\gamma,-}(x) \geq H*m_n^{\gamma}(x) \geq H*m_n^{\gamma,+}(x).
  \end{equation*}
  By Lemma~\ref{lem:cvCDF}, we deduce that, for all $x \in \R$ such that $m^{\gamma}(\{x\})=0$, both the left- and right-hand side above converge to $H*m^{\gamma}(x)$, and by the squeeze lemma, so does $H*m_n^{\gamma}(x)$. By Lemma~\ref{lem:cvCDF} again, we conclude that $m^{\gamma}_n$ converges weakly to $m^{\gamma}$.
\end{proof}
Note that a slight generalisation of the proof, based on the second example in~\cite[Example~2.3, p.~18]{billingsley}, actually allows to prove that the sequence of empirical distributions
\begin{equation*}
  \bar{m}_n := \frac{1}{n} \sum_{k=1}^n \delta_{(x_k^1(n), \ldots, x_k^d(n))} \in \Ps(\R^d)
\end{equation*}
converges weakly to the probability measure $\bar{m} \in \Ps(\R^d)$ defined by
\begin{equation*}
  \bar{m} := \Unif \circ \left((H*m^1)^{-1}, \ldots, (H*m^d)^{-1}\right)^{-1},
\end{equation*}
where $\Unif$ refers to the Lebesgue measure on $[0,1]$. Of course, the marginal distributions of $\bar{m}$ are $m^1, \ldots, m^d$.

We now address the convergence in Wasserstein distance of order $1$ of the discretisation operator. If there exists $\gamma \in \{1, \ldots, d\}$ such that the first order moment of $m^{\gamma}$ is infinite, then we cannot expect the $\Ws^{(d)}_1$ distance between the empirical distribution associated with $\chi_n\bm$ and $\bm$ to converge to $0$, as this distance is always infinite. We however have the following finer result.

\begin{lem}[Wasserstein convergence of the initial discretisation]\label{lem:cvciwass}
  Let $\bm, \bm' \in \dP_{\bm^*}$. Then, for all $p \in [1,+\infty]$,
  \begin{equation*}
    \forall n \geq 1, \qquad ||\chi_n\bm - \chi_n\bm'||_p \leq \left(\frac{n+1}{n}\right)^{1/p} \Ws^{(d)}_p(\bm, \bm'),
  \end{equation*}
  where we take the obvious convention that $((n+1)/n)^{1/\infty}=1$, and
  \begin{equation*}
    \lim_{n \to +\infty} ||\chi_n\bm - \chi_n\bm'||_p = \Ws^{(d)}_p(\bm, \bm').
  \end{equation*}
\end{lem}
\begin{proof}
  Let us fix $\bm=(m^1, \ldots, m^d), \bm'=(m'^1, \ldots, m'^d) \in \dP$ and $\gamma \in \{1, \ldots, d\}$.
  
  On the one hand, recall that by by Remark~\ref{rk:wassemp}, $||\chi_n\bm - \chi_n\bm'||_p$ is the $\Ws^{(d)}_p$ distance between the empirical distributions of $\chi_n \bm$ and $\chi_n \bm'$, therefore by Lemma~\ref{lem:cvci} and Lemma~\ref{lem:lsc}, we deduce that
  \begin{equation*}
    \liminf_{n \to +\infty} ||\chi_n\bm - \chi_n\bm'||_p \geq \Ws^{(d)}_p(\bm, \bm'),
  \end{equation*}
  for all $p \in [1,+\infty]$.
  
  On the other hand, for all $p \in [1,+\infty)$, the Jensen inequality yields
  \begin{equation*}
    \begin{aligned}
      ||\chi_n\bm - \chi_n\bm'||_p^p & = \frac{1}{n} \sum_{\gamma=1}^d \sum_{k=1}^n \left|(n+1)\int_{v=(2k-1)/(2(n+1))}^{(2k+1)/(2(n+1))} \left((H*m^{\gamma})^{-1}(v) - (H*m'^{\gamma})^{-1}(v)\right)\dd v\right|^p\\
      & \leq \frac{n+1}{n} \sum_{\gamma=1}^d \int_{v=1/(2(n+1))}^{1-1/(2(n+1))} \left|(H*m^{\gamma})^{-1}(v) - (H*m'^{\gamma})^{-1}(v)\right|^p\dd v\\
      & \leq \frac{n+1}{n} (\Ws^{(d)}_p(\bm, \bm'))^p,
    \end{aligned}
  \end{equation*}
  therefore
  \begin{equation*}
    ||\chi_n\bm - \chi_n\bm'||_p \leq \left(\frac{n+1}{n}\right)^{1/p} \Ws^{(d)}_p(\bm, \bm'),
  \end{equation*}  
  and consequently
  \begin{equation*}
    \limsup_{n \to +\infty} ||\chi_n\bm - \chi_n\bm'||_p \leq \Ws^{(d)}_p(\bm, \bm'),
  \end{equation*}
  which completes the proof for $p < +\infty$. The case $p=+\infty$ is similar --- actually easier.
\end{proof}


\subsubsection{Construction of the operators $(\bS_t)_{t \geq 0}$}\label{sss:St} By Lemma~\ref{lem:cvci}, for all $\bm \in \Ps(\R)^d$, the sequence of initial configurations $(\chi_n\bm)_{n \geq 1}$ satisfies the assumptions of Theorem~\ref{theo:existence}. Therefore there exists an increasing sequence of integers $(n_{\ell})_{\ell \geq 1}$ along which $\bu[\chi_{n_{\ell}}\bm]$ converges to a probabilistic solution to the system~\eqref{eq:syst} with initial data defined by $u^{\gamma}_0 = H*m^{\gamma}$. The sequence $(n_{\ell})_{\ell \geq 1}$ depends on $\bm$. In the following Proposition~\ref{prop:cSt}, we remove this dependency, which enables us to construct stable semigroup solutions to~\eqref{eq:syst}. 

\begin{prop}[Construction of the operators $(\bS_t)$]\label{prop:cSt}
  Under Assumptions~\eqref{ass:LC} and~\eqref{ass:USH}, let us fix $\bm^* \in \Ps(\R)^d$ and let $\mathcal{N}$ be an unbounded set of positive integers. Then, there exists an increasing sequence $(n_{\ell})_{\ell \geq 1} \subset \mathcal{N}$ such that, for all $\bm \in \dP_{\bm^*}$, the empirical distribution $\upmu[\chi_{n_{\ell}}\bm]$ of the MSPD started at $\chi_{n_{\ell}}\bm$ converges weakly to some probability measure $\bar{\upmu}[\bm] \in \Ms$ when $\ell$ grows to infinity. For all $t \geq 0$, let us denote by 
  \begin{equation}\label{eq:defbS}
    \bS_t\bm = (S^1_t\bm, \ldots, S^d_t\bm) := (\bar{\upmu}^1_t[\bm], \ldots, \bar{\upmu}^d_t[\bm])
  \end{equation}
  the vector of associated marginal distributions. 
  
  The family of operators $\bS_t : \dP_{\bm^*} \to \Ps(\R)^d$ has the following properties.
  \begin{enumerate}[label=(\roman*), ref=\roman*]
    \item\label{it:cSt:1} For all $\bm \in \dP_{\bm^*}$, the vector of CDFs $\bu = (u^1, \ldots, u^d)$ defined by $u^{\gamma}(t,\cdot) = H*(S^{\gamma}_t\bm)$ is a probabilistic solution to the system~\eqref{eq:syst} with initial data $(u^1_0, \ldots, u^d_0)$ defined by $u^{\gamma}_0 = H*m^{\gamma}$. Besides, the measure $\bar{\upmu}[\bm]$ is the image of the Lebesgue measure $\Unif$ on $[0,1]$ by the mapping
    \begin{equation*}
      v \mapsto \left(u^1(t,\cdot)^{-1}(v), \ldots, u^d(t, \cdot)^{-1}(v)\right)_{t \geq 0}.
    \end{equation*}
    \item\label{it:cSt:2} For all $t \geq 0$, for all $\bm \in \dP_{\bm^*}$, $\bS_t\bm \in \dP_{\bm^*}$. Besides, for all $p \in [1,+\infty]$, for all $\bm, \bm' \in \dP_{\bm^*}$, for all $s,t \geq 0$,
    \begin{equation*}
      \Ws^{(d)}_p(\bS_s\bm, \bS_t\bm') \leq \ConstStab_p \Ws^{(d)}_p(\bm, \bm') + |t-s| \ConstBound{p},
    \end{equation*}
    where $\ConstBound{p}$ is defined in~\eqref{eq:ConstBound} and $\ConstStab_p$ is defined in~\eqref{eq:ConstStab}.
    \item\label{it:cSt:3} The family of operators $(\bS_t)_{t \geq 0}$ is a semigroup on $\dP_{\bm^*}$.
  \end{enumerate}
\end{prop}

Following Lemma~\ref{lem:dP}, the space $\dP_{\bm^*}$ metrised by the $\Ws^{(d)}_1$ distance contains a countable and dense subset $\dP_{\bm^*}^0$. By Proposition~\ref{prop:tightness}, for all $\bm \in \dP_{\bm^*}^0$, the family of probability measures $(\upmu[\chi_n\bm])_{n \in \mathcal{N}}$ is tight. By a diagonal extraction procedure, we obtain that there exists an increasing sequence of integers $(n_{\ell})_{\ell \geq 1} \subset \mathcal{N}$ such that, for all $\bm \in \dP_{\bm^*}^0$, $\upmu[\chi_{n_{\ell}}\bm]$ converges weakly to a probability measure $\bar{\upmu}[\bm]$ in $\Ms$. As a consequence, we first define the operator $\bS_t$ on the subset $\dP^0_{\bm^*}$ by~\eqref{eq:defbS}. This operator enjoys the following properties.

\begin{lem}[Stability on $\dP^0_{\bm}$]\label{lem:bmuindP}
  Under the assumptions of Proposition~\ref{prop:cSt},
  \begin{enumerate}[label=(\roman*), ref=\roman*]
    \item\label{it:bmuindP:1} for all $\bm \in \dP^0_{\bm}$, for all $t \geq 0$, $\bS_t\bm \in \dP_{\bm^*}$,
    \item\label{it:bmuindP:2} for all $\bm, \bm' \in \dP^0_{\bm^*}$,
    \begin{equation}\label{eq:lipSt}
      \sup_{t \geq 0} \Ws^{(d)}_1(\bS_t\bm, \bS_t\bm') \leq \ConstStab_1 \Ws^{(d)}_1(\bm, \bm').
    \end{equation}
  \end{enumerate}
\end{lem}
\begin{proof}[Proof of~\eqref{it:bmuindP:1}.]
  Let $\bm=(m^1, \ldots, m^d) \in \dP_{\bm^*}^0$. Following Definition~\ref{defi:dP}, since $\Ws^{(d)}_1(\bm^*, \bm) < +\infty$, it suffices to check that, for all $t \geq 0$,
  \begin{equation*}
    \sum_{\gamma=1}^d \Ws_1(m^{\gamma}, \bar{\upmu}_t^{\gamma}[\bm]) < +\infty.
  \end{equation*}
  Combining Proposition~\ref{prop:tightness} and Lemma~\ref{lem:lsc}, we have
  \begin{equation*}
    \sum_{\gamma=1}^d \Ws_1(m^{\gamma}, \bar{\upmu}_t^{\gamma}[\bm]) \leq \liminf_{\ell \to +\infty} \sum_{\gamma=1}^d \Ws_1(\upmu_0^{\gamma}[\chi_{n_{\ell}}\bm],\upmu_t^{\gamma}[\chi_{n_{\ell}}\bm]).
  \end{equation*}
  But using Remark~\ref{rk:wassemp} and Theorem \ref{theo:stabMSPD}, we rewrite
  \begin{equation*}
    \sum_{\gamma=1}^d \Ws_1(\upmu_0^{\gamma}[\chi_{n_{\ell}}\bm], \upmu_t^{\gamma}[\chi_{n_{\ell}}\bm]) = ||\chi_{n_{\ell}}\bm-\Phi(\chi_{n_{\ell}}\bm;t)||_1 \leq t \ConstBound{1},
  \end{equation*}
  which completes the proof.

  \sk
  \noindent{\em Proof of~\eqref{it:bmuindP:2}.} Let $\bm, \bm' \in \dP_{\bm^*}^0$. By Proposition~\ref{prop:tightness} and Lemma~\ref{lem:lsc}, we have
  \begin{equation*}
    \Ws^{(d)}_1(\bS_t\bm, \bS_t\bm') \leq \liminf_{\ell \to +\infty} \sum_{\gamma=1}^d \Ws_1(\upmu_t^{\gamma}[\chi_{n_{\ell}}\bm], \upmu_t^{\gamma}[\chi_{n_{\ell}}\bm]).
  \end{equation*}
  By Theorem~\ref{theo:stabMSPD},
  \begin{equation*}
    \sum_{\gamma=1}^d \Ws_1(\upmu_t^{\gamma}[\chi_{n_{\ell}}\bm], \upmu_t^{\gamma}[\chi_{n_{\ell}}\bm]) = ||\Phi(\chi_{n_{\ell}}\bm;t) - \Phi(\chi_{n_{\ell}}\bm';t)||_1 \leq \ConstStab_1 ||\chi_{n_{\ell}}\bm - \chi_{n_{\ell}}\bm'||_1,
  \end{equation*}
  and by Lemma~\ref{lem:cvciwass},
  \begin{equation*}
    \lim_{\ell \to +\infty} ||\chi_{n_{\ell}}\bm - \chi_{n_{\ell}}\bm'||_1 = \Ws^{(d)}_1(\bm, \bm'),
  \end{equation*}
  which completes the proof.
\end{proof}

Since the set $\dP_{\bm^*}^0$ is dense in $\dP_{\bm^*}$ and the latter is complete for the $\Ws^{(d)}_1$ distance, we deduce that the operator $\bS_t$ possesses a unique continuous extension to $\dP_{\bm^*}$, which satisfies the same stability estimate~\eqref{eq:lipSt}, for all $\bm, \bm' \in \dP_{\bm^*}$. We now check that $\bS_t\bm$ coincides with the weak limit, when $\ell$ grows to infinity, of $(\upmu_t^1[\chi_{n_{\ell}}\bm], \ldots, \upmu_t^d[\chi_{n_{\ell}}\bm])$.

\begin{lem}[Identification of $\bS_t$]\label{lem:cvbS}
  Under the assumptions of Proposition~\ref{prop:cSt}, for all $t \geq 0$, let $\bS_t$ be defined on $\dP_{\bm^*}$ as above. Then for all $\bm \in \dP_{\bm^*}$, for all $\gamma \in \{1, \ldots, d\}$, $\upmu_t^{\gamma}[\chi_{n_{\ell}}\bm]$ converges weakly to $S^{\gamma}_t \bm$ in $\Ps(\R)$.
\end{lem}
\begin{proof}
  Recall that, if $\bm \in \dP^0_{\bm^*}$, then the result is nothing but the definition of $\bS_t$. Otherwise, we argue as follows. Recalling the Definition~\ref{defi:tWs} of the modified Wasserstein distance $\tilde{\Ws}_1$, we prove that
  \begin{equation}\label{eq:limtWSmu}
    \lim_{\ell \to +\infty} \sum_{\gamma=1}^d \tilde{\Ws}_1(S_t^{\gamma}\bm, \upmu_t^{\gamma}[\chi_{n_{\ell}}\bm]) = 0.
  \end{equation}
  To this aim, let us fix $\epsilon > 0$ and $\bm'=(m'^1, \ldots, m'^d) \in \dP_{\bm^*}^0$ such that $\Ws^{(d)}_1(\bm, \bm') \leq \epsilon$. Then, for all $t \geq 0$, for all $\ell \geq 1$,
  \begin{equation*}
    \begin{aligned}
      & \sum_{\gamma=1}^d \tilde{\Ws}_1(S_t^{\gamma}\bm, \upmu_t^{\gamma}[\chi_{n_{\ell}}\bm])\\
      & \qquad \leq \sum_{\gamma=1}^d \tilde{\Ws}_1(S_t^{\gamma}\bm, S_t^{\gamma}\bm') + \sum_{\gamma=1}^d \tilde{\Ws}_1(S_t^{\gamma}\bm', \upmu_t^{\gamma}[\chi_{n_{\ell}}\bm']) + \sum_{\gamma=1}^d \tilde{\Ws}_1(\upmu_t^{\gamma}[\chi_{n_{\ell}}\bm'], \upmu_t^{\gamma}[\chi_{n_{\ell}}\bm]).
    \end{aligned}
  \end{equation*}
  By the properties of $\tilde{\Ws}_1$ and~\eqref{eq:lipSt},
  \begin{equation*}
    \sum_{\gamma=1}^d \tilde{\Ws}_1(S_t^{\gamma}\bm, S_t^{\gamma}\bm') \leq \sum_{\gamma=1}^d \Ws_1(S_t^{\gamma}\bm, S_t^{\gamma}\bm') \leq \ConstStab_1\sum_{\gamma=1}^d \Ws_1(m^{\gamma}, m'^{\gamma}) \leq \ConstStab_1\epsilon.
  \end{equation*}
  By the properties of $\tilde{\Ws}_1$ and the construction of the sequence $(n_{\ell})_{\ell \geq 1}$,
  \begin{equation*}
    \lim_{\ell \to +\infty} \sum_{\gamma=1}^d \tilde{\Ws}_1(S_t^{\gamma}\bm', \upmu_t^{\gamma}[\chi_{n_{\ell}}\bm']) = 0.
  \end{equation*}
  By the properties of $\tilde{\Ws}_1$ and Theorem~\ref{theo:stabMSPD},
  \begin{equation*}
    \begin{aligned}
      \sum_{\gamma=1}^d \tilde{\Ws}_1(\upmu_t^{\gamma}[\chi_{n_{\ell}}\bm'], \upmu_t^{\gamma}[\chi_{n_{\ell}}\bm]) & \leq \sum_{\gamma=1}^d \Ws_1(\upmu_t^{\gamma}[\chi_{n_{\ell}}\bm'], \upmu_t^{\gamma}[\chi_{n_{\ell}}\bm])\\
      & \leq \ConstStab_1 \sum_{\gamma=1}^d \Ws_1(\upmu_0^{\gamma}[\chi_{n_{\ell}}\bm'], \upmu_0^{\gamma}[\chi_{n_{\ell}}\bm]),
    \end{aligned}
  \end{equation*}
  and it follows from Lemma~\ref{lem:cvciwass} that
  \begin{equation*}
    \lim_{\ell \to +\infty} \sum_{\gamma=1}^d \Ws_1(\upmu_0^{\gamma}[\chi_{n_{\ell}}\bm'], \upmu_0^{\gamma}[\chi_{n_{\ell}}\bm]) = \sum_{\gamma=1}^d \Ws_1(m^{\gamma}, m'^{\gamma}) \leq \epsilon.
  \end{equation*}
  As a consequence, we have
  \begin{equation*}
    \limsup_{\ell \to +\infty} \sum_{\gamma=1}^d \tilde{\Ws}_1(S_t^{\gamma}\bm, \upmu_t^{\gamma}[\chi_{n_{\ell}}\bm]) \leq 2\ConstStab_1 \epsilon,
  \end{equation*}
  and we obtain~\eqref{eq:limtWSmu} by letting $\epsilon$ vanish.
\end{proof}

At this stage, we have constructed a sequence $(n_{\ell})_{\ell \geq 1}$ and a family of operators $\bS_t : \dP_{\bm^*} \to \dP_{\bm^*}$ such that, for all $\bm \in \dP_{\bm^*}$, for all $t \geq 0$, for all $\gamma \in \{1, \ldots, d\}$, 
\begin{equation*}
  S_t^{\gamma}\bm = \lim_{\ell \to +\infty} \upmu^{\gamma}_t[\chi_{n_{\ell}}\bm].
\end{equation*}

We are now ready to complete the proof of Proposition~\ref{prop:cSt}.

\begin{proof}[Proof of~\eqref{it:cSt:1} in Proposition~\ref{prop:cSt}] The first part of the point~\eqref{it:cSt:1} is a corollary of Lemma~\ref{lem:cvbS}: indeed, by the same arguments as in the proof of Theorem~\ref{theo:existence}, the limit of $\upmu^{\gamma}_t[\chi_{n_{\ell}}\bm]$ induces a probabilistic solution to~\eqref{eq:syst}. 

We now show that not only the marginal distributions, but the whole empirical distributions $\upmu[\chi_{n_{\ell}}\bm]$ converge in $\Ms$. By Proposition~\ref{prop:tightness}, the sequence $(\upmu[\chi_{n_{\ell}}\bm])_{\ell \geq 1}$ is tight in $\Ms$; on the other hand, Lemma~\ref{lem:cvbS} shows that the limit $\bar{\upmu}$ of any converging subsequence has marginal distributions given by $\bS_t\bm$. Calling $\bu$ the associated probabilistic solution to~\eqref{eq:syst}, we deduce from Remark~\ref{rk:unicitebmu} that $\bar{\upmu}$ rewrites $\Unif \circ ((u^1(t, \cdot)^{-1}, \ldots, u^d(t, \cdot)^{-1})_{t \geq 0})^{-1}$. This shows that all the converging subsequences of $(\upmu[\chi_{n_{\ell}}\bm])_{\ell \geq 1}$ have the same limit, and therefore implies~\cite[Corollary, p.~59]{billingsley} the whole convergence of $\upmu[\chi_{n_{\ell}}\bm]$ to $\bar{\upmu}[\bm] := \Unif \circ ((u^1(t, \cdot)^{-1}, \ldots, u^d(t, \cdot)^{-1})_{t \geq 0})^{-1}$ in $\Ms$.

\sk
\noindent{\em Proof of~\eqref{it:cSt:2} in Proposition~\ref{prop:cSt}.} It follows from the construction of $\bS_t$ that the latter takes its values in the $\Ws_1$ stability class $\dP_{\bm^*}$. Now let $p \in [1, +\infty]$, $\bm, \bm' \in \dP_{\bm^*}$ and $s, t \geq 0$. By Theorem~\ref{theo:stabMSPD}, for all $\ell \geq 1$,
  \begin{equation*}
    \Ws^{(d)}_p(\upmu_s^{\gamma}[\chi_{n_{\ell}}\bm], \upmu_t^{\gamma}[\chi_{n_{\ell}}\bm']) \leq \ConstStab_p \Ws^{(d)}_p(\upmu_0^{\gamma}[\chi_{n_{\ell}}\bm], \upmu_0^{\gamma}[\chi_{n_{\ell}}\bm']) + |t-s| \ConstBound{p}.
  \end{equation*}  
  We obtain the expected result by applying the lower semicontinuity property of the Wasserstein distance of Lemma~\ref{lem:lsc} and the convergence result of the initial discretisation of Lemma~\ref{lem:cvciwass}.

\sk
\noindent{\em Proof of~\eqref{it:cSt:3} in Proposition~\ref{prop:cSt}.} Let $\bm \in \dP_{\bm^*}$ and let $s,t \geq 0$. To show the semigroup property, we shall prove that
  \begin{equation*}
    \sum_{\gamma=1}^d \tilde{\Ws}_1(S_{t+s}^{\gamma}\bm, S_t^{\gamma}\bS_s\bm) = 0,
  \end{equation*}
  where the modified Wasserstein distance $\tilde{\Ws}_1$ was introduced in Definition~\ref{defi:tWs}. In this purpose, we first remark that, by the flow property for the MSPD, for all $\ell \geq 1$,
  \begin{equation*}
    \forall \gamma \in \{1, \ldots, d\}, \qquad \upmu_{t+s}^{\gamma}[\chi_{n_{\ell}}\bm] = \upmu_t^{\gamma}[\Phi(\chi_{n_{\ell}}\bm;s)]
  \end{equation*}
  therefore we write
  \begin{equation*}
    \sum_{\gamma=1}^d \tilde{\Ws}_1(S_{t+s}^{\gamma}\bm, S_t^{\gamma}\bS_s\bm) \leq \sum_{\gamma=1}^d \tilde{\Ws}_1(S_{t+s}^{\gamma}\bm, \upmu_{t+s}^{\gamma}[\chi_{n_{\ell}}\bm]) + \sum_{\gamma=1}^d \tilde{\Ws}_1(\upmu_t^{\gamma}[\Phi(\chi_{n_{\ell}}\bm;s)], S_t^{\gamma}\bS_s\bm).
  \end{equation*}
  On the one hand, Lemma~\ref{lem:cvbS} yields
  \begin{equation*}
    \lim_{\ell \to +\infty} \sum_{\gamma=1}^d \tilde{\Ws}_1(S_{t+s}^{\gamma}\bm, \upmu_{t+s}^{\gamma}[\chi_{n_{\ell}}\bm]) = 0.
  \end{equation*}
  On the other hand,
  \begin{equation*}
    \begin{aligned}
      & \sum_{\gamma=1}^d \tilde{\Ws}_1(\upmu_t^{\gamma}[\Phi(\chi_{n_{\ell}}\bm;s)], S_t^{\gamma}\bS_s\bm)\\
      & \qquad \leq \sum_{\gamma=1}^d \tilde{\Ws}_1(\upmu_t^{\gamma}[\Phi(\chi_{n_{\ell}}\bm;s)], \upmu_t^{\gamma}[\chi_{n_{\ell}}\bS_s\bm]) + \sum_{\gamma=1}^d \tilde{\Ws}_1(\upmu_t^{\gamma}[\chi_{n_{\ell}}\bS_s\bm], S_t^{\gamma}\bS_s\bm),
    \end{aligned}
  \end{equation*}
  and using Lemma~\ref{lem:cvbS} again, we have
  \begin{equation*}
    \lim_{\ell \to +\infty} \sum_{\gamma=1}^d \tilde{\Ws}_1(\upmu_t^{\gamma}[\chi_{n_{\ell}}\bS_s\bm], S_t^{\gamma}\bS_s\bm) = 0.
  \end{equation*}
  
  It therefore remains to prove that
  \begin{equation*}
    \lim_{\ell \to +\infty} \sum_{\gamma=1}^d \tilde{\Ws}_1(\upmu_t^{\gamma}[\Phi(\chi_{n_{\ell}}\bm;s)], \upmu_t^{\gamma}[\chi_{n_{\ell}}\bS_s\bm]) = 0.
  \end{equation*}
  In this purpose, we use the domination of $\tilde{\Ws}_1$ by $\Ws_1$ and Theorem~\ref{theo:stabMSPD} to obtain
  \begin{equation*}
    \begin{aligned}
      \sum_{\gamma=1}^d \tilde{\Ws}_1(\upmu_t^{\gamma}[\Phi(\chi_{n_{\ell}}\bm;s)], \upmu_t^{\gamma}[\chi_{n_{\ell}}\bS_s\bm]) & \leq \ConstStab_1\sum_{\gamma=1}^d \Ws_1(\upmu_0^{\gamma}[\Phi(\chi_{n_{\ell}}\bm;s)], \upmu_0^{\gamma}[\chi_{n_{\ell}}\bS_s\bm])\\
      & \leq \ConstStab_1 ||\Phi(\chi_{n_{\ell}}\bm;s)-\chi_{n_{\ell}}\bS_s\bm||_1.
    \end{aligned}
  \end{equation*}
  We somehow have to prove that the evolution along the MSPD for a time $s$ asymptotically commutes with the discretisation operation when measured in $\Ws^{(d)}_1$ distance. Let us first note that this is the case for the weak convergence: by Lemma~\ref{lem:cvci}, the empirical distribution of $\chi_{n_{\ell}}\bS_s\bm$ converges weakly to $\bS_s\bm$; while it follows from Lemma~\ref{lem:cvbS} that the empirical distribution of $\Phi(\chi_{n_{\ell}}\bm;s)$ converges weakly to the same limit $\bS_s\bm \in \Ps(\R)^d$.
  
  Let us now remark that by Theorem \ref{theo:stabMSPD}, and Lemma \ref{lem:cvciwass}, 
  \begin{equation*}
    \begin{aligned}
      ||\Phi(\chi_{n_{\ell}}\bm;s)-\chi_{n_{\ell}}\bS_s\bm||_{\infty}& \leq ||\Phi(\chi_{n_{\ell}}\bm;s)-\chi_{n_{\ell}}\bm||_{\infty} + ||\chi_{n_{\ell}}\bm-\chi_{n_{\ell}}\bS_s\bm||_{\infty}\\
      & \leq s \ConstBound{\infty} + \Ws^{(d)}_{\infty}(\bm, \bS_s\bm),
    \end{aligned}
  \end{equation*}
  and it follows from the point~\eqref{it:cSt:2} of Proposition~\ref{prop:cSt} that $\Ws^{(d)}_{\infty}(\bm, \bS_s\bm) \leq s \ConstBound{\infty}$. As a consequence, the right-hand side above is lower than $2s \ConstBound{\infty}$; therefore, letting $\x(n_{\ell}) := \Phi(\chi_{n_{\ell}}\bm;s)$ and $\y(n_{\ell}) := \chi_{n_{\ell}}\bS_s\bm$,
  \begin{equation*}
    \begin{aligned}
      ||\Phi(\chi_{n_{\ell}}\bm;s)-\chi_{n_{\ell}}\bS_s\bm||_1 & = \frac{1}{n_{\ell}} \sum_{\gamma=1}^d \sum_{k=1}^{n_{\ell}} |x_k^{\gamma}(n_{\ell})-y_k^{\gamma}(n_{\ell})|\\
      & = \frac{1}{n_{\ell}} \sum_{\gamma=1}^d \sum_{k=1}^{n_{\ell}} |x_k^{\gamma}(n_{\ell})-y_k^{\gamma}(n_{\ell})| \wedge (2s \ConstBound{\infty})\\
      & = \sum_{\gamma=1}^d \int_{(x,y) \in \R^2} |x-y| \wedge (2s \ConstBound{\infty}) \mathfrak{m}^{\gamma}_{n_{\ell}}(\dd x\dd y),
    \end{aligned}
  \end{equation*}
  where, for all $\gamma \in \{1, \ldots, d\}$, the probability measure $\mathfrak{m}^{\gamma}_{n_{\ell}}$ on $\R^2$ is defined by
  \begin{equation*}
    \mathfrak{m}^{\gamma}_{n_{\ell}}= \frac{1}{n_{\ell}} \sum_{k=1}^{n_{\ell}} \delta_{(x_k^{\gamma}(n_{\ell}), y_k^{\gamma}(n_{\ell}))}.
  \end{equation*}
  This probability measure rewrites $\Unif\circ((H*\upmu_0^{\gamma}[\Phi(\chi_{n_{\ell}}\bm;s)])^{-1},(H*\upmu_0^{\gamma}[\chi_{n_{\ell}}\bS_s\bm])^{-1})^{-1}$. Since both $\upmu_0^{\gamma}[\Phi(\chi_{n_{\ell}}\bm;s)]$ and $\upmu_0^{\gamma}[\chi_{n_{\ell}}\bS_s\bm]$ converge weakly to $S^\gamma_s\bm \in \Ps(\R)$, one deduces from Lemma \ref{lem:cvCDF} that $\mathfrak{m}^{\gamma}_{n_{\ell}}$ converges weakly to $\Unif\circ((H*S^\gamma_s\bm )^{-1},(H*S^\gamma_s\bm )^{-1})^{-1}$, which gives full measure to the diagonal in $\R^2$. Since $(x,y)\mapsto |x-y| \wedge (2s \ConstBound{\infty})$ is continuous and bounded, we conclude that $||\Phi(\chi_{n_{\ell}}\bm;s)-\chi_{n_{\ell}}\bS_s\bm||_1$ tends to $0$.
\end{proof}

The semigroup $(\bS_t)_{t \geq 0}$ still depends on the choice of the sequence $(n_{\ell})_{\ell \geq 1}$. We get rid of this dependency by introducing the uniqueness conditions of Bianchini and Bressan in the next subsection.


\subsection{The Bianchini-Bressan uniqueness conditions}\label{ss:BB} In this subsection, we introduce and adapt the Bianchini-Bressan uniqueness conditions of~\cite{bianbres} to obtain the following lemma.

\begin{lem}[Uniqueness of the semigroup]\label{lem:uniqsg}
  Under Assumptions~\eqref{ass:LC} and~\eqref{ass:USH}, there exists a family of operators $(\bar{\bS}_t)_{t \geq 0}$ on $\Ps(\R)^d$ such that, for all $\bm^* \in \Ps(\R)^d$, any semigroup $(\bS_t)_{t \geq 0}$ on $\dP_{\bm^*}$ obtained by Proposition~\ref{prop:cSt} coincides with the restriction of $(\bar{\bS}_t)_{t \geq 0}$ to $\dP_{\bm^*}$.
\end{lem}
The Bianchini-Bressan conditions are introduced in~\S\ref{sss:BB} below, where we also show that the probabilistic solutions associated with any semigroup obtained by Proposition~\ref{prop:cSt} satisfy these conditions. In order to introduce these conditions it is necessary to give a proper meaning to the {\em Riemann problem} associated with the system~\eqref{eq:syst}. This is done in~\S\ref{sss:Riemann}, which relies on the basic properties of entropy solutions to scalar conservation laws that we recall in~\S\ref{sss:kruzkov}. The proof of Lemma~\ref{lem:uniqsg} is finally completed in~\S\ref{sss:pfuniqsg}.


\subsubsection{Entropy solution to the scalar conservation law}\label{sss:kruzkov} We first recall the following result by Kru\v{z}kov, see~\cite[Theorem~2.3.5 and Proposition~2.3.6, pp.~36-37]{serre}.

\begin{prop}[Existence and uniqueness for the scalar conservation law]\label{prop:kruzkov}
  Let $\Lambda : [0,1] \to \R$ be a Lipschitz continuous function, and let $u_0 : \R \to [0,1]$ be a measurable function. There exists a unique weak solution $u : [0,+\infty) \times \R \to [0,1]$ to the scalar conservation law
  \begin{equation}\label{eq:scalarcl2}
    \left\{\begin{aligned}
      & \partial_t u + \partial_x \left(\Lambda(u)\right) = 0,\\
      & u(0,x) = u_0(x),
    \end{aligned}\right.
  \end{equation}
  satisfying the entropy condition that, for all $c \in [0,1]$,
  \begin{equation*}
    \partial_t |u-c| + \partial_x \left(\sgn(u-c)(\Lambda(u)-\Lambda(c))\right) \leq 0
  \end{equation*}
  in the distributional sense, where
  \begin{equation*}
    \sgn(v) := \begin{cases}
      1 & \text{if $v \geq 0$,}\\
      -1 & \text{if $v < 0$.}
    \end{cases}
  \end{equation*}
  Besides, if the initial datum $u_0$ has a nondecreasing version, then for all $t \geq 0$, $u(t,\cdot)$ has a nondecreasing and right continuous version with left limits, and
  \begin{equation*}
    \lim_{x \to \pm\infty} u(t,x) = \lim_{x \to \pm\infty} u_0(x).
  \end{equation*}
%
\end{prop}

The function $u$ given by Proposition~\ref{prop:kruzkov} is called the {\em entropy} solution to the scalar conservation law~\eqref{eq:scalarcl2}. As was shown by Brenier and Grenier~\cite{bregre} and Jourdain~\cite{jourdain:sticky}, it is the appropriate notion of solution to describe the large-scale behaviour of the Sticky Particle Dynamics. The following lemma is a generalisation of these results and will be useful in the sequel of the subsection. Its proof is postponed to Subsection~\ref{ss:pfEntropyTSPD} in Appendix~\ref{app:proofs}. For all $n \geq 1$, for all $\x \in \Dnd$, we denote by $\tilde{\upmu}[\x]$ the empirical distribution of the Typewise Sticky Particle Dynamics
\begin{equation*}
  \tilde{\upmu}[\x] := \frac{1}{n} \sum_{k=1}^n \delta_{(\tPhi^1_k[\tblambda(\x)](\x;t), \ldots, \tPhi^d_k[\tblambda(\x)](\x;t))_{t \geq 0}},
\end{equation*}
where we recall that the notation $\tPhi^{\gamma}_k[\tblambda(\x)](\x;t)$ was introduced in~\S\ref{sss:tspd}.

\begin{lem}[Large-scale behaviour of the Typewise Sticky Particle Dynamics]\label{lem:EntropyTSPD}
  Let Assumptions~\eqref{ass:C} and~\eqref{ass:USH} hold. Let $\bm = (m^1, \ldots, m^d) \in \Ps(\R)^d$, and let $(\x(n))_{n \geq 1}$ be a sequence of initial configurations such that $\x(n) \in \Dnd$ and, for all $\gamma \in \{1, \ldots, d\}$, the empirical distribution
  \begin{equation*}
    \frac{1}{n} \sum_{k=1}^d \delta_{x_k^{\gamma}(n)}
  \end{equation*}
  converges weakly to $m^{\gamma}$. Let us assume in addition that the following property holds.
  \begin{enumerate}[label=($*$), ref=$*$]
    \item\label{ass:star:EntropyTSPD} For all $\gamma, \gamma' \in \{1, \ldots, d\}$ with $\gamma \not= \gamma'$,
    \begin{equation*}
      \lim_{n \to +\infty} u^{\gamma'}_{n,0}((u^{\gamma}_{n,0})^{-1}(v)^-) = u^{\gamma'}_0((u^{\gamma}_0)^{-1}(v)^-), \qquad \lim_{n \to +\infty} u^{\gamma'}_{n,0}((u^{\gamma}_{n,0})^{-1}(v)) = u^{\gamma'}_0((u^{\gamma}_0)^{-1}(v)), 
    \end{equation*}
    $\dd v$-almost everywhere in $(0,1)$, where $u^{\gamma}_{n,0}$ refers to the empirical CDF of $x^{\gamma}_1(n), \ldots, x^{\gamma}_n(n)$.
  \end{enumerate}
  
  Then $\tilde{\upmu}[\x(n)]$ converges weakly, when $n$ grows to infinity, to the probability measure $\tilde{\upmu} \in \Ms$ defined as the image of the Lebesgue measure $\Unif$ on $[0,1]$ by the mapping
  \begin{equation*}
    v \mapsto \left(\tilde{u}^1(t, \cdot)^{-1}(v), \ldots, \tilde{u}^d(t, \cdot)^{-1}(v)\right)_{t \geq 0}
  \end{equation*}
  where, for all $\gamma \in \{1, \ldots, d\}$, the function $\tilde{u}^{\gamma} : [0,+\infty) \times \R \to [0,1]$ is the entropy solution of the scalar conservation law
  \begin{equation}\label{eq:EntropyTSPD:scl}
    \left\{\begin{aligned}
      & \partial_t \tilde{u}^{\gamma} + \partial_x \left(\tilde{\Lambda}^{\gamma}(\tilde{u}^{\gamma})\right) = 0,\\
      & \tilde{u}^{\gamma}(0, \cdot) = H*m^{\gamma} =: u^{\gamma}_0,
    \end{aligned}\right.
  \end{equation}
  with $\tilde{\Lambda}^{\gamma}(u)$ being defined by
  \begin{equation}\label{eq:tLambda}
    \int_{v=0}^u \lambda^{\gamma}\left(u^1_0((u^{\gamma}_0)^{-1}(v)^-), \ldots, u^{\gamma-1}_0((u^{\gamma}_0)^{-1}(v)^-), v, u^{\gamma+1}_0((u^{\gamma}_0)^{-1}(v)), \ldots, u^d_0((u^{\gamma}_0)^{-1}(v))\right) \dd v.
  \end{equation}
\end{lem}

We note that, by Lemma~\ref{lem:FnGn}, a sufficient condition for the hypothesis~\eqref{ass:star:EntropyTSPD} to hold with any sequence of initial configurations $(\x(n))_{n \geq 1}$ approximating $\bm$ is that, for all $\gamma, \gamma' \in \{1, \ldots, d\}$ with $\gamma \not= \gamma'$, the measures $m^{\gamma}$ and $m^{\gamma'}$ have distinct atoms. 


\subsubsection{The Riemann problem}\label{sss:Riemann} Let us now fix $\xi \in \R$ and $(u^1_-, u^1_+), \ldots, (u^d_-, u^d_+) \in [0,1]^2$, with $u^{\gamma}_- \leq u^{\gamma}_+$ for all $\gamma \in \{1, \ldots, d\}$. The Riemann problem for the system~\eqref{eq:syst} is the problem
\begin{equation}\label{eq:riemann}
  \forall \gamma \in \{1, \ldots, d\}, \qquad \left\{\begin{aligned}
    & \partial_t u^{\gamma} + \lambda^{\gamma}\{\bu\} \partial_x u^{\gamma} = 0,\\
    & u^{\gamma}(0,x) = u^{\gamma}_-\ind{x < \xi} + u^{\gamma}_+\ind{x \geq \xi}.
  \end{aligned}\right.
\end{equation}
Unless $u^{\gamma}_-=0$ and $u^{\gamma}_+=1$, the initial data of this problem are not CDFs, and therefore this system does not {\em a priori} enter the scope of our approach. One can however circumvent this difficulty by formally adding the missing masses $u^{\gamma}_-$ and $1-u^{\gamma}_+$ at the respective points $-\infty$ and $+\infty$. Then by Proposition~\ref{prop:traj}, the trajectories $(X_v^{\gamma}(t))_{t \geq 0}$ associated with any probabilistic solution to the Riemann problem~\eqref{eq:riemann} are expected to satisfy
\begin{equation*}
  \forall v \in (u^{\gamma}_-,u^{\gamma}_+), \quad \forall t \geq 0, \qquad \xi + t \inf_{\bu \in [0,1]^d} \lambda^{\gamma}(\bu) \leq X_v^{\gamma}(t) \leq \xi + t \sup_{\bu \in [0,1]^d} \lambda^{\gamma}(\bu),
\end{equation*}
for coordinates $\gamma$ such that $u^{\gamma}_- < u^{\gamma}_+$. Under Assumption~\eqref{ass:USH}, we deduce that these trajectories evolve in separated space-time cones for positive times, so that the Riemann problem~\eqref{eq:riemann} is actually uncoupled into $d$ scalar problems. Since we noted in Remark~\ref{rk:scalarcl} that, in the scalar case, our definition of probabilistic solutions is consistent with the conservative equation~\eqref{eq:scalarcl}, this motivates the following definition.

\begin{defi}[Solution to the Riemann problem]\label{defi:riemann}
  Under Assumptions~\eqref{ass:C} and~\eqref{ass:USH}, the solution to the Riemann problem~\eqref{eq:riemann} is the function $\bu = (u^1, \ldots, u^d) : [0,+\infty) \times \R \to [0,1]^d$ such that, for all $\gamma \in \{1, \ldots, d\}$, $u^{\gamma}$ is the unique entropy solution to the scalar conservation law
  \begin{equation*}
    \left\{\begin{aligned}
      & \partial_t u^{\gamma} + \partial_x \left(\Lambda^{\gamma}(u^{\gamma})\right) = 0,\\
      & u^{\gamma}(0,x) = u^{\gamma}_-\ind{x < \xi} + u^{\gamma}_+\ind{x \geq \xi},
    \end{aligned}\right.
  \end{equation*}
  where $\Lambda^{\gamma}$ is defined on $[u^{\gamma}_-, u^{\gamma}_+]$ by
  \begin{equation}\label{eq:LambdaRiemann}
    \Lambda^{\gamma}(u) := \int_{v=u^{\gamma}_-}^u \lambda^{\gamma}(u^1_-, \ldots, u^{\gamma-1}_-, v, u^{\gamma+1}_+, \ldots, u^d_+)\dd v.
  \end{equation}
\end{defi}

According to Proposition~\ref{prop:kruzkov}, we shall always implicitly assume that, for all $\gamma \in \{1, \ldots, d\}$, for all $t \geq 0$, the function $u^{\gamma}(t,\cdot)$ is nondecreasing, right continous with left limits.


\subsubsection{The Bianchini-Bressan conditions}\label{sss:BB} Given a function $\bu = (u^1, \ldots, u^d) : [0,+\infty) \times \R \to [0,1]^d$ such that, for all $\gamma \in \{1, \ldots, d\}$, for all $t \geq 0$, $u^{\gamma}(t, \cdot)$ is a CDF on the real line, we define the {\em total variation} of $\bu(t,\cdot)$ on the interval $(a,b) \subset \R$ by
\begin{equation*}
  \TV\{\bu(t,\cdot); (a,b)\} := \sum_{\gamma=1}^d \left(u^{\gamma}(t,b^-) - u^{\gamma}(t,a)\right).
\end{equation*}

For all $(\tau,\xi) \in [0,+\infty) \times \R$, we also denote:
\begin{itemize}
  \item $\bU^{\sharp}_{\bu;\tau,\xi} = (U^{\sharp,1}_{\bu;\tau,\xi}, \ldots, U^{\sharp,d}_{\bu;\tau,\xi})$ the solution, in the sense of Definition~\ref{defi:riemann}, to the Riemann problem~\eqref{eq:riemann} with $u^{\gamma}_- := u^{\gamma}(\tau, \xi^-)$ and $u^{\gamma}_+ := u^{\gamma}(\tau, \xi)$,
  \item $\bU^{\flat}_{\bu;\tau,\xi} = (U^{\flat,1}_{\bu;\tau,\xi}, \ldots, U^{\flat,d}_{\bu;\tau,\xi})$ the solution to the linear problem with constant coefficients
  \begin{equation*}
    \forall \gamma \in \{1, \ldots, d\}, \qquad \left\{\begin{aligned}
      & \partial_t v^{\gamma} + \lambda^{\gamma}\{\bu\}(\tau,\xi) \partial_x v^{\gamma} = 0,\\
      & v^{\gamma}(0,x) = u^{\gamma}(\tau,x).
    \end{aligned}\right.
  \end{equation*}
\end{itemize}
Note that, by the method of characteristics, we have, for all $\gamma \in \{1, \ldots, d\}$,
\begin{equation*}
  \forall t \geq 0, \qquad U^{\flat,\gamma}_{\bu;\tau,\xi}(t,x) = u^{\gamma}\left(\tau, x - \lambda^{\gamma}\{\bu\}(\tau,\xi)t\right).
\end{equation*}

The following definition is adapted from~\cite[Definition~15.1, p.~307]{bianbres}.

\begin{defi}[Bianchini-Bressan conditions]\label{defi:bianbres}
  Let Assumptions~\eqref{ass:C} and~\eqref{ass:USH} hold. Let $\bu = (u^1, \ldots, u^d) : [0,+\infty) \times \R \to [0,1]^d$ such that, for all $\gamma \in \{1, \ldots, d\}$, for all $t \geq 0$, $u^{\gamma}(t, \cdot)$ is a CDF on the real line, and the mapping $t \mapsto \bu(t,\cdot)$ is continuous in $\Ls^1_{\mathrm{loc}}(\R)^d$. The function $\bu$ is said to satisfy the {\em Bianchini-Bressan conditions} for the system~\eqref{eq:syst} if it satisfies the following estimates.
  \begin{enumerate}[label=(\roman*), ref=\roman*]
    \item Shock estimate: $\dd \tau$-almost everywhere on $[0,+\infty)$, for all $\xi \in \R$, for all $\beta' > 0$,
    \begin{equation*}
      \lim_{h \dto 0} \frac{1}{h} \int_{x=\xi-\beta'h}^{\xi+\beta'h} \sum_{\gamma=1}^d |u^{\gamma}(\tau+h,x) - U^{\sharp,\gamma}_{\bu;\tau,\xi}(h,x)| \dd x = 0.
    \end{equation*}
    \item Flat estimate: there exist $C > 0$ and $\beta > 0$ such that, for all $(\tau, \xi) \in [0,+\infty) \times \R$, for all $a,b \in \R$ such that $a < \xi < b$,
    \begin{equation*}
      \limsup_{h \dto 0} \frac{1}{h} \int_{x=a+\beta h}^{b-\beta h} \sum_{\gamma=1}^d |u^{\gamma}(\tau+h,x) - U^{\flat,\gamma}_{\bu;\tau,\xi}(h,x)| \dd x \leq C \left(\TV\{\bu(\tau, \cdot); (a,b)\}\right)^2.
    \end{equation*}
  \end{enumerate}
\end{defi}

In order to use the uniqueness result by Bianchini and Bressan~\cite[Section~15]{bianbres} in our setting, we first check that the probabilistic solutions to~\eqref{eq:syst} obtained in Proposition~\ref{prop:cSt} satisfy the Bianchini-Bressan conditions. This is done in Lemmas~\ref{lem:flat} and~\ref{lem:shock}.

\begin{lem}[Flat estimate]\label{lem:flat}
  Under Assumption~\eqref{ass:LC}, any probabilistic solution $\bu$ to~\eqref{eq:syst}, such that $t \mapsto \bu(t,\cdot)$ is continuous in $\Ls^1_{\mathrm{loc}}(\R)^d$, satisfies the flat estimate of Definition~\ref{defi:bianbres}.
\end{lem}
\begin{proof}
  Let $\bu$ be a probabilistic solution to~\eqref{eq:syst} such that $t \mapsto \bu(t,\cdot)$ is continuous in $\Ls^1_{\mathrm{loc}}(\R)^d$, and let us fix $(\tau,\xi) \in [0,+\infty) \times \R$, $a,b \in \R$ such that $a < \xi < b$. For all $\gamma \in \{1, \ldots, d\}$, for all $(h,x) \in [0,+\infty) \times \R$, let us define
  \begin{equation*}
    w^{\gamma}(h,x) := u^{\gamma}(\tau+h,x) - U^{\flat, \gamma}_{\bu;\tau,\xi}(h,x),
  \end{equation*}
  and let us also denote $\hat{\lambda}^{\gamma} := \lambda^{\gamma}\{\bu\}(\tau,\xi)$. Then we have, for $h > 0$ small enough,
  \begin{equation*}
    \frac{1}{h} \int_{x=a+\beta h}^{b-\beta h} \sum_{\gamma=1}^d |u^{\gamma}(\tau+h,x) - U^{\flat,\gamma}_{\bu;\tau,\xi}(h,x)| \dd x = \sum_{\gamma=1}^d \frac{1}{h} \int_{x=a+\beta h}^{b-\beta h}  |w^{\gamma}(h,x)| \dd x.
  \end{equation*}
  Besides, in the distributional sense,
  \begin{equation*}
    \begin{aligned}
      \partial_h w^{\gamma}(h,x) & = -\lambda^{\gamma}\{\bu\}(\tau+h,x)\dd_x u^{\gamma}(\tau+h,x) + \hat{\lambda}^{\gamma} \partial_x U^{\flat, \gamma}_{\bu;\tau,\xi}(h,x)\\
      & = \left(\hat{\lambda}^{\gamma} - \lambda^{\gamma}\{\bu\}(\tau+h,x)\right)\dd_x u^{\gamma}(\tau+h,x) - \hat{\lambda}^{\gamma} \partial_x w^{\gamma}(h,x).
    \end{aligned}
  \end{equation*}
  As a consequence, letting $\bar{w}^{\gamma}(h,y) := w^{\gamma}(h, y + \hat{\lambda}^{\gamma} h)$, we obtain
  \begin{equation*}
    \partial_h \bar{w}^{\gamma}(h,y) = \left(\hat{\lambda}^{\gamma} - \lambda^{\gamma}\{\bu\}(\tau+h,y + \hat{\lambda}^{\gamma} h)\right)\dd_x u^{\gamma}(\tau+h,y + \hat{\lambda}^{\gamma} h),
  \end{equation*}
  while $\bar{w}^{\gamma}(0,y)=0$. We deduce that, for any $\beta > 0$ and for $h>0$ small enough,
  \begin{equation*}
    \begin{aligned}
      & \frac{1}{h} \int_{x=a+\beta h}^{b-\beta h}  |w^{\gamma}(h,x)| \dd x = \frac{1}{h} \int_{y=a+(\beta-\hat{\lambda}^{\gamma}) h}^{b-(\beta+\hat{\lambda}^{\gamma}) h}  |\bar{w}^{\gamma}(h,y)| \dd y\\
      & \qquad \leq \frac{1}{h} \int_{y=a+(\beta-\hat{\lambda}^{\gamma}) h}^{b-(\beta+\hat{\lambda}^{\gamma}) h} \int_{h'=0}^h |\partial_h \bar{w}^{\gamma}(h',y)| \dd h' \dd y\\
      & \qquad = \frac{1}{h} \int_{h'=0}^h \int_{y=a+(\beta-\hat{\lambda}^{\gamma}) h}^{b-(\beta+\hat{\lambda}^{\gamma}) h} \left|\hat{\lambda}^{\gamma} - \lambda^{\gamma}\{\bu\}(\tau+h',y + \hat{\lambda}^{\gamma} h')\right|\dd_x u^{\gamma}(\tau+h',y + \hat{\lambda}^{\gamma} h')\dd h'.
    \end{aligned}
  \end{equation*}
  For the sake of clarity, we kept the last computation at the formal level, but it can be made rigorous by using suitable test functions.
  
  We now estimate the integral term in the right-hand side above. To this aim, we fix $\beta > \ConstBound{\infty}$, so that $\beta - \hat{\lambda}^{\gamma} > 0$ and $\beta + \hat{\lambda}^{\gamma} > 0$. Then, for all $0 \leq h' \leq h$, for all $y \in [a+(\beta-\hat{\lambda}^{\gamma}) h, b-(\beta+\hat{\lambda}^{\gamma}) h]$, letting $x := y + \hat{\lambda}^{\gamma} h'$ yields
  \begin{equation*}
    \begin{aligned}
      & \left|\hat{\lambda}^{\gamma} - \lambda^{\gamma}\{\bu\}(\tau+h',y + \hat{\lambda}^{\gamma} h')\right| = \left|\lambda^{\gamma}\{\bu\}(\tau,\xi) - \lambda^{\gamma}\{\bu\}(\tau+h',x)\right|\\
      & \qquad \leq \int_{\theta=0}^1 \left| \lambda^{\gamma}\left(u^1(\tau,\xi), \ldots, (1-\theta)u^{\gamma}(\tau,\xi^-) + \theta u^{\gamma}(\tau,\xi), \ldots, u^d(\tau,\xi)\right)\right.\\
      & \qquad\quad  - \left. \lambda^{\gamma}\left(u^1(\tau+h',x), \ldots, (1-\theta)u^{\gamma}(\tau+h',x^-) + \theta u^{\gamma}(\tau+h',x), \ldots, u^d(\tau+h',x)\right)\right|\dd \theta\\
      & \qquad\leq \ConstLip \sum_{\gamma' \not= \gamma} \left|u^{\gamma'}(\tau,\xi) - u^{\gamma'}(\tau+h',x)\right|\\
      & \qquad\quad + \frac{\ConstLip}{2}\left(\left|u^{\gamma}(\tau,\xi^-) - u^{\gamma}(\tau+h',x^-)\right| + \left|u^{\gamma}(\tau,\xi) - u^{\gamma}(\tau+h',x)\right|\right).
    \end{aligned}
  \end{equation*}
  We now prove that, for all $\gamma' \in \{1, \ldots, d\}$, 
  \begin{equation}\label{eq:pfflat:1}
    u^{\gamma'}(\tau, a) \leq u^{\gamma'}(\tau+h', x^-) \qquad \text{and} \qquad u^{\gamma'}(\tau+h', x) \leq u^{\gamma'}(\tau, b^-),
  \end{equation}
  which, on account of the estimation above, implies that
  \begin{equation*}
    \left|\hat{\lambda}^{\gamma} - \lambda^{\gamma}\{\bu\}(\tau+h',y + \hat{\lambda}^{\gamma} h')\right| \leq \ConstLip \TV\{\bu(\tau,\cdot); (a,b)\}.
  \end{equation*}
  By Corollary~\ref{cor:speedprop}, we already have $u^{\gamma'}(\tau,a) \leq u^{\gamma'}(\tau+h', a + \bar{\lambda}^{\gamma'} h')$ and $u^{\gamma'}(\tau, b^-) \geq u^{\gamma'}(\tau+h', (b+\ulambda^{\gamma'}h')^-)$. To obtain~\eqref{eq:pfflat:1}, it now suffices to show that $a + \olambda^{\gamma'} h' < x < b+\ulambda^{\gamma'}h'$. Recalling the definition of $x = y + \hat{\lambda}^{\gamma} h'$ with $a+(\beta-\hat{\lambda}^{\gamma}) h \leq y \leq b-(\beta+\hat{\lambda}^{\gamma}) h$, we deduce that the conclusion follows from the inequalities
  \begin{equation*}
    \olambda^{\gamma'} h' + \hat{\lambda}^{\gamma}(h-h') < \beta h \qquad \text{and} \qquad -\ulambda^{\gamma'} h' - \hat{\lambda}^{\gamma}(h-h') < \beta h,
  \end{equation*}
  which are due to the facts that $\beta > \ConstBound{\infty}$ and $0 \leq h' \leq h$, with $h>0$. As a consequence,
  \begin{equation*}
    \frac{1}{h} \int_{x=a+\beta h}^{b-\beta h}  |w^{\gamma}(h,x)| \dd x \leq \ConstLip \TV\{\bu(\tau,\cdot); (a,b)\} \frac{1}{h} \int_{h'=0}^h \int_{y=a+(\beta-\hat{\lambda}^{\gamma}) h}^{b-(\beta+\hat{\lambda}^{\gamma}) h} \dd_x u^{\gamma}(\tau+h',y + \hat{\lambda}^{\gamma} h')\dd h'.
  \end{equation*}
  But for $0 \leq h' \leq h$, 
  \begin{equation*}
    \begin{aligned}
      \int_{y=a+(\beta-\hat{\lambda}^{\gamma}) h}^{b-(\beta+\hat{\lambda}^{\gamma}) h} \dd_x u^{\gamma}(\tau+h',y + \hat{\lambda}^{\gamma} h') & = u^{\gamma}(\tau+h', b-(\beta+\hat{\lambda}^{\gamma}) h + \hat{\lambda}^{\gamma} h')\\
      & \quad - u^{\gamma}(\tau+h',a+(\beta-\hat{\lambda}^{\gamma}) h + \hat{\lambda}^{\gamma} h')\\
      & \leq u^{\gamma}(\tau,b^-) - u^{\gamma}(\tau,a),
    \end{aligned} 
  \end{equation*}
  by the same arguments as for the proof of~\eqref{eq:pfflat:1}. We finally deduce that
  \begin{equation*}
    \sum_{\gamma=1}^d \frac{1}{h} \int_{x=a+\beta h}^{b-\beta h}  |w^{\gamma}(h,x)| \dd x \leq \ConstLip \left(\TV\{\bu(\tau,\cdot); (a,b)\}\right)^2,
  \end{equation*}
  which completes the proof.
\end{proof}

\begin{lem}[Shock estimate]\label{lem:shock}
  Under Assumptions~\eqref{ass:LC} and~\eqref{ass:USH}, let $\bm^* \in \Ps(\R)^d$ and let $(\bS_t)_{t \geq 0}$ be a semigroup on $\dP_{\bm^*}$ obtained by Proposition~\ref{prop:cSt}. Then, for all $\bm \in \dP_{\bm^*}$, the function $\bu = (u^1, \ldots, u^d) : [0,+\infty) \times \R \to [0,1]^d$ defined by $u^{\gamma}(t,x) := (H*S^{\gamma}_t\bm)(x)$ satisfies the Bianchini-Bressan conditions of Definition~\ref{defi:bianbres}.
\end{lem}
\begin{proof}
  Let $\bm^* \in \Ps(\R)^d$ and let $(\bS_t)_{t \geq 0}$ be defined on $\dP_{\bm^*}$ by $S^{\gamma}_t\bm = \lim_{\ell \to +\infty} \upmu[\chi_{n_{\ell}}\bm]$, where the sequence $(n_{\ell})_{\ell \geq 1}$ is extracted in Proposition~\ref{prop:cSt}. Let us fix $\bm \in \dP_{\bm^*}$ and write $u^{\gamma}(t, \cdot) = H*S^{\gamma}_t\bm$. By Proposition~\ref{prop:cSt} and Remark~\ref{rk:contL1loc}, the function $t \mapsto \bu(t, \cdot)$ is continuous in $\Ls^1_{\mathrm{loc}}(\R)^d$; besides, $\bu$ is a probabilistic solution to~\eqref{eq:syst}, therefore Lemma~\ref{lem:flat} implies that it satisfies the flat estimates. 
  
  Hence, it only remains to check that $\bu$ satisfies the shock estimate. Following the lines of the proof of the necessity part in~\cite[Lemma~15.2, p.~308]{bianbres}, the crucial argument is the following {\em localised} stability estimate with respect to the solution of the Riemann problem: $\dd\tau$-almost everywhere on $[0,+\infty)$, for all $\xi \in \R$, for all $a,b \in \R$ with $a < \xi < b$, we have, for $h$ small enough,
  \begin{equation}\label{eq:pfshock:1}
    \int_{x=a+\ConstBound{\infty} h}^{b-\ConstBound{\infty} h} \sum_{\gamma=1}^d |u^{\gamma}(\tau+h,x) - U^{\sharp,\gamma}_{\bu;\tau,\xi}(h,x)| \dd x \leq \ConstStab_1 \int_{x=a}^b \sum_{\gamma=1}^d |u^{\gamma}(\tau,x) - U^{\sharp,\gamma}_{\bu;\tau,\xi}(0,x)| \dd x.
  \end{equation}
  The proof of~\eqref{eq:pfshock:1} is postponed below. Taking this estimate for granted, let us fix $(\tau,\xi)$ as above and let and $\beta' > 0$. Applying~\eqref{eq:pfshock:1} with $a=\xi-(\ConstBound{\infty}+\beta')h$, $b=\xi+(\ConstBound{\infty}+\beta')h$ and $h>0$ small enough, we obtain
  \begin{equation*}
    \begin{aligned}
      \int_{x=\xi-\beta' h}^{\xi+\beta' h} \sum_{\gamma=1}^d |u^{\gamma}(\tau+h,x) - U^{\sharp,\gamma}_{\bu;\tau,\xi}(h,x)| \dd x & \leq \ConstStab_1 \int_{x=\xi-(\ConstBound{\infty}+\beta')h}^{\xi} \sum_{\gamma=1}^d |u^{\gamma}(\tau,x) - u^{\gamma}(\tau,\xi^-)| \dd x\\
      & \quad + \ConstStab_1 \int_{x=\xi}^{\xi+(\ConstBound{\infty}+\beta')h} \sum_{\gamma=1}^d |u^{\gamma}(\tau,x) - u^{\gamma}(\tau,\xi)| \dd x.
    \end{aligned}
  \end{equation*}
  For all $\gamma \in \{1, \ldots, d\}$, the monotonicity of the function $u^{\gamma}(\tau,\cdot)$ yields
  \begin{equation*}
    \begin{aligned}
      & \int_{x=\xi-(\ConstBound{\infty}+\beta')h}^{\xi} |u^{\gamma}(\tau,x) - u^{\gamma}(\tau,\xi^-)| \dd x \leq (\ConstBound{\infty}+\beta')h \left\{u^{\gamma}(\tau,\xi^-)-u^{\gamma}(\tau,\xi-(\ConstBound{\infty}+\beta')h)\right\},\\
      & \int_{x=\xi}^{\xi+(\ConstBound{\infty}+\beta')h} |u^{\gamma}(\tau,x) - u^{\gamma}(\tau,\xi^-)| \dd x \leq (\ConstBound{\infty}+\beta')h \left\{u^{\gamma}(\tau,\xi+(\ConstBound{\infty}+\beta')h)-u^{\gamma}(\tau,\xi)\right\},
    \end{aligned}
  \end{equation*}
  and, since the function $u^{\gamma}(\tau,\cdot)$ is right continuous and has left limits, the braced terms vanish with $h$. As a conclusion, 
  \begin{equation*}
    \lim_{h \dto 0} \frac{1}{h} \int_{x=\xi-\beta' h}^{\xi+\beta' h} \sum_{\gamma=1}^d |u^{\gamma}(\tau+h,x) - U^{\sharp,\gamma}_{\bu;\tau,\xi}(h,x)| \dd x = 0,
  \end{equation*}
  \ie $\bu$ satisfies the shock estimate.
  
  \sk
  \noindent {\em Proof of~\eqref{eq:pfshock:1}.} By Proposition~\ref{prop:tightness}, $\dd\tau$-almost everywhere, we have 
  \begin{equation*}
    \forall x \in \R, \qquad \Delta_x u^{\gamma}(\tau,x)\Delta_x u^{\gamma'}(\tau,x) = 0,
  \end{equation*}
  for all $\gamma, \gamma' \in \{1, \ldots, d\}$ such that $\gamma \not= \gamma'$. Let us fix $\tau \geq 0$ such that the condition above is satisfied, and $a,b, \xi \in \R$ with $a < \xi < b$. We first define $\bar{\bu}_0 = (\bar{u}^1_0, \ldots, \bar{u}^d_0)$ by
  \begin{equation*}
    \bar{u}^{\gamma}_0(x) := \begin{cases}
      u^{\gamma}(\tau, x) & \text{if $x < a$,}\\
      u^{\gamma}(\tau,\xi^-) & \text{if $a \leq x < \xi$,}\\ 
      u^{\gamma}(\tau,\xi) & \text{if $\xi \leq x < b$,}\\ 
      u^{\gamma}(\tau,x) & \text{if $b \leq x$.} 
    \end{cases}
  \end{equation*}
  We also define the vectors $\bm_{\tau}$ and $\bar{\bm}_{\tau}$ in $\Ps(\R)^d$ by $u^{\gamma}(\tau,\cdot) = H*m^{\gamma}_{\tau}$ and $\bar{u}^{\gamma}_0 = H*\bar{m}^{\gamma}_{\tau}$. It is straightforward that $\Ws^{(d)}_1(\bm_{\tau}, \bar{\bm}_{\tau}) < +\infty$, and by Lemma~\ref{lem:bmuindP}, we have $\bm_{\tau}, \bar{\bm}_{\tau} \in \dP_{\bm^*}$. As a consequence, we deduce from Proposition~\ref{prop:cSt} that, for $h>0$ small enough,
  \begin{equation*}
    \begin{aligned}
      \int_{x=a+\ConstBound{\infty} h}^{b-\ConstBound{\infty} h} \sum_{\gamma=1}^d |u^{\gamma}(\tau+h,x) - (H*S^{\gamma}_h\bar{\bm}_{\tau})(x)|\dd x & \leq \Ws^{(d)}_1(\bS_h\bm_{\tau}, \bS_h\bar{\bm}_{\tau})\\
      & \leq \ConstStab_1 \Ws^{(d)}_1(\bm_{\tau}, \bar{\bm}_{\tau})\\
      & = \ConstStab_1 \sum_{\gamma=1}^d \int_{x \in \R} |u^{\gamma}(\tau,x) - \bar{u}^{\gamma}_0(x)|\dd x\\
      & = \ConstStab_1 \int_{x=a}^b \sum_{\gamma=1}^d|u^{\gamma}(\tau,x) - U^{\sharp, \gamma}_{\bu;\tau,\xi}(0,x)|\dd x,
    \end{aligned}
  \end{equation*}
  where we have used~\eqref{eq:L1W1} to identify the $\Ws_1$ distance between probability measures with the $\Ls^1$ distance of their CDFs. Hence, to complete the proof of~\eqref{eq:pfshock:1} it remains to prove that, for small times $h$, $U^{\sharp, \gamma}_{\bu;\tau,\xi}$ coincides with $H*S^{\gamma}_h\bar{\bm}_{\tau}$ on $(a+\ConstBound{\infty} h, b-\ConstBound{\infty} h)$.
  
  For all $n \geq 1$, we denote $\x(n) = \chi_n\bar{\bm}_{\tau}$, so that
  \begin{equation*}
    x_k^{\gamma}(n) = (n+1) \int_{w=(2k-1)/(2(n+1))}^{(2k+1)/(2(n+1))} (\bar{u}^{\gamma}_0)^{-1}(w) \dd w.
  \end{equation*}
  Let us define $\y(n) \in \Dnd$ by
  \begin{equation*}
    y_k^{\gamma}(n) := \begin{cases}
      x_k^{\gamma}(n) & \text{if $x_k^{\gamma}(n) \not\in (a,b)$,}\\
      \xi & \text{if $x_k^{\gamma}(n) \in (a,b)$.}\\
    \end{cases}
  \end{equation*}
  We first show that, for all $h \geq 0$, $S^{\gamma}_h\bar{\bm}_{\tau}$ is the weak limit, in $\Ps(\R)$, of $\upmu^{\gamma}_h[\y(n_{\ell})]$, where we recall that sequence $(n_{\ell})_{\ell \geq 1}$ is given by Proposition~\ref{prop:cSt}. To this aim, we note that, for all $\gamma \in \{1, \ldots, d\}$, since $\bar{u}^{\gamma}_0$ is flat on $(a,\xi)$ and $(\xi,b)$, then $x_k^{\gamma}(n) \in (a, \xi)$ only if $2(k-1)/(2(n+1)) < u^{\gamma}(\tau,\xi^-) < 2(k+1)/(2(n+1))$, so that there is at most one such $x_k^{\gamma}(n)$. Similarly, there is at most one $x_k^{\gamma}(n)$ in $(\xi,b)$. Hence,
  \begin{equation*}
    ||\x(n)-\y(n)||_1 \leq \frac{d(b-a)}{n}.
  \end{equation*}
  As a consequence, and recalling the Definition~\ref{defi:tWs} of the modified Wasserstein distance $\tilde{\Ws}_1$,
  \begin{equation*}
    \begin{aligned}
      \sum_{\gamma=1}^d \tilde{\Ws}_1(S^{\gamma}_h\bar{\bm}_{\tau}, \upmu^{\gamma}_h[\y(n_{\ell})]) & \leq \sum_{\gamma=1}^d \left(\tilde{\Ws}_1(S^{\gamma}_h\bar{\bm}_{\tau}, \upmu^{\gamma}_h[\x(n_{\ell})]) + \tilde{\Ws}_1(\upmu^{\gamma}_h[\x(n_{\ell})], \upmu^{\gamma}_h[\y(n_{\ell})])\right)\\
      & \leq \sum_{\gamma=1}^d \left(\tilde{\Ws}_1(S^{\gamma}_h\bar{\bm}_{\tau}, \upmu^{\gamma}_h[\x(n_{\ell})]) + \Ws_1(\upmu^{\gamma}_h[\x(n_{\ell})], \upmu^{\gamma}_h[\y(n_{\ell})])\right)\\
      & \leq \sum_{\gamma=1}^d \tilde{\Ws}_1(S^{\gamma}_h\bar{\bm}_{\tau}, \upmu^{\gamma}_h[\x(n_{\ell})]) + \ConstStab_1||\x(n_{\ell})-\y(n_{\ell})||_1,
    \end{aligned}
  \end{equation*}
  and the right-hand side above vanishes when $\ell$ grows to infinity thanks to Lemma~\ref{lem:cvbS}.
  
  We now set 
  \begin{equation*}
    h_{\max} := \frac{(\xi-a) \wedge (b-\xi)}{2\ConstBound{\infty}},
  \end{equation*}
  and show that, for all $h \in [0,h_{\max})$, for all $x \in (a+\ConstBound{\infty} h, b-\ConstBound{\infty} h)$,
  \begin{equation}\label{eq:pfshock:HH}
    H*\upmu^{\gamma}_h[\y(n)](x) = H*\tilde{\upmu}^{\gamma}_h[\y(n)](x),
  \end{equation}
  where $\tilde{\upmu}[\y(n)]$ refers to the empirical distribution of the Typewise Sticky Particle Dynamics started at $\y(n)$. This follows from the observation that, for all $\gamma:k \in \Part$, both $\Phi_k^{\gamma}(\y(n);h)$ and $\tPhi_k^{\gamma}[\tblambda(\y(n))](\y(n);h)$ remain between $y_k^{\gamma} - \ConstBound{\infty} h$ and $y_k^{\gamma} + \ConstBound{\infty} h$, so that:
  \begin{itemize}
    \item if $y_k^{\gamma}(n) \leq a$, then $\Phi_k^{\gamma}(\y(n);h) \leq x$ and $\tPhi_k^{\gamma}[\tblambda(\y(n))](\y(n);h) \leq x$,
    \item if $y_k^{\gamma}(n) > b$, then $\Phi_k^{\gamma}(\y(n);h) > x$ and $\tPhi_k^{\gamma}[\tblambda(\y(n))](\y(n);h) > x$.
  \end{itemize}
  Finally, if $y_k^{\gamma}(n) = \xi$, then the definition of $h_{\max}$ ensures that, on the time interval $[0,h_{\max})$, the particle $\gamma:k$ cannot collide with another particle having an initial position which is not $\xi$. Besides by Assumption~\eqref{ass:USH}, this particle evolves in a space-time cone that is disjoint of the space-time cones in which particles of other types evolve, so that $\Phi_k^{\gamma}(\y(n);h) = \tPhi_k^{\gamma}[\tblambda(\y(n))](\y(n);h)$, which yields~\eqref{eq:pfshock:HH}.
  
  The next step is to use Lemma~\ref{lem:EntropyTSPD} to describe the limit, when $\ell$ grows to infinity, of $\tilde{\upmu}[\y(n_{\ell})]$. We already know, from the argument above, that the empirical distribution of $y^{\gamma}_1(n), \ldots, y^{\gamma}_n(n)$ converges weakly to $\bar{m}^{\gamma}_{\tau}$. We now have to check that the corresponding CDFs satisfy the hypothesis~\eqref{ass:star:EntropyTSPD} there. 
  
  We denote by $\bar{u}^{\gamma}_{n,0}$ the empirical CDF of $y^{\gamma}_1(n), \ldots, y^{\gamma}_n(n)$, and fix $\gamma, \gamma' \in \{1, \ldots, d\}$ such that $\gamma \not= \gamma'$. Recall that $\tau$ is chosen so that the measures $m^{\gamma}_{\tau}$ and $m^{\gamma'}_{\tau}$ have distinct atoms, and that $\bar{m}^{\gamma}_{\tau}$, $\bar{m}^{\gamma'}_{\tau}$ are obtained from $m^{\gamma}_{\tau}$, $m^{\gamma'}_{\tau}$ by moving all the mass on $[a, \xi)$ to $a$, and all the mass on $[\xi, b)$ to $\xi$. As a consequence, the measures $\bar{m}^{\gamma}_{\tau}$ and $\bar{m}^{\gamma'}_{\tau}$ have distinct atoms, except possibly in $a$, $\xi$ and $b$. Following the proof of Lemma~\ref{lem:FnGn},~\eqref{ass:star:EntropyTSPD} is already satisfied $\dd v$-almost everywhere on $(0,1) \setminus (u^{\gamma}(\tau,a^-), u^{\gamma}(\tau, b))$.
  
  Now for all $v \in (u^{\gamma}(\tau,a^-), u^{\gamma}(\tau, \xi^-))$, $(\bar{u}^{\gamma}_{n,0})^{-1}(v) = a = (\bar{u}^{\gamma}_0)^{-1}(v)$ for $n$ large enough. As a consequence,
  \begin{equation*}
    \bar{u}^{\gamma'}_{n,0}((\bar{u}^{\gamma}_{n,0})^{-1}(v)) = \bar{u}^{\gamma'}_{n,0}(a) = \frac{1}{n} \sum_{k=1}^n \ind{y_k^{\gamma'}(n) \leq a}.
  \end{equation*}
  We recall that, by the definition of $\y(n)$, for all $k \in \{1, \ldots, n\}$,
  \begin{itemize}
    \item if $(2k+1)/(2(n+1)) \leq \bar{u}^{\gamma'}_0(a)$, then $y^{\gamma'}_k(n) \leq a$,
    \item if $(2k-1)/(2(n+1)) \geq \bar{u}^{\gamma'}_0(a)$, then $y^{\gamma'}_k(n) \geq \xi > a$,    
  \end{itemize}
  from which we deduce that
  \begin{equation*}
    \left|\bar{u}^{\gamma'}_{n,0}(a) - \bar{u}^{\gamma'}_0(a)\right| \leq \frac{3}{2n},
  \end{equation*}
  therefore
  \begin{equation*}
    \lim_{\ell \to +\infty} \bar{u}^{\gamma'}_{n_{\ell},0}\left((\bar{u}^{\gamma}_{n_{\ell},0})^{-1}(v)\right) = \bar{u}^{\gamma'}_0\left((\bar{u}^{\gamma}_0)^{-1}(v)\right).
  \end{equation*}
  Similarly, 
  \begin{equation*}
    \lim_{\ell \to +\infty} \bar{u}^{\gamma'}_{n_{\ell},0}\left((\bar{u}^{\gamma}_{n_{\ell},0})^{-1}(v)^-\right) = \bar{u}^{\gamma'}_0\left((\bar{u}^{\gamma}_0)^{-1}(v)^-\right)
  \end{equation*}
  is obtained by noting that
  \begin{itemize}
    \item if $(2k+1)/(2(n+1)) \leq \bar{u}^{\gamma'}_0(a^-)$, then $y^{\gamma'}_k(n) < a$,
    \item if $(2k-1)/(2(n+1)) \geq \bar{u}^{\gamma'}_0(a^-)$, then $y^{\gamma'}_k(n) \geq a$.    
  \end{itemize}
  We use the same arguments to address the cases $v \in (u^{\gamma}(\tau,\xi^-), u^{\gamma}(\tau, \xi))$ and $v \in (u^{\gamma}(\tau,\xi), u^{\gamma}(\tau, b))$, which finally shows that the condition~\eqref{ass:star:EntropyTSPD} is satisfied. We can now apply Lemma~\ref{lem:EntropyTSPD} and deduce that $\tilde{\upmu}[\y(n_{\ell})]$ converges weakly to some probability measure $\tilde{\upmu}$ in $\Ms$, such that $\tilde{u}^{\gamma}(t,\cdot) := H*\tilde{\upmu}^{\gamma}_t$ is the entropy solution of the scalar conservation law~\eqref{eq:EntropyTSPD:scl} with initial datum $\bar{u}^{\gamma}_0$. Besides, by~\eqref{eq:pfshock:HH}, we have
  \begin{equation*}
    \forall h < h_{\max}, \quad \forall x \in (a+\ConstBound{\infty} h, b-\ConstBound{\infty} h), \qquad (H*S^{\gamma}_h\bar{\bm}_{\tau})(x) = \tilde{u}^{\gamma}(h,x).
  \end{equation*}
  
  To complete the proof, we finally show that $\tilde{u}^{\gamma}(h,x) = U^{\sharp, \gamma}_{\bu;\tau,\xi}(h,x)$ for $(h,x)$ as above. To this aim, we recall that these functions are the entropy solutions of scalar conservation laws with respective flux functions $\tilde{\Lambda}^{\gamma}$ and $\Lambda^{\gamma}$, which are respectively defined by~\eqref{eq:tLambda} and~\eqref{eq:LambdaRiemann}, and respective initial data $\bar{u}^{\gamma}$ and $u^{\gamma}(\tau, \xi^-) \ind{x < \xi} + u^{\gamma}(\tau, \xi) \ind{x \geq \xi}$. 
  
  By Proposition~\ref{prop:kruzkov}, $U^{\sharp, \gamma}_{\bu;\tau,\xi}(t,x) \in [u^{\gamma}(\tau, \xi^-), u^{\gamma}(\tau, \xi)]$ for all $(t,x) \in [0,+\infty) \times \R$. On the other hand, it is straightforward to check that, on $[u^{\gamma}(\tau, \xi^-), u^{\gamma}(\tau, \xi)]$, the flux functions $\tilde{\Lambda}^{\gamma}$ and $\Lambda^{\gamma}$ only differ by a constant function. In other words, $U^{\sharp, \gamma}_{\bu;\tau,\xi}$ is also the entropy solution of the conservation law with the same flux function $\tilde{\Lambda}^{\gamma}$ as $\tilde{u}^{\gamma}$, but with different initial data. The initial data however coincide on the interval $(a,b)$, which, by~\cite[Proposition~2.3.6, p.~37]{serre}, is enough to ensure that $\tilde{u}^{\gamma}(h,x) = U^{\sharp, \gamma}_{\bu;\tau,\xi}(h,x)$ as long as $h < h_{\max}$ and $x \in (a+\ConstBound{\infty} h, b-\ConstBound{\infty} h)$. This completes the proof.
\end{proof}


\subsubsection{Proof of Lemma~\ref{lem:uniqsg}}\label{sss:pfuniqsg} The proof of the uniqueness result by Bianchini and Bressan~\cite[Sufficiency part in Lemma~15.2, p.~309]{bianbres} is readily adapted to our setting, and we do not reproduce it. We only highlight the fact that it relies on the $\Ls^1$ Lipschitz continuity result~\cite[(15.8) p.~309]{bianbres}, which by~\eqref{eq:L1W1} is our $\Ws_1$ stability estimate of~\eqref{it:cSt:2} in Proposition~\ref{prop:cSt}, and the localised stability estimate~\cite[(13.13) p.~294]{bianbres}, which must be replaced with the following generalisation of our estimate~\eqref{eq:pfshock:1}, the proof of which is postponed below.

\begin{lem}[Localised $\Ls^1$ stability estimate]\label{lem:locL1stab}
  Under the assumptions of Proposition~\ref{prop:cSt}, let $\bm^* \in \Ps(\R)^d$, and let $(\bS_t)_{t \geq 0}$ be a semigroup obtained by Proposition~\ref{prop:cSt}. For all $\bm, \bm' \in \dP_{\bm^*}$, for all $a, b \in \R$ with $a < b$, for $t \geq 0$ small enough,
  \begin{equation*}
    \int_{x=a+\beta t}^{b-\beta t} \sum_{\gamma=1}^d |u^{\gamma}(t,x)-v^{\gamma}(t,x)|\dd x \leq L \int_{x=a}^b \sum_{\gamma=1}^d |u^{\gamma}(0,x)-v^{\gamma}(0,x)|\dd x,
  \end{equation*}
  with $\beta = \ConstBound{\infty}$ and $L=\ConstStab_1$, where $u^{\gamma}(t,x) := (H*S^{\gamma}_t\bm)(x)$ and $v^{\gamma}(t,x) := (H*S^{\gamma}_t\bm')(x)$.
\end{lem}

Then we deduce that, for all $\bm^* \in \Ps(\R)^d$, for all semigroups $(\bS^{(1)}_t)_{t \geq 0}$, $(\bS^{(2)}_t)_{t \geq 0}$ defined on $\dP_{\bm^*}$ and obtained by Proposition~\ref{prop:cSt}, we have, for all $t \geq 0$,
\begin{equation*}
  \bS^{(1)}_t = \bS^{(2)}_t \qquad \text{on $\dP_{\bm^*}$.}
\end{equation*}
As a consequence, for all $t \geq 0$, the operator $\bar{\bS}_t$ defined on $\Ps(\R)^d$ by
\begin{equation*}
  \forall \bm^* \in \Ps(\R)^d, \qquad \bar{\bS}_t\bm^* := \bS_t\bm^*,
\end{equation*}
where $(\bS_t)_{t \geq 0}$ is any semigroup on $\dP_{\bm^*}$ obtained by Proposition~\ref{prop:cSt}, is uniquely defined. In order to emphasise the fact that $\bS_t$ is defined on $\dP_{\bm^*}$, we denote $\bS^{(\bm^*)}_t$. To show that $\bar{\bS}_t$ coincides with $\bS^{(\bm^*)}_t$ on $\dP_{\bm^*}$, we let $\bm' \in \dP_{\bm^*}$. Then, by definition,
\begin{equation*}
  \bar{\bS}_t\bm' = \bS^{(\bm')}_t \bm',
\end{equation*}
where $(\bS^{(\bm')}_t)_{t \geq 0}$ is any semigroup on $\dP_{\bm'}$ obtained by Proposition~\ref{prop:cSt}. But since $\dP_{\bm'} = \dP_{\bm^*}$, we deduce from the uniqueness result of Bianchini and Bressan that $\bS^{(\bm')}_t = \bS^{(\bm^*)}_t$, so that 
\begin{equation*}
  \bar{\bS}_t\bm' = \bS^{(\bm^*)}_t \bm'.
\end{equation*}
This completes the proof of Lemma~\ref{lem:uniqsg}.

\begin{proof}[Proof of Lemma~\ref{lem:locL1stab}]
  The proof is quite similar to the proof of~\eqref{eq:pfshock:1} in Lemma~\ref{lem:shock}. Let $\bm^* \in \Ps(\R)^d$ and let $(\bS_t)_{t \geq 0}$ be a semigroup obtained by Proposition~\ref{prop:cSt}. Let us fix $\bm, \bm' \in \dP_{\bm^*}$, $a, b \in \R$ with $a < b$, and define $u_0^{\gamma} := H*m^{\gamma}$, $v_0^{\gamma} := H*m'^{\gamma}$, $u^{\gamma}(t,x) := (H*S^{\gamma}_t\bm)(x)$ and $v^{\gamma}(t,x) := (H*S^{\gamma}_t\bm')(x)$.
  
  For all $\gamma \in \{1, \ldots, d\}$, we define the CDFs $\bar{u}^{\gamma}_0$ and $\bar{v}^{\gamma}_0$ by
  \begin{equation*}
    \bar{u}^{\gamma}_0(x) := \begin{cases}
      u_0^{\gamma}(x) \wedge v_0^{\gamma}(x) & \text{if $x < a$},\\
      u_0^{\gamma}(x) & \text{if $a \leq x < b$},\\
      u_0^{\gamma}(x) \vee v_0^{\gamma}(x) & \text{if $x \geq b$},
    \end{cases} \qquad \bar{v}^{\gamma}_0(x) := \begin{cases}
      u_0^{\gamma}(x) \wedge v_0^{\gamma}(x) & \text{if $x < a$},\\
      v_0^{\gamma}(x) & \text{if $a \leq x < b$},\\
      u_0^{\gamma}(x) \vee v_0^{\gamma}(x) & \text{if $x \geq b$},
    \end{cases}
  \end{equation*}
  and let $\bar{\bm}, \bar{\bm}' \in \Ps(\R)^d$ be such that $\bar{u}^{\gamma}_0 = H*\bar{m}^{\gamma}$, $\bar{v}^{\gamma}_0 = H*\bar{m}'^{\gamma}$. It follows from the choice of their tails that $\bar{\bm}$ and $\bar{\bm}'$ belong to the $\Ws_1$ stability class $\dP_{\bm^*}$, which allows us to define $\bar{\bu}$, $\bar{\bv}$ by
  \begin{equation*}
    \bar{u}^{\gamma}(t,x) := (H*S^{\gamma}_t\bar{\bm})(x), \qquad \bar{v}^{\gamma}(t,x) := (H*S^{\gamma}_t\bar{\bm}')(x).
  \end{equation*}
  Combining~\eqref{eq:L1W1} and the point~\eqref{it:cSt:2} of Proposition~\ref{prop:cSt}, we write, for $t \geq 0$ small enough,
  \begin{equation*}
    \begin{aligned}
      \int_{x=a+\ConstBound{\infty} t}^{b-\ConstBound{\infty} t} \sum_{\gamma=1}^d |\bar{u}^{\gamma}(t,x)-\bar{v}^{\gamma}(t,x)| \dd x & \leq \sum_{\gamma=1}^d \int_{x \in \R} |\bar{u}^{\gamma}(t,x)-\bar{v}^{\gamma}(t,x)| \dd x\\
      & = \Ws_1^{(d)}(\bS_t\bar{\bm}, \bS_t\bar{\bm}')\\
      & \leq \ConstStab_1 \Ws_1^{(d)}(\bar{\bm}, \bar{\bm}')\\
      & = \ConstStab_1\int_{x=a}^b \sum_{\gamma=1}^d |u_0^{\gamma}(x)-v_0^{\gamma}(x)|\dd x.
    \end{aligned}
  \end{equation*}
  We now define $t_{\max} := (b-a)/(2\ConstBound{\infty}) > 0$ and introduce the space-time set
  \begin{equation*}
    \Delta := \{(t,x) \in [0,+\infty) \times \R : t < t_{\max}, a+\ConstBound{\infty}t < x < b-\ConstBound{\infty}t\}.
  \end{equation*}
  We shall check below that, for $(t,x) \in \Delta$, $\bar{u}^{\gamma}(t,x) = u^{\gamma}(t,x)$. By the same arguments, $\bar{v}^{\gamma}(t,x) = v^{\gamma}(t,x)$, and the proof of Lemma~\ref{lem:locL1stab} is completed.  
  
  Our argument depends on whether $u^{\gamma}_0(a) = u^{\gamma}_0(b^-)$ or $u^{\gamma}_0(a) < u^{\gamma}_0(b^-)$. In the first case, applying Corollary~\ref{cor:speedprop} to both $\bu$ and $\bar{\bu}$ yields $u^{\gamma}(t,x) = u^{\gamma}_0(a) = \bar{u}^{\gamma}_0(a) = \bar{u}^{\gamma}(t,x)$ for all $(t,x) \in \Delta$. In the second case, we let $\x(n) := \chi_n\bm$ and $\bar{\x}(n) := \chi_n\bar{\bm}$. We first note that, by the Definition~\ref{defi:chi} of the discretisation operator, for all $k \in \{1, \ldots, n\}$,
  \begin{itemize}
    \item if $(2k+1)/(2(n+1)) \leq u^{\gamma}_0(a)$, then $x_k^{\gamma}(n) \leq a$ and $\bar{x}_k^{\gamma}(n) \leq a$,
    \item if $(2k-1)/(2(n+1)) \geq u^{\gamma}_0(b)$, then $x_k^{\gamma}(n) \geq b$ and $\bar{x}_k^{\gamma}(n) \geq b$,
    \item if $u^{\gamma}_0(a) < (2k-1)/(2(n+1))$ and $(2k+1)/(2(n+1)) < u^{\gamma}_0(b^-)$, then $a \leq x^{\gamma}_k(n) = \bar{x}^{\gamma}_k(n) \leq b$. 
  \end{itemize}
  In all the cases above, we define $y_k^{\gamma}(n) := x_k^{\gamma}(n)$ and $\bar{y}_k^{\gamma}(n) := \bar{x}_k^{\gamma}(n)$. Let us now assume that $n$ is large enough to ensure that $b-a\geq 1/(n+1)$. 
  \begin{itemize}
    \item If $(2k-1)/(2(n+1)) \leq u^{\gamma}_0(a) < (2k+1)/(2(n+1))$, then $x_k^{\gamma}(n) \leq b$ and $\bar{x}_k^{\gamma}(n) \leq b$, and we let $y_k^{\gamma}(n) := x_k^{\gamma}(n) \wedge a$, $\bar{y}_k^{\gamma}(n) := \bar{x}_k^{\gamma}(n) \wedge a$.
    \item If $(2k-1)/(2(n+1)) < u^{\gamma}_0(b^-) \leq (2k+1)/(2(n+1))$, then $x_k^{\gamma}(n) \geq a$ and $\bar{x}_k^{\gamma}(n) \geq a$, and we let $y_k^{\gamma}(n) := x_k^{\gamma}(n) \vee b$, $\bar{y}_k^{\gamma}(n) := \bar{x}_k^{\gamma}(n) \vee b$.
  \end{itemize}
  There is no difficulty in checking that $\y(n)$ and $\bar{\y}(n)$ are thereby uniquely defined in $\Dnd$, and that
  \begin{equation*}
    ||\x(n) - \y(n)||_1 \leq \frac{2d(b-a)}{n}, \qquad ||\bar{\x}(n) - \bar{\y}(n)||_1 \leq \frac{2d(b-a)}{n},
  \end{equation*}
  so that, following the same arguments as in the proof of Lemma~\ref{lem:shock},
  \begin{equation*}
    S^{\gamma}_t \bm = \lim_{\ell \to +\infty} \upmu^{\gamma}_t[\y(n_{\ell})], \qquad S^{\gamma}_t \bar{\bm} = \lim_{\ell \to +\infty} \upmu^{\gamma}_t[\bar{\y}(n_{\ell})].
  \end{equation*}
  According to the discussion above, for $n$ large enough, for all $\gamma:k \in \Part$, 
  \begin{itemize}
    \item if $y^{\gamma}_k(n) \leq a$ or $y^{\gamma}_k(n) \geq b$, then $\bar{y}^{\gamma}_k(n)$ satisfies the same inequality and, by the definition of $\Delta$, both $\Phi_k^{\gamma}(\y(n);t)$ and $\Phi_k^{\gamma}(\bar{\y}(n);t)$ never intersect $\Delta$;
    \item if $a < y^{\gamma}_k(n) < b$, then $\bar{y}^{\gamma}_k(n)=y^{\gamma}_k(n)$ and, as long as $\Phi_k^{\gamma}(\y(n);t)$ does not reach the boundary of $\Delta$, we have $\Phi_k^{\gamma}(\y(n);t)=\Phi_k^{\gamma}(\bar{\y}(n);t)$, while once the boundary of $\Delta$ is attained, neither $\Phi_k^{\gamma}(\y(n);t)$ nor $\Phi_k^{\gamma}(\bar{\y}(n);t)$ can reenter $\Delta$.
  \end{itemize}
  We deduce that, in all cases, if $(t,x) \in \Delta$, then $\Phi_k^{\gamma}(\y(n);t)$ and $\Phi_k^{\gamma}(\bar{\y}(n);t)$ have the same contribution in $H*\upmu^{\gamma}_t[\y(n)](x)$ and $H*\upmu^{\gamma}_t[\bar{\y}(n)](x)$. Taking the limit of this equality when $n$ grows to infinity, we conclude that $(H*S^{\gamma}_t\bm)(x) = (H*S^{\gamma}_t\bar{\bm})(x)$, for all $(t,x) \in \Delta$, which completes the proof.
\end{proof}


\subsection{Proof of Theorem~\ref{theo:sg}}\label{ss:pfsg} In this last subsection, we combine the results of Proposition~\ref{prop:cSt} and Lemma~\ref{lem:uniqsg} to complete the proof of Theorem~\ref{theo:sg}.

\begin{proof}[Proof of Theorem~\ref{theo:sg}]
  Let $(\bar{\bS}_t)_{t \geq 0}$ be given by Lemma~\ref{lem:uniqsg}. By~\eqref{it:cSt:3} in Proposition~\ref{prop:cSt}, it is immediate that $(\bar{\bS}_t)_{t \geq 0}$ satisfies the semigroup property~\eqref{it:sg:sg} in Theorem~\ref{theo:sg}. Besides, for all $\bm, \bm' \in \Ps(\R)^d$, either $\Ws^{(d)}_1(\bm, \bm') < +\infty$, in which case~\eqref{it:cSt:2} in Proposition~\ref{prop:cSt} yields the stability estimate~\eqref{it:sg:stable} of Theorem~\ref{theo:sg}, or $\Ws^{(d)}_1(\bm, \bm') = +\infty$ and in this case, the right-hand side of the stability estimate~\eqref{it:sg:stable} of Theorem~\ref{theo:sg} is infinite and therefore there is nothing to prove.
  
  We now fix $\bm \in \Ps(\R)^d$ and define $\bu : [0,+\infty) \times \R \to [0,1]^d$ by $u^{\gamma}(t,x) := (H*\bar{S}^{\gamma}_t\bm)(x)$. By Proposition~\ref{prop:cSt}, $\bu$ is a probabilistic solution to~\eqref{eq:syst} with initial data $(u^1_0, \ldots, u^d_0)$ defined by $u^{\gamma}_0 = H*m^{\gamma}$. Besides, by Lemma~\ref{lem:cvci} and Proposition~\ref{prop:tightness}, the sequence of empirical distributions $\upmu[\chi_n\bm]$ is tight in $\Ms$. By Proposition~\ref{prop:cSt} and Lemma~\ref{lem:uniqsg}, all the converging subsequences have the same limit
  \begin{equation*}
    \bar{\upmu}[\bm] := \Unif \circ \left(u^1(t, \cdot)^{-1}, \ldots, u^d(t, \cdot)^{-1}\right)_{t \geq 0}^{-1},
  \end{equation*} 
  so that $\upmu[\chi_n\bm]$ converges weakly to $\bar{\upmu}[\bm]$ in $\Ms$.
\end{proof}


\part*{Appendix and references}

\appendix

\section{Proofs of technical results}\label{app:proofs}


\subsection{Proofs of Propositions~\ref{prop:mspd} and~\ref{prop:continuity}}\label{app:pf:cont} This subsection contains the proofs of Propositions~\ref{prop:mspd} and~\ref{prop:continuity}, which were stated in~\S\ref{sss:mspdprop} and describe some continuity properties of the trajectories of the MSPD.

\begin{proof}[Proof of Proposition~\ref{prop:mspd}]
  We prove by induction on $\Nb(\x)$ that
  \begin{enumerate}[label=(\roman*), ref=\roman*]
    \item\label{it:contsp} the process $(\Phi(\x;t))_{t \geq 0}$ has continuous trajectories in $\Dnd$,
    \item\label{it:flow} for all $s,t \geq 0$, $\Phi(\x;s+t) = \Phi(\Phi(\x;s);t)$.
  \end{enumerate}
  
  Let $\x \in \Dnd$ such that $\Nb(\x) = 0$. Then, by Definition~\ref{defi:mspd}, 
  \begin{equation*}
    \forall t \geq 0, \qquad \Phi(\x;t) = \tPhi[\tblambda(\x)](\x;t),
  \end{equation*}
  and~\eqref{it:contsp} follows from the continuity of the trajectories of $(\tPhi[\tblambda(\x)](\x;t))_{t \geq 0}$. Now, for all $s,t \geq 0$, $\Phi(\x;s+t) = \tPhi[\tblambda(\x)](\x;s+t)$ and $\Phi(\x;s) = \tPhi[\tblambda(\x)](\x;s) =: \x'$. By Corollary~\ref{cor:ttinter}, $\Nb(\x')=0$ and $\tblambda(\x') = \tblambda(\x)$. Hence,
  \begin{equation*}
    \Phi(\Phi(\x;s);t) = \Phi(\x';t) = \tPhi[\tblambda(\x')](\x';t) = \tPhi[\tblambda(\x)](\tPhi[\tblambda(\x)](\x;s);t),
  \end{equation*}
  and the flow property for $(\tPhi[\tblambda(\x)](\cdot;t))_{t \geq 0}$ yields 
  \begin{equation*}
    \tPhi[\tblambda(\x)](\tPhi[\tblambda(\x)](\x;s);t) = \tPhi[\tblambda(\x)](\x;s+t) = \Phi(\x;s+t),
  \end{equation*}
  which results in~\eqref{it:flow}.
  
  Now let $N \geq 0$ such that, for all $\x \in \Dnd$ with $\Nb(\x) \leq N$,~\eqref{it:contsp} and~\eqref{it:flow} are satisfied. Let $\x \in \Dnd$ with $\Nb(\x) = N+1$. In particular, $\Nb(\x) \geq 1$ so that $t^*(\x) < +\infty$, and for all $t \in [0,t^*(\x))$, $\Phi(\x;t) = \tPhi[\tblambda(\x)](\x;t)$. As a consequence, the function $t \mapsto \Phi(\x;t)$ is continuous on $[0, t^*(\x))$. On the other hand, since $\Nb(\x^*) < \Nb(\x) = N+1$, the function $t \mapsto \Phi(\x;t)$ is continuous on $[t^*(\x),+\infty)$. Therefore it remains to prove that the function $t \mapsto \Phi(\x;t)$ is left continuous at the point $t^*(\x)$, where, by definition, it takes the value
  \begin{equation*}
    \Phi(\x; t^*(\x)) = \Phi(\x^*; t^*(\x) - t^*(\x)) = \x^*,
  \end{equation*}
  and we recall that, by definition, $\x^* = \tPhi[\tblambda(\x)](\x;t^*(\x))$. As a consequence, the continuity of the trajectories of $(\tPhi[\tblambda(\x)](\x;t))_{t \geq 0}$ yields
  \begin{equation*}
    \lim_{t \uto t^*(\x)} \Phi(\x;t) = \lim_{t \uto t^*(\x)} \tPhi[\tblambda(\x)](\x;t) = \x^* = \Phi(\x; t^*(\x)),
  \end{equation*}
  which is the expected result.
    
  We finally address~\eqref{it:flow}. Let $s,t \geq 0$. 
  
  {\em Case $s \geq t^*(\x)$.} Then $s+t \geq t^*(\x)$, so that, by Definition~\ref{defi:mspd}, 
  \begin{equation*}
    \Phi(\x; s+t) = \Phi(\x^*; s+t-t^*(\x)) = \Phi(\x^*; s'+t),
  \end{equation*}
  where $s' := s-t^*(\x) \geq 0$. Since, by Corollary~\ref{cor:ttinter}, $\Nb(\x^*) < \Nb(\x)$, then the flow property for $(\Phi(\x^*;t))_{t \geq 0}$ yields $\Phi(\x^*; s'+t) = \Phi(\Phi(\x^*;s'); t) = \Phi(\Phi(\x^*;s-t^*(\x)); t)$ and, using Definition~\ref{defi:mspd} again, $\Phi(\x^*;s-t^*(\x)) = \Phi(\x;s)$. As a conclusion, $\Phi(\x; s+t) = \Phi(\Phi(\x;s); t)$.
  
  {\em Case $s < t^*(\x)$.} Then we write $\x' := \Phi(\x;s) = \tPhi[\tblambda(\x)](\x;s)$, and recall that, by Corollary~\ref{cor:ttinter}, $\tblambda(\x')=\tblambda(\x)$ and $t^*(\x') = t^*(\x) - s$. By Definition~\ref{defi:mspd},
  \begin{equation*}
    \Phi(\x';t) = \left\{\begin{aligned}
      & \tPhi[\tblambda(\x')](\x';t) & \text{if $t < t^*(\x')$},\\
      & \Phi(\x'^*; t-t^*(\x')) & \text{if $t \geq t^*(\x')$}.
    \end{aligned}\right.
  \end{equation*}
  
  If $t < t^*(\x') = t^*(\x) - s$, then combining the flow property for $(\tPhi[\tblambda(\x)](\cdot;t))_{t \geq 0}$ with the equality $\tblambda(\x')=\tblambda(\x)$, we obtain
  \begin{equation*}
    \tPhi[\tblambda(\x')](\x';t) = \tPhi[\tblambda(\x)](\tPhi[\tblambda(\x)](\x;s);t) = \tPhi[\tblambda(\x)](\x;s+t), 
  \end{equation*}
  and, since $s+t < t^*(\x)$, the right-hand side above is worth $\Phi(\x;s+t)$. 
  
  If $t \geq t^*(\x') = t^*(\x) - s$, then by Corollary~\ref{cor:ttinter}, $\x'^* = \x^*$, therefore it is straightforward that
  \begin{equation*}
    \Phi(\x';t) = \Phi(\x'^*; t-t^*(\x')) = \Phi(\x^*; s+t-t^*(\x)) = \Phi(\x;s+t). 
  \end{equation*}
  
  In both cases, we conclude that $\Phi(\Phi(\x;s);t) = \Phi(\x';t) = \Phi(\x;s+t)$, which is~\eqref{it:flow}.
\end{proof}

Before detailing the proof of Proposition~\ref{prop:continuity}, we first define
\begin{equation*}
  \bar{t}(\x) := \inf\{t \geq 0 : \Nb(\Phi(\x;t)) = 0\}.
\end{equation*}
Certainly, if $\Nb(\x)=0$ then $\bar{t}(\x) = 0$, otherwise $\bar{t}(\x) > 0$, and an upper bound on $\bar{t}(\x)$ can be derived as follows.

\begin{lem}[Bound on $\bar{t}(\x)$]\label{lem:bart}
  Under Assumptions~\eqref{ass:C} and~\eqref{ass:USH}, for all $\x \in \Dnd$ such that $\Nb(\x) > 0$, 
  \begin{equation}\label{eq:bart}
    \bar{t}(\x) \leq \frac{1}{\ConstUSH} \sup\{x_j^{\beta}-x_i^{\alpha}, (\alpha:i, \beta:j) \in \Rb(\x)\} < +\infty.
  \end{equation}
\end{lem}
\begin{proof}
  Let $\x \in \Dnd$. For all $t \geq 0$, we have $t \geq \bar{t}(\x)$ if and only if $\Nb(\Phi(\x;t)) = 0$, which is equivalent to the fact that, for all $(\alpha:i, \beta:j) \in (\Part)^2$ with $\alpha < \beta$,
  \begin{equation*}
    \Phi_i^{\alpha}(\x;t) \geq \Phi_j^{\beta}(\x;t),
  \end{equation*}
  that is to say,
  \begin{equation*}
    \int_{s=0}^t (v_i^{\alpha}(\x;s) - v_j^{\beta}(\x;s))\dd s \geq x_j^{\beta} - x_i^{\alpha}.
  \end{equation*}
  Recall that, by~\eqref{eq:typeencadrelambda} and Assumption~\eqref{ass:USH}, since the left-hand side above is larger than $t \ConstUSH$, then a sufficient condition for this inequality to hold is that $t \geq (x_j^{\beta}-x_i^{\alpha})/\ConstUSH$, which yields the bound~\eqref{eq:bart}.
\end{proof}

Recall the Definition~\ref{defi:Drnd} of the dense open set $\Drnd \subset \Dnd$.

\begin{lem}[Properties of $\Drnd$]\label{lem:Drondnd}
  The set $\Drnd$ has the following properties.
  \begin{enumerate}[label=(\roman*), ref=\roman*]
    \item\label{it:Drondnd:1} For all $\x \in \Drnd$, there exists $\eta > 0$ such that, for all $\y \in \barB_1(\x,\eta)$, we have $\y \in \Drnd$ and $\Rb(\x)=\Rb(\y)$.
    \item\label{it:Drondnd:2} The function $t^*$ defined in~\eqref{eq:tstar} is continuous on the set $\{\x \in \Drnd : \Nb(\x) \geq 1\}$.
  \end{enumerate}
\end{lem}
\begin{proof}
  Let $\x \in \Drnd$. Let 
  \begin{equation*}
    \eta := \frac{1}{3n} \min\{|x_i^{\alpha} - x_j^{\beta}|, (\alpha:i, \beta:j) \in (\Part)^2, \alpha < \beta\} > 0.
  \end{equation*}
  Let $\y \in \barB_1(\x, \eta)$. Then, in particular, for all $\alpha:i \in \Part$, $|x_i^{\alpha}-y_i^{\alpha}| \leq n\eta$. Let $(\alpha:i, \beta:j) \in (\Part)^2$ with $\alpha < \beta$.
  
  If $(\alpha:i, \beta:j) \in \Rb(\x)$, then $x_j^{\beta} - x_i^{\alpha} \geq 3n\eta$. Since $|x_j^{\beta} - y_j^{\beta}| \leq n\eta$ and $|x_i^{\alpha} - y_i^{\alpha}| \leq n\eta$, we deduce that $y_j^{\beta} - y_i^{\alpha} \geq n\eta$ so that $y_j^{\beta} > y_i^{\alpha}$ and $(\alpha:i, \beta:j) \in \Rb(\y)$.
  
  Likewise, if $(\alpha:i, \beta:j) \not\in \Rb(\x)$, then $x_i^{\alpha} - x_j^{\beta} \geq 3n\eta$ and $y_i^{\alpha} - y_j^{\beta} \geq n\eta$ so that $(\alpha:i, \beta:j) \not\in \Rb(\y)$. 
  
  As a conclusion, $\Rb(\x)=\Rb(\y)$ and $\y \in \Drnd$.
  
  We now prove that the function $t^*$ is continuous on the set $\{\x \in \Drnd : \Nb(\x) \geq 1\}$. Let us fix a configuration $\x$ in this set. Let $(\y_k)_{k \geq 1}$ be a sequence converging to $\x$ in $\Dnd$. By the first part of the lemma, there is no loss of generality in assuming that, for all $k \geq 1$, $||\x-\y_k||_1 \leq \eta$, where $\eta$ is defined in the first part of the proof, so that $\Rb(\y_k) = \Rb(\x)$. This allows us to write
  \begin{equation*}
    t^*(\y_k) = \min\{\ttinter_{\alpha:i, \beta:j}(\y_k), (\alpha:i, \beta:j) \in \Rb(\x)\}.
  \end{equation*}
  Let us fix $(\alpha:i, \beta:j) \in \Rb(\x)$. We denote $\tau_k := \ttinter_{\alpha:i, \beta:j}(\y_k)$ and prove that $\lim_{k \to +\infty} \tau_k = \ttinter_{\alpha:i, \beta:j}(\x)$. On the one hand, the sequence $(\tau_k)_{k \geq 1}$ is bounded. Indeed, combining Lemma~\ref{lem:bart} with the fact that $\Rb(\y_k)=\Rb(\x)$ and $||\x-\y_k||_1 \leq \eta$, we obtain
  \begin{equation*}
    \tau_k \leq \frac{1}{\ConstUSH}(|x_i^{\alpha}-x_j^{\beta}|+n\eta).
  \end{equation*}
  On the other hand, let $\tau \geq 0$ refer to the limit of a converging subsequence of $(\tau_k)_{k \geq 1}$, that we still index by $k$ for convenience. For all $\y \in \Dnd$ and $t \geq 0$, let 
  \begin{equation*}
    g(\y, t) := \tPhi_j^{\beta}[\tblambda(\x)](\y;t) - \tPhi_i^{\alpha}[\tblambda(\x)](\y;t),
  \end{equation*} 
  so that, for all $\y \in \barB_1(\x, \eta)$, $\ttinter_{\alpha:i, \beta:j}(\y)=t$ if and only if $g(\y,t)=0$. In particular, for all $k \geq 1$, $g(\y_k, \tau_k) = 0$, therefore
  \begin{equation*}
    |g(\x, \tau)| = |g(\x,\tau) - g(\y_k, \tau_k)| \leq |g(\x,\tau)-g(\x,\tau_k)| + |g(\x,\tau_k)-g(\y_k,\tau_k)|.
  \end{equation*}
  By the continuity of the trajectories of the flow $(\tPhi[\tblambda(\x)](\cdot;t))_{t \geq 0}$, $|g(\x,\tau)-g(\x,\tau_k)|$ vanishes when $k$ grows to infinity. Furthermore, Lemma~\ref{lem:contracttPhi} yields
  \begin{equation*}
    \begin{aligned}
      \frac{1}{n}|g(\x,\tau_k)-g(\y_k,\tau_k)| & \leq \frac{1}{n}|\tPhi_i^{\alpha}[\tblambda(\x)](\x;\tau_k) - \tPhi_i^{\alpha}[\tblambda(\x)](\y_k;\tau_k)|\\
      & \qquad + \frac{1}{n}|\tPhi_j^{\beta}[\tblambda(\x)](\x;\tau_k) - \tPhi_j^{\beta}[\tblambda(\x)](\y_k;\tau_k)|\\
      & \leq ||\tPhi[\tblambda(\x)](\x;\tau_k) - \tPhi[\tblambda(\x)](\y_k;\tau_k)||_1\\
      & \leq ||\x - \y_k||_1,
    \end{aligned}
  \end{equation*}
  and the right-hand side also vanishes when $k$ grows to infinity. As a conclusion, $g(\x, \tau)=0$ so that $\tau=\ttinter_{\alpha:i, \beta:j}(\x)$.
  
  Thus, for all $(\alpha:i, \beta:j) \in \Rb(\x)$, the function $\ttinter_{\alpha:i, \beta:j}$ is continuous at $\x$, and we complete the proof by recalling that the minimum of a finite number of continuous functions remains a continuous function.
\end{proof}

For initial configurations $\x \not\in \Drnd$, Lemma~\ref{lem:Drondnd} can be completed by the following lemma.

\begin{lem}[Estimates on the collision times]\label{lem:Ddroitnd}
  Under Assumptions~\eqref{ass:C} and~\eqref{ass:USH}, for all $\x \in \Dnd$, let
  \begin{equation*}
    \Rb'(\x) := \{(\alpha:i,\beta:j) \in (\Part)^2 : \alpha < \beta, x_i^{\alpha}=y_j^{\beta}\},
  \end{equation*}
  and let us define $\eta' > 0$ by
  \begin{equation*}
    \eta' := \frac{1}{3n}\min\{|x_i^{\alpha} - x_j^{\beta}| : (\alpha:i,\beta:j) \in (\Part)^2, \alpha < \beta, (\alpha:i,\beta:j) \not\in \Rb'(\x)\},
  \end{equation*}
  where we take the convention that $\eta' = +\infty$ whenever the minimum above is taken over an empty set. Then, for all $y \in \Dnd$ such that $||\x-\y||_1 \leq \eta' \ConstUSH/\ConstBound{1}$,
  \begin{equation*}
    \inf\{t \geq 0 : \Rb(\Phi(\y;t)) = \Rb(\x)\} \leq \frac{n}{\ConstUSH} ||\x-\y||_1,
  \end{equation*}
  while
  \begin{equation*}
    \sup\{t \geq 0 : \Rb(\Phi(\y;t)) = \Rb(\x)\} \geq \frac{2n\eta'}{\ConstBound{1}}.
  \end{equation*}  
\end{lem}
\begin{proof}
  Let $\y \in \Dnd$ such that $||\x-\y||_1 \leq \eta' \ConstUSH/\ConstBound{1}$. Recall that $\ConstUSH \leq \ConstBound{1}$, so that $||\x-\y||_1 \leq \eta'$, which implies that 
  \begin{equation*}
    \Rb(\y) \subset \Rb(\x) \cup \Rb'(\x).
  \end{equation*}
  Let $(\alpha:i, \beta:j) \in \Rb(\y)$.
  \begin{itemize}
    \item If $(\alpha:i, \beta:j) \in \Rb'(\x)$, then by Assumption~\eqref{ass:USH},
    \begin{equation*}
      \tinter_{\alpha:i, \beta:j}(\y) \leq \frac{1}{\ConstUSH}(y_j^{\beta}-y_i^{\alpha}) = \frac{1}{\ConstUSH}(y_j^{\beta}-x_j^{\beta}+x_i^{\alpha}-y_i^{\alpha}) \leq \frac{n}{\ConstUSH}||\x-\y||_1.
    \end{equation*}
    \item If $(\alpha:i, \beta:j) \in \Rb(\x)$, then by the boundedness of the velocities,
    \begin{equation*}
      \tinter_{\alpha:i, \beta:j}(\y) \geq \frac{1}{\ConstBound{1}}(y_j^{\beta}-y_i^{\alpha}) \geq \frac{1}{\ConstBound{1}}(|x_j^{\beta}-x_i^{\alpha}|-n\eta') \geq \frac{2n\eta'}{\ConstBound{1}}.
    \end{equation*}
  \end{itemize}
  Since the choice of $\y$ ensures that $||\x-\y||_1/\ConstUSH < 2\eta'/\ConstBound{1}$, we conclude that, on the time interval $[n||\x-\y||_1/\ConstUSH, 2n\eta'/\ConstBound{1}]$, we have $\Rb(\Phi(\y;t)) = \Rb(\x)$.
\end{proof}

We are now ready to prove Proposition~\ref{prop:continuity}.
\begin{proof}[Proof of Proposition~\ref{prop:continuity}]
  The proof works by induction on $\Nb(\x)$. 
  
  Let us first fix $\epsilon > 0$ and $\x \in \Dnd$ such that $\Nb(\x)=0$. Let $\delta > 0$, and let $\y \in \barB_1(\x, \delta)$. Then, in particular, for all $\gamma:k \in \Part$, $|x_k^{\gamma} - y_k^{\gamma}| \leq n\delta$. We shall study $||\Phi(\x;t) - \Phi(\y;t)||_1$ on the intervals $[0, \bar{t}(\y))$ and $[\bar{t}(\y),+\infty)$ separately.
  
  If $\bar{t}(\y)=0$ then the interval $[0,\bar{t}(\y))$ is empty. If $\bar{t}(\y) > 0$, that is to say $\Nb(\y) \geq 1$, then we let $t \in [0,\bar{t}(\y))$, and we have
  \begin{equation*}
    ||\Phi(\x;t) - \Phi(\y;t)||_1 \leq ||\x-\y||_1 + \frac{1}{n}\sum_{\gamma=1}^d \sum_{k=1}^n \int_{s=0}^t |v_k^{\gamma}(\x;s) - v_k^{\gamma}(\y;s)| \dd s \leq \delta + 2\ConstBound{1} t.
  \end{equation*}
  Following Lemma~\ref{lem:bart},
  \begin{equation*}
    \bar{t}(\y) \leq \frac{1}{\ConstUSH} \sup\{y_j^{\beta}-y_i^{\alpha}, (\alpha:i, \beta:j) \in \Rb(\y)\}.
  \end{equation*}
  and, for all $(\alpha:i, \beta:j) \in \Rb(\y)$,
  \begin{equation*}
    y_j^{\beta}-y_i^{\alpha} = y_j^{\beta}-x_j^{\beta} + x_j^{\beta}-x_i^{\alpha} + x_i^{\alpha}-y_i^{\alpha} \leq |x_j^{\beta}-y_j^{\beta}| + |x_i^{\alpha}-y_i^{\alpha}| \leq n\delta,
  \end{equation*}
  where we have used the fact that $\Nb(\x)=0$ so that $x_j^{\beta} \leq x_i^{\alpha}$. As a consequence, 
  \begin{equation*}
    \sup_{t \in [0, \bar{t}(\y))} ||\Phi(\x;t) - \Phi(\y;t)||_1 \leq \left(1 + \frac{2n}{\ConstUSH}\ConstBound{1}\right)\delta.
  \end{equation*}
  
  We now study $||\Phi(\x;t) - \Phi(\y;t)||_1$ for $t \geq \bar{t}(\y)$. Letting $\x':=\Phi(\x;\bar{t}(\y))$, $\y':=\Phi(\y;\bar{t}(\y))$ and using Proposition~\ref{prop:mspd}, this amounts to studying $||\Phi(\x';t) - \Phi(\y';t)||_1$ for $t \geq 0$. By the definition of $\bar{t}$, $\Nb(\x')=\Nb(\y')=0$, so that $\tblambda(\x')=\tblambda(\y')$. Hence, for all $t \geq 0$, Lemma~\ref{lem:contracttPhi} yields
  \begin{equation*}
    ||\Phi(\x';t) - \Phi(\y';t)||_1 = ||\tPhi[\tblambda(\x')](\x';t) - \tPhi[\tblambda(\x')](\y';t)||_1 \leq ||\x'-\y'||_1.
  \end{equation*}
  Using the bound obtained on $||\x'-\y'||_1$ above, we finally deduce that
  \begin{equation*}
    \sup_{t \geq 0} ||\Phi(\x;t) - \Phi(\y;t)||_1 \leq \left(1 + \frac{2n}{\ConstUSH}\ConstBound{1}\right)\delta,
  \end{equation*}
  so that the conclusion follows from taking $\delta$ small enough for the inequality
  \begin{equation*}
    \left(1 + \frac{2n}{\ConstUSH}\ConstBound{1}\right)\delta \leq \epsilon
  \end{equation*}  
  to hold.
  
  We now let $N \geq 0$ such that, for all $\x \in \Dnd$ such that $\Nb(\x) \leq N$, the conclusion of Proposition~\ref{prop:continuity} holds. Let us fix $\epsilon > 0$ and $\x \in \Dnd$, such that $\Nb(\x) = N+1$. We are willing to construct $\delta > 0$ such that, for all $\y \in \barB_1(\x,\delta)$,
  \begin{equation*}
    \sup_{t \geq 0} ||\Phi(\x;t) - \Phi(\y;t)||_1 \leq \epsilon.
  \end{equation*}
  
  First, by Corollary~\ref{cor:ttinter}, $\Nb(\x^*) \leq N$, therefore there exists $\delta^* > 0$ such that, for all $\y \in \Dnd$, if $||\x^* - \Phi(\y;t^*(\x))||_1 \leq \delta^*$, then
  \begin{equation*}
    \sup_{t \geq 0} ||\Phi(\x^*;t)-\Phi(\Phi(\y;t^*(\x));t)||_1 \leq \epsilon,
  \end{equation*}
  that is to say, thanks to the flow property stated in Proposition~\ref{prop:mspd},
  \begin{equation*}
    \sup_{t \geq t^*(\x)} ||\Phi(\x;t)-\Phi(\y;t)||_1 \leq \epsilon.
  \end{equation*}
  
  We now prove that there exists $\delta > 0$ such that, for all $\y \in \barB_1(\x,\delta)$, $\sup_{t \in [0,t^*(\x)]} ||\Phi(\x;t)-\Phi(\y;t)||_1 \leq \epsilon$ and $||\x^* - \Phi(\y;t^*(\x))||_1 \leq \delta^*$; which we shall actually do at once by constructing $\delta > 0$ such that, for all $\y \in \barB_1(\x,\delta)$,
  \begin{equation*}
    \sup_{t \in [0,t^*(\x)]} ||\Phi(\x;t)-\Phi(\y;t)||_1 \leq \epsilon \wedge \delta^*.
  \end{equation*}
 
  To this aim, we first assume that $\x \in \Drnd$. Then, by~\eqref{it:Drondnd:1} in Lemma~\ref{lem:Drondnd}, there exists $\eta > 0$ such that, for all $\y \in \barB_1(\x,\eta)$, $\Rb(\x)=\Rb(\y)$, and therefore $\tblambda(\x) = \tblambda(\y) =: \bblambda$. As a consequence, for all $t \in [0, t^*(\x) \wedge t^*(\y)]$, 
  \begin{equation*}
    \Phi(\x;t) = \tPhi[\bblambda](\x;t), \qquad \Phi(\y;t) = \tPhi[\bblambda](\y;t),
  \end{equation*}
  so that Lemma~\ref{lem:contracttPhi} yields
  \begin{equation*}
    \forall t \in [0, t^*(\x) \wedge t^*(\y)], \qquad ||\Phi(\x;t) - \Phi(\y;t)||_1 \leq ||\x-\y||_1.
  \end{equation*}
  Letting $\x' := \Phi(\x;t^*(\x) \wedge t^*(\y))$, $\y' := \Phi(\y;t^*(\x) \wedge t^*(\y))$, one still has the trivial bound, for all $t \in [t^*(\x) \wedge t^*(\y), t^*(\x)]$,
  \begin{equation*}
    \begin{aligned}
      ||\Phi(\x;t) - \Phi(\y;t)||_1 & \leq ||\x'-\y'||_1 + 2\ConstBound{1} (t-t^*(\x) \wedge t^*(\y))\\
      & \leq ||\x-\y||_1 + 2\ConstBound{1}|t^*(\x) - t^*(\y)|.
    \end{aligned}
  \end{equation*}
  As a conclusion, for $\y \in \Dnd$ such that $||\x-\y||_1 \leq \eta$,
  \begin{equation*}
    \sup_{t \in [0,t^*(\x)]} ||\Phi(\x;t)-\Phi(\y;t)||_1 \leq ||\x-\y||_1 + 2\ConstBound{1}|t^*(\x) - t^*(\y)|.
  \end{equation*}
  By Lemma~\ref{lem:Drondnd}, there exists $\delta > 0$ such that, for all $\y \in \Dnd$ such that $||\x-\y||_1 \leq \delta$, the right-hand side above is lower than $\epsilon \wedge \delta^*$. This completes the proof of the case $\x \in \Drnd$.
  
  Without assuming that $\x \in \Drnd$, we proceed as follows. Let $\eta' > 0$ be given by Lemma~\ref{lem:Ddroitnd}. Let us note that, since $\Nb(\x) \geq 1$, then $\eta' < +\infty$. Besides, the proof of Lemma~\ref{lem:Ddroitnd} shows that $t^*(\x) \geq 3n\eta'/\ConstBound{1}$. Let us denote
  \begin{equation*}
    t' := \frac{2n\eta'}{\ConstBound{1}} \in (0, t^*(\x)).
  \end{equation*}
  Then $\Phi(\x;t') \in \Drnd$, and $\Rb(\Phi(\x;t'))=\Rb(\x)$. As a consequence, using the argument above, we obtain that there exists $\delta' > 0$ such that, for all $\y \in \Dnd$ such that $\Phi(\y;t') \in \barB_1(\Phi(\x;t'),\delta')$, then
  \begin{equation*}
    \sup_{t \in [t',t^*(\x)]} ||\Phi(\x;t)-\Phi(\y;t)||_1 \leq \epsilon \wedge \delta^*.
  \end{equation*}
  Now, for all $\y \in \Dnd$ such that $||\x-\y||_1 \leq n\eta' \ConstUSH/\ConstBound{1}$, then
  \begin{equation*}
    t'' := \inf\{t \geq 0 : \Rb(\Phi(\y;t)) = \Rb(\x)\} \leq t',
  \end{equation*}
  and
  \begin{equation*}
    \sup_{t \in [0,t'']} ||\Phi(\x;t)-\Phi(\y;t)||_1 \leq ||\x-\y||_1 + 2\ConstBound{1} t'' \leq \left(1 + 2n \frac{\ConstBound{1}}{\ConstUSH}\right)||\x-\y||_1,
  \end{equation*}
  where the bound on $t''$ follows from Lemma~\ref{lem:Ddroitnd}. On the other hand, using Lemma~\ref{lem:Ddroitnd} again, we obtain that, on the time interval $[t'',t']$, $\Rb(\Phi(\y;t))=\Rb(\x)=\Rb(\Phi(\x;t))$, therefore by Lemma~\ref{lem:contracttPhi},
  \begin{equation*}
    \sup_{t \in [t'',t']} ||\Phi(\x;t)-\Phi(\y;t)||_1 \leq ||\Phi(\x;t'')-\Phi(\y;t'')||_1 \leq \left(1 + 2n \frac{\ConstBound{1}}{\ConstUSH}\right)||\x-\y||_1.
  \end{equation*}
  As a consequence, letting
  \begin{equation*}
    \delta := \min\left(n\eta' \frac{\ConstUSH}{\ConstBound{1}}, \frac{\epsilon \wedge \delta'}{1 + 2n \ConstBound{1}/\ConstUSH}\right),
  \end{equation*}
  we conclude that, for all $\y \in \barB_1(\x,\delta)$,
  \begin{equation*}
    ||\Phi(\x;t')-\Phi(\y;t')||_1 \leq \delta',
  \end{equation*}
  while
  \begin{equation*}
    \sup_{t \in [0,t']} ||\Phi(\x;t)-\Phi(\y;t)||_1 \leq \epsilon,
  \end{equation*}
  which completes the proof.
\end{proof}


\subsection{Proof of Proposition~\ref{prop:closedness}}\label{app:pf:closedness} Before proving Proposition~\ref{prop:closedness}, we state and prove the technical Lemmas~\ref{lem:ellsol} and~\ref{lem:FnGn}.

\begin{lem}[An extended change of variable formula]\label{lem:ellsol}
  Let $\ell : [0,1] \times \R \to \R$ be a measurable and bounded function, and $F$ be a CDF on the real line. Then
  \begin{equation}\label{eq:ellsol}
    \int_{x \in \R} \int_{\theta=0}^1 \ell((1-\theta)F(x^-)+\theta F(x), x) \dd \theta \dd F(x) = \int_{v=0}^1 \ell(v, F^{-1}(v)) \dd v.
  \end{equation}
\end{lem}
\begin{proof}
  Let us split the integral in the left-hand side of~\eqref{eq:ellsol} in two parts, depending on whether $\Delta F(x) = 0$ or $\Delta F(x) > 0$. On the one hand, using Lemma~\ref{lem:CDFm1},
  \begin{equation*}
    \begin{aligned}
      & \int_{x \in \R} \int_{\theta=0}^1 \ind{\Delta F(x)=0} \ell((1-\theta)F(x^-)+\theta F(x), x) \dd \theta \dd F(x) = \int_{x \in \R} \ind{\Delta F(x)=0} \ell(F(x), x) \dd F(x)\\
      & \qquad = \int_{v=0}^1 \ind{\Delta F(F^{-1}(v))=0} \ell(F(F^{-1}(v)), F^{-1}(v)) \dd v,
    \end{aligned} 
  \end{equation*}
  and it follows from~\eqref{it:pseudoinv:1} in Lemma~\ref{lem:pseudoinv} that, if $\Delta F(F^{-1}(v))=0$, then $F(F^{-1}(v))=v$. As a consequence,
  \begin{equation*}
    \int_{x \in \R} \int_{\theta=0}^1 \ind{\Delta F(x)=0} \ell((1-\theta)F(x^-)+\theta F(x), x) \dd \theta \dd F(x) = \int_{v=0}^1 \ind{\Delta F(F^{-1}(v))=0} \ell(v, F^{-1}(v)) \dd v.
  \end{equation*}
  
  On the other hand, 
  \begin{equation*}
    \begin{aligned}
      & \int_{x \in \R} \int_{\theta=0}^1 \ind{\Delta F(x)>0} \ell((1-\theta)F(x^-)+\theta F(x), x) \dd \theta \dd F(x) \\
      & \qquad = \int_{v=0}^1\ind{\Delta F(F^{-1}(v))>0}  \int_{\theta=0}^1 \ell((1-\theta)F(F^{-1}(v)^-)+\theta F(F^{-1}(v)), F^{-1}(v)) \dd v \dd \theta\\
      & \qquad = \int_{v=0}^1\ind{\Delta F(F^{-1}(v))>0} \frac{1}{\Delta F(F^{-1}(v))} \int_{w=F(F^{-1}(v)^-)}^{F(F^{-1}(v))} \ell(w,F^{-1}(v)) \dd w \dd v\\
      & \qquad = \int_{v=0}^1 \int_{w=0}^1 \ind{\Delta F(F^{-1}(v))>0, F(F^{-1}(v)^-) < w \leq F(F^{-1}(v))} \frac{\ell(w,F^{-1}(v))}{\Delta F(F^{-1}(v))} \dd w \dd v.
    \end{aligned} 
  \end{equation*}
  The key observation here is that, if $v \in (0,1)$ is such that $\Delta F(F^{-1}(v)) > 0$, then, for all $w$ such that
  \begin{equation*}
    F(F^{-1}(v)^-) < w \leq F(F^{-1}(v)),
  \end{equation*}
  one has $F^{-1}(w) = F^{-1}(v)$. As a consequence, the right-hand side above rewrites
  \begin{equation*}
    \begin{aligned}
      & \int_{v=0}^1 \int_{w=0}^1 \ind{\Delta F(F^{-1}(v))>0, F(F^{-1}(v)^-) < w \leq F(F^{-1}(v))} \frac{\ell(w,F^{-1}(v))}{\Delta F(F^{-1}(v))} \dd w \dd v\\
      & \qquad = \int_{v=0}^1 \int_{w=0}^1 \ind{\Delta F(F^{-1}(w))>0, F(F^{-1}(v)^-) < w \leq F(F^{-1}(v))} \frac{\ell(w,F^{-1}(w))}{\Delta F(F^{-1}(w))} \dd w \dd v\\
      & \qquad = \int_{w=0}^1 \ind{\Delta F(F^{-1}(w))>0}\frac{\ell(w,F^{-1}(w))}{\Delta F(F^{-1}(w))} \int_{v=0}^1 \ind{F(F^{-1}(v)^-) < w \leq F(F^{-1}(v))} \dd v \dd w.
    \end{aligned} 
  \end{equation*}
  We now complete the proof by checking that, $\dd w$-almost everywhere, if $\Delta F(F^{-1}(w))>0$ then
  \begin{equation*}
    \int_{v=0}^1 \ind{F(F^{-1}(v)^-) < w \leq F(F^{-1}(v))} \dd v = \Delta F(F^{-1}(w)).
  \end{equation*}
  To this aim, we note that for all $w \in (0,1)$ such that $\Delta F(F^{-1}(w))>0$,
  \begin{equation*}
    \begin{aligned}
      \int_{v=0}^1 \ind{F(F^{-1}(v)^-) < w \leq F(F^{-1}(v))} \dd v & = \int_{x \in \R} \ind{F(x^-) < w \leq F(x)} \dd F(x)\\
      & = \sum_{x : \Delta F(x) > 0} \ind{F(x^-) < w \leq F(x)} \Delta F(x),
    \end{aligned}
  \end{equation*}
  where we have used Lemma~\ref{lem:CDFm1} at the first line. 
  
  Recall that, by~\eqref{it:pseudoinv:1} in Lemma~\ref{lem:pseudoinv}, $F(F^{-1}(w)^-) \leq w \leq F(F^{-1}(w))$. As a consequence, if $w$ is not taken from the countable set of values of $F(x^-)$ when $x$ is an atom of $\dd F$, then the sum above contains exactly one positive term, which corresponds to $x=F^{-1}(w)$ and therefore writes $\Delta F(F^{-1}(w))$.
\end{proof}

\begin{lem}[Convergence of composed CDFs]\label{lem:FnGn}
  Let $(F_n)_{n \geq 1}$ and $(G_n)_{n \geq 1}$ be two sequences of CDFs on $\R$ and $F$ and $G$ be two CDFs on $\R$, such that:
  \begin{itemize}
    \item for all $x \in \R$ such that $\Delta F(x)=0$, $\lim_{n \to +\infty} F_n(x) = F(x)$,
    \item for all $x \in \R$ such that $\Delta G(x)=0$, $\lim_{n \to +\infty} G_n(x) = G(x)$,
    \item for all $x \in \R$, $\Delta F(x) \Delta G(x) = 0$.
  \end{itemize}
  Then, $\dd v$-almost everywhere,
  \begin{equation*}
    \lim_{n \to +\infty} G_n(F_n^{-1}(v)) = G(F^{-1}(v)) \qquad \text{and} \qquad \lim_{n \to +\infty} G_n(F_n^{-1}(v)^-) = G(F^{-1}(v)^-).
  \end{equation*}
\end{lem}
\begin{proof}
  By Lemma~\ref{lem:cvCDF}, $F_n^{-1}(v)$ converges to $F^{-1}(v)$, $\dd v$-almost everywhere in $(0,1)$. We now check that, for all $x \in \R$ such that $\Delta G(x) > 0$, the set $\{v \in (0,1) : F^{-1}(v)=x\}$ is negligible with respect to the Lebesgue measure on $(0,1)$. Since the function $F^{-1}$ is nondecreasing, this set is an interval, and if there exist $\uv < \ov$ such that $F^{-1}(\uv)=F^{-1}(\ov)=x$, then $F(x^-) \leq \uv < \ov \leq F(x)$, which is a contradiction with the fact that $\Delta F(x) \Delta G(x) = 0$.
  
  As a consequence, $\dd v$-almost everywhere, $F_n^{-1}(v)$ converges to $F^{-1}(v)$ and $\Delta G(F^{-1}(v)) = 0$. Let us fix $v \in (0,1)$ satisfying these two conditions, and write $G_n(F_n^{-1}(v)) = G_n(F^{-1}(v)) + G_n(F_n^{-1}(v)) - G_n(F^{-1}(v))$. On the one hand,
  \begin{equation*}
    \lim_{n \to +\infty} G_n(F^{-1}(v)) = G(F^{-1}(v)),
  \end{equation*}
  since $\Delta G(F^{-1}(v)) = 0$. On the other hand, by the Dominated Convergence Theorem, for all $\epsilon > 0$, there exists $\delta > 0$ such that
  \begin{equation*}
    \int_{x \in \R} \ind{F^{-1}(v) - \delta \leq x \leq F^{-1}(v) + \delta} \dd G(x) \leq \epsilon.
  \end{equation*}
  Besides, for $n$ large enough, $F_n^{-1}(v) \in (F^{-1}(v) - \delta, F^{-1}(v) + \delta)$, so that
  \begin{equation*}
    |G_n(F_n^{-1}(v)) - G_n(F^{-1}(v))| \leq \int_{x \in \R} \ind{F^{-1}(v) - \delta \leq x \leq F^{-1}(v) + \delta} \dd G_n(x).
  \end{equation*}
  We now deduce from the characterisation of weak convergence on closed sets in the Portmanteau Theorem~\cite[Theorem~2.1, p.~16]{billingsley} that
  \begin{equation*}
    \limsup_{n \to +\infty} |G_n(F_n^{-1}(v)) - G_n(F^{-1}(v))| \leq \epsilon,
  \end{equation*}
  which completes the proof of the first assertion. 
  
  To prove the second assertion, we follow the same arguments and first show that $G_n(F^{-1}(v)^-)$ converges to $G(F^{-1}(v)^-)$ --- which, in fact, is $G(F^{-1}(v))$. To this aim, we take $\epsilon$ small and such that $\Delta G(F^{-1}(v)-\epsilon) = 0$, so that
  \begin{equation*}
    G(F^{-1}(v)-\epsilon) \leq \liminf_{n \to +\infty} G_n(F^{-1}(v)^-) \leq \limsup_{n \to +\infty} G_n(F^{-1}(v)^-) \leq G_n(F^{-1}(v)),
  \end{equation*}
  and, since $\epsilon$ can be chosen arbitrarily small, the result follows from the fact that $\Delta G(F^{-1}(v)) = 0$. The sequel of the proof is identical to the first case.
\end{proof}

We are now ready to prove Proposition~\ref{prop:closedness}.

\begin{proof}[Proof of Proposition~\ref{prop:closedness}]
  Let $(\bu_n)_{n \geq 1}$ and $\bu$ satisfy the assumptions of Proposition~\ref{prop:closedness}. Let us fix $\bvarphi = (\varphi^1, \ldots, \varphi^d) \in \Cs^{1,0}_{\mathrm{c}}([0,+\infty)\times\R, \R^d)$ and $\gamma \in \{1, \ldots, d\}$. For all $t \geq 0$, the set of points $x \in \R$ such that $\Delta_x u^{\gamma}(t,x) > 0$ is at most countable, therefore $\dd x$-almost everywhere, $u_n^{\gamma}(t,x)$ converges to $u^{\gamma}(t,x)$. By the Dominated Convergence Theorem, we deduce that
  \begin{equation*}
    \begin{aligned}
      & \lim_{n \to +\infty} \int_{t=0}^{+\infty} \int_{x \in \R} \partial_t\varphi^{\gamma}(t,x) u_n^{\gamma}(t,x) \dd x\dd t + \int_{x \in \R} \varphi^{\gamma}(0,x) u_{0,n}^{\gamma}(x) \dd x\\
      & \qquad = \int_{t=0}^{+\infty} \int_{x \in \R} \partial_t\varphi^{\gamma}(t,x) u^{\gamma}(t,x) \dd x\dd t + \int_{x \in \R} \varphi^{\gamma}(0,x) u_0^{\gamma}(x) \dd x.
    \end{aligned}
  \end{equation*}
  The main difficulty of the proof actually lies in checking that
  \begin{equation}\label{eq:realdiff}
    \lim_{n \to +\infty} \int_{t=0}^{+\infty} \int_{x \in \R} \varphi^{\gamma}(t,x) \lambda^{\gamma}\{\bu_n\}(t,x) \dd_x u_n^{\gamma}(t,x) \dd t = \int_{t=0}^{+\infty} \int_{x \in \R} \varphi^{\gamma}(t,x) \lambda^{\gamma}\{\bu\}(t,x) \dd_x u^{\gamma}(t,x) \dd t.
  \end{equation}
  
  In the scalar case,~\eqref{eq:ippscal} yields, for all $t \geq 0$,
  \begin{equation*}
    \int_{x \in \R} \varphi(t,x) \lambda\{u_n\}(t,x) \dd_x u_n(t,x) = - \int_{x \in \R} \partial_x\varphi(t,x) \Lambda(u_n(t,x))\dd x
  \end{equation*}
  and, similarly,
  \begin{equation*}
    \int_{x \in \R} \varphi(t,x) \lambda\{u\}(t,x) \dd_x u(t,x) = - \int_{x \in \R} \partial_x\varphi(t,x) \Lambda(u(t,x))\dd x,
  \end{equation*}
  so that the limit~\eqref{eq:realdiff} is easy to obtain, at least for test functions having a continuous partial derivative $\partial_x\varphi$.

  In the general case, Lemma~\ref{lem:ellsol} above allows us to rewrite~\eqref{eq:realdiff} under the following equivalent form:
  \begin{equation*}
    \begin{aligned}
      & \lim_{n \to +\infty} \int_{t=0}^{+\infty} \int_{v=0}^1 \varphi^{\gamma}\left(t,u_n^{\gamma}(t,\cdot)^{-1}(v)\right) \lambda^{\gamma}\left(u_n^1(t, u^{\gamma}_n(t,\cdot)^{-1}(v)), \ldots, v, \ldots, u_n^d(t, u^{\gamma}_n(t,\cdot)^{-1}(v))\right) \dd v \dd t\\
      & \qquad = \int_{t=0}^{+\infty} \int_{v=0}^1 \varphi^{\gamma}\left(t,u^{\gamma}(t,\cdot)^{-1}(v)\right) \lambda^{\gamma}\left(u^1(t, u^{\gamma}(t,\cdot)^{-1}(v)), \ldots, v, \ldots, u^d(t, u^{\gamma}(t,\cdot)^{-1}(v))\right) \dd v \dd t.
    \end{aligned}
  \end{equation*}
  By the Dominated Convergence Theorem and thanks to the continuity of the functions $\varphi^{\gamma}(t,\cdot)$ and $\lambda^{\gamma}$, this identity follows if we first prove that, $\dd t$-almost everywhere, $\dd v$-almost everywhere, for all $\gamma, \gamma' \in \{1, \ldots, d\}$ with $\gamma\not=\gamma'$, 
  \begin{equation*}
    \lim_{n \to +\infty} u_n^{\gamma}(t,\cdot)^{-1}(v) = u^{\gamma}(t,\cdot)^{-1}(v), \qquad \lim_{n \to +\infty} u_n^{\gamma'}\left(t, u_n^{\gamma}(t,\cdot)^{-1}(v)\right) = u^{\gamma'}\left(t, u^{\gamma}(t,\cdot)^{-1}(v)\right).
  \end{equation*}
  These equalities are obtained by applying Lemma~\ref{lem:FnGn} above at all times $t$ such that
  \begin{equation*}
    \forall x \in \R, \qquad \Delta_x u_n^{\gamma}(t,x)\Delta_x u_n^{\gamma'}(t,x)=0.
  \end{equation*}
  On account of Condition~\eqref{eq:mutgammutgamp}, this is the case $\dd t$-almost everywhere, which completes the proof.
\end{proof}


\subsection{Proof of Lemma~\ref{lem:gooddense}}\label{app:pf:gooddense} We now detail the proof of Lemma~\ref{lem:gooddense}, which asserts that the set of good configurations $\Good$ is dense in $\Dnd$ under Assumptions~\eqref{ass:C} and~\eqref{ass:USH}, and Condition~\eqref{cond:ND}.

\begin{proof}[Proof of Lemma~\ref{lem:gooddense}]
  Let us begin the proof by recalling the chain of inclusions
  \begin{equation*}
    \Good \subset \Drnd \subset \Dnd,
  \end{equation*}
  and that $\Drnd$ is dense in $\Dnd$. As a consequence, it suffices to prove that, for all $\x \in \Drnd$, for all $\epsilon > 0$, there exists $\y \in \Good$ such that $||\x-\y||_1 \leq \epsilon$. The reader will not be surprised that the proof works by induction on $\Nb(\x)$.
  
  If $\x \in \Drnd$ and $\Nb(\x)=0$, then $\x \in \Good$ and there is nothing to prove. Now let $N \geq 0$ such that any $\x \in \Drnd$ with $\Nb(\x) \leq N$ belongs to the closure of $\Good$. Let $\x \in \Drnd$ with $\Nb(\x)=N+1$; in particular, $t^*(\x) < +\infty$. Let us fix 
  \begin{equation*}
    t^*(\x) < t' < t'' < t^*(\x) + t^*(\x^*),
  \end{equation*}
  such that, in the MSPD started at $\x$, there is no self-interaction on the time interval $(t^*(\x), t'')$, see Figure~\ref{fig:pf:gooddense:1}. 

  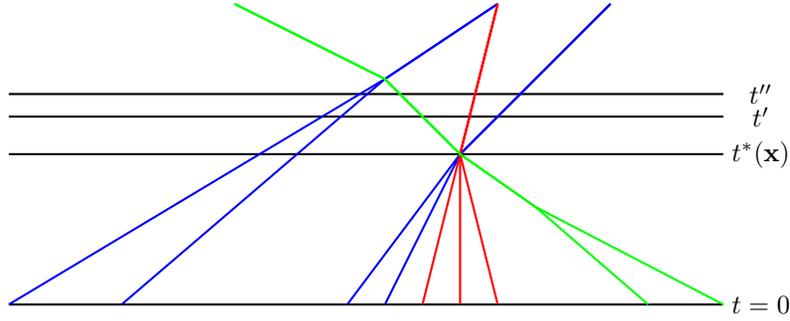
\begin{figure}[ht]
    \begin{pspicture}(10,4)
      \psline[linecolor=black](0,0)(9.5,0)
      \psline[linecolor=black](0,2)(9.5,2)
      \psline[linecolor=black](0,2.5)(9.5,2.5)
      \psline[linecolor=black](0,2.8)(9.5,2.8)
      
      \psline[linecolor=blue](0,0)(5,3)(6.5,4)
      \psline[linecolor=blue](1.5,0)(5,3)(6.5,4)
      \psline[linecolor=blue](4.5,0)(6,2)(8,4)
      \psline[linecolor=blue](5,0)(6,2)(8,4)
      
      \psline[linecolor=red](5.5,0)(6,2)(6.5,4)
      \psline[linecolor=red](6,0)(6,2)(6.5,4)
      \psline[linecolor=red](6.5,0)(6,2)(6.5,4)
      
      \psline[linecolor=green](8.5,0)(7,1.3)(6,2)(5,3)(3,4)
      \psline[linecolor=green](9.5,0)(7,1.3)(6,2)(5,3)(3,4)
      
      \rput(10,0){\textcolor{black}{$t=0$}}
      \rput(10,2){\textcolor{black}{$t^*(\x)$}}
      \rput(10,2.5){\textcolor{black}{$t'$}}
      \rput(10,2.8){\textcolor{black}{$t''$}}
      
    \end{pspicture}
    \caption{The choices of $t'$ and $t''$ to ensure that, on the time interval $(t^*(\x), t'']$, there is neither a self-interaction nor a collision in the MSPD started at $\x$.}
    \label{fig:pf:gooddense:1}
  \end{figure}
  
  We shall prove in Step~1 below that, for all $\epsilon > 0$, there exists $\x' \in \barB_1(\x,\epsilon)$ and $s' \in (0,t^*(\x))$ such that:
  \begin{itemize}
    \item in the MSPD started at $\x'$, there is no self-interaction on the time interval $[s',t^*(\x)]$,
    \item for all $t \geq t^*(\x)$, $\Phi(\x;t)=\Phi(\x';t)$.
  \end{itemize}
  As a consequence, we shall assume, without loss of generality, that $\x$ satisfies the following property: there exists $s' \in (0,t^*(\x))$ such that, in the MSPD started at $\x$, there is no self-interaction on the time interval $[s',t^*(\x)]$, see Figure~\ref{fig:pf:gooddense:2}. 
  
  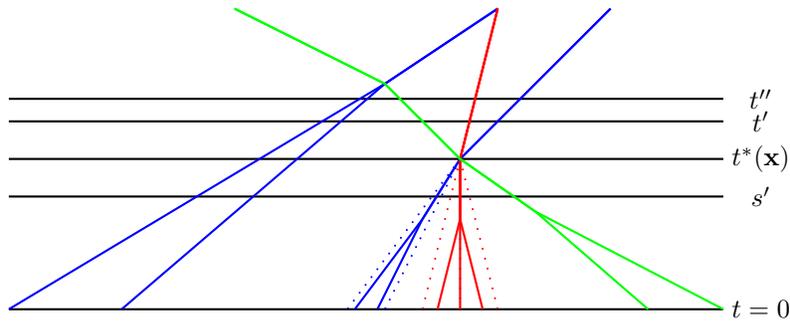
\begin{figure}[ht]
    \begin{pspicture}(10,4)
      \psline[linecolor=black](0,0)(9.5,0)
      \psline[linecolor=black](0,1.5)(9.5,1.5)
      \psline[linecolor=black](0,2)(9.5,2)
      \psline[linecolor=black](0,2.5)(9.5,2.5)
      \psline[linecolor=black](0,2.8)(9.5,2.8)
      
      \psline[linecolor=blue](0,0)(5,3)(6.5,4)
      \psline[linecolor=blue](1.5,0)(5,3)(6.5,4)
      \psline[linecolor=blue,linestyle=dotted](4.5,0)(6,2)
      \psline[linecolor=blue,linestyle=dotted](5,0)(6,2)
      \psline[linecolor=blue](4.6,0)(5.5,1.2)(6,2)(8,4)
      \psline[linecolor=blue](4.9,0)(5.5,1.2)(6,2)(8,4)
      
      \psline[linecolor=red,linestyle=dotted](5.5,0)(6,2)
      \psline[linecolor=red,linestyle=dotted](6,0)(6,2)
      \psline[linecolor=red,linestyle=dotted](6.5,0)(6,2)
      \psline[linecolor=red](5.7,0)(6,1.2)(6,2)(6.5,4)
      \psline[linecolor=red](6,0)(6,1.2)(6,2)(6.5,4)
      \psline[linecolor=red](6.3,0)(6,1.2)(6,2)(6.5,4)
      
      \psline[linecolor=green](8.5,0)(7,1.3)(6,2)(5,3)(3,4)
      \psline[linecolor=green](9.5,0)(7,1.3)(6,2)(5,3)(3,4)
      
      \rput(10,0){\textcolor{black}{$t=0$}}
      \rput(10,1.5){\textcolor{black}{$s'$}}
      \rput(10,2){\textcolor{black}{$t^*(\x)$}}
      \rput(10,2.5){\textcolor{black}{$t'$}}
      \rput(10,2.8){\textcolor{black}{$t''$}}
      
    \end{pspicture}
    \caption{The shrinking of particles having a self-interaction at time $t^*(\x)$ allows to select $s' < t^*(\x)$ such that there is no self-interaction on the time interval $[s',t'']$.}
    \label{fig:pf:gooddense:2}
  \end{figure}
  
  Then, arguing as in the proof of Lemma~\ref{lem:radial}, there exists $\epsilon_0 > 0$ such that, for all $\y \in \barB_1(\x,\epsilon_0)$,
  \begin{enumerate}
    \item $\y \in \Drnd$ and $\Rb(\y)=\Rb(\x)$,
    \item $\Phi(\y;s') \in \Drnd$ and $\Rb(\Phi(\y;s'))=\Rb(\Phi(\x;s'))$, which implies that, in the MSPD started at $\y$, there is no collision on the time interval $[0,s']$,
    \item for all $\gamma:k \in \Part$, $\clu_k^{\gamma}(\y;s') = \clu_k^{\gamma}(\x;s')$,
    \item $\Phi(\y;t') \in \Drnd$ and $\Rb(\Phi(\y;t')) = \Rb(\Phi(\x;t'))$, which implies that 
    \begin{equation*}
      \{(\alpha:i,\beta:j) \in \Rb(\x) : \tinter_{\alpha:i,\beta:j}(\x) \in [t^*(\x),t']\} = \{(\alpha:i,\beta:j) \in \Rb(\y) : \tinter_{\alpha:i,\beta:j}(\y) \in (s',t')\}
    \end{equation*}
    \ie the collisions in the MSPD started at $\x$ on the time interval $[t^*(\x),t']$ (or, equivalently, $(s',t')$) involve the same pairs of particles as the collisions in the MSPD started at $\y$ on the time interval $(s',t')$,
    \item $\Phi(\y;t'') \in \Drnd$ and $\Rb(\Phi(\y;t'')) = \Rb(\Phi(\x;t''))$, which implies that, in the MSPD started at $\y$, there is no collision on the time interval $[t',t'']$,
    \item for all $\gamma:k \in \Part$, $\clu_k^{\gamma}(\y;t'') = \clu_k^{\gamma}(\x;t'')$, which implies that, in the MSPD started at $\y$, there is no self-interaction on the time interval $[s',t'']$.
  \end{enumerate}
  
  Let us fix $\epsilon \in (0, \epsilon_0]$. The sequel of the proof is as follows: in Step~2, we construct $\y_0 \in \barB_1(\x,\epsilon/2)$ such that, in the MSPD started at $\y_0$, the collisions on the time interval $[0,t'']$ (or, equivalently, $(s',t')$), are binary. Of course, $||\x-\y_0||_1 \leq \epsilon_0$, therefore $\y_0$ satisfies all the conditions above; in particular, in the MSPD started at $\y_0$, the self-interactions are separated from collisions on the time interval $[0,t'']$. In Step~3, we show that there exists $\eta \in (0,\epsilon/2]$ such that, for all $\y \in \barB_1(\y_0,\eta)$, the collisions on the time interval $[0,t'']$ in the MSPD started at $\y$ remain binary. In Step~4, we construct $\eta' > 0$ such that, for all $\y' \in \barB_1(\Phi(\y_0;t'),\eta')$, there exists $\y \in \barB_1(\y_0,\eta)$ such that $\Phi(\y';t''-t') = \Phi(\y;t'')$.
  
  Taking the result of these four steps for granted, let us explain how to complete the proof. By construction, the collisions in the time interval $[0,t']$ in the MSPD started at $\y_0$ are binary and separated from self-interactions. Besides, $\Nb(\Phi(\y_0;t')) = \Nb(\Phi(\x;t')) \leq N$, therefore there exists $\y' \in \barB_1(\Phi(\y_0;t'),\eta')$ such that $\y' \in \Good$. Let $\y \in \barB_1(\y_0, \eta)$ be given by Step~4. Then, on the one hand,
  \begin{equation*}
    ||\x-\y||_1 \leq ||\x-\y_0||_1+||\y_0-\y||_1 \leq \frac{\epsilon}{2} + \eta \leq \epsilon,
  \end{equation*}
  while, on the other hand, 
  \begin{itemize}
    \item since $\y \in \barB_1(\y_0,\eta)$, the collisions are binary and separated from self-interactions on the time interval $[0,t'']$ in the MSPD started at $\y$,
    \item since $\y' \in \Good$, the collisions are binary and separated from self-interactions in the MSPD started at $\Phi(\y';t''-t')=\Phi(\y;t'')$.
  \end{itemize}
  As a consequence of the flow property for the MSPD, $\y \in \Good$ and the proof is completed.
  
  Let us now give a detailed proof of Steps~1 to~4.
    
  \sk
  \noindent {\em Step~1.} We separate self-interactions from collisions by shrinking groups of particles involved in self-interactions at time $t^*(\x)$ around their centre of mass, as is depicted on Figure~\ref{fig:pf:gooddense:2}. Let us fix $\epsilon > 0$ and assume that there exist $\gamma \in \{1, \ldots, d\}$ and $\uk < \ok$ such that 
  \begin{equation*}
    \clu_{\uk}^{\gamma}(\x;t^*(\x)) = \clu_{\ok}^{\gamma}(\x;t^*(\x)) = \gamma:\uk\cdots\ok, \qquad \clu_{\uk}^{\gamma}(\x;t^*(\x)^-) \not= \clu_{\ok}^{\gamma}(\x;t^*(\x)^-),
  \end{equation*}
  that is to say, a self-interaction occurs at time $t^*(\x)$ between the particles $\gamma:\uk, \ldots, \gamma:\ok$. Let us define
  \begin{equation*}
    \xi := \frac{1}{\ok-\uk+1} \sum_{k=\uk}^{\ok} x_k^{\gamma},
  \end{equation*}
  and denote by $\x^{\rho}$ the configuration in $\Dnd$ such that, for all $\gamma':k' \in \Part$,
  \begin{equation*}
    (x^{\rho})_{k'}^{\gamma'} := \left\{\begin{aligned}
      & x_{k'}^{\gamma'} & \text{if $\gamma':k' \not\in \gamma:\uk\cdots\ok$},\\
      & (1-\rho)\xi + \rho x_{k'}^{\gamma'} & \text{otherwise},
    \end{aligned}\right.
  \end{equation*}
  for all $\rho \in [0,1]$. Then, it is easily seen that, for all $\gamma':k' \not\in \gamma:\uk\cdots\ok$,
  \begin{equation*}
    \forall t \in [0,t'], \qquad \Phi_{k'}^{\gamma'}(\x^{\rho};t) = \Phi_{k'}^{\gamma'}(\x;t).
  \end{equation*}
  Besides, we claim that
  \begin{enumerate}
    \item $\inf\{t \geq 0 : \Phi_{\uk}^{\gamma}(\x^{\rho};t) = \Phi_{\ok}^{\gamma}(\x^{\rho};t)\} = \rho t^*(\x)$,
    \item for all $k \in \{\uk, \ldots, \ok\}$, for all $t \geq \rho t^*(\x)$, $\Phi_k^{\gamma}(\x^{\rho};t) = \Phi_k^{\gamma}(\x^0;t)$,
    \item for all $k \in \{\uk, \ldots, \ok\}$, for all $t \geq t^*(\x)$, $\Phi_k^{\gamma}(\x^{\rho};t) = \Phi_k^{\gamma}(\x;t)$.
  \end{enumerate}
  The first point is obtained by elementary geometry if there is no self-interaction between times $0$ and $t^*(\x)$. Otherwise, let $c_1, \ldots, c_r$ denote the distinct elements of the set
  \begin{equation*}
    \{\clu_k^{\gamma}(\x;t^*(\x)^-), k \in \{\uk, \ldots, \ok\}\}.
  \end{equation*}
  Let us recall that in the proof of Lemma~\ref{lem:radial}, we made the observation that, in the Local Sticky Particle Dynamics, the centre of mass travels at constant velocity whatever the composition of the clusters. Applying this remark to each generical cluster $c_i$, we write
  \begin{equation*}
    t^*(\x) = \inf\{t \geq 0 : \Phi_{\uk}^{\gamma}(\tilde{\x};t) = \Phi_{\ok}^{\gamma}(\tilde{\x};t)\},
  \end{equation*}
  where $\tilde{\x}$ is derived from $\x$ by the following procedure: for all $i \in \{1, \ldots, r\}$, for all $\gamma:k \in c_i$, replace the coordinate $x_k^{\gamma}$ in $\x$ with
  \begin{equation*}
    \tilde{x}_k^{\gamma} := \frac{1}{|c_i|} \sum_{\gamma:k' \in c_i} x_{k'}^{\gamma}.
  \end{equation*}
  Then, in the MSPD started at $\tilde{\x}$, the particles $\gamma:\uk, \ldots, \gamma:\ok$ do not have self-interactions between times $0$ and $t^*(\x)$, so that the argument above yields
  \begin{equation*}
    \inf\{t \geq 0 : \Phi_{\uk}^{\gamma}(\tilde{\x}^{\rho};t) = \Phi_{\ok}^{\gamma}(\tilde{\x}^{\rho};t)\} = \rho \inf\{t \geq 0 : \Phi_{\uk}^{\gamma}(\tilde{\x};t) = \Phi_{\ok}^{\gamma}(\tilde{\x};t)\} = \rho t^*(\x),
  \end{equation*}
  where $\tilde{\x}^{\rho}$ is derived from $\tilde{\x}$ in the same fashion as $\x^{\rho}$ is derived from $\x$. To complete the argument, we now have to check that
  \begin{equation*}
    \inf\{t \geq 0 : \Phi_{\uk}^{\gamma}(\x^{\rho};t) = \Phi_{\ok}^{\gamma}(\x^{\rho};t)\} = \inf\{t \geq 0 : \Phi_{\uk}^{\gamma}(\tilde{\x}^{\rho};t) = \Phi_{\ok}^{\gamma}(\tilde{\x}^{\rho};t)\}.
  \end{equation*}
  This follows from the fact that the operations mapping $\x$ to $\tilde{\x}$ and $\x$ to $\x^{\rho}$ are commutative; therefore the equality above is obtained by the same geometric arguments as in the case $\rho=1$.
  
  The second and third points above easily follow.
  
  Finally, the configuration $\x^{\rho}$ satisfies
  \begin{equation*}
    ||\x-\x^{\rho}||_1 = \frac{1-\rho}{n} \sum_{k=\uk}^{\ok} |\xi-x_k^{\gamma}|,
  \end{equation*}
  so that for $\rho$ close enough to $1$, $||\x-\x^{\rho}||_1 \leq \epsilon$ while the self-interactions between the particles $\gamma:\uk, \ldots, \gamma:\ok$ in the MSPD started at $\x^{\rho}$ occur before $\rho t^*(\x) < t^*(\x)$, without modifying neither the trajectories of the other particles on $[0,t^*(\x)]$, nor the trajectories of all the particles after $t^*(\x)$ with respect to the MSPD started at $\x$. Applying the argument to the finite number of groups of particles having a self-interaction at time $t^*(\x)$, we conclude that there exists $\x' \in \barB_1(\x, \epsilon)$ and $s' \in (0,t^*(\x))$ such that, in the MSPD started at $\x'$, there is no self-interaction in the time interval $[s',t^*(\x)]$.
  
  \sk
  \noindent {\em Step~2.} We now blow up the non-binary collisions by shifting the initial positions, as is described on Figure~\ref{fig:pf:gooddense:3}. Let us assume that there exist
  \begin{equation*}
    \gamma_1 < \cdots < \gamma_r, \qquad r \geq 3,
  \end{equation*}
  such that, in the MSPD started at $\x$, a collision occurs at the space-time point $(\xi^*,t^*(\x))$ between clusters of type $\gamma_1, \ldots, \gamma_r$. For all $i \in \{1, \ldots, r\}$, let us denote by $c_i$ the cluster of type $\gamma_i$ involved in the collision. For $\theta > 0$, let us define the configuration $\x^{\theta,1}$ as follows: for all $\gamma:k \in \Part$,
  \begin{equation*}
    (x^{\theta,1})_k^{\gamma} := \left\{\begin{aligned}
      & x_k^{\gamma} & \text{if $\gamma:k \not\in c_3 \cup \cdots \cup c_r$},\\
      & x_k^{\gamma} + \theta & \text{if $\gamma:k \in c_3 \cup \cdots \cup c_r$}.
    \end{aligned}\right.
  \end{equation*}
  Note that
  \begin{equation*}
    ||\x-\x^{\theta,1}||_1 \leq \frac{\theta}{n} (|c_3|+\cdots+|c_r|),
  \end{equation*}
  so that $\theta$ can be chosen small enough to ensure that $\x^{\theta,1} \in \barB_1(\x,\epsilon_0)$, and therefore satisfies all the conditions stated in the introduction of the proof. In particular, on the time interval $[s',t']$, the collisions in the MSPD started at $\x^{\theta,1}$ remain the same as in the MSPD started at $\x$.
  
  Then, it is straightforwardly checked that, in the MSPD started at $\x^{\theta,1}$,
  \begin{itemize}
    \item there is a binary collision between $c_1$ and $c_2$ at the space-time point $(\xi^*, t^*(\x))$,
    \item there is a collision between $c_3, \ldots, c_r$ at the space-time point $(\xi^*+\theta, t^*(\x))$,
    \item if $\tau_{i,j}$ refers to the instant of collision between the clusters $c_i$ and $c_j$, then
    \begin{equation*}
      \forall j \in \{3, \ldots, r\}, \qquad t^*(\x) < \tau_{1,j} < \tau_{2,j}.
    \end{equation*}
  \end{itemize}
  More precisely, the boundedness of the velocities yields
  \begin{equation*}
    \tau_{1,j} \geq t^*(\x) + \frac{\theta}{2\ConstBound{1}},
  \end{equation*}
  while Assumption~\eqref{ass:USH} yields
  \begin{equation*}
    \tau_{2,j} \leq t^*(\x) + \frac{\theta}{\ConstUSH}.
  \end{equation*}
  
  Let us now define the configuration $\x^{\theta,2}$ by, for all $\gamma:k \in \Part$,
  \begin{equation*}
    (x^{\theta,2})_k^{\gamma} := \left\{\begin{aligned}
      & (x^{\theta,1})_k^{\gamma} & \text{if $\gamma:k \not\in c_4 \cup \cdots \cup c_r$},\\
      & (x^{\theta,1})_k^{\gamma} + \theta \left(\frac{2\ConstBound{1}}{\ConstUSH}-1\right) & \text{if $\gamma:k \in c_4 \cup \cdots \cup c_r$}.
    \end{aligned}\right.
  \end{equation*}
  Then, the same arguments as above ensure that, for $\theta$ small enough, in the MSPD started at $\x^{\theta,2}$,
  \begin{itemize}
    \item there is a binary collision between $c_1$ and $c_2$ at time $t^*(\x)$,
    \item there are binary collisions between $c_3$ and $c_1$, then between $c_3$ and $c_2$, at respective times $\tau_{1,3}$ and $\tau_{2,3}$ such that
    \begin{equation*}
      t^*(\x) < \tau_{1,3} < \tau_{2,3} \leq t^*(\x) + \frac{\theta}{\ConstUSH},
    \end{equation*}
    \item all the collisions between clusters $c_1, c_2, c_3$ on the one hand and $c_4, \ldots, c_r$ on the other hand occur after the time $t^*(\x) + \theta/\ConstUSH$.
  \end{itemize}
  
  Iterating the argument, we finally construct a configuration $\x^{\theta, r-2}$ such that
  \begin{equation*}
    ||\x-\x^{\theta,r-2}||_1 \leq C\theta
  \end{equation*}
  for some constant $C$ depending only on $\ConstBound{1}$, $\ConstUSH$, $n$ and $d$, and, for $\theta$ small enough, in the MSPD started at $\x^{\theta, r-2}$, if $\tau_{i,j}$ refers to the instant of collision between $c_i$ and $c_j$, then, for all $j \in \{3, \ldots, r\}$,
  \begin{equation*}
    \tau_{j-2,j-1} \leq \tau_{1,j} < \tau_{2,j} < \cdots < \tau_{j-1,j}.
  \end{equation*}
  We complete Step~2 by applying the argument to blow-up all the non-binary collisions, and finally take $\theta$ small enough for the resulting configuration $\y_0$ to be such that $||\x-\y_0||_1 \leq \epsilon/2$.

  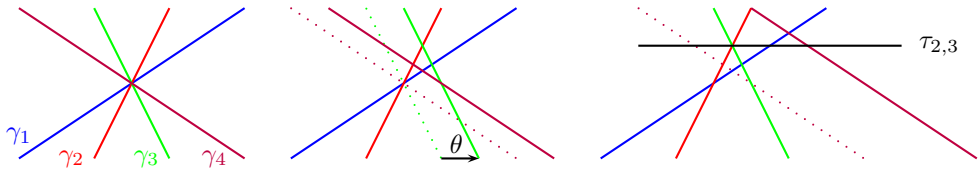
\begin{figure}[ht]
    \begin{pspicture}(3,2)
      \psline[linecolor=blue](0,0)(3,2)
      \psline[linecolor=red](1,0)(2,2)
      \psline[linecolor=green](2,0)(1,2)
      \psline[linecolor=purple](3,0)(0,2)
      \rput(0,.3){\textcolor{blue}{$\gamma_1$}}
      \rput(.7,0){\textcolor{red}{$\gamma_2$}}
      \rput(1.7,0){\textcolor{green}{$\gamma_3$}}
      \rput(2.6,0){\textcolor{purple}{$\gamma_4$}}
    \end{pspicture}
    \hskip 5mm
    \begin{pspicture}(3.5,2)
      \psline[linecolor=blue](0,0)(3,2)
      \psline[linecolor=red](1,0)(2,2)
      \psline[linecolor=green,linestyle=dotted](2,0)(1,2)
      \psline[linecolor=purple,linestyle=dotted](3,0)(0,2)
      \psline[linecolor=green](2.5,0)(1.5,2)
      \psline[linecolor=purple](3.5,0)(.5,2)
      \psline[linecolor=black]{->}(2,0)(2.5,0)
      \rput(2.2,.2){\textcolor{black}{$\theta$}}
    \end{pspicture}
    \hskip 5mm
    \begin{pspicture}(5,2)
      \psline[linecolor=blue](0,0)(3,2)
      \psline[linecolor=red](1,0)(2,2)
      \psline[linecolor=green](2.5,0)(1.5,2)
      \psline[linecolor=purple,linestyle=dotted](3.5,0)(.5,2)
      \psline[linecolor=purple](5,0)(2,2)
      \psline[linecolor=black](.5,1.5)(4,1.5)
      \rput(4.5,1.5){\textcolor{black}{$\tau_{2,3}$}}
    \end{pspicture}
    \caption{Blowing up the non-binary collisions: the left-hand figure represents a collision involving four clusters $\gamma_1, \ldots, \gamma_4$. In the central figure, the clusters $\gamma_3$ and $\gamma_4$ are shifted on the right of a distance $\theta$. In the right-hand figure, the cluster $\gamma_4$ is shifted on the right in order to ensure that its first collision with one of the three other clusters occurs after $\tau_{2,3}$, therefore after all the collisions between the clusters $\gamma_1$, $\gamma_2$ and $\gamma_3$. The minimal shift distance remains proportional to $\theta$.}
    \label{fig:pf:gooddense:3}
  \end{figure}
  
  \sk
  \noindent {\em Step~3.} We begin by noting that, for all $\eta \in (0,\epsilon/2]$, for all $\y \in \barB_1(\y_0,\eta)$, $||\x-\y||_1 \leq \epsilon_0$, therefore $\y$ satisfies all the conditions stated in the introduction of the proof. In particular, in the MSPD started at $\y$, there is no self-interaction on the time interval $[s',t'']$, while all the collisions occuring on the time interval $[0,t'']$ actually occur on the time interval $(s',t')$, and they involve the same pairs of clusters than in the MSPD started at $\y_0$. By Step~2, it is known that the corresponding collisions are binary in the MSPD started at $\y_0$. Let $R$ refer to the subset of $\Rb(\y_0)$ defined by
  \begin{equation*}
    R := \{(\alpha:i,\beta:j) \in \Rb(\y_0) : \tinter_{\alpha:i,\beta:j}(\y_0) \in (s',t')\} = \Rb(\y_0) \setminus \Rb(\Phi(\y_0;t')).
  \end{equation*}
  By Proposition~\ref{prop:continuity}, one can construct $\eta \in (0,\epsilon/2]$ such that, for all $\y \in \barB_1(\y_0, \eta)$, for all $(\alpha:i, \beta:j)$ and $(\alpha':i',\beta':j') \in R$, if $\Xiinter_{\alpha:i,\beta:j}(\y_0) \not= \Xiinter_{\alpha':i',\beta':j'}(\y_0)$, then $\Xiinter_{\alpha:i,\beta:j}(\y) \not= \Xiinter_{\alpha':i',\beta':j'}(\y)$. This implies that, in the MSPD started at $\y$, the collisions on $[0,t']$ are binary.

  \sk
  \noindent {\em Step~4.} By Condition~\eqref{cond:ND} and Lemma~\ref{lem:sGNLclu}, there exists $\eta'_1>0$ such that, for all $\y' \in \barB_1(\Phi(\y_0;t'),\eta'_1)$, for all $\gamma:k \in \Part$,
  \begin{equation*}
    \clu_k^{\gamma}(\y';t''-t') = \clu_k^{\gamma}(\y_0;t'').
  \end{equation*}
  Besides, by Lemma~\ref{lem:Drondnd}, there exists $\eta'_2>0$ such that, for all $\y' \in \barB_1(\Phi(\y_0;t'),\eta'_2)$, $\Rb(\y')=\Rb(\Phi(\y_0;t'))$ and $t^*(\y') > t''-t'$, which implies
  \begin{equation}\label{eq:pf:gooddense:etap2}
    ||\Phi(\y';t''-t')-\Phi(\y_0;t'')||_1 \leq ||\y'-\Phi(\y_0;t')||_1,
  \end{equation}
  thanks to Lemma~\ref{lem:contracttPhi}.
  
  For $\y' \in \barB_1(\Phi(\y_0;t'),\eta'_1 \wedge \eta'_2)$ and $\y'' := \Phi(\y';t''-t')$, we are willing to construct $\y$, close to $\y_0$, such that $\Phi(\y;t'') = \y''$. It is therefore necessary to describe how to follow the MSPD flow {\em backward}, and we shall construct a process $(\Psi(\y'';s))_{s \in [0,t''-s']}$ such that
  \begin{equation*}
    \forall s \in [0, t'' - s'], \qquad \Phi(\Psi(\y'';s);t''-s) = \y''.
  \end{equation*}  
  Of course, there is generically not a unique fashion to do so; since clusters containing several particles in $\y''$ could split at any time $s \geq 0$. In order to ensure that $\Psi(\y'';s)$ remains as close as possible to $\Phi(\y_0;t''-s)$, we define the backward dynamics $(\Psi(\y'';s))_{s \in [0,t''-s']}$ so that clusters never split. 
  
  Let us carry this task out by defining the {\em backward frozen dynamics} independently of the setting of the proof. Let $\z \in \Dnd$, and let $c_1, \ldots, c_L$ refer to the partition of $\Part$ into generical clusters such that, for all $\gamma:k \in \Part$, the generical cluster $c_l$ containing $\gamma:k$ is the largest set of particles of type $\gamma$ having the same position as $\gamma:k$ in the configuration $\z$. The generical cluster $c_l$ shall be called the {\em frozen cluster} of the particle $\gamma:k$.
  
  For all $s \geq 0$, we define the process $(\Psi(\z;s))_{s \geq 0}$ as follows. For all $l \in \{1, \ldots, L\}$, the initial velocity of all the particles in the frozen cluster $c_l$ is set to 
  \begin{equation}\label{eq:pf:gooddense:frozen}
    - \frac{1}{|c_l|} \sum_{\gamma:k \in c_l} \tlambda_k^{\gamma}(\z).
  \end{equation}
  Then, frozen clusters travel at constant velocity. When two frozen clusters of the same type collide, they stick together and form a frozen cluster with velocity determined by conservation of mass and momentum. When clusters of different types collide, say at time $s^*$, they remain formed and the new velocity of each cluster is given by~\eqref{eq:pf:gooddense:frozen}, where $\z$ is replaced with $\Psi(\z;s^*)$ instead.
  
  This backward frozen dynamics is generally {\em not} the MSPD with reverse characteristic fields $-\blambda$; since, in the latter dynamics, frozen clusters of a type $\gamma$ such that $-\partial_{\gamma}\lambda^{\gamma}>0$ would instantaneously split. However, it can be interpreted as a variant of the MSPD, where the initial velocity of the particle $\gamma:k$ in the frozen cluster $c_l$ is defined by~\eqref{eq:pf:gooddense:frozen} instead of $-\tlambda_k^{\gamma}(\z)$. This ensures that frozen clusters do not split, and stick together at collisions with frozen clusters of the same type --- which we shall refer to as {\em self-interactions} for the backward frozen dynamics. 
  
  As a consequence, the proof of Proposition~\ref{prop:continuity} can be slightly adapted to yield the following statement: if $\z \in \Dnd$ and $s \geq 0$ are such that, in the backward frozen dynamics started at $\z$, there is no self-interaction on the time interval $[0,s]$, then for all $\epsilon > 0$, there exists $\delta > 0$ such that, for all $\z' \in \barB_1(\z,\delta)$ having the property that the frozen clusters are the same in the configurations $\z$ and $\z'$, we have:
  \begin{itemize}
    \item there is no self-interaction in the backward frozen dynamics started at $\z'$ on the time interval $[0,s]$, which implies that the frozen clusters are the same in the configurations $\Psi(\z;r)$ and $\Psi(\z';r)$ for all $r \in [0,s]$,
    \item the following continuity property holds:
    \begin{equation*}
      \sup_{r \in [0,s]} ||\Psi(\z;r)-\Psi(\z';r)||_1 \leq \epsilon.
    \end{equation*}
  \end{itemize}
  We shall refer to these two points as Property~($*$).
  
  Let us now come back to the construction of $\y$ , close to $\y_0$, and such that $\Phi(\y;t'')=\y''$. Let $\y_0'' := \Phi(\y_0;t'')$. Since there is no self-interaction in the MSPD started at $\y_0$ on the time interval $[s',t'']$, it is straightforwardly checked that, for all $s \in [0,t''-s']$,
  \begin{equation*}
    \Psi(\y_0'';s) = \Phi(\y_0;t''-s).
  \end{equation*}
  Let $\epsilon > 0$ to be precised below. Let $\delta > 0$ associated to $\epsilon$ by Property~($*$), and let us define
  \begin{equation*}
    \eta' := \eta'_1 \wedge \eta'_2 \wedge \delta.
  \end{equation*}
  Let us now fix $\y' \in \barB_1(\Phi(\y_0;t'),\eta')$ and denote $\y'':=\Phi(\y';t''-t')$. Then the fact that $\eta' \leq \eta'_1 \wedge \eta'_2$ implies that the frozen clusters are the same in the configurations $\y''$ and $\y_0''$, and~\eqref{eq:pf:gooddense:etap2} combined with $\eta' \leq \delta'$ yield $\y'' \in \barB_1(\y_0'',\delta)$. As a consequence, Property~($*$) ensures that: one the hand, there is no self-interaction in the backward frozen dynamics started at $\y''$ on the time interval $[0,t''-s']$, therefore
  \begin{equation*}
    \forall s \in [0,t''-s'], \qquad \Phi(\Psi(\y'';s);t''-s) = \y'';
  \end{equation*}
  on the other hand,
  \begin{equation*}
    \sup_{s \in [0,t''-s']} ||\Psi(\y'';s)-\Psi(\y_0'';s)||_1 \leq \epsilon,
  \end{equation*}
  which in particular implies that
  \begin{equation*}
    ||\Psi(\y'';t''-s') - \Phi(\y_0;s')||_1 \leq \epsilon.
  \end{equation*}
  Besides, the frozen clusters are the same in the configurations $\Psi(\y'';t''-s')$ and $\Phi(\y_0;s')$.
  
  Recall the construction of $\eta>0$ carried out in Step~3. To complete the proof, it remains to fix a value of $\epsilon$ ensuring that, for all $\z \in \barB_1(\Phi(\y_0;s'), \epsilon)$ having the same frozen clusters as $\Phi(\y_0;s')$, one can construct a configuration $\y \in \barB_1(\y_0,\eta)$ such that $\Phi(\y;s') = \z$, and apply the result to $\z = \Psi(\y'';t''-s')$. In other words, we now have to take self-interactions into account, which was not the case for the backward frozen dynamics. On the other hand, since $s' < t^*(\y_0)$, we do not have to care about collisions between clusters of different types, therefore the problem can be addressed cluster by cluster. This enables us to use the following trick: for all frozen clusters $c$ in $\Phi(\y_0;s')$, for all $\gamma:k \in c$, let us define
  \begin{equation*}
    y_k^{\gamma} := (y_0)_k^{\gamma} + h_c,
  \end{equation*}
  where 
  \begin{equation*}
    h_c := z_k^{\gamma}-\Phi_k^{\gamma}(\y_0;s')
  \end{equation*}
  does not depend on the choice of $\gamma:k$ in $c$. Then 
  \begin{equation*}
    ||\y-\y_0||_1 \leq \epsilon,
  \end{equation*}
  and by Proposition~\ref{prop:continuity}, $\epsilon$ can be chosen small enough to prevent particles belonging to different frozen clusters in $\Phi(\y_0;s')$ from colliding in the MSPD started at $\y$. Under this condition, it is easily checked that, for all $s \in [0,s']$, for all frozen clusters $c$ in $\Phi(\y_0;s')$,
  \begin{equation*}
    \forall \gamma:k \in c, \qquad \Phi_k^{\gamma}(\y;s) = \Phi_k^{\gamma}(\y_0;s) + h_c.
  \end{equation*}
  In particular, $\Phi(\y;s')=\z$ and we complete the proof by taking $\epsilon \leq \eta$.
\end{proof}


\subsection{Proof of Lemma~\ref{lem:EntropyTSPD}}\label{ss:pfEntropyTSPD} This subsection contains the proof of Lemma~\ref{lem:EntropyTSPD}, which generalises, in an adequate way in view of Lemma~\ref{lem:shock}, the convergence results of the Sticky Particle Dynamics to the entropy solution of the corresponding scalar conservation law of~\cite{bregre, jourdain:sticky}.

\begin{proof}[Proof of Lemma~\ref{lem:EntropyTSPD}]
  We first note that, on account of Assumption~\eqref{ass:C}, Proposition~\ref{prop:kruzkov} ensures that the functions $\tilde{u}^1, \ldots, \tilde{u}^d$ introduced in Lemma~\ref{lem:EntropyTSPD} are well defined and, for all $t \geq 0$, $\tilde{u}^{\gamma}(t, \cdot)$ is a CDF on the real line. On the other hand, using the same arguments as in the proof of Proposition~\ref{prop:tightness}, there is no difficulty in checking that the sequence $(\tilde{\upmu}[\x(n)])_{n \geq 1}$ is tight in $\Ms$. Calling $\tilde{\upmu}_{\infty}$ the limit of a converging subsequence and defining $\tilde{u}^{\gamma}_{\infty}(t,x) := H*(\tilde{\upmu}_{\infty})^{\gamma}_t(x)$, we shall show below that $\tilde{u}^{\gamma}_{\infty}(t,x) = \tilde{u}^{\gamma}(t,x)$, for all $(t,x) \in [0,+\infty) \times \R$. By the same arguments as in Remark~\ref{rk:unicitebmu}, this allows to identify $\tilde{\upmu}_{\infty}$ with $\tilde{\upmu}$ and thereby completes the proof. 
  
  Let us fix $n \geq 1$ and $\gamma \in \{1, \ldots, d\}$. For simplicity, we will denote $x_k(t) := \tPhi^{\gamma}_k[\tblambda(\x(n))](\x(n);t)$. Let us define
  \begin{equation*}
    \tilde{u}^{\gamma}_n(t,x) := H*\tilde{\upmu}^{\gamma}_t[\x(n)](x) = \frac{1}{n} \sum_{k=1}^n \ind{x_k(t) \leq x},
  \end{equation*}
  and recall that, by Lemma~\ref{lem:cvCDF}, $\tilde{u}^{\gamma}_n(t,x)$ converges to $\tilde{u}^{\gamma}_{\infty}(t,x)$ for all $x \in \R$ such that $\Delta_x \tilde{u}^{\gamma}_{\infty}(t,x)=0$. We also fix $c \in [0,1]$, $\varphi \in \Cs^{1,1}_{\mathrm{c}}([0,+\infty)\times\R, \R)$ such that $\varphi(t,x) \geq 0$ for all $(t,x) \in [0,+\infty)\times\R$, and define the quantity
  \begin{equation*}
    \begin{aligned}
      & A_n := \int_{x \in \R} |\tilde{u}^{\gamma}_n(0,x)-c|\varphi(0,x)\dd x\\
      & + \int_{t=0}^{+\infty} \int_{x \in \R} \left(|\tilde{u}^{\gamma}_n(t,x) - c|\partial_t \varphi(t,x) + \sgn(\tilde{u}^{\gamma}_n(t,x)-c) \left(\tilde{\Lambda}^{\gamma}(\tilde{u}^{\gamma}_n(t,x)) - \tilde{\Lambda}^{\gamma}(c)\right) \partial_x \varphi(t,x)\right) \dd x\dd t.
    \end{aligned}
  \end{equation*}
  Since, for all $t \geq 0$, $\tilde{u}^{\gamma}_n(t,x)$ converges $\dd x$-almost everywhere to $\tilde{u}^{\gamma}_{\infty}(t,x)$, and the function $u \mapsto \sgn(u-c)(\tilde{\Lambda}^{\gamma}(u)-\tilde{\Lambda}^{\gamma}(c))$ is continuous, the Dominated Convergence Theorem shows that $A_n$ converges to $A_{\infty}$ defined as $A_n$ but with $\tilde{u}^{\gamma}_{\infty}$ instead of $\tilde{u}^{\gamma}_n$. We will prove below that $\liminf_{n \to +\infty} A_n \geq 0$, which implies that the distribution $\partial_t|\tilde{u}^{\gamma}_{\infty} - c| + \partial_x (\sgn(\tilde{u}^{\gamma}_{\infty} - c)(\tilde{\Lambda}^{\gamma}(\tilde{u}^{\gamma}_{\infty}) - \tilde{\Lambda}^{\gamma}(c)))$ is nonpositive, so that Proposition~\ref{prop:kruzkov} ensures that $\tilde{u}^{\gamma}_{\infty} = \tilde{u}^{\gamma}$.
  
  Let us denote by $k_0$ the unique integer in $\{0, \ldots, n\}$ such that $(k_0-1)/n < c \leq k_0/n$. We also take the convention to define $x_0(t) = -\infty$ and $x_{n+1}(t) = +\infty$. Then, letting
  \begin{equation*}
    \psi(t,x) := \int_{y=x}^{+\infty} \varphi(t,y)\dd y,
  \end{equation*}
  we have
  \begin{equation*}
    \begin{aligned}
      & \int_{x \in \R} |\tilde{u}^{\gamma}_n(0,x)-c|\varphi(0,x)\dd x = \sum_{k=0}^n \int_{x=x_k(0)}^{x_{k+1}(0)} \left|\frac{k}{n} - c\right| \varphi(0,x)\dd x\\
      & \qquad = \sum_{k=0}^{k_0-1} \left(c-\frac{k}{n}\right) \left(\psi(0,x_k(0))-\psi(0,x_{k+1}(0))\right)+ \sum_{k=k_0}^n \left(\frac{k}{n}-c\right) \left(\psi(0,x_k(0))-\psi(0,x_{k+1}(0))\right)\\
      & \qquad = -\frac{1}{n} \sum_{k=1}^{k_0} \psi(0,x_k(0)) + \frac{1}{n} \sum_{k=k_0+1}^n \psi(0,x_k(0)) + 2\left(\frac{k_0}{n}-c\right)\psi(0,x_{k_0}(0));
    \end{aligned} 
  \end{equation*}
  and similarly, for all $t \geq 0$,
  \begin{equation*}
    \begin{aligned}
      & \int_{x \in \R} |\tilde{u}^{\gamma}_n(t,x) - c|\partial_t \varphi(t,x) \dd x\dd t\\
      & \qquad = -\frac{1}{n} \sum_{k=1}^{k_0} \partial_t\psi(t,x_k(t)) + \frac{1}{n} \sum_{k=k_0+1}^n \partial_t\psi(t,x_k(t)) + 2\left(\frac{k_0}{n}-c\right)\partial_t\psi(t,x_{k_0}(t)),
    \end{aligned}
  \end{equation*}
  while
  \begin{equation*}
    \begin{aligned}
      & \int_{x \in \R} \sgn(\tilde{u}^{\gamma}_n(t,x)-c) \left(\tilde{\Lambda}^{\gamma}(\tilde{u}^{\gamma}_n(t,x)) - \tilde{\Lambda}^{\gamma}(c)\right) \partial_x \varphi(t,x) \dd x\dd t\\
      & \qquad = -\sum_{k=1}^{k_0} \left(\tilde{\Lambda}^{\gamma}\left(\frac{k}{n}\right) - \tilde{\Lambda}^{\gamma}\left(\frac{k-1}{n}\right)\right)\partial_x\psi(t,x_k(t))\\
      & \qquad \quad + \sum_{k=k_0+1}^n \left(\tilde{\Lambda}^{\gamma}\left(\frac{k}{n}\right) - \tilde{\Lambda}^{\gamma}\left(\frac{k-1}{n}\right)\right)\partial_x\psi(t,x_k(t))+ 2\left(\tilde{\Lambda}^{\gamma}\left(\frac{k_0}{n}\right) - \tilde{\Lambda}^{\gamma}(c)\right)\partial_x\psi(t,x_{k_0}(t)).
    \end{aligned}
  \end{equation*}
  
  As a consequence, we rewrite
  \begin{equation*}
    A_n = - \frac{1}{n}\sum_{k=1}^{k_0} A_{k,n} + \frac{1}{n}\sum_{k=k_0+1}^n A_{k,n} + \frac{2}{n} A'_{k_0,n},
  \end{equation*}
  where
  \begin{equation*}
    A_{k,n} := \psi(0,x_k(0)) + \int_{t=0}^{+\infty} \left(\partial_t\psi(t,x_k(t)) + n\left(\tilde{\Lambda}^{\gamma}\left(\frac{k}{n}\right) - \tilde{\Lambda}^{\gamma}\left(\frac{k-1}{n}\right)\right)\partial_x\psi(t,x_k(t))\right) \dd t,
  \end{equation*}
  and
  \begin{equation*}
    \begin{aligned}
      A'_{k_0,n} & := (k_0-nc)\psi(0,x_{k_0}(0))\\
      & \quad + \int_{t=0}^{+\infty} \left((k_0-nc)\partial_t\psi(t,x_{k_0}(t)) + n\left(\tilde{\Lambda}^{\gamma}\left(\frac{k_0}{n}\right) - \tilde{\Lambda}^{\gamma}(c)\right)\partial_x\psi(t,x_{k_0}(t))\right) \dd t.
    \end{aligned}
  \end{equation*}
  On the other hand, we recall that, following~\eqref{eq:vspd}, for all $k \in \{1, \ldots, n\}$, for all $T>0$,
  \begin{equation*}
    \psi(T, x_k(T)) = \psi(0, x_k(0)) + \int_{t=0}^T \left(\partial_t \psi(t,x_k(t)) + v_k(t) \partial_x \psi(t,x_k(t))\right)\dd t,
  \end{equation*}
  where $v_k(t) := v_k[\tlambda^{\gamma}(\x(n))](\rx^{\gamma}(n);t)$ with the notations of Subsection~\ref{ss:mspd}. Taking $T$ large enough for the left-hand side above to vanish for all $k$, we obtain
  \begin{equation*}
    A'_{k_0,n} = \int_{t=0}^{+\infty} \left(n\left(\tilde{\Lambda}^{\gamma}\left(\frac{k_0}{n}\right) - \tilde{\Lambda}^{\gamma}(c)\right) - (k_0-nc)v_{k_0}(t)\right)\partial_x\psi(t,x_{k_0}(t))\dd t.
  \end{equation*}  
  We first note that, by the definition of $\tilde{\Lambda}^{\gamma}$, $v_k(t)$ and since $|k_0-nc| \leq 1$, we have
  \begin{equation*}
    |A'_{k_0,n}| \leq 2 \sup_{\bu \in [0,1]^d} |\lambda^{\gamma}(\bu)| \int_{t=0}^{+\infty} \varphi(t,x_{k_0}(t))\dd t,
  \end{equation*}
  so that
  \begin{equation*}
    \lim_{n \to +\infty} \frac{2}{n} A'_{k_0,n} = 0.
  \end{equation*}
  
  Let us now write 
  \begin{equation*}
    -\frac{1}{n}\sum_{k=1}^{k_0} A_{k,n} + \frac{1}{n}\sum_{k=k_0+1}^n A_{k,n} = \frac{1}{n}\int_{t=0}^{+\infty}\left(-\sum_{k=1}^{k_0} a_{k,n}(t) +\sum_{k=k_0+1}^n a_{k,n}(t)\right)\partial_x \psi(t, x_k(t))\dd t
  \end{equation*}
  with
  \begin{equation*}
    a_{k,n}(t) := n \left(\tilde{\Lambda}^{\gamma}\left(\frac{k}{n}\right)-\tilde{\Lambda}^{\gamma}\left(\frac{k-1}{n}\right)\right) - v_k(t),
  \end{equation*}
  and denote by $\gamma:\uk_0 \cdots \ok_0$ the cluster of the particle $\gamma:k_0$ at time $t \geq 0$. Grouping particles by clusters, we have
  \begin{equation*}
    \begin{aligned}
      & - \sum_{k=1}^{\uk_0-1} a_{k,n}(t) = -n \left(\tilde{\Lambda}^{\gamma}\left(\frac{\uk_0-1}{n}\right)-\tilde{\Lambda}^{\gamma}(0)\right) + \sum_{k=1}^{\uk_0-1} \tilde{\lambda}^{\gamma}_k(\x(n)),\\
      & \sum_{k=\uk_0+1}^n a_{k,n}(t) = n \left(\tilde{\Lambda}^{\gamma}(1)-\tilde{\Lambda}^{\gamma}\left(\frac{\ok_0}{n}\right)\right) - \sum_{k=\uk_0+1}^n \tilde{\lambda}^{\gamma}_k(\x(n)),
    \end{aligned}
  \end{equation*}
  while
  \begin{equation*}
    \begin{aligned}
      - \sum_{k=\uk_0}^{k_0} a_{k,n}(t) + \sum_{k=k_0+1}^{\ok_0} a_{k,n}(t) & = -n\left(\tilde{\Lambda}^{\gamma}\left(\frac{k_0}{n}\right)-\tilde{\Lambda}^{\gamma}\left(\frac{\uk_0-1}{n}\right)\right) + \sum_{k=\uk_0}^{k_0} \tilde{\lambda}^{\gamma}_k(\x(n))\\
      & \quad + n\left(\tilde{\Lambda}^{\gamma}\left(\frac{\ok_0}{n}\right)-\tilde{\Lambda}^{\gamma}\left(\frac{k_0}{n}\right)\right) - \sum_{k=k_0+1}^{\ok_0} \tilde{\lambda}^{\gamma}_k(\x(n))\\
      & \quad + \sum_{k=\uk_0}^{k_0} \left(v_{k_0}(t)-\tilde{\lambda}^{\gamma}_k(\x(n))\right) - \sum_{k=k_0+1}^{\ok_0} \left(v_{k_0}(t) - \tilde{\lambda}^{\gamma}_k(\x(n))\right).
    \end{aligned} 
  \end{equation*}
  The key point of the proof is now that, by the stability condition as is stated in Lemma~\ref{lem:extstab}, we have
  \begin{equation*}
    \sum_{k=\uk_0}^{k_0} \left(v_{k_0}(t)-\tilde{\lambda}^{\gamma}_k(\x(n))\right) \leq 0, \qquad  - \sum_{k=k_0+1}^{\ok_0} \left(v_{k_0}(t) - \tilde{\lambda}^{\gamma}_k(\x(n))\right) \leq 0.
  \end{equation*}
  As a consequence, 
  \begin{equation*}
    \begin{aligned}
      -\sum_{k=1}^{k_0} a_{k,n}(t) +\sum_{k=k_0+1}^n a_{k,n}(t) \leq I_n(t) & := -n\left(\tilde{\Lambda}^{\gamma}\left(\frac{k_0}{n}\right)-\tilde{\Lambda}^{\gamma}(0)\right) + \sum_{k=1}^{k_0} \tilde{\lambda}^{\gamma}_k(\x(n))\\
      & \quad + n\left(\tilde{\Lambda}^{\gamma}(1)-\tilde{\Lambda}^{\gamma}\left(\frac{k_0}{n}\right)\right) + \sum_{k=k_0+1}^n \tilde{\lambda}^{\gamma}_k(\x(n)).
    \end{aligned}
  \end{equation*}
  Recalling the definition of $\tilde{\Lambda}^{\gamma}$ on the one hand, and rewriting
  \begin{equation*}
    \begin{aligned}
      \tilde{\lambda}^{\gamma}_k(\x(n)) & = n \int_{v=(k-1)/n}^{k/n} \lambda^{\gamma}\left(\frac{1}{n} \sum_{k'=1}^n \ind{x^1_{k'}(n) < x^{\gamma}_k(n)}, \ldots, v, \ldots, \frac{1}{n} \sum_{k'=1}^n \ind{x^d_{k'}(n) \leq x^{\gamma}_k(n)}\right)\dd v\\
      & = n \int_{v=(k-1)/n}^{k/n} \lambda^{\gamma}\left(u^1_{n,0}((u^{\gamma}_{n,0})^{-1}(v)^-), \ldots, v, \ldots, u^1_{n,0}((u^{\gamma}_{n,0})^{-1}(v))\right)\dd v
    \end{aligned}
  \end{equation*}
  with $u^{\gamma}_{n,0}$ the empirical CDF of $x^{\gamma}_1(n), \ldots, x^{\gamma}_n(n)$ on the other hand, we get the estimation
  \begin{equation*} 
    \begin{split}
      |I_n(t)| \leq n \int_{v=0}^1 & \left|\lambda^{\gamma}\left(u^1_0((u^{\gamma}_0)^{-1}(v)^-), \ldots, v, \ldots, u^d_0((u^{\gamma}_0)^{-1}(v))\right)\right.\\
      & \left.-\lambda^{\gamma}\left(u^1_{n,0}((u^{\gamma}_{n,0})^{-1}(v)^-), \ldots, v, \ldots, u^d_{n,0}((u^{\gamma}_{n,0})^{-1}(v))\right)\right|\dd v.
    \end{split}
  \end{equation*}
  Combining the assumption~\eqref{ass:star:EntropyTSPD} made in the statement of Lemma~\ref{lem:EntropyTSPD} with the continuity of $\lambda^{\gamma}$, we deduce that the integrand above converges to $0$, $\dd v$-almost everywhere. By the Dominated Convergence Theorem, this implies that, for all $t \geq 0$,
  \begin{equation*}
    \limsup_{n \to +\infty} \frac{1}{n}\left(-\sum_{k=1}^{k_0} a_{k,n}(t) +\sum_{k=k_0+1}^n a_{k,n}(t)\right) \leq \lim_{n \to +\infty} \frac{1}{n} I_n(t) = 0.
  \end{equation*}
  Since $\partial_x \psi(t,x) = -\phi(t,x) \leq 0$, we finally deduce from Fatou's Lemma that
  \begin{equation*}
    \liminf_{n \to +\infty} \left(-\frac{1}{n} \sum_{k=1}^{k_0} A_{k,n} + \frac{1}{n} \sum_{k=k_0+1}^n A_{k,n}\right) \geq 0,
  \end{equation*}
  which, together with the uniform boundedness of $|A'_{k_0,n}|$, completes the proof.
\end{proof}

In the scalar case, the proof is shortened as the function $\tilde{\Lambda}$ is nothing but the antiderivative $\Lambda$ of $\lambda$. As a consequence, $n(\Lambda(k/n)-\Lambda((k-1)/n))$ exactly coincides with the velocity $\tilde{\lambda}_k$ of the $k$-th particle, so that the quantity $I_n(t)$ in the proof above is already $0$ for $n$ fixed. This implies
\begin{equation*}
  - \frac{1}{n} \sum_{k=1}^{k_0} A_{k,n} + \frac{1}{n} \sum_{k=k_0+1}^n A_{k,n} \geq 0,
\end{equation*}
from which we can observe that if $c=k_0/n$, then $A'_{k_0,n}=0$ so that the Kru\v{z}kov entropy inequality is satisfied by the discrete solution $u_n(t,x)$. In particular, taking $c=0$ and $c=1$, we deduce that $u_n$ is a weak solution to the scalar conservation law~\eqref{eq:scalarcl2}. Following~\cite[Lemma~3.3]{boujam99}, if the flux function $\Lambda$ is concave, then $u_n$ actually coincides with the entropy solution to the scalar conservation law with discrete initial datum $u_n(0,x)$.


\section{Index of notations}\label{app:notations}

The following table contains most of the notations used in this article. The last column refers to the number of the subsection or of the paragraph in which the corresponding notation first appears.

\begin{center}
  \begin{pspicture}(15,1)
    \rput(.5,.5){\LARGE\Rightscissors}
    \psline[linestyle=dashed]{-}(1,.5)(15,.5)
  \end{pspicture}
\end{center}

\begin{supertabular}{p{3cm} p{10cm} l}
  \hline
  
  Symbol & Meaning & \\
  
  \hline
  
  $d$ & Size of the system & \ref{ss:hypsys}\\
  $\bu = (u^1, \ldots, u^d)$ & Vector of conserved quantities & \ref{ss:hypsys}\\
  $\blambda = (\lambda^1, \ldots, \lambda^d)$ & Charactersitic fields & \ref{ss:hypsys}\\
  $\bu_0 = (u^1_0, \ldots, u^d_0)$ & Initial data & \ref{ss:hypsys}\\

  $\bm = (m^1, \ldots, m^d)$ & Vector of probability measures corresponding to initial data & \ref{ss:diagsys}\\
  $H*\cdot$ & Convolution with the Heaviside function & \ref{ss:diagsys}\\
  $\Lambda$ & Flux function in the scalar case & \ref{ss:diagsys}\\
  
  $\Ps(E)$ & Set of probability measures on $E$ & \ref{sss:not:prob}\\
  $\Ms$ & Set of probability measures on $\Cs([0,+\infty), \R^d)$ & \ref{sss:not:probC}\\
  $\upmu^{\gamma}_t$ & Marginal distribution of $\upmu \in \Ms$ & \ref{sss:not:probC}\\
  $\pi^{\gamma}_t$ & Projection operator & \ref{sss:not:probC}\\

  \hline
  
  $\ConstBound{p}$ & Boundedness constant on $\blambda$ & \ref{ss:ass}\\
  $\ConstLip$ & Lipschitz constant on $\blambda$ & \ref{ss:ass}\\
  $\ConstUSH$ & Uniform strict hyperbolicity constant on $\blambda$ & \ref{ss:ass}\\
  
  $\gamma:k$ & Generical label of a particle in the MSPD & \ref{ss:mspd:intro}\\
  $\Part$ & Set of indices $\gamma:k$ & \ref{ss:mspd:intro}\\
  $\Dn$ & Configuration space for the Sticky Particle Dynamics & \ref{ss:mspd:intro}\\
  $\Dnd$ & Configuration space for the MSPD & \ref{ss:mspd:intro}\\
  $\x, \y, \z$ & Generical configurations for the MSPD & \ref{ss:mspd:intro}\\
  $\Phi(\x;t)$ & Flow of the MSPD & \ref{ss:mspd:intro}\\
  
  $\Delta F(x)$ & Jump of the CDF $F$ at the point $x$ & \ref{ss:CDF}\\
  $F^{-1}$ & Pseudo-inverse of the CDF $F$ & \ref{ss:CDF}\\
  $\Unif$ & Lebesgue measure on $[0,1]$ & \ref{ss:CDF}\\
  $\Delta u(t,x)$ & Jump of the CDF $u(t,\cdot)$ at the point $x$ & \ref{ss:CDF}\\
  $u(t,\cdot)^{-1}$ & Pseudo-inverse of the CDF $u(t,\cdot)$ & \ref{ss:CDF}\\
  
  $\lambda^{\gamma}\{\bu\}(t,x)$ & Velocity function in the definition of probabilistic solution & \ref{sss:probsol}\\
  $\bvarphi = (\varphi^1, \ldots, \varphi^d)$ & Test function in the definition of probabilistic solution & \ref{sss:probsol}\\
  $\upmu[\x]$ & Empirical distribution of the MSPD started at $\x$ & \ref{sss:pfexist}\\
  $\bu[\x]$ & Vector of empirical CDFs of the MSPD started at $\x$ & \ref{sss:pfexist}\\
  $\x(n)$ & Sequence of initial configurations approximating $\bu_0$ & \ref{sss:pfexist}\\
  $\bar{\upmu}$ & Limit of $\upmu[\x(n)]$ & \ref{sss:pfexist}\\
  \multicolumn{2}{l}{$\bX_v = (X^1_v(t), \ldots, X^d_v(t))_{t \geq 0}$ Trajectories} & \ref{sss:introtraj}\\
  $(\bbX(t))_{t \geq 0}$ & Probabilistic representation of solutions & \ref{sss:introtraj}\\
  
  $||\x-\y||_p$ & Normalised $\Ls^p$ distance on $\Dnd$ & \ref{ss:discrstab}\\
  $\ConstStab_p$ & Stability constant & \ref{ss:discrstab}\\
  $\Ratio$ & $3\ConstLip/\ConstUSH$ & \ref{ss:discrstab}\\
  
  $\Ws_p(m,m')$ & Wasserstein distance between $m$ and $m'$ & \ref{sss:wass}\\
  $\mathfrak{m}$ & Coupling of $m$ and $m'$ & \ref{sss:wass}\\
  $\Ws^{(d)}_p(\bm, \bm')$ & Distance on the Cartesian product $\Ps(\R)^d$ & \ref{sss:wass}\\
  $\chi_n$ & Discretisation operator & \ref{sss:sg}\\
  $(\bar{\bS}_t)_{t \geq 0}$ & Semigroup of solutions & \ref{sss:sg}\\

  \hline
  
  $\rblambda = (\barlambda_1, \ldots, \barlambda_n)$ & Initial velocity vector for the Sticky Particle Dynamics & \ref{sss:spd}\\
  $(\phi[\rblambda](\rx;t))_{t \geq 0}$ & Flow of the Sticky Particle Dynamics & \ref{sss:spd}\\
  $\clu_k[\rblambda](\rx;t)$ & Cluster of the $k$-th particle in the Sticky Particle Dynamics & \ref{sss:spd}\\
  $v_k[\rblambda](\rx;t)$ & Velocity of the $k$-th particle in the Sticky Particle Dynamics & \ref{sss:spd}\\
  $\gamma_k[\rblambda](\rx;t)$ & Coordinates of the reflection term at the boundary of $\Dn$ & \ref{sss:spd}\\
  $D_K$ & Configuration space for the Local Sticky Particle Dynamics & \ref{sss:locspd}\\
  
  $\Rb(\x)$ & Pair of colliding particles in the MSPD & \ref{ss:mspd}\\
  $\Nb(\x)$ & Cardinality of $\Rb(\x)$ & \ref{ss:mspd}\\
  $\omega^{\gamma'}_{\gamma:k}(\x)$ & Rank of $\gamma:k$ in the system of type $\gamma'$ & \ref{ss:mspd}\\
  $\tlambda^{\gamma}_k(\x)$ & Initial velocity of $\gamma:k$ in the MSPD & \ref{ss:mspd}\\
  $\tblambda(\x)$ & Array of initial velocities in the MSPD & \ref{ss:mspd}\\
  
  $\bblambda(\x)$ & Array of initial velocities in the Typewise SPD & \ref{sss:tspd}\\
  $(\tPhi[\bblambda](\x;t))_{t \geq 0}$ & Flow of the Typewise SPD & \ref{sss:tspd}\\
  $\ttinter_{\alpha:i, \beta:j}(\x)$ & Collision time between $\alpha:i$ and $\beta:j$ in the Typewise SPD & \ref{sss:tspd}\\
  $t^*(\x)$ & First collision time in the Typewise SPD & \ref{sss:tspd}\\
  $\x^*$ & Configuration at the first collision time & \ref{sss:tspd}\\
  
  $v^{\gamma}_k(\x;t)$ & Velocity of $\gamma:k$ in the MSPD & \ref{sss:mspd}\\
  $c = \gamma : \uk \cdots \ok$ & Generical cluster & \ref{sss:mspd}\\
  $\type(c)$ & Type of the generical cluster $c$ & \ref{sss:mspd}\\
  $|c|$ & Cardinality of the generical cluster $c$ & \ref{sss:mspd}\\
  $\clu^{\gamma}_k(\x;t)$ & Velocity of $\gamma:k$ in the MSPD & \ref{sss:mspd}\\

  $B_p(\x,\delta)$, $\barB_p(\x,\delta)$ & Open and closed balls in $\Dnd$ & \ref{sss:mspdprop}\\
  
  $\tinter_{\alpha:i, \beta:j}(\x)$ & Collision time between $\alpha:i$ and $\beta:j$ in the MSPD & \ref{sss:mspdcoll}\\

  $\Ttau_{\gamma:k}(\x)$ & Collision times of $\gamma:k$ in the MSPD & \ref{sss:mspdloc}\\
  $T^- \wedge \Ttau_{\gamma:k}(\x)$ & Largest collision times of $\gamma:k$ on $[0,T)$ & \ref{sss:mspdloc}\\
  
  \hline
  
  $\ulambda^{\gamma}$, $\olambda^{\gamma}$ & Bounds on the characteristic field $\lambda^{\gamma}$ & \ref{ss:reptraj}\\
  $\beta$ & A $\Cs^1$ increasing bijection $[0,1] \to [0,1]$ & \ref{ss:veltraj}\\
  
  \hline
  
  $\rho$ & Spreading constant for rarefaction coordinates & \ref{ss:rar}\\
  $\omega_F(\delta)$ & Modulus of continuity of $F$ & \ref{ss:rar}\\
  
  \hline
  
  \multicolumn{2}{l}{$\Xiinter_{\alpha:i,\beta:j}(\x) = (\xiinter_{\alpha:i,\beta:j}(\x), \tinter_{\alpha:i,\beta:j}(\x))$ Space-time point of collision} & \ref{sss:collisions}\\
  $\Ibinter(\x)$ & Set of space-time points of collision & \ref{sss:collisions}\\
  $\clu^{\gamma}_k(\x;t^-)$ & Left limit of a cluster & \ref{sss:collisions}\\
  $\Ibself_{\gamma:k, \gamma:k'}(\x)$ & Set of space-time points of self-interactions & \ref{sss:collisions}\\
  $\Ibself(\x)$ & Set of all space-time points of self-interactions & \ref{sss:collisions}\\
  
  $\Drnd$ & Configurations with no collisions at initial time & \ref{sss:defDrond}\\
  
  $\Good$ & Good configurations & \ref{sss:goodconf}\\
  
  $\sim$ & Equivalence relation on $\Rb(\x)$ & \ref{sss:BinColl}\\
  $\Classe(\x)$ & Equivalence classes, or collisions & \ref{sss:BinColl}\\
  $\Mb(\x)$ & Number of collisions & \ref{sss:BinColl}\\
  $\classe = a \times b$ & Generical collision & \ref{sss:BinColl}\\
  \multicolumn{2}{l}{$\Xi(\x;\classe) = (\xi(\x;\classe), T(\x;\classe))$ Space-time point of collision} & \ref{sss:BinColl}\\
  $\Classe_{\gamma:k}(\x)$ & Ordered set of collisions involving $\gamma:k$ & \ref{sss:BinColl}\\
  $\Tbmax_{\gamma:k}(\x)$ & Last collision time of $\gamma:k$ & \ref{sss:BinColl}\\
  
  $d_{\gamma:k}(t)$ & Distance between $\Phi^{\gamma}_k(\x;t)$ and $\Phi^{\gamma}_k(\y;t)$ & \ref{sss:locstab}\\
  
  $(e_m)_{0 \leq m \leq M}$ & Auxiliary system & \ref{sss:coupling}\\
  $(\totmass_m)_{0 \leq m \leq M}$ & Total mass of the auxiliary system & \ref{sss:coupling}\\
  $\bar{\mu}$ & Forward shift of a function $\mu : \Part \to \{0, \ldots, M\}$ & \ref{sss:coupling}\\
  $\Mnu$ & Space of specific functions $\mu : \Part \to \{0, \ldots, M\}$ & \ref{sss:coupling}\\
  
  $\Gamma^-_m(\gamma:k)$ & Set of type paths & \ref{sss:totmass}\\
  $F(g)$ & Foot of a type path $g$ & \ref{sss:totmass}\\
  $c_{m'}(g)$ & Cluster in a type path & \ref{sss:totmass}\\
  $w^-_m(g)$ & Weight of the type path $g$ & \ref{sss:totmass}\\
  $H_m(g)$ & History of the type path $g$ & \ref{sss:totmass}\\
  
  $\Xi^{\dxi, \dtau}$ & $(\dxi, \dtau)$-box around the space-time point $\Xi$ & \ref{sss:radial}\\
  $\rho^*$ & Shrinking factor allowing to keep Condition~{\rm (\hyperref[cond:C]{LHM})} & \ref{sss:radial}\\
  
  \hline
  
  $\dP_{\bm^*}$ & $\Ws_1$ stability class & \ref{sss:wassfur}\\
  $\tilde{\Ws}_1(m,m')$ & Modified $\Ws_1$ distance & \ref{sss:wassfur}\\
  
  $\bar{\upmu}[\bm]$ & Limit of $\upmu[\chi_n\bm]$ & \ref{sss:St}\\
  $(\bS_t)_{t \geq 0}$ & Semigroup on $\dP_{\bm^*}$ & \ref{sss:St}\\
  
  $\Lambda^{\gamma}$ & Flux function in the Riemann problem & \ref{sss:Riemann}\\
  
  $\bU^{\sharp}_{\bu;\tau,\xi}$ & Solution to the Riemann problem & \ref{sss:BB}\\
  $\bU^{\flat}_{\bu;\tau,\xi}$ & Solution to the linearised problem & \ref{sss:BB}\\
  
  \hline
\end{supertabular}


\subsection*{Acknowledgements} We warmly thank our colleague Régis Monneau (CERMICS), who plainly deserves to sign this paper for his important contribution.



\end{document}